\documentclass[10pt]{article}
\usepackage{amsmath}
\usepackage{amssymb}
\usepackage{amscd}
\usepackage{graphicx}
\usepackage{amsthm}       
\newtheorem{thm}{Theorem}[section]
\newtheorem{prop}[thm]{Proposition}
\newtheorem{lem}[thm]{Lemma}
\newtheorem{cor}[thm]{Corollary}  \theoremstyle{definition}
\newtheorem{df}[thm]{Definition}   \theoremstyle{definition}
\newtheorem{ques}[thm]{Question}
\newtheorem{rem}[thm]{Remark}                \theoremstyle{plain}
 \theoremstyle{definition}
\newtheorem{ex}[thm]{Example}   
 
\def\CC{\Bbb{C}}
\def\RR{\Bbb{R}}  
\def\DD{\Bbb{D}}

\def\CCI{\hat{\CC}}        \def\NN{\Bbb{N}} 
\def\B1{{\rm\kern.32em\vrule    width.12em       height1.4ex
depth-.05ex\kern-.28em 1}}
\def\G{\Gamma}
\def\g{\gamma }
\def\l{\lambda }
\def\ov{\overline}
\def\GN{\Gamma ^{\NN }}

\def\CMX{\text{CM}(Y)}
\def\emCMX{\text{{\em CM}}(Y)}
\def\OCM{\text{OCM}}

\def\OCMX{\text{OCM}(Y)}
\def\emOCMX{\text{{\em OCM}}(Y)}

\def\HMX{\text{HM}(Y)}

\def\NHM{\text{NHM}}

\def\NHMX{\text{NHM}(Y)}

\def\CPn{\Bbb{C} \Bbb{P}^{n}}
\def\MHDt{\text{MHD}(\tau )}
\def\emMHDt{\text{{\em MHD}}(\tau )}
\def\Rat{\text{Rat}}
\def\emRat{\text{{\em Rat}}}
\def\Ratp{\text{Rat}_{+}}
\def\emRatp{\text{{\em Rat}}_{+}}
\def\suppt{\text{supp}\, \tau}

\def\Cpt{\text{Cpt}}
\def\emCpt{\text{\em Cpt}}
\def\Hol{\text{H\"{o}l}}
\def\emHol{\text{{\em H\"{o}l}}}
\def\Min{\text{Min}}
\def\emMin{\text{\em Min}}

\def\LSfc{\text{LS}({\cal U}_{f,\tau }(\hat{\Bbb{C}}))}
\def\emLSfc{\text{{\em LS}}({\cal U}_{f,\tau }(\hat{\Bbb{C}}))}

\def\LSfac{\text{LS}({\cal U}_{f,\tau ,\ast }(\hat{\Bbb{C}}))}
\def\emLSfac{\text{{\em LS}}({\cal U}_{f,\tau ,\ast }(\hat{\Bbb{C}}))}

\def\LSfk{\text{LS}({\cal U}_{f,\tau }(S_{\tau }))}
\def\emLSfk{\text{{\em LS}}({\cal U}_{f,\tau }(S_{\tau }))}

\def\LSfl{\text{LS}({\cal U}_{f,\tau }(L))}
\def\emLSfl{\text{{\em LS}}({\cal U}_{f,\tau }(L))}

\def\Ufc{{\cal U}_{f,\tau }(\hat{\Bbb{C}})}
\def\Uvc{{\cal U}_{v,\tau }(\hat{\Bbb{C}})}
\def\Ufac{{\cal U}_{f,\tau ,\ast }(\hat{\Bbb{C}})}
\def\Uvac{{\cal U}_{v,\tau ,\ast }(\hat{\Bbb{C}})}

\def\Uvk{{\cal U}_{v,\tau }(S_{\tau })}

\def\Uvak{{\cal U}_{v,\tau ,\ast }(S_{\tau })}

\def\Ufl{{\cal U}_{f,\tau }(L)}
\def\Uvl{{\cal U}_{v,\tau }(L)}

\def\Uval{{\cal U}_{v,\tau ,\ast }(L)}

\begin{document}
\title{Random complex dynamics and\\ semigroups of holomorphic maps
\footnote{Published in Proc. London Math. Soc. (2011), 102 (1), 50--112. 
2000 Mathematics Subject Classification. 
37F10, 30D05. Keywords: Random dynamical systems, random complex dynamics, 
random iteration, Markov process, rational semigroups, polynomial semigroups,   
Julia sets, fractal geometry, cooperation principle, noise-induced order.}}

\author{Hiroki Sumi\\  
Department of Mathematics, 
Graduate School of Science, Osaka University\\ 
1-1, Machikaneyama, Toyonaka, Osaka, 560-0043, Japan \\ 
{\bf E-mail: sumi@math.sci.osaka-u.ac.jp}\\ 
http://www.math.sci.osaka-u.ac.jp/\textasciitilde sumi/welcomeou-e.html
\date{May 15, 2010}
}
\maketitle
\begin{abstract}
We investigate the random dynamics of rational maps on the Riemann sphere $\CCI $ and the dynamics 
of semigroups of rational maps on $\CCI .$ 
We show that regarding random complex dynamics of polynomials, 
in most cases, the chaos of the averaged system disappears,  
due to the cooperation of the generators. 
We investigate the iteration and spectral properties of transition operators.
We show that under certain conditions, in the limit stage, 
``singular functions on the complex plane'' appear. 
In particular, we consider the functions $T$ which represent the probability of tending to infinity 
with respect to the random dynamics of polynomials.  
Under certain conditions  
these functions $T$ are  
 complex analogues of the devil's staircase and Lebesgue's singular functions. 
 More precisely, we show that these functions $T$ are continuous on $\CCI $ and vary only on 
the Julia sets of associated semigroups.  
Furthermore, by using ergodic theory and potential theory, 
we investigate the non-differentiability and regularity of these 
functions. 
We find many phenomena which can hold in the random complex dynamics and the dynamics of 
semigroups of rational maps, but cannot hold in the usual iteration 
dynamics of a single holomorphic map. 
We carry out a systematic study of these phenomena and their mechanisms. 
\end{abstract}
\section{Introduction}
In this paper, we investigate the random dynamics of rational maps on the Riemann sphere $\CCI $ and the dynamics 
of rational semigroups (i.e., semigroups of non-constant rational maps 
where the semigroup operation is functional composition) on $\CCI .$ 
We see that the both fields are related to each other very deeply.
In fact, we develop both theories simultaneously.  

One motivation for research in complex dynamical systems is to describe some
 mathematical models on ethology. For 
 example, the behavior of the population 
 of a certain species can be described by the 
 dynamical system associated with iteration of a polynomial 
 $f(z)= az(1-z)$ 
 such that $f$ preserves the unit interval and 
 the postcritical set in the plane is bounded 
 (cf. \cite{D}). However, when there is a change in the natural environment,  
some species have 
 several strategies to survive in nature. 
From this point of view, 
 it is very natural and important not only to consider the dynamics 
of iteration, where the same survival strategy (i.e., function) is repeatedly applied, but also 
to consider random 
 dynamics, where a new strategy might be applied at each time step.  
The first study of random complex dynamics was given by J. E. Fornaess and  N. Sibony (\cite{FS}). 
For research on random complex dynamics of quadratic polynomials, 
see \cite{Br1, Br2, BBR, Bu1, Bu2, GQL}.  
For research on random dynamics of polynomials (of general degrees) 
with bounded planar postcritical set, see the author's works \cite{S11, S10,SdpbpI, SdpbpII, SdpbpIII, Ssugexp}. 
  
The first study of dynamics of rational semigroups was 
conducted by
A. Hinkkanen and G. J. Martin (\cite{HM}),
who were interested in the role of the
dynamics of polynomial semigroups (i.e., semigroups of non-constant polynomial maps) while studying
various one-complex-dimensional
moduli spaces for discrete groups,
and
by F. Ren's group (\cite{GR}), 
 who studied 
such semigroups from the perspective of random dynamical systems.
Since the Julia set $J(G)$ of a finitely generated rational semigroup 
$G=\langle h_{1},\ldots, h_{m}\rangle $ has 
``backward self-similarity,'' i.e.,  
$J(G)=\bigcup _{j=1}^{m}h_{j}^{-1}(J(G))$ (see Lemma~\ref{l:bss} and \cite[Lemma 1.1.4]{S1}),  
the study of the dynamics of rational semigroups can be regarded as the study of  
``backward iterated function systems,'' and also as a generalization of the study of 
self-similar sets in fractal geometry.  
   
For recent work on the dynamics of rational semigroups, 
see the author's papers \cite{S1}--\cite{Ssugexp}, \cite{Skokyu10}, and 
\cite{SS, SU3, SU1, SU2, SU4}. 

 In order to consider the random dynamics of a family of polynomials on $\CCI $, 
 let $T_{\infty }(z)$ be the probability of tending to $\infty \in \CCI $ 
starting with the initial value $z\in \CCI .$ 
In this paper, we see that under certain conditions, 
the function 
$T_{\infty }: \CCI \rightarrow [0,1]$ is continuous on $\CCI $ and has 
some singular properties (for instance, varies only on a thin fractal set, the so-called  
Julia set of a polynomial semigroup), and this function is a complex analogue of 
the devil's staircase (Cantor function) or Lebesgue's singular functions 
(see Example~\ref{ex:dc1}, Figures~\ref{fig:dcjulia}, \ref{fig:dcgraphgrey2}, and \ref{fig:dcgraphudgrey2}). 
Before going into detail, let us recall the definition of 
the devil's staircase (Cantor function) and Lebesgue's singular functions. 
Note that the following definitions look a little bit different from those in \cite{YHK}, 
but it turns out that they are equivalent to those in \cite{YHK}.  
\begin{df}[\cite{YHK}]
Let $\varphi :\RR \rightarrow [0,1] $ be the unique bounded function 
which satisfies the following functional equation:
\begin{equation}\label{dsdef}
\frac{1}{2}\varphi (3x)+\frac{1}{2}\varphi (3x-2)\equiv \varphi (x),\  
\varphi |_{(-\infty ,0]}\equiv 0,\ \varphi |_{[1,+\infty )}\equiv 1.  
\end{equation}
The function $\varphi |_{[0,1]}: [0,1]\rightarrow [0,1]$
is called the {\bf devil's staircase (or Cantor function)}.
\end{df} 
\begin{rem}
The above $\varphi :\RR \rightarrow [0,1]$ is continuous on $\RR $ and varies precisely on 
the Cantor middle third set. Moreover, it is monotone (see Figure~\ref{fig:dsgraph1-1}). 
\end{rem} 
\begin{df}[\cite{YHK}]
Let $0<a<1$ be a constant. 
We denote by $\psi _{a}: \RR \rightarrow [0,1]$ the unique bounded function which 
satisfies the following 
functional equation: 
\begin{equation}\label{Lebesguedef}
a\psi _{a} (2x)+(1-a)\psi _{a} (2x-1)\equiv \psi _{a} (x), \ 
\psi _{a} |_{(-\infty ,0]}\equiv 0,\ \psi _{a} |_{[1,+\infty )}\equiv 1.  
\end{equation}
For each $a\in (0,1)$ with $a\neq 1/2$, the function $L_{a}:=\psi _{a}|_{[0,1]}:[0,1]\rightarrow [0,1]$ is called {\bf Lebesgue's singular function} 
with respect to the parameter $a.$ 
\end{df} 
\begin{rem}
The function $\psi _{a}:\RR \rightarrow [0,1]$ is continuous on $\RR $, monotone on $\RR $, and 
strictly monotone on $[0,1]$. Moreover, if $a\neq 1/2$, then for almost every $x\in [0,1]$ with respect to the one-dimensional Lebesgue 
measure, the derivative of $\psi _{a}$ at $x$ is equal to zero (see Figure~\ref{fig:dsgraph1-1}). 
For the details on the devil's staircase and Lebesgue's singular functions and their related topics, 
see \cite{YHK,HY}.  
\end{rem}
 \begin{figure}[htbp]
\caption{(From left to right) The graphs of the devil's staircase and Lebesgue's singular function.}
\ \ \ \ \ \ \ \ \ \ \ \ \ \ \ \ \ \ \ \ \ \ \ \ \ \ \ \ \ \ \ 
\ \ \ \ \ \ \ \ \ \ \ \ \ \ \ 
\includegraphics[width=2.1cm,height=2.1cm, origin =c, angle =-90]{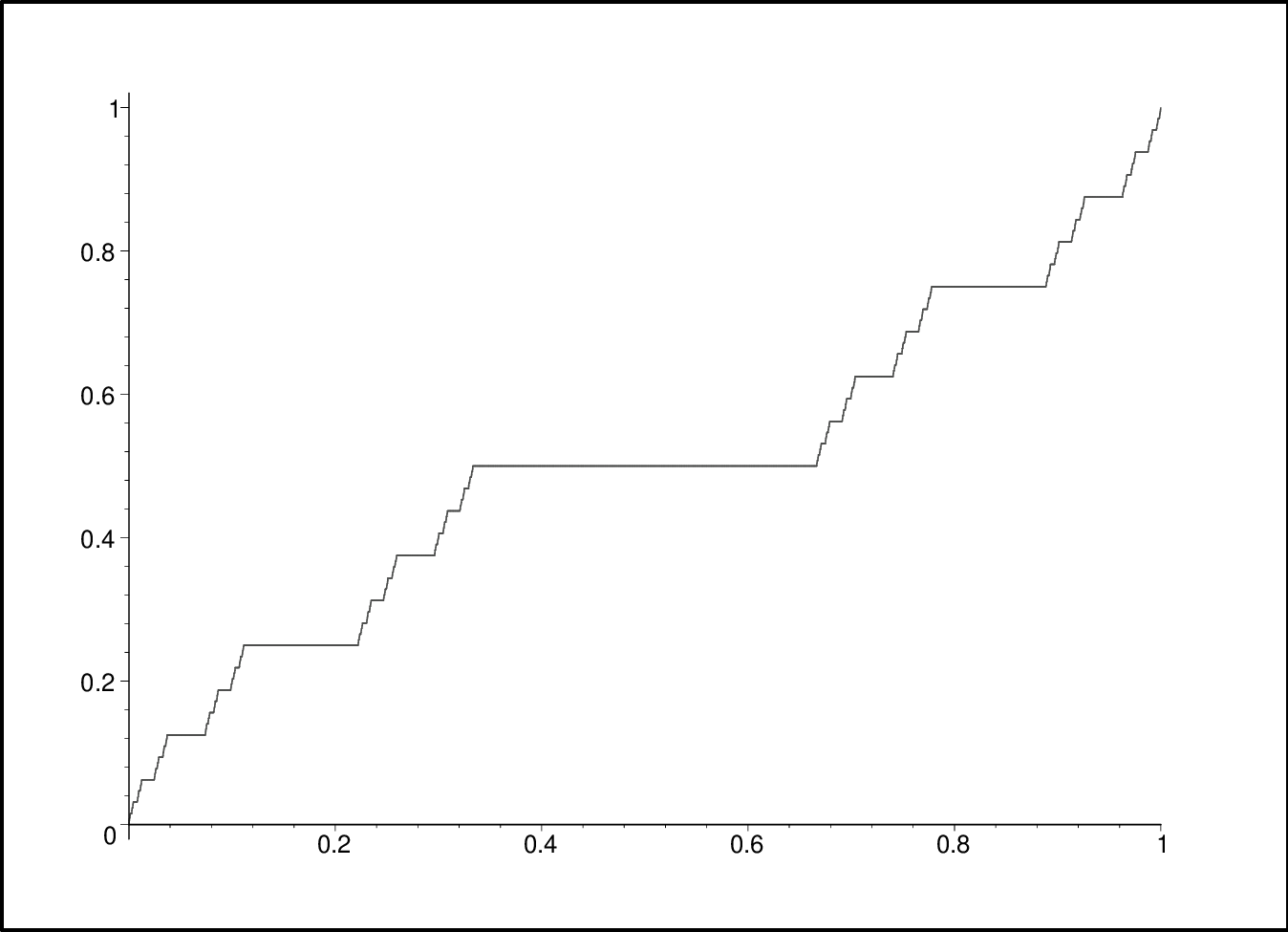}
\label{fig:dsgraph1-1}
\includegraphics[width=2.1cm,height=2.1cm, origin =c, angle =-90]{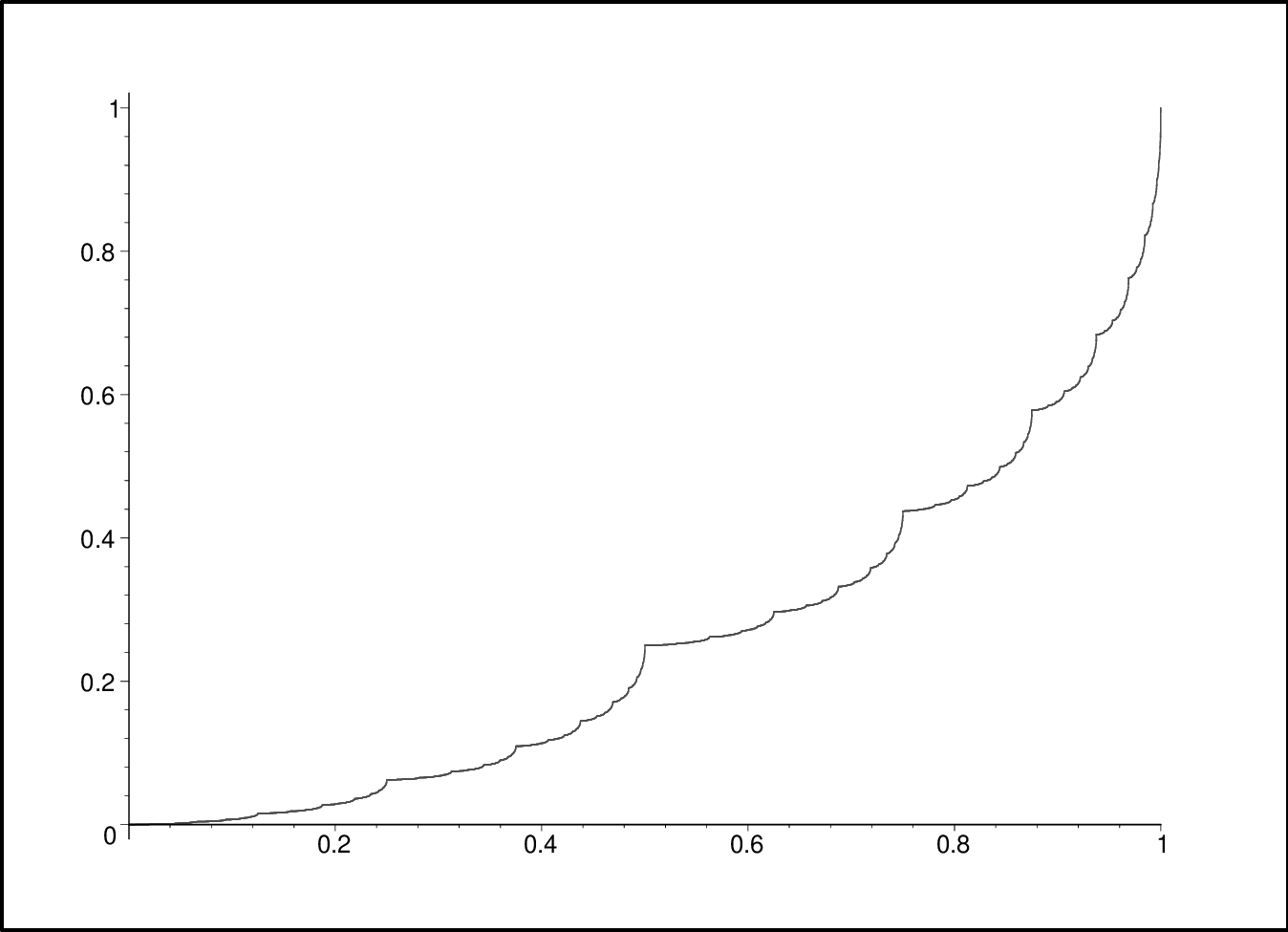}\label{fig:lebfcn}
\end{figure}
These singular functions defined on $[0,1]$ can be redefined by using random dynamical systems 
on $\RR $ as follows. 
Let $f_{1}(x):= 3x, f_{2}(x):=3(x-1)+1\ (x\in \RR )$ and 
we consider the random dynamical system (random walk) 
on $\RR $ such that at every step we choose $f_{1}$ with probability $1/2 $ and $f_{2}$ with probability $1/2.$ 
We set $\hat{\RR }:= \RR \cup \{ \pm \infty \} .$ We denote by $T_{+\infty }(x)$ the 
probability of tending to $+\infty \in \hat{\RR }$ starting with the initial value $x\in \RR .$   
Then, we can see that the function $T_{+\infty }|_{[0,1]}:[0,1]\rightarrow [0,1]$ is equal to the devil's staircase. 

 Similarly, let $g_{1}(x):= 2x, g_{2}(x):= 2(x-1)+1\ (x\in \RR )$ and let 
 $0<a<1$ be a constant. We consider the random dynamical system on $\RR $ such that 
 at every step we choose the map $g_{1}$ with probability $a$ and the map $g_{2}$ with 
probability $1-a.$ Let $T_{+\infty ,a}(x)$ be the probability of tending to 
$+\infty $ starting with the initial value $x\in \RR .$ 
Then, we can see that the function $T_{+\infty ,a}|_{[0,1]}: [0,1]\rightarrow [0,1]$ is 
equal to Lebesgue's singular function $L_{a}$ with respect to the parameter $a.$ 
   
We remark that in most of the literature, the theory of random dynamical systems has not been used directly to investigate 
these singular functions on the interval, although some researchers have used 
it implicitly. 

 One of the main purposes of this paper is to consider the complex analogue of the above story. 
In order to do that, we have to investigate the independent and identically-distributed (abbreviated by i.i.d.) 
random dynamics of rational maps and 
the dynamics of semigroups of rational maps on $\CCI $ simultaneously. 
We develop both the theory of random dynamics of rational maps and that of 
the dynamics of semigroups of rational maps. The author thinks this is the best strategy  
since when we want to investigate 
one of them, we need to investigate the other. 
  
To introduce the main idea of this paper, 
we let $G$ be a rational semigroup and denote by $F(G)$ the Fatou set of $G$, which is defined to be  
the maximal open subset of $\CCI $ where $G$ is equicontinuous with respect to the spherical distance on $\CCI $.    
We call $J(G):=\CCI \setminus F(G)$ the Julia set of $G.$  
The Julia set is backward invariant under each element $h\in G$, but 
might not be forward invariant. This is a difficulty of the theory of rational semigroups. 
Nevertheless, we ``utilize'' this as follows.  
The key to investigating random complex dynamics is to consider the 
following {\bf kernel Julia set} of $G$, which is defined by 
$J_{\ker }(G)=\bigcap _{g\in G}g^{-1}(J(G)).$ This is the largest forward 
invariant subset of $J(G)$ under the action of $G.$ Note that 
if $G$ is a group or if $G$ is a commutative semigroup, 
then $J_{\ker }(G)=J(G).$ 
However, for a general rational semigroup $G$ generated by a family of 
rational maps $h$ with $\deg (h)\geq 2$, it may happen that 
$\emptyset =J_{\ker }(G)\neq J(G) $ (see subsection~\ref{Conjkeremp}, section~\ref{Examples}).  

Let Rat be the space of all non-constant rational maps on the Riemann sphere $\CCI $, 
endowed with the distance $\kappa $ which is defined by 
$\kappa (f,g):=\sup _{z\in \CCI }d(f(z),g(z))$, where $d$ denotes the spherical distance on $\CCI .$  
Let Rat$_{+}$ be the space of all rational maps $g$ with $\deg (g)\geq 2.$ Let 
${\cal P}$ be the space of all polynomial maps $g$ with $\deg (g)\geq 2.$  
Let $\tau $ be a Borel probability measure on Rat with compact support. 
We consider the i.i.d. random dynamics on $\CCI $ such that 
at every step we choose a map $h\in \mbox{Rat}$ according to $\tau .$ 
Thus this determines a time-discrete Markov process with time-homogeneous transition probabilities 
on the phase space 
$\CCI $ such that for each $x\in \CCI $ and 
each Borel measurable subset $A$ of $\CCI $, 
the transition probability 
$p(x,A)$ of the Markov process is defined as $p(x,A)=\tau (\{ g\in \Rat \mid g(x)\in A\} ).$ 
Let $G_{\tau }$ be the 
rational semigroup generated by the support of $\tau .$ 
Let $C(\CCI )$ be the space of all complex-valued continuous functions on $\CCI $ endowed with 
the supremum norm. 
Let $M_{\tau }$ be the operator on $C(\CCI )$ 
defined by $M_{\tau }(\varphi )(z)=\int \varphi (g(z)) d\tau (g).$ 
This $M_{\tau }$ is called the transition operator of the Markov process induced by $\tau .$ 
For a topological space $X$, let ${\frak M}_{1}(X)$ be the space of all 
Borel probability measures on $X$ endowed with the topology 
induced by the weak convergence (thus $\mu _{n}\rightarrow \mu $ in ${\frak M}_{1}(X)$ if and only if 
$\int \varphi d\mu _{n}\rightarrow \int \varphi d\mu $ for each bounded continuous function $\varphi :X\rightarrow \RR $). 
Note that if $X$ is a compact metric space, then ${\frak M}_{1}(X)$ is compact and metrizable. 
For each $\tau \in {\frak M}_{1}(X)$, we denote by supp$\, \tau $ the topological support of $\tau .$  
Let ${\frak M}_{1,c}(X)$ be the space of all Borel probability measures $\tau $ on $X$ such that supp$\,\tau $ is 
compact.     
Let $M_{\tau }^{\ast }:{\frak M}_{1}(\CCI )\rightarrow {\frak M}_{1}(\CCI )$ 
be the dual of $M_{\tau }$.  
This $M_{\tau }^{\ast }$ can be regarded as the ``averaged map'' 
on the extension ${\frak M}_{1}(\CCI )$ of $\CCI $ (see Remark~\ref{r:Phi}).  
We define the ``Julia set'' $J_{meas}(\tau )$ of 
the dynamics of $M_{\tau }^{\ast }$ as the set of all elements $\mu \in {\frak M}_{1}(\CCI )$ 
satisfying that for each neighborhood $B$ of $\mu $, $\{ (M_{\tau }^{\ast })^{n}|_{B}:B\rightarrow {\frak M}_{1}(\CCI )\} _{n\in \NN }$ 
is not equicontinuous on $B$ (see Definition~\ref{d:ytau}). 
For each sequence $\gamma =(\gamma _{1}, \gamma _{2},\ldots )\in (\Rat )^{\NN }$, 
we denote by $J_{\gamma }$ the set of non-equicontinuity of the sequence 
$\{ \gamma _{n}\circ \cdots \circ \gamma _{1}\} _{n\in \NN }$ with respect to the spherical distance on $\CCI .$  This $J_{\gamma }$ is called the Julia set of 
$\gamma .$  Let $\tilde{\tau }:=\otimes _{j=1}^{\infty }\tau \in {\frak M}_{1}((\Rat)^{\NN }).$ 
 
We prove the following theorem. 
\begin{thm}[Cooperation Principle I, see Theorem~\ref{kerJthm1} and Proposition~\ref{Jkeremptygenprop1}]
\label{t:thmA}
Let $\tau \in {\frak M}_{1,c}({\emRat}).$ Suppose   
 that $J_{\ker }(G_{\tau })=\emptyset .$ Then  
 $J_{meas}(\tau )=\emptyset .$ Moreover, for $\tilde{\tau }$-a.e. $\gamma \in (\emRat) ^{\NN }$, 
 the $2$-dimensional Lebesgue measure of $J_{\gamma }$ is equal to zero. 
\end{thm}   
This theorem means that if all the maps in the support of $\tau $ cooperate,  
the set of sensitive initial values of the averaged system disappears. Note that 
for any $h\in \mbox{Rat}_{+}$, $J_{meas}(\delta _{h})\neq \emptyset .$ 
Thus the above result deals with a phenomenon which can hold in 
the random complex dynamics but cannot hold in the usual iteration dynamics of 
a single rational map $h$ with $\deg (h)\geq 2.$  

From the above result and some further detailed arguments, we prove the following theorem. 
To state the theorem, for a $\tau \in {\frak M}_{1,c}(\Rat)$, we denote by $U_{\tau }$ 
the space of all finite linear combinations of unitary eigenvectors of $M_{\tau }:C(\CCI )\rightarrow 
C(\CCI )$, where an eigenvector  is said to be unitary if the absolute value of 
the corresponding eigenvalue is equal to one. Moreover, 
we set ${\cal B}_{0,\tau }:= \{ \varphi \in C(\CCI )\mid M_{\tau }^{n}(\varphi )\rightarrow 0\} .$ 
Under the above notations, we have the following.   
\begin{thm}[Cooperation Principle II: Disappearance of Chaos, see Theorem~\ref{t:mtauspec}]
\label{t:thmB}\ \\ 
Let $\tau \in {\frak M}_{1,c}(\emRat).$ Suppose that  
$J_{\ker }(G_{\tau })=\emptyset $ and $J(G_{\tau })\neq \emptyset $.  
Then we have all of the following statements. 
\begin{itemize}
\item[{\em (1)}]
There exists a direct decomposition 
$C(\CCI )= U_{\tau }\oplus {\cal B}_{0,\tau }$.  
Moreover, $\dim _{\CC } U_{\tau }<\infty $ and ${\cal B}_{0,\tau }$ is a closed subspace 
of $C(\CCI ).$   
Moreover, there exists a non-empty $M_{\tau }^{\ast }$-invariant compact subset $A$ of ${\frak M}_{1}(\CCI )$ 
with finite topological dimension such that 
for each $\mu \in {\frak M}_{1}(\CCI )$, 
$d((M_{\tau }^{\ast })^{n}(\mu ), A)\rightarrow 0$ in ${\frak M}_{1}(\CCI )$ 
as $n\rightarrow \infty $. 
Furthermore, each element of $U_{\tau }$ is locally constant on 
$F(G_{\tau })$. 
Therefore each element of $U_{\tau }$ is a continuous function 
on $\CCI $ which varies only on the Julia set $J(G_{\tau }).$  
\item[{\em (2)}]
For each $z\in \CCI $, there exists a Borel subset ${\cal A}_{z}$ 
of $(\emRat)^{\NN }$ with $\tilde{\tau }({\cal A}_{z})=1$ with the following property.
\begin{itemize}
\item For each $\gamma =(\gamma _{1},\gamma _{2},\ldots )\in {\cal A}_{z}$, 
there exists a number $\delta =\delta (z,\gamma )>0$ such that 
$\mbox{diam}(\gamma _{n}\cdots \gamma _{1}(B(z,\delta )))\rightarrow 0$ as 
$n\rightarrow \infty $, where diam denotes the diameter with respect to the 
spherical distance on $\CCI $, and $B(z,\delta )$ denotes the ball with center $z$ and radius 
$\delta .$ 
\end{itemize}
\item[{\em (3)}] 
There exists at least one and at most finitely many minimal sets for $(G_{\tau },\CCI )$, 
where we say that a non-empty compact subset $L$ of $\CCI $ is a minimal set for $(G_{\tau },\CCI )$ 
if $L$ is minimal in 
$\{ C\subset \CCI \mid \emptyset \neq C \mbox{ is compact}, \forall g\in G_{\tau }, g(C)\subset C\} $ 
with respect to inclusion. 
\item[{\em (4)}] 
Let $S_{\tau }$ be the union of minimal sets for $(G_{\tau },\CCI )$. Then 
for each $z\in \CCI $ there exists a Borel subset ${\cal C}_{z}$ of $(\emRat)^{\NN }$ with 
$\tilde{\tau }({\cal C}_{z})=1$ such that for each $\gamma =(\gamma _{1},\gamma _{2},\ldots )\in {\cal C}_{z}$, 
$d(\gamma _{n}\cdots \gamma _{1}(z),S_{\tau })\rightarrow 0$ as $n\rightarrow \infty .$ 
\end{itemize}
\end{thm}
This theorem means that if all the maps in the support of $\tau $ cooperate, the chaos of the averaged system 
disappears. 
Theorem~\ref{t:thmB} describes new phenomena which can hold in random complex dynamics but 
cannot hold in the usual iteration dynamics of a single $h\in \mbox{Rat}_{+}.$ For example, 
for any $h\in \mbox{Rat}_{+}$, if we take a point $z\in J(h)$, where $J(h)$ denotes the Julia set of the semigroup generated by 
$h$, then for any ball $B$ with 
$B\cap J(h)\neq \emptyset $, $h^{n}(B)$ expands as $n\rightarrow \infty $, and 
we have infinitely many minimal sets (periodic cycles) of $h.$ 

In Theorem~\ref{t:mtauspec}, we completely investigate the structure of $U_{\tau }$ and the 
set of unitary eigenvalues of $M_{\tau }$ (Theorem~\ref{t:mtauspec}). 
Using the above result, we show that if $\dim _{\CC } U_{\tau }>1$ and 
int$(J(G_{\tau }))=\emptyset $ where int$(\cdot )$ denotes the set of interior points,  
then $F(G_{\tau })$ has infinitely many connected components (Theorem~\ref{t:mtauspec}-\ref{t:mtauspeccfi}).  
Thus the random complex dynamics can be applied to the theory of dynamics of rational semigroups. 
The key to proving Theorem~\ref{t:thmB} (Theorem~\ref{t:mtauspec})
 is to show that for almost every $\gamma =(\gamma _{1},\gamma _{2},\ldots )\in (\Rat )^{\NN }$ 
with respect to $\tilde{\tau }:=\otimes _{j=1}^{\infty }\tau $ and for each compact set $Q$ contained in a connected component $U$ 
of $F(G_{\tau })$, $\mbox{diam}\gamma _{n}\circ \cdots \circ \gamma _{1}(Q)\rightarrow 0$ as $n\rightarrow \infty .$ 
This is shown by using careful arguments on the hyperbolic metric of each connected component of $F(G_{\tau }).$   
Combining this with the decomposition theorem on ``almost periodic operators'' on Banach spaces from \cite{Ly}, 
we prove Theorem~\ref{t:thmB} (Theorem~\ref{t:mtauspec}).  

Considering these results, we have the following natural question: ``When is the kernel Julia set empty?'' 
Since the kernel Julia set of $G$ is forward invariant under $G$, Montel's theorem implies that 
if $\tau $ is a Borel probability measure on ${\cal P}$ with compact support, and 
if the support of $\tau $ contains an admissible subset of ${\cal P}$ (see Definition~\ref{d:adm}), then 
$J_{\ker }(G_{\tau })=\emptyset $ (Lemma~\ref{l:aigpjke}). 
In particular, if the support of $\tau $ contains an interior point with respect to the topology of ${\cal P}$, 
then $J_{\ker }(G_{\tau })=\emptyset $ (Lemma~\ref{l:ignejke}).  
 From this result, it follows  that 
for any Borel probability measure $\tau $ on ${\cal P}$ with compact support, 
there exists a Borel probability measure $\rho $ with finite support, such that $\rho $ is arbitrarily close to 
$\tau $,  such that the support of $\rho $ is arbitrarily close to the support 
of $\tau $ , and such that $J_{\ker }(G_{\rho })=\emptyset $ (Proposition~\ref{p:v1v2rho}). The above results mean that in a certain sense, 
$J_{\ker }(G_{\tau })=\emptyset $
for most Borel probability measures $\tau $ on 
${\cal P}.$ 
Summarizing these results we can state the following. 
\begin{thm}[Cooperation Principle III, see Lemmas~\ref{l:ignejke}, \ref{l:aigpjke}, Proposition~\ref{p:v1v2rho}] 
Let ${\frak M}_{1,c}({\cal P})$ be endowed with the topology ${\cal O}$ such that 
$\tau _{n}\rightarrow \tau $ in $({\frak M}_{1,c}({\cal P}), {\cal O})$ 
if and only if {\em (a)} $\int \varphi d\tau _{n}\rightarrow \int \varphi d\tau $ for each 
bounded continuous function $\varphi $ on ${\cal P}$, and {\em (b)} {\em supp}$\, \tau _{n}\rightarrow ${\em supp}$\, \tau $ with respect 
to the Hausdorff metric. We set $A:= \{ \tau \in {\frak M}_{1,c}({\cal P})\mid J_{\ker }(G_{\tau })=\emptyset \} $ 
and $B:=\{ \tau \in {\frak M}_{1,c}({\cal P})\mid J_{\ker }(G_{\tau })=\emptyset , \sharp \mbox{{\em supp}}\, \tau <\infty \} $.  
Then we have all of the following. 
\begin{itemize}
\item[{\em (1)}]
$A$ and $B$ are dense in $({\frak M}_{1,c}({\cal P}),{\cal O})$.  
\item[{\em (2)}] 
If the interior of the support of $\tau $ is not empty with respect to the topology of ${\cal P}$, 
then $\tau \in A.$ 
\item[{\em (3)}] 
For each $\tau \in A$,  
the chaos of the averaged system of the Markov process induced by $\tau $ disappears (more precisely, 
all the statements in Theorems~\ref{t:thmA}, \ref{t:thmB} hold). 
\end{itemize}
\end{thm}
In the subsequent paper \cite{Sprep}, we investigate more detail on the above result (some results of \cite{Sprep} are 
announced in \cite{Skokyu10}).   

We remark that in 1983, by numerical experiments, K. Matsumoto and I. Tsuda (\cite{MT}) observed that 
if we add some uniform noise to the dynamical system associated with iteration of a chaotic map on the unit interval $[0,1]$, 
then under certain conditions, 
the quantities which represent chaos (e.g.,  entropy, Lyapunov exponent, etc.) decrease. More precisely, 
they observed that the entropy decreases and the Lyapunov exponent turns negative. They called this phenomenon 
``noise-induced order'', and many physicists have investigated it by numerical experiments, although there has been 
only a few mathematical supports for it.

Moreover, in this paper, we introduce 
``mean stable'' rational semigroups in subsection~\ref{Almostst}. 
If $G$ is mean stable, then $J_{\ker }(G)=\emptyset $ and a small perturbation 
$H$ of $G$ is still mean stable.    
We show that if $\Gamma $ is a compact subset of Rat$_{+}$ 
and if 
the semigroup $G$ generated by $\Gamma $ is semi-hyperbolic (see Definition~\ref{d:sh}) and $J_{\ker }(G)=\emptyset$,   
then there exists a neighborhood ${\cal V}$ of $\Gamma $ in the space of
non-empty compact subset of Rat such that 
for each $\Gamma '\in {\cal V}$, the semigroup $G'$ generated by 
$\Gamma '$ is mean stable, and $J_{\ker }(G')=\emptyset .$

 By using the above results, we investigate 
the random dynamics of polynomials. Let $\tau $ be a Borel probability 
measure on ${\cal P}$ with compact support.  
Suppose that $J_{\ker }(G_{\tau })=\emptyset $  
and the smallest filled-in Julia set $\hat{K}(G_{\tau })$ (see Definition~\ref{d:sfj}) of $G_{\tau }$ 
is not empty. 
Then we show that the function $T_{\infty ,\tau }$ of probability of tending to $\infty \in \CCI $ 
belongs to $U_{\tau }$ and is not constant (Theorem~\ref{kerJthm2}). 
Thus $T_{\infty ,\tau }$ is non-constant and continuous on $\CCI $ and varies only on $J(G_{\tau }).$ 
Moreover, the function $T_{\infty ,\tau }$ is characterized as the unique 
Borel measurable bounded function $\varphi :\CCI \rightarrow \RR $ which satisfies 
$M_{\tau }(\varphi )=\varphi ,$ $\varphi |_{F_{\infty }(G_{\tau })}\equiv 1$, and 
$\varphi |_{\hat{K}(G_{\tau })}\equiv 0$, where $F_{\infty }(G_{\tau })$ denotes the 
connected component of the Fatou set $F(G_{\tau })$ of $G_{\tau }$ containing $\infty $ 
(Proposition~\ref{p:chtinfty}).  From these results, we can show that $T_{\infty ,\tau }$ 
has a kind of ``monotonicity,'' and applying it, we get information regarding the structure of 
the Julia set $J(G_{\tau })$ of $G_{\tau }$ (Theorem~\ref{kerJthm3}).  
We call the function $T_{\infty ,\tau }$ a {\bf devil's coliseum}, 
especially when int$(J(G_{\tau }))=\emptyset $
 (see Example~\ref{ex:dc1}, Figures~\ref{fig:dcjulia}, \ref{fig:dcgraphgrey2}, 
and \ref{fig:dcgraphudgrey2}). Note that for any $h\in {\cal P}$, 
$T_{\infty ,\delta _{h}}$ is not continuous at any point of $J(h)\neq \emptyset .$ 
Thus the above results deal with a phenomenon which can hold in the random complex dynamics, but 
cannot hold in the usual iteration dynamics of a single polynomial. 

 It is a natural question to ask about the regularity of non-constant $\varphi \in  U_{\tau }$ 
 (e.g., $\varphi =T_{\infty ,\tau }$) on the Julia set $J(G_{\tau }).$ 
For a rational semigroup $G$, we set 
$P(G):=\overline{\bigcup _{h\in G}\{ \mbox{all critical values of } h:\CCI \rightarrow \CCI \}}$, 
where the closure is taken in $\CCI $, and we say that $G$ is {\bf hyperbolic} if $P(G)\subset F(G).$ 
If $G$ is generated by $\{ h_{1},\ldots ,h_{m}\} $ as a semigroup, we write 
$G=\langle h_{1},\ldots ,h_{m}\rangle .$ 
We prove the following theorem.
\begin{thm}[see Theorem~\ref{t:hnondiff} and Theorem~\ref{t:hdiffornd}]
\label{t:thmC}
Let $m\geq 2$ and let $(h_{1},\ldots ,h_{m})\in {\cal P}^{m}$. 
Let $G=\langle h_{1},\ldots ,h_{m}\rangle $. 
Let $0< p_{1},p_{2},\ldots ,p_{m}<1$ with $\sum _{i=1}^{m}p_{i}=1.$ 
Let $\tau =\sum _{i=1}^{m}p_{i}\delta _{h_{i}}.$  
Suppose that $h_{i}^{-1}(J(G))\cap h_{j}^{-1}(J(G))=\emptyset $ for each $(i,j)$ with $i\neq j$ 
and suppose also that 
$G$ is  hyperbolic.  
Then we have all of the following statements. 
\begin{itemize}
\item[{\em (1)}] 
$J_{\ker }(G_{\tau })=\emptyset $, 
int$(J(G_{\tau }))=\emptyset $,  
 and $\dim _{H}(J(G))<2$, where $\dim _{H}$ denotes the 
Hausdorff dimension with respect to the spherical distance on $\CCI .$ 
\item[{\em (2)}] 
Suppose further that at least one of the following conditions {\em (a)(b)(c)} holds. 
\begin{itemize}
\item[{\em (a)}] $\sum _{j=1}^{m}p_{j}\log (p_{j}\deg (h_{j}))>0.$ 
\item[{\em (b)}] $P(G)\setminus \{ \infty \} $ is 
bounded in $\CC $. 
\item[{\em (c)}] $m=2.$
\end{itemize}
Then there exists a non-atomic ``invariant measure'' $\lambda $ on $J(G)$ with supp$\, \lambda =J(G)$ 
and an uncountable dense subset $A$ of $J(G)$ with $\lambda (A)=1$ and $\dim _{H}(A)>0$, 
such that 
for  every $z\in A$ and for each non-constant 
$\varphi \in U_{\tau }$, the pointwise H\"{o}lder exponent 
of $\varphi $ at $z$, which is defined to be 
$$\inf \{\alpha \in \RR \mid \limsup _{y\rightarrow z}\frac{|\varphi (y)-\varphi (z)|}{|y-z|^{\alpha }}=\infty \},$$
 is strictly less than $1$ and $\varphi $ is not differentiable at $z$ (Theorem~\ref{t:hnondiff}).  
\item[{\em (3)}] 
In {\em (2)} above, 
the pointwise H\"{o}lder exponent of $\varphi $ at $z$ can be represented in terms of 
$p_{j}, \log (\deg (h_{j}))$ and the integral of the sum of the values of the Green's function of 
the basin of $\infty $ for the sequence $\gamma =(\gamma _{1},\gamma _{2},\ldots )\in \{ h_{1},\ldots, h_{m}\} ^{\NN }$ 
at the finite critical points of $\gamma _{1}$ (Theorem~\ref{t:hnondiff}). 
\item[{\em (4)}]  
Under the assumption of (2), for almost every point $z\in J(G)$ with respect to the 
$\delta $-dimensional Hausdorff measure $H^{\delta }$ 
where $\delta =\dim _{H}(J(G))$,  the pointwise H\"{o}lder exponent of a non-constant $\varphi \in  U_{\tau }$ at $z$  
can be represented in terms of the $p_{j}$ and the derivatives of $h_{j}$ (Theorem~\ref{t:hdiffornd}). 
\end{itemize}
\end{thm}
Combining Theorems~\ref{t:thmA}, \ref{t:thmB}, \ref{t:thmC}, it follows that 
under the assumptions of Theorem~\ref{t:thmC}, the chaos of the averaged system disappears in the $C^{0}$ ``sense'', 
but it remains in the $C^{1}$ ``sense''. 
From Theorem~\ref{t:thmC}, we also obtain that if $p_{1}$ is small enough, 
then for almost every $z\in J(G)$ with respect to $H^{\delta }$ and for each $\varphi \in U_{\tau }$, 
$\varphi $ is differentiable at $z$ and the derivative of $\varphi $ at $z$ is equal to zero, even though  
a non-constant $\varphi \in U_{\tau }$ is not differentiable at any point 
of an uncountable dense subset of $J(G)$ (Remark~\ref{r:dinondi}).  
 To prove these results, we use Birkhoff's ergodic theorem, potential theory,  
 the Koebe distortion theorem and thermodynamic formalisms in ergodic theory.  
We can construct many examples of $(h_{1},\ldots ,h_{m})\in {\cal P}^{m}$ 
such that $h_{i}^{-1}(J(G))\cap h_{j}^{-1}(J(G))=\emptyset $ for each $(i,j)$ with $i\neq j,$ 
where $G=\langle h_{1},\ldots ,h_{m}\rangle $, $G$ is hyperbolic,  
$\hat{K}(G)\neq \emptyset $, and  
$ U_{\tau }$ possesses non-constant elements (e.g., $T_{\infty ,\tau }$) for any 
$\tau =\sum _{i=1}^{m}p_{i}\delta _{h_{i}}$ (see Proposition~\ref{semihyposcexprop}, Example~\ref{ex:dc1}, 
 Proposition~\ref{p:hyphigh}, Proposition~\ref{p:hypdispur}, and Remark~\ref{r:manydhk}). 

We also investigate the topology of the Julia sets $J_{\gamma }$ of sequences $\gamma \in (\mbox{supp}\, \tau )^{\NN }$, 
where $\tau $ is a Borel probability measure on ${\cal P}$ with compact support. 
We show that if $P(G_{\tau })\setminus \{ \infty \} $ is not bounded in $\CC $, 
then for almost every sequence $\gamma $ with respect to 
$\tilde{\tau }:=\otimes _{j=1}^{\infty }\tau $, 
the Julia set $J_{\gamma }$ of $\gamma $ has uncountably many connected components (Theorem~\ref{Punbddthm1}).  
This generalizes \cite[Theorem 1.5]{Br1} and \cite[Theorem 2.3]{BBR}.  
Moreover, we show that $\hat{K}(G_{\tau })=\emptyset $ if and only if  
$T_{\infty ,\tau }\equiv 1$, and that if $\hat{K}(G_{\tau })=\emptyset $,   
then for almost every $\gamma $ with respect to $\tilde{\tau }$, 
the $2$-dimensional Lebesgue measure of filled-in Julia set $K_{\gamma }$ (see Definition~\ref{d:kg}) 
of $\gamma $ is equal to zero  and $K_{\gamma }=J_{\gamma }$ has uncountably many connected 
components (Theorem~\ref{Tequals1thm1} and Example~\ref{ex:kead}). 
These results generalize \cite[Theorem 2.2]{BBR} and one of the statements of \cite[Theorem 2.4]{Br1}. 
 
Another matter of considerable interest is what happens when $J_{\ker }(G_{\tau })\neq \emptyset .$ 
We show that if $\tau $ is a Borel probability measure on Rat$_{+}$ with compact support and 
$G_{\tau }$ is ``semi-hyperbolic'' (see Definition~\ref{d:sh}), 
then $J_{\ker }(G_{\tau })\neq \emptyset $ if and only if $J_{meas}(\tau )\neq \emptyset $ 
(Theorem~\ref{t:dimjpt0jkn}). We define several types of ``smaller Julia sets'' of $M_{\tau }^{\ast }$. 
We denote by $J_{pt}^{0}(\tau )$ the ``pointwise Julia set'' of $M_{\tau }^{\ast }$  
restricted to $\CCI $ (see Definition~\ref{d:manyFJ}). 
We show that if $G_{\tau }$ is semi-hyperbolic, 
then $\dim _{H}(J_{pt}^{0}(\tau ))<2$ (Theorem~\ref{t:dimjpt0jkn}). 
Moreover, if $J_{\ker }(G_{\tau })\neq \emptyset $, $G_{\tau }$ is semi-hyperbolic,  and $\sharp \mbox{supp}\, \tau <\infty $, 
then $\overline{J_{pt}^{0}(\tau )}=J(G_{\tau })$ (Theorem~\ref{t:dimjpt0jkn}).  
Thus the dual of the transition operator of the Markov process induced by $\tau $ can detect the Julia set of $G_{\tau }.$ 
To prove these results, we utilize some observations concerning  
semi-hyperbolic rational semigroups that may be found in \cite{S4,S7}. 
In particular, the continuity of $\gamma \mapsto J_{\gamma }$ 
is required. (This is non-trivial, and 
does not hold for an arbitrary rational semigroup.)  

Moreover, even when $J_{\ker }(G_{\tau })\neq \emptyset $, it is shown that 
if $J_{\ker }(G_{\tau })$ is included in the unbounded component of 
the complement of the intersection of the set of non-semi-hyperbolic points of $G_{\tau }$ and $J(G_{\tau })$, 
then for almost every $\gamma \in {\cal P}^{\NN }$ with respect to $\tilde{\tau }$, 
the $2$-dimensional Lebesgue measure of the Julia set $J_{\gamma }$ of $\gamma $ is equal to 
zero (Theorem~\ref{t:jkucuh}).  
To prove this result, we again utilize observations concerning the kernel Julia set of $G_{\tau }$, 
and non-constant limit functions must be handled carefully (Lemmas~\ref{Lkerlem}, \ref{l:jkejgcuh} and \ref{l:jkucuhjkj}).  

As pointed out in the previous paragraphs, we find many new phenomena 
which can hold in random complex dynamics and the dynamics of rational semigroups, 
but cannot hold in the usual iteration dynamics of a single rational map. 
These new phenomena and their mechanisms are systematically investigated. 

 In the proofs of all results, we employ the skew product map 
 associated with the support of $\tau $ (Definition~\ref{d:sp}), and some detailed observations  
concerning the skew product are required. It is a new idea to use the kernel Julia set  
of the associated semigroup to investigate random complex dynamics. Moreover, 
it is both natural and new to combine the theory of random complex dynamics and the theory of rational semigroups. 
Without considering the Julia sets of rational semigroups, we are unable to discern the singular properties of 
the non-constant finite linear combinations $\varphi $ (e.g., $\varphi =T_{\infty ,\tau }$, a devil's coliseum) of the unitary eigenvectors of $M_{\tau }$.   
 
 In section~\ref{Pre}, we give some fundamental notations and definitions. 
In section~\ref{Results}, we present the main results of this paper. 
In section~\ref{Tools}, we introduce the basic tools used to prove the main results. 
In section~\ref{Proofs}, we provide the proofs of the main results. 
In section~\ref{Examples}, we give many examples to which the main results are applicable. 

 In the subsequent paper \cite{Sprep}, we investigate the stability and bifurcation of 
 $M_{\tau }$ (some results of \cite{Sprep} are announced in \cite{Skokyu10}).  
\ 

\noindent {\bf Acknowledgment:} The author thanks Rich Stankewitz for valuable comments.  
This work was supported by JSPS Grant-in-Aid for Scientific Research(C) 21540216. 
\section{Preliminaries}
\label{Pre}
In this section, we give some basic definitions and notations on the dynamics of semigroups 
of holomorphic maps and the i.i.d. random dynamics of holomorphic maps.

\ 

\noindent {\bf Notation:} 
Let $(X,d)$ be a metric space, $A$ a subset of $X$, and $r>0$. We set 
$B(A,r):= \{ z\in X\mid d(z,A)<r\} .$ Moreover, 
for a subset $C$ of $\CC $, we set 
$D(C,r):= \{ z\in \CC \mid \inf _{a\in C}|z-a|<r\} .$ 
Moreover, for any topological space $Y$ and for any subset $A$ of $Y$, we denote by int$(A)$ the set of all interior points of $A.$

\begin{df}
Let $Y$ be a metric space. 
We set $\CMX := \{ f:Y\rightarrow Y\mid f \mbox{ is continuous}\} $ 
endowed with the compact-open topology.  
Moreover, we set $\OCMX := \{ f\in \CMX \mid 
f \mbox{ is an open map} \} $ endowed 
with the relative topology from $\CMX .$ 
Furthermore, we set 
$C(Y):= \{ \varphi :Y\rightarrow \CC \mid \varphi \mbox{ is continuous }\} .$ 
When $Y$ is compact, we endow $C(Y)$ with the supremum norm $\| \cdot \| _{\infty }.$ 
Moreover, for a subset ${\cal F}$ of $C(Y)$, we set 
${\cal F}_{nc}:=\{ \varphi \in {\cal F}\mid \varphi \mbox{ is not constant}\} .$ 
\end{df}
\begin{df}
Let $Y$ be a complex manifold. We set 
$\HMX := \{ f: Y\rightarrow Y \mid f \mbox{ is holomorphic}\} $ endowed 
with the compact open topology. Moreover, 
we set $\NHMX := \{ f\in \HMX \mid f \mbox{ is not constant}\} $ endowed with the 
compact open topology.  
\end{df}
\begin{rem}
$\CMX $,  $\OCMX $,  $\HMX $, and $\NHMX $ are semigroups with the semigroup operation being 
functional composition. 
\end{rem}
\begin{df} A {\bf rational semigroup} is a semigroup  
generated by a family of non-constant rational maps on 
the Riemann sphere $\CCI $ with the semigroup operation being 
functional composition(\cite{HM,GR}). A 
{\bf polynomial semigroup } is a 
semigroup generated by a family of non-constant 
polynomial maps. 
We set 
Rat : $=\{ h:\CCI \rightarrow \CCI \mid 
h \mbox { is a non-constant rational map}\} $
endowed with the distance $\kappa $ which is defined 
by $\kappa (f,g):=\sup _{z\in \CCI }d(f(z),g(z))$, where $d$ denotes the 
spherical distance on $\CCI .$   
Moreover, we set 
$\Ratp:=\{ h\in \mbox{Rat}\mid \deg (h)\geq 2\} $ endowed with the 
relative topology from Rat. 
Furthermore, we set 
${\cal P}:= \{ g:\CCI \rightarrow \CCI \mid 
g \mbox{ is a polynomial}, \deg (g)\geq 2\} $
endowed with the relative topology from 
Rat.  

\end{df}
\begin{df}
Let $Y$ be a compact metric space and 
let $G$ be a subsemigroup of $\CMX. $   
The {\bf Fatou set }of $G$ is defined to be  
$F(G):=\\ \{ z\in Y \mid \exists \mbox{ neighborhood } U \mbox{ of }z$  
\mbox{s.t.} $\{ g|_{U}:U\rightarrow Y \} _{g\in G}$ is equicontinuous 
on  $U \} .$ (For the definition of equicontinuity, see \cite{Be}.) 
The {\bf Julia set }of $G$ is defined to be 
 $J(G):= Y \setminus F(G).$ 
If $G$ is generated by $\{ g_{i}\} _{i}$, then 
we write $G=\langle g_{1},g_{2},\ldots \rangle .$
If $G$ is generated by a subset $\G $ of $\CMX$, then we write
 $G=\langle \G \rangle .$  
For finitely many elements $g_{1},\ldots, g_{m}\in \CMX$, 
we set  $F(g_{1},\ldots ,g_{m}):=F(\langle g_{1},\ldots ,g_{m}\rangle )$ 
and $J(g_{1},\ldots ,g_{m}):=J(\langle g_{1},\ldots ,g_{m}\rangle )$.  
For a subset $A$ of $Y$, we set 
$G(A):= \bigcup _{g\in G}g(A)$ and 
$G^{-1}(A):= \bigcup _{g\in G}g^{-1}(A).$ 
We set $G^{\ast }:= G\cup \{ \mbox{Id}\} $, where 
Id denotes the identity map. 
\end{df}
By using the method in \cite{HM,GR}, it is easy to see that the following 
lemma holds. 
\begin{lem}
\label{ocminvlem}
Let $Y$ be a compact metric space and 
let $G$ be a subsemigroup of  $\emOCMX. $ Then 
for each $h\in G$, $h(F(G))\subset F(G)$ and $h^{-1}(J(G))\subset J(G).$ 
Note that the equality does not hold in general. 
\end{lem}

The following is the key to investigating random complex dynamics. 
\begin{df}Let $Y$ be a compact metric space and 
let $G$ be a subsemigroup of $\CMX. $  
We set $J_{\ker }(G):= \bigcap _{g\in G}g^{-1}(J(G)).$ 
This is called the {\bf kernel Julia set} of $G.$  
\end{df}
\begin{rem}
\label{r:kjulia}
Let $Y$ be a compact metric space and 
let $G$ be a subsemigroup of $\CMX. $
(1) $J_{\ker }(G)$ is a compact subset of $J(G).$ 
(2) For each $h\in G$, 
$h(J_{\ker }(G))\subset J_{\ker }(G).$  
(3) If $G$ is a rational semigroup and 
if $F(G)\neq \emptyset $, then 
int$(J_{\ker }(G))=\emptyset .$ 
(4) If $G$ is generated by a single map or if $G$ is a group, then 
$J_{\ker }(G)=J(G).$ However, for a general rational semigroup $G$, 
it may happen that $\emptyset =J_{\ker }(G)\neq J(G)$ (see subsection~\ref{Conjkeremp} and section~\ref{Examples}).  
\end{rem}
The following {\bf postcritical set} is important when we investigate the dynamics of 
rational semigroups. 
\begin{df}
For a rational semigroup $G$,  
let $P(G):= 
\overline{\bigcup _{g\in G}\{ \mbox{all critical values of } 
g:\CCI \rightarrow \CCI \}}$ where the closure is taken in $\CCI .$  
This is called the {\bf postcritical set} of $G.$ 
\end{df} 
\begin{rem}
\label{r:pg}
If $\Gamma \subset \mbox{Rat}$ and $G=\langle \Gamma \rangle $, then 
$P(G)=\overline{G^{\ast }(\bigcup _{h\in \G }\{ \mbox{all critical values of } h\} )}.$ 
From this one may know the figure of $P(G)$, in the finitely generated case, using a 
computer.   
\end{rem}
\begin{df}Let $G$ be a rational semigroup.
Let $N$ be a positive integer. 
We denote by $SH_{N}(G)$ the set of points $z\in \CCI $ 
satisfying that  
there exists a positive number $\delta $  
such that for each $g\in G$, 
$\deg (g:V\rightarrow B(z,\delta ))\leq N$, 
for each connected component $V$ of $g^{-1}(B(z,\delta )).$ 
Moreover, 
we set $UH(G):= \CCI \setminus \bigcup _{N\in \NN }SH_{N}(G).$ 
\end{df}
\begin{df}
\label{d:sh}
Let $G$ be a rational semigroup. We say that 
$G$ is {\bf hyperbolic} if $P(G)\subset F(G).$ 
We say that $G$ is {\bf semi-hyperbolic} if $UH(G)\subset F(G).$ 
\end{df}
\begin{rem}
\label{r:uhsp}
We have $UH(G)\subset P(G).$ If $G$ is hyperbolic, then $G$ is semi-hyperbolic.
\end{rem}
It is sometimes important to investigate the dynamics of sequences of maps. 
\begin{df}
Let $Y$ be a compact metric space. 
For each $\gamma =(\gamma _{1},\gamma _{2},\ldots )\in 
(\CMX)^{\NN }$ and each $m,n\in \NN $ with $m\geq n$, we set
$\gamma _{m,n}=\gamma _{m}\circ \cdots \circ \gamma _{n}$ and we set  
$$F_{\gamma }:= \{ z\in Y \mid 
\exists \mbox{ neighborhood } U \mbox{ of } z \mbox{ s.t. } 
\{ \gamma _{n,1}\} _{n\in \NN } \mbox{ is equicontinuous on }  
U\} $$ and $J_{\gamma }:= Y \setminus F_{\gamma }.$ 
The set $F_{\gamma }$ is called the {\bf Fatou set} of the sequence $\gamma $ and 
the set $J_{\gamma }$ is called the {\bf Julia set} of the sequence $\gamma .$ 
\end{df}
\begin{rem}
Let $Y=\CCI $ and let $\gamma \in (\Ratp) ^{\NN }$. Then by \cite[Theorem 2.8.2]{Be}, $J_{\gamma }\neq \emptyset .$ 
Moreover, if $\Gamma $ is a non-empty compact subset of $\Ratp$ and $\gamma \in \Gamma ^{\NN }$, 
then by \cite{S4}, $J_{\gamma }$ is a perfect set and $J_{\gamma }$ has uncountably many points.  
\end{rem}
We now give some notations on random dynamics. 
\begin{df}
\label{d:d0}
For a topological space $Y$, we denote by 
${\frak M}_{1}(Y)$ the space of all Borel probability measures on  $Y$ endowed 
with the topology such that 
$\mu _{n}\rightarrow \mu $ in ${\frak M}_{1}(Y)$ if and only if 
for each bounded continuous function $\varphi :Y\rightarrow \CC $, 
$\int \varphi \ d\mu _{n}\rightarrow \int \varphi \ d\mu .$ 
 Note that if $Y$ is a compact metric space, then 
${\frak M}_{1}(Y)$ is a compact metric space with the metric 
$d_{0}(\mu _{1},\mu _{2}):=\sum _{j=1}^{\infty }\frac{1}{2^{j}}
\frac{|\int \phi _{j}d\mu _{1}-\int \phi _{j}d\mu _{2}|}{1+|\int \phi _{j}d\mu _{1}-\int \phi _{j}d\mu _{2}|}$, 
where $\{ \phi _{j}\} _{j\in \NN }$ is a dense subset of $C(Y).$  
Moreover, for each $\tau \in {\frak M}_{1}(Y)$, 
we set $\suppt :=\{ z\in Y\mid \forall \mbox{ neighborhood } U \mbox{ of }z,\ 
\tau (U)>0\} .$ Note that $\suppt $ is a closed subset of $Y.$ 
Furthermore, 
we set ${\frak M}_{1,c}(Y):= \{ \tau \in {\frak M}_{1}(Y)\mid \suppt \mbox{ is compact}\} .$ 

For a complex Banach space ${\cal B}$, we denote by ${\cal B}^{\ast }$ the 
space of all continuous complex linear functionals $\rho :{\cal B}\rightarrow \CC $, 
endowed with the weak$^{\ast }$ topology. 
\end{df}
For any $\tau \in {\frak M}_{1}(\CMX)$, we will consider the i.i.d. random dynamics on $Y $ such that 
at every step we choose a map $g\in \CMX $ according to $\tau $ 
(thus this determines a time-discrete Markov process with time-homogeneous transition probabilities 
on the phase space 
$Y$ such that for each $x\in Y$ and 
each Borel measurable subset $A$ of $Y$, 
the transition probability 
$p(x,A)$ of the Markov process is defined as $p(x,A)=\tau (\{ g\in \CMX \mid g(x)\in A\} )$).  
\begin{df} 
\label{d:ytau}
Let $Y$ be a compact metric space.  
Let $\tau \in {\frak M}_{1}(\CMX).$  
\begin{enumerate}
\item  We set $\Gamma _{\tau }:= \mbox{supp}\, \tau $ (thus $\Gamma _{\tau }$ is a 
closed subset of $\CMX $).     
Moreover, we set $X_{\tau }:= (\Gamma _{\tau })^{\NN }$ $
 (=\{ \gamma  =(\gamma  _{1},\gamma  _{2},\ldots )\mid \gamma  _{j}\in \Gamma _{\tau }\ (\forall j)\} )$ endowed with the product topology.  
Furthermore, we set $\tilde{\tau }:= \otimes\displaystyle _{j=1}^{\infty }\tau .$ 
This is the unique Borel probability measure 
on $X_{\tau }$ such that for each cylinder set 
$A=A_{1}\times \cdots \times A_{n}\times \Gamma _{\tau }\times  \Gamma _{\tau }\times \cdots $ in 
$X_{\tau }$, $\tilde{\tau }(A)=\prod _{j=1}^{n}\tau (A_{j}).$ 
 We denote by $G_{\tau }$ the subsemigroup of $\CMX$ generated by 
the subset $\Gamma _{\tau }$ of $\CMX .$   
\item 
Let $M_{\tau }$ be the operator 
on $C(Y ) $ defined by $M_{\tau }(\varphi )(z):=\int _{\Gamma _{\tau }}\varphi (g(z))\ d\tau (g).$ 
$M_{\tau }$ is called the {\bf transition operator} of the Markov process induced by $\tau .$ 
Moreover, let $M_{\tau }^{\ast }: C(Y)^{\ast }
\rightarrow C(Y)^{\ast }$ be the dual of $M_{\tau }$, which is defined as 
$M_{\tau }^{\ast }(\mu )(\varphi )=\mu (M_{\tau }(\varphi ))$ for each 
$\mu \in C(Y)^{\ast }$ and each $\varphi \in C(Y).$ 
Remark: we have $M_{\tau }^{\ast }({\frak M}_{1}(Y))\subset {\frak M}_{1}(Y)$ and  
for each $\mu \in {\frak M}_{1}(Y )$ and each open subset $V$ of $Y $, 
we have $M_{\tau }^{\ast }(\mu )(V)=\int _{\Gamma _{\tau }}\mu (g^{-1}(V))\ d\tau (g).$ 
\item  
We denote by $F_{meas }(\tau )$ the 
set of $\mu \in {\frak M}_{1}(Y )$
satisfying that there exists a neighborhood $B$ 
of $\mu $ in ${\frak M}_{1}(Y )$ such that 
the sequence  $\{ (M_{\tau }^{\ast })^{n}|_{B}: 
B\rightarrow {\frak M}_{1}(Y ) \} _{n\in \NN }$
is equicontinuous on $B.$
We set $J_{meas}(\tau ):= {\frak M}_{1}(Y )\setminus 
F_{meas}(\tau ).$
\item We denote by $F_{meas}^{0}(\tau )$ the 
set of $\mu \in {\frak M}_{1}(Y)$ satisfying that 
the sequence$\{ (M_{\tau }^{\ast })^{n}:
{\frak M}_{1}(Y )\rightarrow {\frak M}_{1}(Y )\} _{n\in \NN }$ is 
equicontinuous at the one point $\mu .$ 
We set $J_{meas}^{0}(\tau ):= {\frak M}_{1}(Y )\setminus F_{meas}^{0}(\tau ).$ 
\end{enumerate}
\end{df}
\begin{rem}
We have $F_{meas}(\tau )\subset F_{meas}^{0}(\tau )$ and $J_{meas}^{0}(\tau )\subset J_{meas}(\tau ).$ 
\end{rem}
\begin{rem}
\label{r:ggt}
Let $\Gamma $ be a closed subset of Rat. Then there exists a 
$\tau \in {\frak M}_{1}(\mbox{Rat})$ such that $\Gamma _{\tau }=\Gamma .$ 
By using this fact, we sometimes apply the results on random complex dynamics  
to the study of the dynamics of rational semigroups. 
\end{rem}
\begin{df}
\label{d:Phi}
Let $Y$ be a compact metric space. 
Let $\Phi :Y \rightarrow {\frak M}_{1}(Y )$ be the topological embedding 
defined by: $\Phi (z):=\delta _{z}$, where $\delta _{z}$ denotes the 
Dirac measure at $z.$ Using this topological embedding $\Phi :Y \rightarrow {\frak M}_{1}(Y )$, 
we regard $Y $ as a compact subset of ${\frak M}_{1}(Y ).$ 
\end{df}
\begin{rem}
\label{r:Phi}
If $h\in \CMX $ and $\tau =\delta _{h}$, then 
we have $M_{\tau }^{\ast }\circ \Phi = \Phi \circ h$ on $Y.$ 
Moreover, for a general $\tau \in {\frak M}_{1}(\CMX)$, 
$M_{\tau }^{\ast }(\mu )=\int h_{\ast }(\mu )d\tau (h)$ for each 
$\mu \in {\frak M}_{1}(Y).$ 
Therefore, for a general $\tau \in {\frak M}_{1}(\CMX)$, 
the map $M_{\tau }^{\ast }:{\frak M}_{1}(Y)\rightarrow {\frak M}_{1}(Y)$ 
can be regarded as the ``averaged map'' on the extension ${\frak M}_{1}(Y)$ of 
$Y.$ 
\end{rem}

\begin{rem}
If $\tau =\delta _{h}\in {\frak M}_{1}(\Ratp)$ with $h\in  \Ratp $, then 
$J_{meas}(\tau )\neq \emptyset $. In fact, 
using the embedding $\Phi :\CCI \rightarrow {\frak M}_{1}(\CCI )$, 
we have $\emptyset \neq \Phi (J(h))\subset J_{meas}(\tau ).$     
\end{rem}
The following is an important and interesting object in random dynamics. 
\begin{df}
Let $Y$ be a compact metric space and let $A$ be a subset of $Y.$ 
Let $\tau \in {\frak M}_{1}(\CMX).$ For each $z\in Y$,  
we set 
$T_{A,\tau }(z):= \tilde{\tau }(\{ \gamma =(\gamma _{1},\gamma _{2},\ldots )\in X_{\tau }\mid 
d(\gamma _{n,1}(z),A)\rightarrow 0 \mbox{ as } n\rightarrow \infty \}).$ 
This is the probability of tending to $A$ starting with the initial value $z\in Y.$    
For any $a\in Y$,  we set $T_{a,\tau }:=T_{\{ a\} ,\tau }.$ 
\end{df}
\section{Results}
\label{Results} 
In this section, we present the main results of this paper.  
\subsection{General results and properties of $M_{\tau }$}
\label{Genres} 

In this subsection, we present some general results and 
some results on properties of the iteration of $M_{\tau }:C(\CCI )\rightarrow C(\CCI )$ and 
$M_{\tau }^{\ast }:C(\CCI )^{\ast }\rightarrow C(\CCI )^{\ast }.$ 
The proofs are given in subsection~\ref{pfgenres}. 
We need some notations. 

\begin{df}
Let $Y$ be a $n$-dimensional smooth manifold.  
We denote by Leb$_{n}$ the two-dimensional Lebesgue measure on $Y .$ 
\end{df}
\begin{df}
Let ${\cal B}$ be a complex vector space and let $M:{\cal B}\rightarrow {\cal B}$ be a linear operator.  
Let $\varphi \in {\cal B}$ and $a\in \CC $ be such that 
$\varphi \neq 0, |a|=1$, and $M(\varphi )=a\varphi .$ Then we say that  
$\varphi $ is a unitary eigenvector of $M$ with respect to $a$, 
and we say that $a$ is a unitary eigenvalue.   
\end{df}
\begin{df}
Let $Y$ be a compact metric space and let 
$\tau \in {\frak M}_{1}(\CMX).$ 
Let $K$ be a non-empty subset of $Y$ such that 
$G(K)\subset K$. 
We denote by ${\cal U}_{f,\tau }(K)$ the set of 
all unitary eigenvectors of $M_{\tau }:C(K)\rightarrow C(K)$. Moreover, 
we denote by ${\cal U}_{v,\tau }(K)$ the set of all 
unitary eigenvalues of $M_{\tau }:C(K)\rightarrow C(K).$  
Similarly, we denote by 
${\cal U}_{f,\tau ,\ast}(K)$ the set of all unitary eigenvectors of 
$M_{\tau }^{\ast }:C(K)^{\ast }\rightarrow C(K)^{\ast }$, 
 and we denote by ${\cal U}_{v,\tau ,\ast }(K)$ the set of all 
unitary eigenvalues of $M_{\tau }^{\ast }:C(K)^{\ast }\rightarrow C(K)^{\ast }.$ 
\end{df}
\begin{df}
Let $V$ be a complex vector space and let $A$ be a subset of $V.$ 
We set $\mbox{LS}(A):= \{ \sum _{j=1}^{m}a_{j}v_{j}\mid a_{1},\ldots ,a_{m}\in \CC , 
v_{1},\ldots ,v_{m}\in A, m\in \NN \} .$ 
\end{df}
\begin{df}
Let $Y$ be a topological space and let $V$ be a subset of $Y.$ 
We denote by $C_{V}(Y)$ the space of all $\varphi \in C(Y)$ such that 
for each connected component $U$ of $V$, there exists a constant $c_{U}\in \CC $ with 
$\varphi |_{U}\equiv c_{U}.$ 
\end{df}
\begin{rem}
$C_{V}(Y)$ is a linear subspace of $C(Y).$ Moreover, if 
$Y$ is compact, metrizable, and locally connected and $V$ is an open subset of $Y$, then 
$C_{V}(Y)$ is a closed subspace of $C(Y).$ Furthermore, 
if $Y$ is compact, metrizable, and locally connected, $\tau \in {\frak M}_{1}(\CMX)$, and $G_{\tau }$ is a 
subsemigroup of $\OCMX$, then $M_{\tau }(C_{F(G_{\tau })}(Y))\subset C_{F(G_{\tau })}(Y).$  
\end{rem}
\begin{df}
For a topological space $Y$, we denote by Cpt$(Y)$ the space of all non-empty compact subsets of $Y$. 
If $Y$ is a metric space, we endow Cpt$(Y)$  
with the Hausdorff metric.
\end{df}
\begin{df}
\label{d:minimal}
Let $Y$ be a metric space and let $G$ be a subsemigroup of $\CMX .$ Let $K\in \Cpt(Y).$ 
We say that $K$ is a minimal set for $(G,Y)$ if 
$K$ is minimal among the space 
$\{ L\in \Cpt(Y)\mid G(L)\subset L\} $ with respect to inclusion. 
Moreover, we set $\Min(G,Y):= \{ K\in \Cpt(Y)\mid K \mbox{ is a minimal set for } (G,Y)\} .$  
\end{df}
\begin{rem}
\label{r:minimal}
Let $Y$ be a metric space and let $G$ be a subsemigroup of $\CMX .$
By Zorn's lemma, it is easy to see that 
if $K_{1}\in \Cpt(Y)$ and $G(K_{1})\subset K_{1}$, then 
there exists a $K\in \Min(G,Y)$ with $K\subset K_{1}.$  
Moreover, it is easy to see that 
for each $K\in \Min(G,Y)$ and each $z\in K$, 
$\overline{G(z)}=K.$ In particular, if $K_{1},K_{2}\in \Min(G,Y)$ with $K_{1}\neq K_{2}$, then 
$K_{1}\cap K_{2}=\emptyset .$ Moreover, 
by the formula $\overline{G(z)}=K$, we obtain that 
for each $K\in \Min(G,Y)$, either (1) $\sharp K<\infty $ or (2) $K$ is perfect and 
$\sharp K>\aleph _{0}.$ 
Furthermore, it is easy to see that 
if $\Gamma \in \Cpt(\CMX), G=\langle \Gamma \rangle $, and 
$K\in \Min(G,Y)$, then $K=\bigcup _{h\in \Gamma }h(K).$   
\end{rem}
\begin{df}
Let $Y$ be a compact metric space. 
Let $\rho \in  C(Y)^{\ast }.$  
We denote by $a(\rho )$ the set of points $z\in Y$ 
which satisfies that there exists a neighborhood $U$ of $z$ in $Y$ 
such that for each $\varphi \in C(Y)$ with $\mbox{supp}\, \varphi \subset U$, 
$\rho (\varphi )=0.$ 
We set $\mbox{supp}\, \rho := Y\setminus a(\rho ).$ 
\end{df}

\begin{df}
Let $\{ \varphi _{n}:U\rightarrow \CCI \} _{n=1}^{\infty }$ be a sequence of 
holomorphic maps on an open set $U$ of $\CCI .$ Let 
$\varphi :U\rightarrow \CCI $ be a holomorphic map. 
We say that $\varphi $ is a limit function of 
$\{ \varphi _{n}\} _{n=1}^{\infty }$ if 
there exists a strictly increasing sequence $\{ n_{j}\} _{j=1}^{\infty }$ in 
$\NN $ such that $\varphi _{n_{j}}\rightarrow \varphi $ as $j\rightarrow \infty $ locally 
uniformly on $U.$ 
\end{df}
\begin{df}
For a topological space $Z$, we denote by Con$(Z)$ the set of all 
connected components of $Z.$  
\end{df}
\begin{df}
Let $G$ be a rational semigroup. 
We set $J_{res}(G):= \{ z\in J(G)\mid \forall U\in \mbox{Con}(F(G)), z\not\in \partial U\} .$ 
This is called the {\bf residual Julia set} of $G.$ 
\end{df}
We now present the main results. 
\begin{thm}[Cooperation Principle I]
\label{kerJthm1} 
Let $\tau \in {\frak M}_{1,c}(\text{{\em NHM}}(\CC \Bbb{P}^{n}))$, 
where $\CC \Bbb{P}^{n}$ denotes the $n$-dimensional complex projective space.   
Suppose that   
$J_{\ker}(G_{\tau })=\emptyset .$ 
Then, 
$F_{meas}(\tau )={\frak M}_{1}(\CC \Bbb{P}^{n})$, and 
for $\tilde{\tau }$-a.e. $\gamma  \in (\mbox{{\em NHM}}(\CC \Bbb{P}^{n}))^{\NN }$,   
$\mbox{{\em Leb}}_{2n}(J_{\gamma  })=0.$  
\end{thm}
\begin{thm}[Cooperation Principle II: Disappearance of Chaos]
\label{t:mtauspec}
Let $\tau \in {\frak M}_{1,c}(\emRat)$ and let 
$S_{\tau }:=\bigcup _{L\in \emMin (G_{\tau },\CCI )}L.$  Suppose that  
$J_{\ker }(G_{\tau })=\emptyset $ and $J(G_{\tau })\neq \emptyset $.
Then, all of the following statements 1,$\ldots ,$21 hold. 
\begin{enumerate}
\item \label{t:mtauspec2}
Let ${\cal B}_{0,\tau }:= \{ \varphi \in C(\CCI )\mid M_{\tau }^{n}(\varphi )\rightarrow 0 \mbox{ as }n\rightarrow \infty \} .$ 
Then, ${\cal B}_{0,\tau }$ is a closed subspace of $C(\CCI )$ and there exists a direct sum decomposition 
$C(\CCI )=\mbox{{\em LS}}({\cal U}_{f,\tau }(\CCI ))\oplus {\cal B}_{0,\tau }.$ 
Moreover, $\mbox{{\em LS}}({\cal U}_{f,\tau }(\CCI ))\subset C_{F(G_{\tau })}(\CCI )$ and 
$\dim _{\CC }(\mbox{{\em LS}}({\cal U}_{f,\tau }(\CCI )))<\infty .$ 
\item \label{t:mtauspec2-1}
Let $q:=\dim_{\CC }(\emLSfc ).$ 
Let $\{ \varphi _{j}\} _{j=1}^{q}$ be a basis of 
$\emLSfc$ such that for each $j=1,\ldots ,q$, there exists an $\alpha _{j}\in \Uvc $ with 
$M_{\tau }(\varphi _{j})=\alpha _{j}\varphi _{j}.$ Then, 
there exists a unique family $\{ \rho _{j}:C(\CCI )\rightarrow \CC \} _{j=1}^{q}$ of 
complex linear functionals such that for each $\varphi \in C(\CCI )$, 
$\| M_{\tau }^{n}(\varphi -\sum _{j=1}^{q}\rho _{j}(\varphi )\varphi _{j})\| _{\infty }
\rightarrow 0$ as $n\rightarrow \infty .$ 
Moreover, $\{ \rho _{j}\} _{j=1}^{q}$ satisfies all of the following.
\begin{itemize}
\item[{\em (a)}] For each $j=1,\ldots ,q$, $\rho _{j}:C(\CCI )\rightarrow \CC $ is continuous. 
\item[{\em (b)}] For each $j=1,\ldots ,q$, $M_{\tau }^{\ast }(\rho _{j})=\alpha _{j}\rho _{j}.$ 
\item[{\em (c)}] For each $(i,j)$, $\rho _{i}(\varphi _{j})=\delta _{ij}.$ 
Moreover, $\{ \rho _{j}\} _{j=1}^{q}$ is a basis of $\emLSfac.$  
\item[{\em (d)}] For each $j=1,\ldots ,q$, {\em supp}$\, \rho _{j}\subset S_{\tau }.$ 
\end{itemize} 
\item \label{t:mtauspecj3} 
We have $\sharp J(G_{\tau })\geq 3.$ In particular, 
for each $U\in \mbox{{\em Con}}(F(G_{\tau }))$, 
we can take the hyperbolic metric on $U.$ 

\item \label{t:mtauspec1}
There exists a Borel measurable subset ${\cal A}$ of $(\emRat)^{\NN }$ with $\tilde{\tau }({\cal A})=1$ such that 
\begin{itemize}
\item[{\em (a)}]
for each $\gamma \in {\cal A}$ and for each $U\in \mbox{{\em Con}}(F(G_{\tau }))$, 
each limit function of $\{ \gamma _{n,1}|_{U}\} _{n=1}^{\infty }$ is constant, and 
\item[{\em (b)}]
for each $\gamma \in {\cal A}$ and for each $Q\in \mbox{{\em Cpt}}(F(G_{\tau }))$, 
$\sup _{a\in Q}\| \gamma _{n,1}'(a)\| _{h}\rightarrow 0$ as $n\rightarrow \infty $, 
where $\| \gamma _{n,1}'(a) \| _{h}$ denotes the norm of the derivative 
of $\gamma _{n,1}$ at a point $a$ measured from the hyperbolic metric on the element $U_{0}\in \mbox{{\em Con}}(F(G_{\tau }))$ with 
$a\in U_{0}$ to that on the element $U_{n}\in \mbox{{\em Con}}(F(G_{\tau }))$ with $\gamma _{n,1}(a)\in U_{n}.$ 
\end{itemize} 
\item \label{t:mtauspecaz} For each $z\in \CCI $, there exists a Borel subset ${\cal A}_{z}$ 
of $(\emRat)^{\NN }$ with $\tilde{\tau }({\cal A}_{z})=1$ with the following property.
\begin{itemize}
\item For each $\gamma =(\gamma _{1},\gamma _{2},\ldots )\in {\cal A}_{z}$, 
there exists a number $\delta =\delta (z,\gamma )>0$ such that 
$\mbox{diam}(\gamma _{n}\cdots \gamma _{1}(B(z,\delta )))\rightarrow 0$ as 
$n\rightarrow \infty $, where diam denotes the diameter with respect to the 
spherical distance on $\CCI $, and $B(z,\delta )$ denotes the ball with center $z$ and radius 
$\delta .$ 
\end{itemize}

\item \label{t:mtauspec3}
$\sharp \emMin(G_{\tau },\CCI )<\infty .$
\item \label{t:mtauspec4}
Let $W:= \bigcup _{A\in \text{{\em Con}}(F(G_{\tau })), A\cap S_{\tau }\neq \emptyset }A$. 
Then $S_{\tau }$ is compact. Moreover, for each $z\in \CCI $ there exists a Borel measurable subset 
${\cal C}_{z}$ of $(\emRat)^{\NN }$ with $\tilde{\tau }({\cal C}_{z})=1$ such that 
for each $\gamma \in {\cal C}_{z}$, there exists an $n\in \NN $ with 
$\gamma _{n,1}(z)\in W$ and   
$d(\gamma _{m,1}(z),S_{\tau })\rightarrow 0$ as $m\rightarrow \infty .$ 
\item \label{t:mtauspec5}
Let $L\in \emMin(G_{\tau}, \CCI )$ and 
$r_{L}:= \dim _{\CC }(\emLSfl ).$ Then, 
${\cal U}_{v,\tau }(L)$ is a finite subgroup of $S^{1}$ with 
$\sharp {\cal U}_{v,\tau }(L)=r_{L}.$  
Moreover, there exists an $a_{L}\in S^{1}$   
and a family $\{ \psi _{L,j}\} _{j=1}^{r_{L}}$ in ${\cal U}_{f,\tau }(L)$ 
such that 
\begin{itemize}
\item[{\em (a)}]
\label{t:mtauspec5a}
$a_{L}^{r_{L}}=1$, ${\cal U}_{v,\tau }(L)=\{ a_{L}^{j}\} _{j=1}^{r_{L}}$, 
\item[{\em (b)}]
\label{t:mtauspec5b}
$M_{\tau }(\psi _{L,j})=a_{L}^{j}\psi _{L,j} $ for each $j=1,\ldots ,r_{L}$, 
\item[{\em (c)}] 
\label{t:mtauspec5c}
$\psi _{L,j}=(\psi _{L,1})^{j}$ for each $j=1,\ldots ,r_{L}$, and 
\item[{\em (d)}] 
\label{t:mtauspec5d}
$\{ \psi _{L,j}\} _{j=1}^{r_{L}}$ is a basis of $\mbox{{\em LS}}({\cal U}_{f,\tau }(L)).$ 
\end{itemize}
\item \label{t:mtauspec6}
Let $\Psi _{S_{\tau }}: \mbox{{\em LS}}({\cal U}_{f,\tau }(\CCI ))\rightarrow 
C(S_{\tau })$ be the map defined by $\varphi \mapsto \varphi |_{S_{\tau }}.$ 
Then, $\Psi _{S_{\tau }}(\mbox{{\em LS}}({\cal U}_{f,\tau }(\CCI )))=\emLSfk$ and 
$\Psi _{S_{\tau }}: \emLSfc \rightarrow \emLSfk $ is a linear isomorphism. 
Furthermore, $\Psi _{S_{\tau }}\circ M_{\tau }=M_{\tau }\circ \Psi _{S_{\tau }}$ 
on $\emLSfc .$ 
\item \label{t:mtauspec7}
$\Uvc=\Uvk=\bigcup _{L\in \emMin(G_{\tau },\CCI )}\Uvl=\bigcup _{L\in \emMin(G_{\tau },\CCI )}\{ a_{L}^{j}\} _{j=1}^{r_{L}}$ 
and $\dim _{\CC }(\emLSfc )=\sum _{L\in \emMin(G_{\tau },\CCI )}r_{L}.$ 
\item \label{t:mtauspec7-1}
$\Uvac=\Uvc ,\ \Uvak=\Uvk$,  and $\Uval=\Uvl$ for each $L\in \emMin(G_{\tau },\CCI ).$ 
\item \label{t:mtauspec7-2} 
Let $L\in \emMin(G_{\tau },\CCI ).$ 
Let $\Lambda _{r_{L}}:= \{ g_{1}\circ \cdots \circ g_{r_{L}}\mid 
\forall j, g_{j}\in \Gamma _{\tau }\} .$ Moreover, 
let $G_{\tau }^{r_{L}}:= \langle \Lambda _{r_{L}}\rangle .$ 
Then, $r_{L}=\sharp \emMin(G_{\tau }^{r_{L}},L).$  
\item \label{t:mtauspec8}
There exists a basis $\{ \varphi _{L,i}\mid L\in \emMin(G_{\tau },\CCI ), i=1,\ldots ,r_{L}\} $ 
of $\emLSfc $ and a basis $\{ \rho _{L,i}\mid L\in \emMin(G_{\tau },\CCI ), i=1,\ldots ,r_{L}\} $ 
of $\emLSfac $ such that for each $L\in \emMin(G_{\tau },\CCI )$ and for each 
$i=1,\ldots ,r_{L}$, we have all of the following. 
\begin{itemize}
\item[{\em (a)}]
$M_{\tau }(\varphi _{L,i})=a_{L}^{i}\varphi _{L,i}.$ 
\item[{\em (b)}] 
$|\varphi _{L,i}||_{L}\equiv 1.$ 
\item[{\em (c)}] $\varphi _{L,i}|_{L'}\equiv 0$ for any  $L'\in \emMin(G_{\tau },\CCI )$ with $L'\neq L.$ 
\item[{\em (d)}] $\varphi _{L,i}|_{L}=(\varphi _{L,1}|_{L})^{i}.$  
\item[{\em (e)}] $\mbox{{\em supp}}\, \rho _{L,i}=L.$ 
\item[{\em (f)}] $\rho _{L,i}(\varphi _{L,j})=\delta _{ij}$ for each $j=1,\ldots ,r_{L}.$   
\end{itemize}
\item \label{t:mtauspecdual}
For each $\nu \in {\frak M}_{1}(\CCI )$, 
$d_{0}((M_{\tau }^{\ast })^{n}(\nu ), \emLSfac \cap {\frak M}_{1}(\CCI ))\rightarrow 0$ as 
$n\rightarrow \infty. $ Moreover, 
$\dim _{T}(\emLSfac \cap {\frak M}_{1}(\CCI ))\leq 2\dim _{\CC }\emLSfc <\infty $, 
where $\dim _{T}$ denotes the topological dimension.  
\item \label{t:mtauspec9}
For each $L\in \emMin(G_{\tau },\CCI )$, 
$T_{L,\tau }:\CCI \rightarrow [0,1]$ is continuous 
and   
$M_{\tau }(T_{L,\tau })=T_{L,\tau }$.   
Moreover, $\sum _{L\in \emMin(G_{\tau },\CCI )}T_{L,\tau }(z)=1$ for each 
$z\in \CCI .$ 
\item \label{t:mtauspec9-2}
If $\sharp \emMin (G_{\tau },\CCI )\geq 2$, then 
{\em (a)} for each $L\in \emMin (G_{\tau },\CCI )$, $T_{L,\tau }(J(G_{\tau }) )=[0,1]$, 
and {\em (b)} $\dim _{\CC }(\emLSfc )>1.$   
\item \label{t:mtauspec10}
$S_{\tau }= \{ \overline{z\in F(G)\cap S_{\tau }\mid \exists g\in G_{\tau } \mbox{ {\em s.t.} } g(z)=z, |m(g,z)|<1} \} $, 
where the closure is taken in $\CCI $, and $m(g,z)$ denotes the multiplier (\cite{Be}) of $g$ at the fixed point $z.$  
\item \label{t:mtauspec11}
If $\Gamma _{\tau } \cap \emRatp \neq \emptyset $, then 
$S_{\tau }= \{ \overline{z\in F(G)\cap S_{\tau }\mid \exists g\in G_{\tau }\cap \emRatp \mbox{ {\em s.t.} } g(z)=z, |m(g,z)|<1} \}\subset UH(G_{\tau })\subset P(G_{\tau }).$ 
\item \label{t:mtauspec12} 
If $\dim _{\CC }(\emLSfc )>1$, then 
for any $\varphi \in \emLSfc _{nc}$ there exists  
an uncountable subset $A$ of $\CC $ such that 
for each $t\in A$, 
$\emptyset \neq \varphi ^{-1}(\{ t\} )\cap J(G_{\tau })\subset J_{res}(G_{\tau }).$ 
\item \label{t:mtauspeccfi}
If $\dim _{\CC }(\emLSfc )>1$ and {\em int}$(J(G_{\tau }))=\emptyset $, 
then $\sharp \mbox{{\em Con}}(F(G_{\tau }))=\infty .$ 
\item \label{t:mtauspecdeg1}
Suppose that $G_{\tau }\cap \mbox{{\em Aut}}(\CCI )\neq \emptyset $, where 
{\em Aut}$(\CCI )$ denotes the set of all holomorphic automorphisms on $\CCI .$ If 
there exists a loxodromic or parabolic element of $G_{\tau }\cap \mbox{{\em Aut}}(\CCI )$, then 
$\sharp \mbox{\em Min}(G_{\tau },\CCI )=1$ and $\dim _{\CC }(\mbox{\em LS}({\cal U}_{f,\tau }(\CCI )))=1.$ 
\end{enumerate}
\end{thm}
\begin{rem}
\label{r:sjgg3}
Let $G$ be a rational semigroup with $G\cap \Ratp \neq \emptyset .$ Then by \cite[Theorem 4.2.4]{Be}, 
$\sharp (J(G))\geq 3.$ 
\end{rem}
\begin{rem}
\label{r:uminj}Let $\tau \in {\frak M}_{1,c}(\mbox{Rat})$ be such that 
$J_{\ker }(G_{\tau })=\emptyset $ and $J(G_{\tau })\neq \emptyset .$ 
The union $S_{\tau }$ of minimal sets for $(G_{\tau },\CCI )$ may meet $J(G_{\tau }).$ 
See Example~\ref{r:parajkere}. 
\end{rem}
\begin{rem}
\label{r:ls1}
Let $\tau \in {\frak M}_{1,c}(\mbox{Rat})$ be such that 
$J_{\ker }(G_{\tau })=\emptyset $ and $J(G_{\tau })\neq \emptyset .$ 
Then $\dim _{\CC }(\LSfc )>1$ if and only if $(\LSfc )_{nc}\neq \emptyset .$ 
\end{rem}
\begin{df}
\label{d:sfj}
Let $G$ be a polynomial semigroup. 
We set \\ 
$\hat{K}(G):= \{ z\in \CC \mid \{ g(z)\mid g\in G\} \mbox{ is bounded in }\CC \} .$  
$\hat{K}(G)$ is called the {\bf smallest filled-in Julia set} of $G.$ 
For any $h\in {\cal P}$, we set $K(h):=\hat{K}(\langle h\rangle ).$ This is called the filled-in Julia set of $h.$ 
\end{df}
\begin{rem}
Let $\tau \in {\frak M}_{1,c}({\cal P})$ be such that 
$J_{\ker }(G_{\tau })=\emptyset $ and $\hat{K}(G_{\tau })\neq \emptyset .$ 
Then  
 $\sharp \Min (G_{\tau },\CCI )\geq 2.$ Thus 
by Theorem~\ref{t:mtauspec}-\ref{t:mtauspec9-2}, 
$\dim _{\CC }(\LSfc )>1.$ 
\end{rem}
\begin{rem}
\label{r:pjkekie}
There exist many examples of $\tau \in {\frak M}_{1,c}({\cal P})$ 
such that $J_{\ker }(G_{\tau })=\emptyset , \hat{K}(G_{\tau })\neq \emptyset $ 
and int$(J(G_{\tau }))=\emptyset $ (see Proposition~\ref{semihyposcexprop}, 
Proposition~\ref{p:hyphigh}, Proposition~\ref{p:hypdispur}, Theorem~\ref{t:hnondiff}, 
and \cite[Theorem 2.3]{S2}).  
\end{rem}
\subsection{Properties on $T_{\infty ,\tau }$}
\label{ProT}
In this subsection, we present some results on  properties of $T_{\infty ,\tau }$ 
for a $\tau \in {\frak M}_{1,c}({\cal P}).$  
Moreover, we present some results on the structure of $J(G_{\tau })$ for a $\tau \in {\frak M}_{1,c}({\cal P})$ 
with $J_{\ker }(G_{\tau })=\emptyset .$ 
The proofs are given in subsection~\ref{pfProT}. 
 
By Theorem~\ref{kerJthm1} or Theorem~\ref{t:mtauspec}, we obtain the following result. 
\begin{thm}
\label{kerJthm2}
Let $\tau \in {\frak M}_{1,c}({\cal P})$.  
Suppose that 
$J_{\ker}(G_{\tau })=\emptyset .$ Then, the function 
$T_{\infty ,\tau }:\CCI \rightarrow [0,1]$ is continuous on the whole \nolinebreak $\CCI $, 
and $M_{\tau }(T_{\infty ,\tau })=T_{\infty ,\tau }.$   
\end{thm}
\begin{rem}
Let $h\in {\cal P}$ and let $\tau := \delta _{h}.$ Then, $T_{\infty ,\tau }(\CCI )=\{ 0,1\} $ and 
$T_{\infty ,\tau }$ is not continuous at every point in $J(h)\neq \emptyset .$  
\end{rem}
On the one hand, we have the following, due to Vitali's theorem. 
\begin{lem}
\label{Tlem1}
Let $\tau \in {\frak M}_{1,c}({\cal P})$.   
Then, for each 
connected component $U$ of $F(G_{\tau })$, there exists a constant
 $C_{U}\in [0,1]$ such that 
 $T_{\infty ,\tau }|_{U}\equiv C_{U}.$ 
\end{lem}
\begin{df}
Let $G$ be a polynomial semigroup. 
If $\infty \in F(G)$, then we denote by $F_{\infty }(G)$ the connected component 
of $F(G)$ containing $\infty .$ (Note that if $G$ is generated by 
a compact subset of ${\cal P}$, then $\infty \in F(G).$) 
\end{df}
We give a characterization of $T_{\infty ,\tau }.$ 
\begin{prop}
\label{p:chtinfty}
Let $\tau \in {\frak M}_{1,c}({\cal P}).$ Suppose that 
$J_{\ker }(G_{\tau })=\emptyset $ and $\hat{K}(G_{\tau })\neq \emptyset .$ 
Then, there exists a unique bounded Borel measurable function $\varphi :\CCI \rightarrow \RR $ 
such that $\varphi =M_{\tau }(\varphi )$,  
$\varphi |_{F_{\infty }(G_{\tau })}\equiv 1$ and 
$\varphi |_{\hat{K}(G_{\tau })}\equiv 0.$ 
Moreover, $\varphi =T_{\infty ,\tau }.$ 
\end{prop}
\begin{rem}
Combining Theorem~\ref{kerJthm2} and Lemma~\ref{Tlem1}, it follows that 
under the assumptions of Theorem~\ref{kerJthm2}, if $T_{\infty ,\tau }\not\equiv 1$, 
then the function $T_{\infty ,\tau }$ is continuous on $\CCI $ and varies only on the Julia set 
$J(G_{\tau })$ of $G_{\tau }.$ In this case, the function $T_{\infty ,\tau }$ is called the 
{\bf devil's coliseum} (see Figures~\ref{fig:dcgraphgrey2}, \ref{fig:dcgraphudgrey2}). 
This is a complex analogue of the devil's staircase or Lebesgue's singular functions. 
We will see the monotonicity of this function $T_{\infty ,\tau }$ in Theorem~\ref{kerJthm3}.   
\end{rem}
In order to present the result on the monotonicity of the function 
$T_{\infty ,\tau }:\CCI \rightarrow [0,1]$, the level set of $T_{\infty ,\tau }|_{J(G_{\tau })}$ 
and the structure of the Julia set $J(G_{\tau })$, 
we need the following notations.   
\begin{df}
Let $K_{1},K_{2}\in \mbox{Cpt}(\CCI ).$  
\begin{enumerate}
\item 
``$K_{1}<_{s}K_{2}$'' indicates that $K_{1}$ is included in the union of all bounded 
components of $\CC \setminus K_{2}.$
\item ``$K_{1}\leq _{s}K_{2}$'' indicates that $K_{1}<_{s}K_{2}$ or $K_{1}=K_{2}.$
\end{enumerate}
\end{df}
\begin{rem}
This ``$\leq _{s}$'' is a partial order in $\mbox{Cpt}(\CCI ).$ 
This ``$\leq _{s}$'' is called the {\bf surrounding order.} 
\end{rem}
We present a necessary and sufficient condition for $T_{\infty ,\tau }$ to be the constant function $1.$ 
\begin{lem}
\label{kemptylem1}
Let $\tau \in {\frak M}_{1,c}({\cal P})$.  
Then, the following {\em (1)}, {\em (2)}, and {\em (3)} are equivalent.
{\em (1)} $T_{\infty ,\tau }\equiv 1.$
{\em (2)} $T_{\infty ,\tau }|_{J(G_{\tau })}\equiv 1.$ 
{\em (3)} $\hat{K}(G_{\tau })=\emptyset .$ 

\end{lem}

By Theorem~\ref{kerJthm2} and Lemma~\ref{Tlem1}, we obtain the following result.  
\begin{thm}[Monotonicity of $T_{\infty, \tau }$ and the structure of $J(G_{\tau })$]
\label{kerJthm3}   
Let $\tau \in {\frak M}_{1,c}({\cal P})$.  Suppose that
 $J_{\ker }(G_{\tau })=\emptyset $ and    
$\hat{K}(G_{\tau })\neq \emptyset .$ Then, 
we have all of the following.
\begin{enumerate}
\item \label{kerJthm3-1}
 {\em int}$(\hat{K}(G_{\tau }))\neq \emptyset .$ 
\item \label{kerJthm3-2}
$T_{\infty ,\tau }(J(G_{\tau }))=[0,1].$ 
\item \label{kerJthm3-3}
For each $t_{1},t_{2}\in [0,1]$ with $0\leq t_{1}<t_{2}\leq 1$, 
we have $T_{\infty ,\tau }^{-1}(\{ t_{1}\} )<_{s}T_{\infty ,\tau }^{-1}(\{ t_{2}\} )\cap 
J(G_{\tau }).$ 
\item \label{kerJthm3-4}
For each $t\in (0,1)$, we have  
$\hat{K}(G_{\tau })<_{s}T_{\infty ,\tau }^{-1}(\{ t\} )\cap J(G_{\tau })<_{s}
\overline{F_{\infty }(G_{\tau })}.$ 
\item \label{kerJthm3-5}
There exists an uncountable dense subset $A$ of $[0,1]$ with $\sharp ([0,1]\setminus A)\leq \aleph _{0}$ 
such that 
for each $t\in A$, we have   
$\emptyset \neq T_{\infty ,\tau }^{-1}(\{ t\} )\cap J(G_{\tau })\subset J_{res}(G_{\tau }).$ 
\end{enumerate} 

\end{thm}
\begin{rem}
If $G$ is generated by a single map $h\in {\cal P}$, then 
$\partial \hat{K}(G)=\partial F_{\infty }(G)=J(G)$ and so 
$\hat{K}(G)$ and $\overline{F_{\infty }(G)}$ cannot be separated. 
However, under the assumptions of Theorem~\ref{kerJthm3}, the theorem implies that  
$\hat{K}(G_{\tau })$ and $\overline{F_{\infty }(G_{\tau })}$ are separated by 
the uncountably many level sets $\{ T_{\infty ,\tau }|_{J(G_{\tau })}^{-1}(\{ t\} )\} _{ t\in (0,1)} $, 
and that these level sets are totally ordered with respect to the surrounding order, 
respecting the usual order in $(0,1).$ Note that there are many $\tau \in {\frak M}_{1,c}({\cal P})$ 
such that $J_{\ker }(G_{\tau })=\emptyset $ and $\hat{K}(G_{\tau })\neq \emptyset $. 
See section~\ref{Examples}.     
\end{rem}
\begin{rem}
\label{rem:anygetau}
For each $\Gamma \in \mbox{Cpt}(\mbox{Rat})$, there exists a $\tau \in {\frak M}_{1}(\mbox{Rat})$ such that 
$\G _{\tau }=\G .$ Thus, Theorem~\ref{kerJthm3} tells us the information of 
the Julia set of a polynomial semigroup $G$ generated by a compact subset $\G $ of 
${\cal P}$ such that $J_{\ker }(G)=\emptyset $ and $\hat{K}(G)\neq \emptyset .$ 
\end{rem}

%
Applying Theorem~\ref{kerJthm2} and Lemma~\ref{Tlem1}, we obtain the following result. 
\begin{thm}
\label{thm:jkreyintj}
Let $\G $ be a non-empty compact subset of ${\cal P}$ and let $G=\langle \Gamma \rangle .$  
Suppose that $\hat{K}(G)\neq \emptyset $ and $J_{\ker }(G)=\emptyset $. 
Then, at least one of the following statements {\em (a)} and {\em (b)} holds. 

\vspace{2mm} 
{\em (a)} {\em int}$(J(G))\neq \emptyset .$ \ 
{\em (b)} $\sharp \{ U\in \mbox{{\em Con}}(F(G))\mid U\neq F_{\infty }(G) \mbox{ and } U\not\subset \mbox{{\em int}}(\hat{K}(G))\} =\infty .$  
\end{thm}
\begin{rem}
There exist finitely generated polynomial semigroups $G$ in ${\cal P}$ 
such that int$(J(G))\neq \emptyset $ and $J(G)\neq \CCI $ (see \cite{HM2}, Example~\ref{ex:pjne}).  
\end{rem}
\subsection{Planar postcritical set and the condition that $\hat{K}(G_{\tau })=\emptyset $}
\label{Planar}
In this subsection, we present some results which are deduced from the condition that the planar postcritical set is unbounded. 
Moreover, we present some results which are deduced from the condition that $\hat{K}(G_{\tau })=\emptyset .$ 
The proofs are given in subsection~\ref{pfPlanar}.  
\begin{df}
For a polynomial semigroup $G$, we set 
$P^{\ast }(G):= P(G)\setminus \{ \infty \} .$ This is called the 
{\bf planar postcritical set} of the polynomial semigroup $G.$ 
\end{df} 
\begin{df}
Let $Y$ be a complete metric space. We say that a subset $A$ of $Y$ is residual 
if $A$ contains a countable intersection of open dense subsets of $Y.$ 
Note that by Baire's category theorem, a residual subset $A$ of $Y$ is dense in $Y.$ 
\end{df} 
The following theorem generalizes \cite[Theorem 1.5]{Br1} and  \cite[Theorem 2.3]{BBR}. 
\begin{thm}
\label{Punbddthm1}
Let $\G \in \mbox{{\em Cpt}}({\cal P})$ and let 
$G=\langle \G \rangle .$  
Suppose that $P^{\ast }(G)$ is not bounded in $\CC .$ 
Then, there exists a residual subset ${\cal U}$ of $\GN $ such that 
for each $\tau \in {\frak M}_{1}({\cal P})$ with $\G _{\tau }=\G $, 
we have $\tilde{\tau }({\cal U})=1$, and such that 
for each $\g \in {\cal U}$,   
the Julia set $J_{\gamma }$ of $\gamma $ has 
uncountably many connected components.  
\end{thm}
\begin{ques}
What happens if  $\hat{K}(G_{\tau })=\emptyset $ (i.e., if $T_{\infty ,\tau }\equiv 1$) ?
\end{ques}
\begin{df}
\label{d:kg}
Let $\gamma =(\gamma _{1},\gamma _{2},\ldots )\in {\cal P}^{\NN }.$ We set 
$K_{\gamma }:= \{ z\in \CC \mid \{ \gamma _{n,1}(z)\} _{n\in \NN } 
\mbox{ is bounded in } \CC \} .$ Moreover, 
we set 
$A_{\infty ,\g }:= \{ z\in \CCI \mid \gamma _{n,1}(z)\rightarrow \infty \} .$ 
\end{df}

\begin{thm}
\label{Tequals1thm1}
Let $\tau \in {\frak M}_{1,c}({\cal P})$. Suppose that $\hat{K}(G_{\tau })=\emptyset $. Then, 
we have all of the following statements 1,$\ldots $,4. 
\begin{enumerate}
\item $J_{\ker }(G_{\tau })=\emptyset .$ 
\item $F_{meas}(\tau )={\frak M}_{1}(\CCI )$ and $(M_{\tau }^{\ast })^{n}(\nu )\rightarrow \delta _{\infty }$ 
as $n\rightarrow \infty $ uniformly on $\nu \in {\frak M}_{1}(\CCI ).$   
\item $T_{\infty ,\tau }\equiv 1$ on $\CCI .$ 
\item For $\tilde{\tau }$-a.e. $\gamma \in {\cal P}^{\NN }$,  
{\em (a)} {\em Leb}$_{2}(K_{\gamma })=0$, 
{\em (b)} $K_{\gamma }=J_{\gamma }$, and 
{\em (c)} $K_{\gamma }=J_{\gamma }$ has uncountably many connected components. 


\end{enumerate}

\end{thm}
\begin{rem}
\label{r:pnckem}
Let $\tau \in {\frak M}_{1,c}({\cal P})$. Suppose that $J_{\ker }(G_{\tau })=\emptyset .$ 
From Theorem~\ref{kerJthm2} and Theorem~\ref{Tequals1thm1}, it follows that  
$\hat{K}(G_{\tau })\neq \emptyset $ if and only if 
$(\LSfc )_{nc}\neq \emptyset .$ 
\end{rem}
\begin{ex}
\label{ex:ke}
Let $\tau \in {\frak M}_{1,c}({\cal P})$ and suppose that 
there exist two elements $h_{1},h_{2}\in \Gamma _{\tau }$ such that 
$K(h_{1})\cap K(h_{2})=\emptyset .$ Then $\hat{K}(G_{\tau })=\emptyset .$ 
For more examples of $\tau $ with $\hat{K}(G_{\tau })=\emptyset $, 
see Example~\ref{ex:kead}. 
\end{ex}
\subsection{Conditions to be Leb$_{2}(J_{\gamma })=0$ for $\tilde{\tau }$-a.e. $\gamma $ (even if $J_{\ker }(G_{\tau })\neq \emptyset $)}
\label{Conleb}
In this subsection, we present some sufficient conditions to be 
Leb$_{2}(J_{\gamma })=0$ for $\tilde{\tau }$-a.e. $\gamma $.  
More precisely, we show that even if $J_{\ker }(G_{\tau })\neq \emptyset $, 
under certain conditions, for $\tilde{\tau }$-a.e. $\gamma $,  for Leb$_{2}$-a.e. $z\in \CCI $, 
there exists a number $n_{0}\in \NN $ such that for each $n$ with $n\geq n_{0}$, 
$\gamma _{n,1}(z)\in F(G_{\tau }).$ 
The proofs are given in subsection~\ref{pfConleb}. 
We also define other kinds of Julia sets of $M_{\tau }^{\ast }.$ 
\begin{df}
\label{d:manyFJ}
Let $Y$ be a compact metric space. 
Let $\tau \in {\frak M}_{1}(\CMX ).$ 
Regarding $Y $ as a compact subset of ${\frak M}_{1}(Y)$ as in Definition~\ref{d:Phi}, we use 
the following notation.   
\begin{enumerate}
\item 
We denote by $F_{pt}(\tau )$ the set of 
$z\in Y $ satisfying that 
there exists a neighborhood $B$ of $z $ in $Y $ such that 
the sequence$\{ (M_{\tau }^{\ast })^{n}|_{B}:B\rightarrow 
{\frak M}_{1}(Y )\} _{n\in \NN }$ is equicontinuous on $B.$ 
We set $J_{pt}(\tau ):= Y \setminus F_{pt}(\tau ).$ 
\item Similarly, we denote by 
$F_{pt }^{0}(\tau )$ the set of 
$z\in Y $ such that 
the sequence $\{ (M_{\tau }^{\ast })^{n}|_{Y }: Y 
\rightarrow {\frak M}_{1}(Y)\} _{n\in \NN }$ is 
equicontinuous at the one point $z\in Y.$ 
We set $J_{pt}^{0}(\tau ):= Y \setminus F_{pt}^{0}(\tau ).$ 
\end{enumerate}

\end{df}
\begin{rem}
We have $F_{pt}(\tau )\subset F_{pt}^{0}(\tau )$ and  
$J_{pt}^{0}(\tau )\subset J_{pt}(\tau )\cap J_{meas}^{0}(\tau )$.
\end{rem}
We also need the following notations on the skew products. 
In fact, we heavily use the idea and the notations of the dynamics of skew products, 
to prove many results of this paper. 

\begin{df}
\label{d:sp}
Let $Y$ be a compact metric space and 
let $\Gamma $ be a non-empty compact subset of  
$\CMX .$   
We define a map $f:\GN \times Y \rightarrow 
\GN \times Y $ as follows: 
For a point $(\gamma  ,y)\in \GN \times Y $ where 
$\gamma =(\gamma _{1},\gamma _{2},\ldots )$, we set 
$f(\gamma  ,y):= (\sigma (\gamma  ), \gamma  _{1}(y))$, where 
$\sigma :\GN \rightarrow \GN $ is the shift map, that is, 
$\sigma (\gamma  _{1},\gamma  _{2},\ldots )=(\gamma  _{2},\gamma  _{3},\ldots ).$ 
The map $f:\GN \times Y \rightarrow \GN \times Y $ 
is called the {\bf skew product associated with the 
generator system }$\G .$ 
Moreover, we use the following notation. 
\begin{enumerate}
\item 
Let 
$\pi : \GN \times \CCI \rightarrow \GN $  
and $\pi _{Y }:\GN \times Y \rightarrow Y $ be 
the canonical projections. 
For each 
$\gamma  \in \GN $ and $n\in \NN $, we set 
$f_{\gamma  }^{n}:= f^{n}|_{\pi ^{-1}\{ \gamma  \} } : 
\pi ^{-1}\{ \gamma  \} \rightarrow 
\pi ^{-1}\{ \sigma ^{n}(\gamma  )\} .$  
Moreover, we set 
$f_{\gamma  ,n}:= \gamma  _{n}\circ \cdots \circ \gamma  _{1}.$ 
\item 
For each $\gamma \in \Gamma ^{\NN }$, 
we set $J^{\gamma  }:= \{ \gamma  \} \times J_{\gamma  }
\ (\subset \GN \times Y )$. 
Moreover, we set 
$\tilde{J}(f):= \overline{\bigcup _{\gamma  \in \GN }J^{\gamma  }}$, 
where the closure is taken in the product space $\GN \times Y .$ 
Furthermore, we set $\tilde{F}(f):= (\Gamma ^{\NN }\times Y)\setminus \tilde{J}(f).$  
\item 
For each $\gamma  \in \GN $, we set 
$\hat{J}^{\gamma  ,\Gamma }:= \pi ^{-1}\{  \gamma  \} \cap \tilde{J}(f)$, 
$\hat{F}^{\gamma ,\Gamma }:= \pi ^{-1}(\{ \gamma \} )\setminus \hat{J}^{\gamma ,\G }$,  
$\hat{J}_{\gamma ,\Gamma }:= \pi _{Y }(\hat{J}^{\gamma ,\Gamma })$, 
and $\hat{F}_{\g, \G }:= Y\setminus \hat{J}_{\g ,\G }.$   
Note that $J_{\gamma }\subset \hat{J}_{\gamma ,\Gamma }.$   
\item When $\Gamma \subset \mbox{Rat}$, 
for each $z=(\gamma  ,y)\in \GN \times \CCI $, 
we set $f'(z):=(\gamma  _{1})'(y).$ 
\end{enumerate}
\end{df}
\begin{rem}
Under the above notation, let $G=\langle \Gamma \rangle .$  
Then   
$\pi _{Y }(\tilde{J}(f))\subset J(G)$ and 
$\pi \circ f=\sigma \circ \pi $ on $\Gamma ^{\NN }\times Y .$ 
Moreover, for each $\g \in \Gamma ^{\NN }$, 
$\gamma _{1}(J_{\gamma })\subset J_{\sigma (\gamma )}$, 
$\gamma _{1}(\hat{J}_{\gamma ,\Gamma })\subset \hat{J}_{\sigma (\gamma ),\Gamma }$, 
and $f(\tilde{J}(f))\subset \tilde{J}(f)$ (see Lemma~\ref{genskewprodinvlem1}). 
Furthermore, if $\Gamma \in \Cpt(\Rat)$, then 
for each $\gamma \in \Gamma ^{\NN }$, 
$\gamma _{1}(J_{\gamma })=J_{\sigma (\gamma )}$, $\gamma _{1}^{-1}(J_{\sigma (\gamma )})=J_{\gamma }$, 
$\gamma _{1}(\hat{J}_{\gamma ,\Gamma })=\hat{J}_{\sigma (\gamma ),\Gamma }$, 
$\gamma _{1}^{-1}(\hat{J}_{\sigma (\gamma ),\Gamma })=\hat{J}_{\gamma ,\Gamma }$,  
$f(\tilde{J}(f))=\tilde{J}(f)=f^{-1}(\tilde{J}(f))$, and $f(\tilde{F}(f))=\tilde{F}(f)=f^{-1}(\tilde{F}(f))$ 
(see \cite[Lemma 2.4]{S4}). 
\end{rem}

We now present the results. 
Even if $J_{\ker}(G_{\tau })\neq \emptyset $, we have the following. 
\begin{thm}
\label{t:jkucuh}
Let $\tau \in {\frak M}_{1,c}({\cal P})$.   
Suppose that  
$J_{\ker }(G_{\tau })$ is included in the unbounded component of  
$\CC \setminus (UH(G_{\tau })\nolinebreak \cap \nolinebreak J(G_{\tau }))$.   
 Then, we have the following.
\begin{enumerate}
\item  
For $\tilde{\tau }$-a.e.  
$\gamma \in X_{\tau }$,  
{\em Leb}$_{2}(J_{\gamma })=\mbox{{\em Leb}}_{2}(\hat{J}_{\g ,\G _{\tau }})=0.$ 
\item For {\em Leb}$_{2}$-a.e. $y\in \CCI $, 
there exists a Borel subset ${\cal A}_{y}$ of $X_{\tau }$ with $\tilde{\tau }({\cal A}_{y})=1$ 
such that for each $\gamma \in {\cal A}_{y}$, 
there exists an $n=n(y,\gamma )\in \NN $ 
with $\gamma _{n,1}(y)\in F(G_{\tau }).$ 
\item {\em Leb}$_{2}(J_{pt}^{0}(\tau ))=0.$
\item For {\em Leb}$_{2}$-a.e. point $y\in \CCI $, 
$T_{\infty ,\tau }$ is continuous at $y.$ 
\end{enumerate}
\end{thm}
\begin{rem}
\label{r:jkup}
Let $\tau \in {\frak M}_{1,c}({\cal P}). $ 
If $J_{\ker }(G_{\tau })$ is included in the unbounded component of 
$\CC \setminus (P(G_{\tau })\cap J(G_{\tau }))$, then 
 $J_{\ker }(G_{\tau })$ is included in the unbounded component of 
$\CC \setminus (UH(G_{\tau })\cap J(G_{\tau }))$ (see Remark~\ref{r:uhsp}).  
\end{rem}
%
%
\begin{rem}
\label{r:pjker}
Let $\tau \in {\frak M}_{1,c}({\cal P}).$ 
Suppose that for each $h\in \G _{\tau }$, 
$h$ is a real polynomial and each critical value of $h$ in $\CC $ belongs to $\RR .$ 
Suppose also that for each $z\in P(G_{\tau })\cap J(G_{\tau })$, 
there exists an element $g_{z}\in G_{\tau }$ such that 
$g_{z}(z)\in F(G_{\tau }).$ 
Then $J_{\ker }(G_{\tau })$ is included in the unbounded component of 
$\CC \setminus (UH(G_{\tau })\cap J(G_{\tau })).$ 
\end{rem}
\subsection{Conditions to be $J_{\ker }(G)=\emptyset $}
\label{Conjkeremp}
In this subsection, we present some sufficient conditions to be $J_{\ker }(G)=\emptyset .$ 
The proofs are given in subsection~\ref{pfConjkeremp}. 
 
The following is a natural question. 
\begin{ques}
When do we have that $J_{\ker}(G)=\emptyset $?
\end{ques}
We give several answers to this question. 
\begin{lem}
\label{l:ignejke}
Let $\Gamma $ be a subset of {\em Rat} such that the interior of $\Gamma $ with respect to the 
topology of {\em Rat} is not empty. 
Let $G=\langle \Gamma \rangle .$ Suppose that $F(G)\neq \emptyset .$ 
Then, $J_{\ker }(G)=\emptyset .$  
\end{lem}
\begin{df}
Let $\Lambda $ be a finite dimensional complex manifold and 
let $\{ g_{\lambda }\} _{\lambda \in \Lambda }$ be a 
family of rational maps on $\CCI .$ We say that 
$\{ g_{\lambda }\} _{\lambda \in \Lambda }$ is 
a holomorphic family of rational maps if 
the map $(z,\lambda )\in \CCI \times \Lambda \mapsto g_{\lambda }(z)\in \CCI $ is 
holomorphic on $\CCI \times \Lambda .$ We say that 
 $\{ g_{\lambda }\} _{\lambda \in \Lambda }$ is 
a holomorphic family of polynomials if $\{ g_{\lambda }\} _{\lambda \in \Lambda }$ is 
a holomorphic family of rational maps and each $g_{\lambda }$ is a polynomial.   
\end{df}
\begin{df}
\label{d:adm} 
Let ${\cal Y}$ be a subset of ${\cal P}.$ 
\begin{enumerate}
\item 
We say that ${\cal Y}$ is {\em admissible} if 
for each $z_{0}\in \CC $ there exists a holomorphic family of 
polynomials $\{ g_{\lambda }\} _{\lambda \in \Lambda }$ such that 
$\{ g_{\lambda }\mid \lambda \in \Lambda \} \subset {\cal Y}$ and 
the map $\lambda \mapsto g_{\lambda }(z_{0})$ is nonconstant in $\Lambda .$ 
\item We say that ${\cal Y}$ is {\em strongly admissible} if 
for each $(z_{0},h_{0})\in \CC \times {\cal Y}$ there exists a
 holomorphic family $\{ g_{\lambda }\} _{\lambda \in \Lambda }$ of polynomials and a point 
 $\lambda _{0}\in \Lambda $ 
 such that $\{ g_{\lambda }\mid \lambda \in \Lambda \} \subset {\cal Y} ,$ 
$g_{\lambda _{0}}=h_{0},$ and the map $\lambda \mapsto g_{\lambda }(z_{0})\in \CC $ 
is nonconstant in any neighborhood of $\lambda _{0}$ in $\Lambda .$

\end{enumerate}
\end{df}
\begin{ex}\ 
\begin{enumerate}
\item Let ${\cal Y}$ be a strongly admissible subset of ${\cal P}.$ 
Let ${\cal Y}$ be endowed with the relative topology from ${\cal P}.$
If $\G $ is a non-empty  open subset  of ${\cal Y}$, then 
$\G $ is strongly admissible. 
If $\G '$ is a subset of ${\cal Y}$ such that the interior of $\G '$ 
in ${\cal Y}$ is not empty, then $\G '$ is admissible. 
\item ${\cal P}$ is strongly admissible. If $\Gamma $ is a subset of ${\cal P}$ such that 
the interior of $\Gamma $ in ${\cal P}$ is not empty, then 
$\G $ is admissible. 
\item 
For a fixed $h_{0}\in {\cal P},  
{\cal Y}:=\{ h_{0}+c\mid c\in \CC \} $ is a strongly admissible closed subset of ${\cal P}.$  
If $\G $ is a subset of ${\cal Y}$ such that the interior of $\G $ in ${\cal Y}$ is not empty, 
then $\G $ is admissible. 
\end{enumerate}
\end{ex}
\begin{lem}
\label{l:aigpjke}
Let $\Gamma $ be a relative compact admissible subset of ${\cal P}.$  
Let $G=\langle \G \rangle .$  
Then, $J_{\ker }(G)=\emptyset .$ 
\end{lem}
\begin{prop}
\label{p:v1v2rho}
Let ${\cal Y}$ be a closed subset of an open subset of ${\cal P}.$ Suppose that ${\cal Y}$ is 
strongly admissible. 
Let $\tau \in {\frak M}_{1,c}({\cal Y})$. 
Let $V_{1}$ be any neighborhood of $\tau $ in ${\frak M}_{1}({\cal Y})$ and 
$V_{2}$ be any neighborhood of $\Gamma _{\tau }$ in {\em Cpt}$(\CCI ).$ 
Then, there exists an element $\rho \in {\frak M}_{1}({\cal Y})$ such that 
$\rho \in V_{1}$, $\Gamma _{\rho }\in V_{2}$, $\sharp \Gamma _{\rho }<\infty $, 
and $J_{\ker }(G_{\rho })=\emptyset .$ 
\end{prop}
\begin{rem}[Cooperation Principle III]
\label{r:CPIII}
By Lemma~\ref{l:aigpjke}, Proposition~\ref{p:v1v2rho}, Theorems~\ref{kerJthm1}, \ref{t:mtauspec}, 
we can state that for most $\tau \in {\frak M}_{1,c}({\cal P})$, 
the chaos of the averaged system of the Markov process induced by $\tau $ disappears. 
In the subsequent paper \cite{Sprep}, we investigate the further detail regarding  this result.  
Some results of \cite{Sprep} are announced in \cite{Skokyu10}. 
\end{rem}
\begin{ex}
\label{ex:kead}
Let $\tau \in {\frak M}_{1,c}({\cal P})$ be such that $\Gamma _{\tau }$ is admissible. 
Suppose that there exists an element $h\in \Gamma _{\tau }$ with int$(K(h))=\emptyset .$ 
Then $\hat{K}(G_{\tau })=\emptyset $ and the statements in Theorem~\ref{Tequals1thm1} hold.  
For, if $\hat{K}(G_{\tau })\neq \emptyset $, 
then since $\Gamma _{\tau }$ is admissible and since $G_{\tau }(\hat{K}(G_{\tau }))\subset \hat{K}(G_{\tau })$, 
we have int$(\hat{K}(G_{\tau }))\neq \emptyset .$ However, since int$(K(h))=\emptyset $, 
this is a contradiction. Thus $\hat{K}(G_{\tau })=\emptyset .$ 

From the above argument, 
we obtain many examples of $\tau \in {\frak M}_{1,c}({\cal P})$ such that 
$\hat{K}(G_{\tau })=\emptyset .$ 
For example, if $h(z)=z^{2}+c$ belongs to the boundary of the Mandelbrot set and 
$\Gamma _{\tau }$ contains a neighborhood of $h$ in the $c$-plane, 
then from the above argument, $\hat{K}(G_{\tau })=\emptyset $ and the 
statements in Theorem~\ref{Tequals1thm1} hold. 
Thus the above argument generalizes 
\cite[Theorem 2.2]{BBR} and 
a statement in \cite[Theorem 2.4]{Br1}.  
\end{ex}
\subsection{Mean stability}
\label{Almostst}
In this subsection, we introduce mean stable rational semigroups, and 
we present some results on mean stability. 
The proofs are given in subsection~\ref{pfAlmostst}. 
\begin{df}
\label{d:as}
Let $Y$ be a compact metric space and let $\Gamma \in \Cpt(\CMX).$ 
Let $G=\langle \Gamma \rangle .$ 
We say that $G$ is {\bf mean stable} if 
there exist non-empty open subsets $U,V$ of $F(G)$ and a number $n\in \NN $ 
such that all of the following hold.
\begin{itemize}
\item[(1)]
$\overline{V}\subset U$ and $\overline{U}\subset F(G).$ 
\item[(2)]
For each $\gamma \in \Gamma ^{\NN }$, 
$\gamma _{n,1}(\overline{U})\subset V.$ 
\item[(3)]
For each point $z\in Y$, there exists an element 
$g\in G$ such that  
$g(z)\in U.$ 
\end{itemize} 
Note that this definition does not depend on the choice of a compact set $\Gamma $ which generates 
$G.$  Moreover, for a $\Gamma \in \Cpt(\CMX)$, 
we say that $\Gamma $ is mean stable if $\langle \Gamma \rangle $ is mean stable. 
Furthermore, for a $\tau \in {\frak M}_{1,c}(\CMX)$, we say that $\tau $ is mean stable if $G_{\tau }$ is mean stable. 
\end{df}
\begin{rem}
\label{r:asjkere}
It is easy to see that if $G$ is mean stable, then $J_{\ker }(G)=\emptyset .$ 
\end{rem}
By Montel's theorem, it is easy to see that the following lemma holds.
\begin{lem}
\label{l:asnbd}
Let $\Gamma \in \emCpt(\emRat)$ be mean stable. Suppose 
$\sharp (\CCI \setminus V)\geq 3$, where $V$ is the open set coming from Definition~\ref{d:as}. 
Then there exists a neighborhood ${\cal U}$ of $\Gamma $ in $\emCpt(\emRat)$ with respect to 
the Hausdorff metric such that 
 each $\Gamma '\in {\cal U}$ is mean stable. 
\end{lem}

\begin{prop}
\label{p:hyppjke}
Let $\G \in \mbox{{\em Cpt}}({\emRatp})$.   
Suppose that $J_{\ker }(\langle \G \rangle)=\emptyset $ and $\langle \G\rangle $ is semi-hyperbolic. 
Then there exists an open neighborhood $U$ of $\G $ in 
$ \mbox{{\em Cpt}}(\mbox{{\em Rat}})$ such that for each 
$\G '\in U$, $\G '$ is mean stable and 
$J_{\ker}(\langle \G '\rangle )=\emptyset .$  
\end{prop}
\begin{rem}
\label{r:hyppjke}
Let $\G \in \mbox{Cpt}({\Ratp})$.   
Suppose that $J_{\ker }(\langle \G \rangle)=\emptyset $ and $\langle \G\rangle $ is semi-hyperbolic. 
Then for a small perturbation $\G '$ of $\G $, 
$\G '$ is mean stable, which is the consequence of Proposition~\ref{p:hyppjke}, 
but $\langle \G '\rangle $ may not be semi-hyperbolic. See Proposition~\ref{semihyposcexprop}-(c).  
\end{rem}
\begin{prop}
\label{p:asmt}
Let $\tau \in {\frak M}_{1,c}$ be mean stable. 
Suppose that $J(G_{\tau })\neq \emptyset .$ 
Let $V$ be the set coming from Definition~\ref{d:as}. 
Let $S_{\tau }:=\bigcup _{L\in \emMin (G_{\tau },\CCI )}L.$ 
Then we have all of the following.
\begin{enumerate}
\item \label{p:asmt1}
$S_{\tau }\subset \overline{G_{\tau }^{\ast }(\overline{V})}\subset F(G_{\tau }).$ 
\item \label{p:asmt2} 
Let $W:=\bigcup _{A\in \mbox{{\em Con}}(F(G_{\tau })), A\cap S_{\tau }\neq \emptyset } A.$ 
Let ${\cal U}_{W}:=\{ \varphi \in C_{W}(W)\mid \exists a\in S^{1}, M_{\tau }(\varphi )=a\varphi , 
\varphi \neq 0\} $ Moreover, 
let $\Psi _{W}:\emLSfc\rightarrow C_{W}(W)$ be the map defined by 
$\varphi \mapsto \varphi |_{W}.$ 
Then $\Psi _{W}(\emLSfc)=\mbox{{\em LS}}({\cal U}_{W})$ and 
$\Psi _{W}: \emLSfc \rightarrow \mbox{{\em LS}}({\cal U}_{W})$ 
is a linear isomorphism. 
\item \label{p:asmt3}
Let $Z:=\bigcup _{A\in \mbox{{\em Con}}(F(G_{\tau })), A\cap \overline{G^{\ast }(\overline{V})}\neq \emptyset } A.$ 
Let ${\cal U}_{Z}:=\{ \varphi \in C_{Z}(Z)\mid \exists a\in S^{1}, M_{\tau }(\varphi )=a\varphi , 
\varphi \neq 0\} $ Moreover, 
let $\Psi _{Z}:\emLSfc\rightarrow C_{Z}(Z)$ be the map defined by 
$\varphi \mapsto \varphi |_{Z}.$ 
Then $\Psi _{Z}(\emLSfc)=\mbox{{\em LS}}({\cal U}_{Z})$ and 
$\Psi _{Z}: \emLSfc \rightarrow \mbox{{\em LS}}({\cal U}_{Z})$ 
is a linear isomorphism. 
\end{enumerate}
\end{prop}
\begin{rem}
\label{r:asmtf}
Under the assumptions and notation of Proposition~\ref{p:asmt}, 
we have $\dim _{\CC }C_{W}(W)<\infty $ and $\dim _{\CC }C_{Z}(Z)<\infty .$ 
Thus, in order to seek $\Ufc $ and $\Uvc $, it suffices to consider the 
eigenvectors and eigenvalues of the matrix representation of 
$M_{\tau }$ on the finite dimensional linear space $C_{W}(W)$ or $C_{Z}(Z).$ 
\end{rem}
\begin{rem}
\label{r:gastv}
Let $\Gamma \in \Cpt(\Ratp)$ and let 
$G=\langle \Gamma \rangle .$ 
\begin{enumerate}
\item 
Suppose that $G$ is semi-hyperbolic and $J_{\ker }(G)=\emptyset .$ 
Then by Proposition~\ref{p:hyppjke}, $G$ is mean stable. 
Moreover, by Lemma~\ref{l:shatt1}, the set $V$ in Definition~\ref{d:as} can be taken to be a  
small neighborhood of $A(G)$ in $F(G)$, where 
$A(G):=\overline{G(\{ z\in \CCI \mid \exists g\in G \mbox{ s.t. }g(z)=z, |m(g,z)|<1\} ) }. $
In this case, $\{ A\in \mbox{Con }(F(G))\mid A\cap \overline{G^{\ast }(\overline{V})}\neq \emptyset \} 
=\{ A\in \mbox{Con}(F(G))\mid A\cap A(G)\neq \emptyset .\} .$ 
\item 
Similarly, suppose that $G$ is hyperbolic and $J_{\ker }(G)=\emptyset .$ 
Then by Proposition~\ref{p:hyppjke}, $G$ is mean stable. 
Moreover, by Lemma~\ref{l:shatt1}, the set $V$ in Definition~\ref{d:as} can be taken to be a 
small neighborhood of $P(G)$ in $F(G)$. 
In this case, $\{ A\in \mbox{Con }(F(G))\mid A\cap \overline{G^{\ast }(\overline{V})}\neq \emptyset \} 
=\{ A\in \mbox{Con}(F(G))\mid A\cap P(G)\neq \emptyset .\} .$ 
\end{enumerate}
\end{rem}
\subsection{Necessary and Sufficient conditions to be $J_{\ker }(G_{\tau })\neq \emptyset $}
\label{Suffnec}
In this subsection, we present some results on necessary and sufficient conditions to be 
$J_{\ker }(G_{\tau })\neq \emptyset .$ 
The proofs are given in subsection~\ref{Suffnec}. 

The following is a natural question. 
\begin{ques}
What happens if $J_{\ker }(G_{\tau })\neq \emptyset $?
\end{ques}
\begin{df}
Let $Y$ be a compact metric space with $\dim _{H}(Y)<\infty $ and 
let $\tau \in \frak{M}_{1,c}(\CMX)$. 
Since the function $\gamma \mapsto \dim _{H}(\hat{J}_{\g, \G _{\tau }})$ is Borel measurable and 
since $(\sigma , \tilde{\tau })$ is ergodic, there exists a 
number $a\in [0,\infty )$ such that for $\tilde{\tau }$-a.e. $\gamma \in \Gamma _{\tau }$, 
$\dim _{H}(\hat{J}_{\g ,\Gamma _{\tau }})=a.$ 
We set $\MHDt := a.$  
\end{df}
\begin{rem}
\label{r:mhdt}
Let $\Gamma \in \Cpt(\mbox{Rat}_{+})$ and let 
$G=\langle \Gamma \rangle$.   
Suppose that $G$ is semi-hyperbolic and  
$F(G)\neq \emptyset .$ Then, 
$\gamma \mapsto J_{\g }$ is continuous on $\Gamma ^{\NN }$  with respect to the Hausdorff metric 
(this is non-trivial) 
and 
for each $\gamma \in \Gamma ^{\NN }$, $J_{\g }=\hat{J}_{\g, \Gamma}$ 
(see Lemma~\ref{l:shatt1} and \cite[Theorem 2.14]{S4}). 
 Moreover, there exists 
a constant $0\leq b<2 $ such that for each $\gamma \in \Gamma ^{\NN }$, 
$\dim _{H}(J_{\g })\leq b$ (see Lemma~\ref{l:shatt1} and \cite[Theorem 1.16]{S7}).  
Note that if we do not assume semi-hyperbolicity, then $\gamma \mapsto J_{\gamma }$ is not 
continuous in general. 
\end{rem}
\begin{thm}
\label{t:dimjpt0jkn}
Let $\tau \in {\frak M}_{1,c}(\emRatp )$. 
Suppose that $G_{\tau }$ is semi-hyperbolic and  
$F(G_{\tau })\neq \emptyset .$ Then, 
we have all of the following.
\begin{enumerate}
\item \label{t:dimjpt0jkn1}
$\dim _{H}(J_{pt}^{0}(\tau ))\leq \emMHDt <2.$ 
\item \label{t:dimjpt0jkn2}
$J_{\ker }(G_{\tau }) \subset J_{pt}^{0}(\tau ).$  
\item \label{t:dimjpt0jkn3}
$F_{meas}(\tau )={\frak M}_{1}(\CCI )$ if and only if 
$J_{\ker }(G_{\tau })=\emptyset .$ 
If $J_{\ker }(G_{\tau })\neq \emptyset $, 
then $J_{meas}(\tau )={\frak M}_{1}(\CCI ).$ 
\item \label{t:dimjpt0jkn4}
If, in addition to the assumption, $\sharp \Gamma _{\tau } <\infty $, 
then we have the following.
\begin{enumerate}
\item[{\em (a)}] \label{t:dimjpt0jkn4a}
$G_{\tau }^{-1}(J_{\ker }(G_{\tau }))\subset J_{pt}^{0}(\tau ).$
\item[{\em (b)}]
 \label{t:dimjpt0jkn4b}
Either $F_{meas}(\tau )={\frak M}_{1}(\CCI ) $ or $J_{pt}(\tau )=J(G_{\tau }).$ 
\end{enumerate} 
\end{enumerate}
\end{thm}
\begin{rem}
\label{r:hypdeg2sh}
Let $G$ be a hyperbolic rational semigroup with $G\cap \Ratp \neq \emptyset .$  
Then, $G$ is semi-hyperbolic and $F(G)\neq \emptyset .$ 
\end{rem}
\subsection{Singular properties and regularity of 
non-constant finite linear combinations of unitary eigenvectors of $M_{\tau }$}
\label{Singular}
In this subsection, we present some 
results on singular properties and regularity of non-constant 
finite linear combinations $\varphi $ of unitary eigenvectors of $M_{\tau }:C(\CCI )\rightarrow 
C(\CCI )$. It turns out that under certain conditions, such $\varphi $ 
is non-differentiable at each point of an uncountable dense subset of $J(G_{\tau })$ 
(see Theorem~\ref{t:hnondiff}).  Moreover, we investigate the pointwise H\"{o}lder exponent 
of such $\varphi $ (see Theorem~\ref{t:hnondiff} and Theorem~\ref{t:hdiffornd}).   
The proofs are given in subsection~\ref{pfSingular}. 
\begin{lem}
\label{l:disjker}
Let $m\in \NN $ with $m\geq 2.$ 
Let $Y$ be a compact metric space and let $h_{1},h_{2},\ldots ,h_{m}\in \emOCMX.$ 
Let $G=\langle h_{1},\ldots ,h_{m}\rangle .$ Suppose that 
for each $(i,j)$ with $i\neq j$, $h_{i}^{-1}(J(G))\cap h_{j}^{-1}(J(G))=\emptyset .$ 
Then, $J_{\ker }(G)=\emptyset .$ 
\end{lem}
\begin{df}
For each $m\in \NN $, we set 
${\cal W}_{m}:=\{ (p_{1},\ldots ,p_{m})\in (0,1)^{m}\mid \sum _{j=1}^{m}p_{j}=1\} .$ 
\end{df}
\begin{lem}
\label{l:lsncnonc}
Let $m\in \NN $ with $m\geq 2.$ 
Let $h=(h_{1},\ldots ,h_{m})\in (\emRat )^{m}$ and let 
$G=\langle h_{1},\ldots ,h_{m}\rangle .$ 
Let $p=(p_{1},\ldots ,p_{m})\in {\cal W}_{m}$ and let 
$\tau =\sum _{j=1}^{m}p_{j}\delta _{h_{j}}.$ 
Suppose that $J(G)\neq \emptyset $ and that 
$h_{i}^{-1}(J(G))\cap h_{j}^{-1}(J(G))=\emptyset $ for 
each $(i,j)$ with $i\neq j.$ 
Then {\em int}$(J(G))=\emptyset $ and for each $\varphi \in (\mbox{{\em LS}}({\cal U}_{f,\tau }(\CCI )))_{nc}$, 
$$J(G)=\{ z\in \CCI \mid \mbox{ for any neighborhood } U \mbox{ of } z,  
\varphi |_{U} \mbox{ is non-constant} \} .$$ 
\end{lem}

\begin{df}
Let $U$ be a domain in $\CCI $ and let 
$g:U\rightarrow \CCI $ be a meromorphic function. 
For each $z\in U$, we denote by $\| g'(z)\| _{s}$ the 
norm of the derivative of $g$ at $z$ with respect to the spherical metric. 
\end{df}
\begin{df}
\label{d:dfu}
Let $m\in \NN .$ 
Let $h=(h_{1},\ldots, h_{m})\in (\Rat)^{m}$ be an element such that 
$h_{1},\ldots ,h_{m}$ are 
mutually distinct. 
We set 
$\Gamma := \{ h_{1},\ldots ,h_{m}\} .$ 
Let $f:\Gamma ^{\NN }\times 
\CCI \rightarrow \Gamma ^{\NN }\times \CCI $ be the 
skew product associated with $\Gamma .$  
Let $\mu \in {\frak M}_{1}(\Gamma ^{\NN }\times \CCI )$ be an $f$-invariant Borel probability measure. 
For each $p=(p_{1},\ldots ,p_{m})\in {\cal W}_{m}$, 
we define a function $\tilde{p}:\Gamma ^{\NN }\times \CCI \rightarrow \RR $ by 
$\tilde{p}(\gamma ,y):=p_{j}$  if $\gamma _{1}=h_{j} $ 
(where $\gamma =(\gamma _{1},\gamma _{2},\ldots )$), and  we set 
$$u(h,p,\mu ):= 
\frac{-(\int _{\Gamma ^{\NN }\times \CCI }\log \tilde{p}\ d\mu  )}
{\int _{\Gamma ^{\NN }\times \CCI }\log \| f'\| _{s} \ d\mu }$$ 
(when the integral of the denominator converges). 
\end{df}
\begin{df} \label{d:green} 
Let $h=(h_{1},\ldots , h_{m})\in {\cal P}^{m}$ be an element such that 
$h_{1},\ldots ,h_{m}$ are 
mutually distinct. 
We set 
$\Gamma := \{ h_{1},\ldots ,h_{m}\} .$ 
For any $(\gamma ,y)\in \Gamma ^{\NN }\times \CC $, 
let $G_{\gamma }(y):= \lim _{n\rightarrow \infty }\frac{1}{\deg (\gamma _{n,1})}
\log ^{+}|\gamma _{n,1}(y)|$, 
where $\log ^{+}a:=\max\{\log a,0\} $ for each $a>0.$  
By the arguments in \cite{Se}, for each $\gamma \in \Gamma ^{\NN }$, 
$G_{\gamma }(y) $ exists, 
$G_{\gamma }$ is subharmonic on $\CC $, and 
$G_{\gamma }|_{A_{\infty ,\gamma }}$ is equal to the Green's function on 
$A_{\infty ,\gamma }$ with pole at $\infty $.
Moreover, $(\gamma ,y)\mapsto G_{\gamma }(y)$ is continuous on $\Gamma ^{\NN }\times \CC .$ 
Let $\mu _{\gamma }:=dd^{c}G_{\gamma }$, where $d^{c}:=\frac{i}{2\pi }(\overline{\partial }-\partial ).$ 
Note that by the argument in \cite{J1,J2}, 
$\mu _{\gamma }$ is a Borel probability measure on $J_{\gamma }$ such that 
$\mbox{supp}\, \mu _{\gamma }=J_{\gamma }.$ 
Furthermore, for each $\gamma \in \Gamma ^{\NN }$, 
let $\Omega (\gamma )=\sum _{c} G_{\gamma }(c)$, where $c$ runs over all critical points of 
$\gamma _{1}$ in $\CC $, counting multiplicities.   
\end{df}
\begin{rem}
\label{r:maxrelent}
Let $h=(h_{1},\ldots ,h_{m})\in (\Ratp)^{m}$ be an element such that 
$h_{1},\ldots ,h_{m}$ are mutually distinct. 
Let $\Gamma =\{ h_{1},\ldots ,h_{m}\} $ and 
let $f:\Gamma ^{\NN }\times \CCI \rightarrow \Gamma ^{\NN }\times \CCI $ 
be the skew product map associated with $\Gamma .$ 
Moreover, let $p=(p_{1},\ldots ,p_{m})\in {\cal W}_{m}$ and 
let $\tau =\sum _{j=1}^{m}p_{j}\delta _{h_{j}}\in {\frak M}_{1}(\Gamma ).$ 
Then, there exists a unique $f$-invariant ergodic Borel probability measure 
$\mu $ on $\Gamma ^{\NN }\times \CCI $ such that $\pi _{\ast }(\mu )=\tilde{\tau }$ and   
$h_{\mu }(f|\sigma )=\max _{\rho \in {\frak E}_{1}(\Gamma ^{\NN }\times \CCI ): 
f_{\ast }(\rho )=\rho, \pi _{\ast }(\rho )=\tilde{\tau} }  h_{\rho }(f|\sigma )=\sum _{j=1}^{m}p_{j}\log (\deg (h_{j}))$, 
where $h_{\rho }(f|\sigma )$ denotes the relative metric entropy 
of $(f,\rho )$ with respect to $(\sigma, \tilde{\tau })$, and 
${\frak E}_{1}(\cdot )$ denotes the space of ergodic measures (see \cite{S3}).  
This $\mu $ is called the {\bf maximal relative entropy measure} for $f$ with respect to 
$(\sigma ,\tilde{\tau }).$   
\end{rem}
\begin{df}
\label{d:phe}
Let $V$ be a non-empty open subset of $\CCI .$ Let $\varphi :V \rightarrow \CC $ be a function and 
let $y\in V $ be a point. Suppose that $\varphi $ is bounded around $y.$ 
Then we set  
$$\mbox{H\"{o}l}(\varphi ,y):= 
\inf \{ \beta \in \RR \mid \limsup _{z\rightarrow y}
\frac{|\varphi (z)-\varphi (y)|}{d(z,y)^{\beta }}=\infty \}, $$
where $d$ denotes the spherical distance. 
This is called the {\bf pointwise H\"{o}lder exponent of $\varphi $ at $y.$} 
\end{df}
\begin{rem} 
If $\mbox{H\"{o}l}(\varphi ,y)<1$, then 
$\varphi $ is non-differentiable at $y.$ 
If 
 $\mbox{H\"{o}l}(\varphi ,y)>1$, then 
$\varphi $ is differentiable at $y$ and the derivative at $y$ is equal to $0.$
\end{rem}
We now present a result on non-differentiability of non-constant finite linear combinations of 
unitary eigenvectors of $M_{\tau }$ 
at almost every point in $J(G_{\tau })$ with respect to the projection of the 
maximal relative entropy measure. 
\begin{thm}[{\bf Non-differentiability of $\varphi \in (\LSfc )_{nc}$ 
at points in $J(G_{\tau })$}]
\label{t:hnondiff}
Let $m\in \NN $ with $m\geq 2.$ 
Let $h=(h_{1},\ldots ,h_{m})\in (\emRatp)^{m}$ and we set 
$\Gamma := \{ h_{1},h_{2},\ldots ,h_{m}\} .$ 
Let $G=\langle h_{1},\ldots ,h_{m}\rangle .$ 
Let $p=(p_{1},\ldots ,p_{m})\in {\cal W}_{m}.$  
Let $f:\Gamma ^{\NN }\times \CCI \rightarrow \Gamma ^{\NN }\times \CCI $ be the  
skew product associated with $\Gamma .$  
Let 
$\tau := \sum _{j=1}^{m}p_{j}\delta _{h_{j}}\in {\frak M}_{1}(\Gamma )
\subset {\frak M}_{1}({\cal P }).$
Let 
$\mu \in {\frak M}_{1}(\Gamma ^{\NN }\times \CCI )$ be the maximal relative entropy measure 
 for $f:\Gamma ^{\NN }\times \CCI \rightarrow \Gamma ^{\NN }\times \CCI $ with respect to 
 $(\sigma ,\tilde{\tau }).$ 
Moreover, let 
$\lambda := (\pi _{\CCI })_{\ast }(\mu )\in {\frak M}_{1}(\CCI ).$ 
Suppose that  
$G$ is hyperbolic, and 
$h_{i}^{-1}(J(G))\cap h_{j}^{-1}(J(G))=\emptyset $ for each 
$(i,j)$ with $i\neq j$. 
Then, we have all of the following. 
\begin{enumerate}
\item \label{t:hnondiff0} $G_{\tau }=G$ is mean stable and $J_{\ker }(G)=\emptyset .$ 
\item \label{t:hnondiff1}
$0<\dim _{H}(J(G))<2 .$ 
\item \label{t:hnondiff2} 
{\em supp} $\lambda =J(G).$
\item \label{t:hnondiff3} 
For each $z\in J(G)$, $\lambda (\{ z\} )=0.$ 
\item \label{t:hnondiff4} 
There exists a Borel subset $A$ of $J(G)$ with $\lambda (A)=1$ such that 
for each $z_{0}\in A$ and each $\varphi \in (\mbox{{\em LS}}({\cal U}_{f,\tau }(\CCI )))_{nc}$, 
$
\mbox{{\em H\"{o}l}}(\varphi , z_{0})=
u(h,p,\mu ).$ 
\item \label{t:hnondiff4-1}
If $h=(h_{1},\ldots ,h_{m})\in {\cal P}^{m}$, then 
$$
u(h,p,\mu )=\frac{-(\sum _{j=1}^{m}p_{j}\log p_{j})}
{\sum _{j=1}^{m}p_{j}\log \deg (h_{j})+\int _{\Gamma ^{\NN }}\Omega (\gamma )\ d\tilde{\tau }(\gamma )}
$$
and 
\begin{align*}
2 > & \dim _{H}(\{ z\in J(G) \mid \mbox{ for each }\varphi \in (\mbox{{\em LS}}({\cal U}_{f,\tau }(\CCI )))_{nc},\ 
\emHol (\varphi ,z)=u(h,p,\mu )\} )\\ 
\geq & \frac{\sum _{j=1}^{m}p_{j}\log \deg (h_{j})-\sum _{j=1}^{m}p_{j}\log p_{j}}
{\sum _{j=1}^{m}p_{j}\log \deg (h_{j})+\int _{\Gamma ^{\NN }}\Omega (\gamma )\ d\tilde{\tau }(\gamma )}>0. 
\end{align*} 
\item \label{t:hnondiff5} 
Suppose $h=(h_{1},\ldots ,h_{m})\in {\cal P}^{m}.$ 
Moreover, suppose that at least one of the following {\em (a)}, {\em (b)}, and {\em (c)} holds:
{\em (a)} $\sum _{j=1}^{m}p_{j}\log (p_{j}\deg (h_{j}))>0.$
{\em (b)} $P^{\ast }(G)$ is bounded in $\CC .$ 
{\em (c)} $m=2.$ 
Then, $u(h,p,\mu )<1$ and  
for each non-empty open subset $U$ of $J(G)$ there exists an uncountable dense subset $A_{U}$ of $U$ such that 
for each $z\in A_{U}$ and each $\varphi \in (\mbox{{\em LS}}({\cal U}_{f,\tau }(\CCI )))_{nc}$, 
$\varphi $ is non-differentiable at $z.$ 
\end{enumerate}
 
\end{thm}
\begin{rem}
By Theorems~\ref{t:mtauspec} and \ref{t:hnondiff}, it follows that 
under the assumptions of Theorem~\ref{t:hnondiff}, the chaos 
of the averaged system disappears in the $C^{0}$ ``sense'', but it 
remains in the $C^{1}$ ``sense''. 
\end{rem}
We now present a result on the representation of pointwise H\"{o}lder 
exponent of $\varphi \in (\LSfc )_{nc}$ at almost every point 
in $J(G_{\tau })$ with respect to the 
$\delta $-dimensional Hausdorff measure, where $\delta =\dim _{H}(J(G_{\tau })).$ 
\begin{thm}
\label{t:hdiffornd}
Let $m\in \NN $ with $m\geq 2.$ 
Let $h=(h_{1},\ldots ,h_{m})\in (\emRatp)^{m}$ and we set 
$\Gamma := \{ h_{1},h_{2},\ldots ,h_{m}\} .$ 
Let $G=\langle h_{1},\ldots ,h_{m}\rangle .$ 
Let $p=(p_{1},\ldots ,p_{m})\in {\cal W}_{m}.$  
Let $f:\Gamma ^{\NN }\times \CCI \rightarrow \Gamma ^{\NN }\times \CCI $ be the  
skew product associated with $\Gamma .$  
Let 
$\tau := \sum _{j=1}^{m}p_{j}\delta _{h_{j}}\in {\frak M}_{1}(\Gamma )
\subset {\frak M}_{1}(\emRatp).$
Suppose that  
$G$ is hyperbolic and 
$h_{i}^{-1}(J(G))\cap h_{j}^{-1}(J(G))=\emptyset $ for each 
$(i,j)$ with $i\neq j$. 
Let $\delta := \dim _{H}J(G)$ and let $H^{\delta }$ be the 
$\delta$-dimensional Hausdorff measure.  
Let $\tilde{L}: C(\tilde{J}(f))\rightarrow C(\tilde{J}(f))$ be the operator 
defined by $\tilde{L}(\varphi )(z)=\sum _{f(w)=z}\varphi (w)\| f'(w)\| _{s} ^{-\delta }.$ 
Moreover, let $L:C(J(G))\rightarrow C(J(G))$ be the operator defined by 
$L(\varphi )(z)=\sum _{j=1}^{m}\sum _{h_{j}(w)=z}\varphi (w)\| h_{j}'(w)\| _{s}^{-\delta }.$ 
Then, we have all of the following.
\begin{enumerate}
\item \label{t:hdiffornd0} 
$G_{\tau }=G$ is mean stable and $J_{\ker }(G)=\emptyset .$ 
\item \label{t:hdiffornd1}
There exists a unique element 
$\tilde{\nu } \in {\frak M}_{1}(\tilde{J}(f))$ such that 
$\tilde{L}^{\ast } (\tilde{\nu })=\tilde{\nu }.$ 
Moreover, the limits 
$\tilde{\alpha }=\lim _{n\rightarrow \infty }\tilde{L}^{n}(1)\in C(\tilde{J}(f))$ and 
$\alpha =\lim _{n\rightarrow \infty }L^{n}(1)\in C(J(G))$ exist, 
where $1$ denotes the constant function taking its value $1$. 
\item \label{t:hdiffornd2}
Let $\nu :=(\pi _{\CCI })_{\ast }(\tilde{\nu })\in {\frak M}_{1}(J(G)).$ Then 
$0<\delta <2$, $0<H^{\delta }(J(G))<\infty $, and 
$\nu =\frac{H^{\delta }}{H^{\delta }(J(G))}.$
\item \label{t:hdiffornd3} 
Let $\tilde{\rho }:= \tilde{\alpha }\tilde{\nu }\in {\frak M}_{1}(\tilde{J}(f)).$  
Then $\tilde{\rho }$ is $f$-invariant and ergodic. Moreover, $\min _{z\in J(G)}\alpha (z)>0.$ 
\item \label{t:hdiffornd4} 
There exists a Borel subset of $A$ of $J(G)$ with $H^{\delta }(A)=H^{\delta }(J(G))$ 
such that for each $z_{0}\in A$ and each $\varphi \in (\mbox{{\em LS}}({\cal U}_{f,\tau }(\CCI )))_{nc},$   
$$\mbox{{\em H\"{o}l}}(\varphi ,z_{0})=
u(h,p,\tilde{\rho })=
\frac{-\sum _{j=1}^{m}(\log p_{j})\int _{h_{j}^{-1}(J(G))}\alpha (y)\ dH^{\delta }(y)}{\sum _{j=1}^{m}\int _{h_{j}^{-1}(J(G))}
\alpha (y)\log \| h_{j}'(y)\| _{s}\ dH^{\delta }(y)}.$$
\end{enumerate}
\end{thm}
\begin{rem}
\label{r:tinftyuns}
Let $m\in \NN $ with $m\geq 2.$ 
Let $h=(h_{1},\ldots ,h_{m})\in {\cal P}^{m}$ and let 
$G=\langle h_{1},\ldots ,h_{m}\rangle .$ 
Let $p=(p_{1},\ldots ,p_{m})\in {\cal W}_{m}$ and 
let $\tau =\sum _{j=1}^{m}p_{j}\delta _{h_{j}}.$ 
Suppose that $\hat{K}(G)\neq \emptyset $, 
$G$ is hyperbolic, and $h_{i}^{-1}(J(G))\cap h_{j}^{-1}(J(G))=\emptyset $ 
for each $(i,j)$ with $i\neq j.$ Then, by 
Lemma~\ref{l:disjker} and Theorem~\ref{kerJthm2}, 
$T_{\infty ,\tau }\in (\mbox{LS}({\cal U}_{f,\tau }(\CCI )))_{nc}.$ 
\end{rem}
\begin{rem}
\label{r:dinondi}
Let $m\in \NN $ with $m\geq 2.$ 
Let $h=(h_{1},\ldots ,h_{m})\in {\cal P}^{m}$ and we set 
$\Gamma := \{ h_{1},\ldots ,h_{m}\} .$ 
Let $G=\langle h_{1},\ldots ,h_{m}\rangle .$ 
Let $p=(p_{1}, \ldots ,p_{m})\in {\cal W}_{m}.$  
Let $f:\Gamma ^{\NN }\times \CCI \rightarrow \Gamma ^{\NN }\times \CCI $ be the  
skew product associated with $\Gamma .$  
Let 
$\tau := \sum _{j=1}^{m}p_{j}\delta _{h_{j}}\in {\frak M}_{1}(\Gamma )
\subset {\frak M}_{1}({\cal P }).$
Suppose that 
$\hat{K}(G)\neq \emptyset $, $G$ is hyperbolic, and 
$h_{i}^{-1}(J(G))\cap h_{j}^{-1}(J(G))=\emptyset $ for each 
$(i,j)$ with $i\neq j.$ Moreover, suppose we have at least one of the following 
(a),(b),(c):  
(a) $\sum _{j=1}^{m}p_{j}\log (p_{j}\deg (h_{j}))>0$. (b) $P^{\ast }(G)$ is bounded in $\CC .$ 
(c) $m=2.$ 
Then, 
combining Theorem~\ref{t:hnondiff}, Theorem~\ref{t:hdiffornd}, and Remark~\ref{r:tinftyuns}, 
it follows that there exists a number $q>0$ such that 
if $p_{1}<q$, then we have all of the following. 
\begin{enumerate}
\item Let $\mu $ be the maximal relative entropy measure 
for $f$ with respect to $(\sigma ,\tilde{\tau }).$ 
Let $\lambda =(\pi _{\CCI})_{\ast }\mu \in {\frak M}_{1}(J(G)).$ 
Then for $\lambda $-a.e. $z_{0}\in J(G)$ and for any $\varphi \in \LSfc _{nc}$ (e.g., $\varphi =T_{\infty ,\tau }$),  
$\limsup _{n\rightarrow \infty }\frac{|\varphi (y)-\varphi (z_{0})|}
{|y-z_{0}|}=\infty $ and $\varphi $ is not differentiable at $z_{0}.$ 
\item 
Let $\delta =\dim _{H}(J(G))$ and let $H^{\delta }$ be the $\delta $-dimensional 
Hausdorff measure. Then $0<H^{\delta }(J(G))<\infty $ and 
for $H^{\delta }$-a.e. $z_{0}\in J(G)$ and for any $\varphi \in \LSfc $ (e.g., $\varphi =T_{\infty ,\tau }$),  
$\limsup _{n\rightarrow \infty }\frac{|\varphi (y)-\varphi (z_{0})|}
{|y-z_{0}|}=0 $ and $\varphi $ is differentiable at $z_{0}.$ 
\end{enumerate}
\end{rem}
Combining Theorem~\ref{t:mtauspec} and Theorem~\ref{t:hnondiff}, 
we obtain the following result. 
\begin{cor}
\label{c:hypdisjc}
Let $m\in \NN $ with $m\geq 2.$ 
Let $h=(h_{1},\ldots ,h_{m})\in {\cal P}^{m}$ and we set 
$\Gamma := \{ h_{1},\ldots ,h_{m}\} .$ 
Let $G=\langle h_{1},\ldots ,h_{m}\rangle .$ 
Let $p=(p_{1}, \ldots ,p_{m})\in {\cal W}_{m}.$  
Let $f:\Gamma ^{\NN }\times \CCI \rightarrow \Gamma ^{\NN }\times \CCI $ be the  
skew product associated with $\Gamma .$  
Let 
$\tau := \sum _{j=1}^{m}p_{j}\delta _{h_{j}}\in {\frak M}_{1}(\Gamma )
\subset {\frak M}_{1}({\cal P }).$
Suppose that 
$\hat{K}(G)\neq \emptyset $, $G$ is hyperbolic, and 
$h_{i}^{-1}(J(G))\cap h_{j}^{-1}(J(G))=\emptyset $ for each 
$(i,j)$ with $i\neq j.$ Moreover, suppose we have at least one of 
the following {\em (a), (b), (c):}   
{\em (a)} $\sum _{j=1}^{m}p_{j}\log (p_{j}\deg (h_{j}))>0.$ 
{\em (b)} $P^{\ast }(G)$ is bounded in $\CC .$ 
{\em (c)} $m=2.$  
Let $\varphi \in C(\CCI )$. Then, we have exactly one of the following {\em (i)} and {\em (ii)}.
\begin{itemize}
\item[{\em (i)}]
There exists a constant function $\zeta \in C(\CCI )$ such that 
$M_{\tau }^{n}(\varphi )\rightarrow  \zeta $ as $n\rightarrow \infty $ in $C(\CCI ).$ 
\item[{\em (ii)}]
There exists an element $\psi \in (\mbox{{\em LS}}({\cal U}_{f,\tau }(\CCI )))_{nc}$ 
and a number $l\in \NN $ such that 
\begin{itemize}
\item 
$M_{\tau }^{l}(\psi )=\psi $, 
\item 
$\{ M_{\tau }^{j}(\psi )\} _{j=0}^{l-1}\subset (\mbox{{\em LS}}({\cal U}_{f,\tau }(\CCI )))_{nc}\subset C_{F(G)}(\CCI )$, 
\item 
there exists an uncountable dense subset $A$ of $J(G)$  such that for each 
$z_{0}\in A$ and each $j$, $M_{\tau }^{j}(\psi )$ is not differentiable at $z_{0}$, and 
\item $M_{\tau }^{nl+j}(\varphi )\rightarrow M_{\tau }^{j}(\psi )$ as $n\rightarrow \infty $ for each 
$j=0,\ldots ,l-1.$ 
\end{itemize} 
\end{itemize}
\end{cor}
We present a result on H\"{o}lder continuity of $\varphi \in \LSfc .$ 
\begin{thm}
\label{t:hholder}
Let $m\in \NN $ with $m\geq 2.$ 
Let $h=(h_{1},\ldots ,h_{m})\in \emRatp^{m}$ and we set 
$\Gamma := \{ h_{1},\ldots ,h_{m}\} .$ 
Let $G=\langle h_{1},\ldots ,h_{m}\rangle .$ 
Let $p=(p_{1}, \ldots ,p_{m})\in {\cal W}_{m}$ and 
let $\tau := \sum _{j=1}^{m}p_{j}\delta _{h_{j}}\in {\frak M}_{1}(\Gamma )
\subset {\frak M}_{1}(\emRatp ).$ 
Suppose that  
 $G$ is hyperbolic and 
$h_{i}^{-1}(J(G))\cap h_{j}^{-1}(J(G))=\emptyset $ for each 
$(i,j)$ with $i\neq j.$  
Then, $G$ is mean stable and there exists an $\alpha >0$ such that 
for each $\varphi \in \mbox{{\em LS}}({\cal U}_{f,\tau }(\CCI ))$, 
$\varphi :\CCI \rightarrow [0,1]$ is $\alpha $-H\"{o}lder continuous on 
$\CCI .$ 
\end{thm}
\begin{rem}
In the proof of Theorem~\ref{t:hnondiff}, we use the Birkhoff ergodic theorem and the 
Koebe distortion theorem, 
in order to show that for each $\varphi \in (\mbox{LS}({\cal U}_{f,\tau }))_{nc}$, 
$\mbox{H\"{o}l}(\varphi ,z_{0})=u(h,p,\mu ).$ 
Moreover, we apply potential theory in order to calculate 
$u(h,p,\mu )$ by using $p$, $\deg (h_{j})$, and $\Omega (\gamma ).$  
\end{rem}
\section{Tools}
\label{Tools} 
In this section, we give some basic tools to prove the main results.

\begin{lem}[Lemma 0.2 in \cite{S4}]
\label{l:bss}
Let $Y$ be a compact metric space and let $\Gamma \in \emCpt(\emOCMX)$.  
Let $G=\langle \Gamma \rangle .$  
Then, $J(G)=\bigcup _{h\in \Gamma }h^{-1}(J(G)).$ In particular, if 
$G=\langle h_{1},\ldots ,h_{m}\rangle \subset \emOCMX $, then 
$J(G)=\bigcup _{j=1}^{m}h_{j}^{-1}(J(G)).$ This property is called the {\bf backward self-similarity}.  
\end{lem}
\begin{proof}
By Lemma~\ref{ocminvlem}, $J(G)\supset \bigcup _{h\in \Gamma }h^{-1}(J(G)).$ 
By using the method in the proof of \cite[Lemma 0.2]{S4}, 
we easily see that $J(G)\subset \bigcup _{h\in \Gamma }h^{-1}(J(G)).$ 
Thus, $J(G)= \bigcup _{h\in \Gamma }h^{-1}(J(G)).$
\end{proof}
\noindent {\bf Notation:} 
Let $Y$ be a topological space. Let $\mu \in {\frak M}_{1}(Y)$ and 
let $\varphi :Y\rightarrow \RR $ be a bounded continuous function. 
Then we set $\mu (\varphi ):= \int _{Y} \varphi \ d\mu .$

\begin{lem}
\label{FJmeaslem1}
Let $Y$ be a compact metric space and 
let $\tau \in {\frak M}_{1}(\emCMX).$ 
Then, we have the following.
\begin{enumerate}
\item
\label{FJmeaslem1-1} 
$(M_{\tau }^{\ast })^{-1}(F_{meas}(\tau ))\subset F_{meas}(\tau )$,\ 
and $(M_{\tau }^{\ast })^{-1}(F_{meas}^{0}(\tau ))\subset F_{meas}^{0}(\tau ).$ 
\item 
\label{FJmeaslem1-2}
Let $y\in Y $ be a point. Then, 
$y\in F_{pt}(\tau )$ if and only if 
 there exists a neighborhood $U$ of $y$ in 
$Y $ such that for any $\phi \in C(Y )$, 
the sequence $\{ z\mapsto M_{\tau }^{n}(\phi )(z)\} _{n\in \NN }$ 
of functions on $U$ is equicontinuous on $U.$ 
Similarly, $y\in F_{pt}^{0}(\tau )$ if and only if 
for any $\phi \in C(Y )$, the sequence 
$\{ z\mapsto M_{\tau }^{n}(\phi )(z)\} _{\in \NN }$ of 
functions on $Y $ is equicontinuous at the one point 
$y.$ 
\item 
\label{FJmeaslem1-3} 
$F_{meas}(\tau )\cap Y \subset F_{pt}(\tau ).$ 
\item 
\label{FJmeaslem1-4}
$F_{meas}^{0}(\tau )\cap Y =F_{pt }^{0}(\tau ).$ 
\item 
\label{FJmeaslem1-5}
$F(G_{\tau })\subset F_{pt}(\tau ).$ 
\item 
\label{FJmeaslem1-6}
$F_{pt}^{0}(\tau )=Y $ if and only if 
$F_{meas}(\tau )={\frak M}_{1}(Y ).$  
\end{enumerate}
\end{lem}
\begin{proof}
Since $M_{\tau }^{\ast }:{\frak M}_{1}(Y)\rightarrow {\frak M}_{1}(Y)$ is continuous, 
it is easy to see that statement \ref{FJmeaslem1-1} holds. 

Let $\{ \phi _{j}\} _{j\in \NN }$ be a dense subset of $C(Y)$ and 
let $d_{0}$ be as in Definition~\ref{d:d0}. 
 We now prove statement \ref{FJmeaslem1-2}. 
 Let $y\in F_{pt}(\tau ).$ Then there exists a neighborhood 
 $U$ of $y$ in $X$ with the following property that 
 for each $z\in U$ and each $\epsilon >0$ there exists a 
 $\delta =\delta (z,\epsilon )>0$ such that 
 if $d(z,z')<\delta, z'\in U$ then for each $n\in \NN $, 
$d_{0}((M_{\tau }^{\ast })^{n}(\delta _{z}), (M_{\tau }^{\ast })^{n}(\delta _{z'}))<\epsilon .$ 
Let $z\in U$ and let $\epsilon >0.$ 
Let $\phi \in C(Y)$ be any element and let $\phi _{j}$ be such that 
$\| \phi -\phi _{j}\| _{\infty }<\epsilon .$ 
Let $\delta =\delta (z,\frac{\epsilon }{2^{j}}).$ 
Then for each $n\in \NN $ and each $z'\in U$ with $d(z,z')<\delta $,  
$\frac{|(M_{\tau }^{\ast })^{n}(\delta _{z}))(\phi _{j})-((M_{\tau }^{\ast })^{n}(\delta _{z'}))(\phi _{j})|}
{1+ | ((M_{\tau }^{\ast })^{n}(\delta _{z}))(\phi _{j})- ((M_{\tau }^{\ast })^{n}(\delta _{z'}))(\phi _{j})|}<\epsilon .$ 
Hence for each $n\in \NN $ and each $z'\in U$ with $d(z,z')<\delta $,  
$|((M_{\tau }^{\ast })^{n}(\delta _{z}))(\phi _{j})-((M_{\tau }^{\ast })^{n}(\delta _{z'}))(\phi _{j})|
<\frac{\epsilon }{1-\epsilon }.$   
It follows that for each $n\in \NN $ and each $z'\in U$ with $d(z,z')<\delta $,  
\begin{align*}
|((M_{\tau }^{\ast })^{n}(\delta _{z}))(\phi )- ((M_{\tau }^{\ast })^{n}(\delta _{z'}))(\phi )|
& \leq  |   ((M_{\tau }^{\ast })^{n}(\delta _{z}))(\phi )-  ((M_{\tau }^{\ast })^{n}(\delta _{z}))(\phi _{j})| \\ 
& \ \ + | ((M_{\tau }^{\ast })^{n}(\delta _{z}))(\phi _{j})- ((M_{\tau }^{\ast })^{n}(\delta _{z'}))(\phi _{j})| \\ 
& \ \ + | ((M_{\tau }^{\ast })^{n}(\delta _{z'}))(\phi _{j})-  ((M_{\tau }^{\ast })^{n}(\delta _{z'}))(\phi )| \\ 
& \leq 2\epsilon +\frac{\epsilon }{1-\epsilon }.
\end{align*}
Therefore, $\{ z\mapsto M_{\tau }^{n}(\phi )(z)\} _{n\in \NN } $ is equicontinuous on $U.$ 
To show the converse, let $y\in X$ and suppose that there exists a neighborhood $U$ of $y$ in $X$ 
such that  for any $\phi \in C(Y)$, 
$\{ z\mapsto M_{\tau }^{n}(\phi )(z)\} _{n\in \NN } $ is equicontinuous on $U.$
Let $z\in U.$ 
For each $\epsilon >0$, there exists an $n_{0}\in \NN $ such that 
$\sum _{n\geq n_{0}}\frac{1}{2^{n}}<\epsilon .$ 
Moreover, there exists a $\delta >0$ such that 
if $z'\in U$ and $d(z,z')<\delta $, then 
for each $n\in \NN $ and each $j=1,\ldots , n_{0}$, 
$|M_{\tau }^{n}(\phi _{j})(z)-M_{\tau }^{n}(\phi _{j})(z')|<\epsilon /n_{0}.$ 
It follows that if $z'\in U$  and $d(z,z')<\delta $, then for each $n\in \NN $,  
$d_{0}((M_{\tau }^{\ast })^{n}(\delta _{z}), (M_{\tau }^{\ast })^{n}(\delta _{z'}))\leq 2\epsilon .$ 
Therefore, $y\in F_{pt}(\tau ).$ 
Thus, we have proved that $y\in F_{pt}(\tau )$ if and only if 
there exists a neighborhood $U$ of $y$ such that 
for any $\phi \in C(Y)$, 
$\{ z\mapsto M_{\tau }^{n}(\phi )(z)\} _{n\in \NN } $ is equicontinuous on $U.$ 
Similarly, we can prove that 
$y\in F_{pt}^{0}(\tau )$ if and only if 
for any $\phi \in C(Y)$, 
$\{ z\mapsto M_{\tau }^{n}(\phi )(z)\} _{n\in \NN } $ is equicontinuous at the one point $y.$ 
Hence, we have proved statement~\ref{FJmeaslem1-2}.  

 Statement \ref{FJmeaslem1-3} easily follows from the definition of $F_{meas}(\tau )$ and $F_{pt}(\tau ).$ 
 
We now prove statement \ref{FJmeaslem1-4}. 
From the definition of $F_{meas}^{0}(\tau )$ and $F_{pt}^{0}(\tau )$, 
it is easy to see that $F_{meas}^{0}(\tau )\cap Y \subset F_{pt}^{0}(\tau ).$ 
To show the opposite inclusion, let $y\in F_{pt}^{0}(\tau ).$ 
Let $\epsilon >0$ and let $\phi \in C(Y ).$ 
Then there exists a $\delta _{1}>0$ such that 
for each $y'\in Y $ with $d(y,y')<\delta _{1}$ and each $n\in \NN $, 
we have $|M_{\tau }^{n}(\phi )(y)-M_{\tau }^{n}(\phi )(y')|<\epsilon .$ 
Moreover, 
there exists a $\delta _{2}>0$ such that 
for each $\mu \in {\frak M}_{1}(Y )$ with $d_{0}(\delta _{y},\mu )<\delta _{2}$, 
we have $\mu (\{ y'\in Y \mid d(y',y)\geq \delta _{1}\} )<\epsilon .$ 
Hence, for each $\mu \in {\frak M}_{1}(Y )$ with $d_{0}(\delta _{y},\mu )<\delta _{2}$ and 
for each 
$n\in \NN $,  
we have 
\begin{align*}
|((M_{\tau }^{\ast })^{n}(\delta _{y}))(\phi )-((M_{\tau }^{\ast })^{n}(\mu ))(\phi )|
& =   |\int _{B(y,\delta _{1})}M_{\tau }^{n}(\phi )(y)\ d\mu (y')-
   \int _{B(y,\delta _{1})}M_{\tau }^{n}(\phi )(y')\ d\mu (y')| \\ 
& \ \ + |\int _{Y \setminus B(y,\delta _{1})}M_{\tau }^{n}(\phi )(y)\ d\mu (y')-
   \int _{Y \setminus B(y,\delta _{1})}M_{\tau }^{n}(\phi )(y')\ d\mu (y')| \\   
& \leq \int _{B(y,\delta _{1})}|M_{\tau }^{n}(\phi )(y)-M_{\tau }^{n}(\phi )(y')|\ d\mu (y') +2\epsilon \| \phi \| _{\infty } \\  
& \leq \epsilon +2\epsilon \| \phi \| _{\infty }.   
\end{align*}
Hence, $\delta _{y}\in F_{meas}^{0}(\tau ).$ 
Therefore, $F_{pt}^{0}(\tau )\subset F_{meas}^{0}(\tau )\cap Y .$ 
Thus, we have proved statement~\ref{FJmeaslem1-4}. 

  We now prove statement~\ref{FJmeaslem1-5}. 
Let $y\in F(G_{\tau }).$ 
Then there exists a neighborhood $B$ of $y$ in $Y$ 
such that $G_{\tau }$ is equicontinuous on $B.$ 
Let $\phi \in C(Y)$ and let $\epsilon >0.$ 
Since $\phi :Y\rightarrow \RR $ is uniformly continuous, 
there exists a $\delta _{1}>0$ such that 
for each $z,z'\in Y$ with $d(z,z')<\delta _{1}$, we have 
$|\phi (z)-\phi (z')|<\epsilon .$ 
Let $z\in B.$ Since $G_{\tau }$ is equicontinuous on $B$, there exists a 
$\delta _{2}>0$ such that 
for each $z'\in B$ with $d(z,z')<\delta _{2}$ and for each $g\in G_{\tau }$, 
we have $d(g(z),g(z'))<\delta _{1}.$ 
Hence, for each $z'\in B$ with $d(z,z')<\delta _{2}$ and for each $n\in \NN $, we have 
\begin{align*}
|M_{\tau }^{n}(\phi )(z)-M_{\tau }^{n}(\phi )(z')|
& =|\int \phi (\gamma _{n,1}(z))\ d\tilde{\tau }(\gamma )-\int \phi (\gamma _{n,1}(z'))\ d\tilde{\tau }(\gamma )|\\ 
& \leq \int |\phi (\gamma _{n,1}(z))-\phi (\gamma _{n,1}(z'))|\ d\tilde{\tau }(\gamma )<\epsilon .
\end{align*}  
From statement \ref{FJmeaslem1-2}, it follows that 
$y\in F_{pt}(\tau ).$ Therefore, $F(G_{\tau })\subset F_{pt}(\tau ).$ 
Thus, we have proved statement~\ref{FJmeaslem1-5}. 

 We now prove statement~\ref{FJmeaslem1-6}. 
 It is easy to see that if $F_{meas}(\tau )={\frak M}_{1}(Y)$ then $F_{pt}^{0}(\tau )=Y.$ 
 To show the converse, 
 suppose $F_{pt}^{0}(\tau )=Y.$ 
Then $F_{pt}(\tau )=Y.$ 
Suppose that there exists an element $\mu \in J_{meas}^{0}(\tau ).$  
Then there exists an element $\phi \in C(Y)$, an $\epsilon >0$, a strictly increasing sequence 
$\{ n_{j} \} _{j\in \NN }$ of positive integers, and a sequence $\{ \mu _{j}\} _{j\in \NN }$ 
in ${\frak M}_{1}(Y)$ with 
$\mu _{j}\rightarrow \mu $  such that for each $j\in \NN $, 
\begin{equation}
\label{FJmeaslempfeq1}
|((M_{\tau }^{\ast })^{n_{j}}(\mu ))(\phi )-((M_{\tau }^{\ast })^{n_{j}}(\mu _{j}))(\phi )|\geq \epsilon .
\end{equation}
Combining $F_{pt}(\tau )=Y$ and the Ascoli-Arzela theorem, 
we may assume that there exists an element $\psi \in C(Y)$ such that 
$M_{\tau }^{n_{j}}(\phi )\rightarrow \psi $ as $j\rightarrow \infty .$ 
Hence, for each large $j\in \NN $, 
$\| M_{\tau }^{n_{j}}(\phi )-\psi \| _{\infty }< \frac{\epsilon }{3}.$ 
Moreover, since $\mu _{j}\rightarrow \mu $, 
we have that for each large $j\in \NN $, 
$|\mu _{j}(\psi )-\mu (\psi )|<\frac{\epsilon }{3}.$ 
It follows that for a large $j\in \NN $, 
\begin{align*}
|((M_{\tau }^{\ast })^{n_{j}}(\mu ))(\phi )-((M_{\tau }^{\ast })^{n_{j}})(\mu _{j})(\phi )|
& \leq |((M_{\tau }^{\ast })^{n_{j}}(\mu ))(\phi )-\mu (\psi )|  
 + |\mu (\psi )-\mu _{j}(\psi )| \\ 
&\ \ \ + |\mu _{j}(\psi )-((M_{\tau }^{\ast })^{n_{j}})(\mu _{j})(\phi )|\\ 
&<\epsilon .
\end{align*}
However, this contradicts (\ref{FJmeaslempfeq1}). 
Hence, $F_{meas}^{0}(\tau )={\frak M}_{1}(Y).$ 
Therefore, $F_{meas}(\tau )={\frak M}_{1}(Y).$ 
Thus, we have proved statement~\ref{FJmeaslem1-6}. 

 Hence, we have completed the proof of Lemma~\ref{FJmeaslem1}. 
\end{proof}
\begin{lem}
\label{FJmeaslem2}
Let $Y$ be a compact metric space and let 
$\tau \in {\frak M}_{1,c}(\emCMX )$ with $\G _{\tau }\subset \emOCMX .$  
Let $y\in Y $ be a point. Suppose that 
$\tilde{\tau }\left(\{ \gamma  =(\gamma  _{1},\gamma  _{2},\gamma  _{3},\ldots )\in X_{\tau }\mid 
y\in \bigcap _{j=1}^{\infty }\gamma  _{1}^{-1}\cdots \gamma  _{j}^{-1}(J(G_{\tau }))\} \right) 
=0.$ Then, we have that 
$y\in F_{pt}^{0}(\tau )=F_{meas}^{0}(\tau )\cap Y .$ 
\end{lem}
\begin{proof}
By the assumption of our lemma and Lemma~\ref{ocminvlem}, 
we obtain that for $\tilde{\tau }$-a.e. $\gamma \in X_{\tau }$,  
$\lim _{n\rightarrow \infty }1_{F(G_{\tau })}(\gamma _{n,1}(y))=1.$ 
Hence $\lim _{n\rightarrow \infty }\int _{X_{\tau }}1_{F(G_{\tau })}(\gamma _{n,1}(y))\ d\tilde{\tau }(\gamma )=1.$ 
Therefore, for a given $\epsilon >0$, 
there exists an $n_{0}\in \NN  $ such that for each 
$n\in \NN $ with $n\geq n_{0}$, 
$\tilde{\tau }(\{ \gamma \in X_{\tau }\mid \gamma _{n,1}(y)\in F(G_{\tau })\} )\geq 1-\epsilon .$ 
Since $F(G_{\tau })$ is an open subset of a compact metric space, $F(G_{\tau })$ is a countable union of 
compact subsets of $F(G_{\tau }).$ Hence, there exists a compact subset $K$ of $F(G_{\tau })$ such that 
$\tilde{\tau }(\{ \gamma \in X_{\tau }\mid \gamma _{n_{0},1}(y)\in K\} )\geq 1-2\epsilon .$ 
Since $G_{\tau }$ is equicontinuous on the compact set $K$, 
for a given $\phi \in C(Y)$, there exists a $\delta _{1}>0$ such that 
for each $z\in K$, $z'\in Y$ with $d(z,z')<\delta _{1}$ and for each $l\in \NN $, 
$|M_{\tau }^{l}(\phi )(z)-M_{\tau }^{l}(\phi )(z')|<\epsilon .$ 
Moreover, since $\G _{\tau }$ is compact, 
there exists a $\delta _{2}>0$ such that 
for each $y'\in Y$ with $d(y,y')<\delta _{2}$ and for each $\gamma \in X _{\tau }$,  
$d(\gamma _{n_{0}, 1}(y), \gamma _{n_{0}, 1}(y'))<\delta _{1}.$ 
It follows that for each $y'\in Y$ with $d(y,y')<\delta _{2}$ and for each $l\in \NN $, 
\begin{align*}
|M_{\tau }^{n_{0}+l}(\phi )(y)-M_{\tau }^{n_{0}+l}(\phi )(y')|
& =|M_{\tau }^{n_{0}}(M_{\tau }^{l}(\phi ))(y)-M_{\tau }^{n_{0}}(M_{\tau }^{l}(\phi ))(y')|\\ 
& =|\int _{X_{\tau }}\left( M_{\tau }^{l}(\phi )(\gamma _{n_{0}, 1}(y))-M_{\tau }^{l}(\phi )
(\gamma _{n_{0}, 1}(y'))\right)\ d\tilde{\tau }(\gamma )|\\ 
& \leq \int _{\{ \gamma \in X_{\tau }\mid \gamma _{n_{0}, 1}(y)\in K\} }
|M_{\tau }^{l}(\gamma _{n_{0}, 1}(y))-M_{\tau }^{l}(\phi )(\gamma _{n_{0}, 1}(y')|\ 
d\tilde{\tau }(\gamma ) \\ 
& \ \ \ +\int _{\{ \gamma \in X_{\tau }\mid \gamma _{n_{0}, 1}(y)\not\in K\} }
 |M_{\tau }^{l}(\gamma _{n_{0}, 1}(y))-M_{\tau }^{l}(\phi )(\gamma _{n_{0}, 1}(y')|\ 
 d\tilde{\tau }(\gamma ) \\ 
& \leq \epsilon +2\epsilon \cdot 2\| \phi \| _{\infty }. 
\end{align*} 
Therefore, by Lemma~\ref{FJmeaslem1}-\ref{FJmeaslem1-2}, 
we obtain that $y\in F_{pt}^{0}(\tau ).$ 
Thus, we have completed the proof of Lemma~\ref{FJmeaslem2}. 
\end{proof}
\begin{lem}
\label{genskewprodinvlem1}
Let $Y$ be a compact metric space and let $\Gamma \in \emCpt (\emCMX ).$
Let $f:\GN \times Y \rightarrow \GN \times Y$ be the 
skew product associated with $\Gamma .$ 
Then, $f(\tilde{J}(f))\subset \tilde{J}(f)$ and 
for each $\g \in \G $, 
$\gamma _{1}(J_{\gamma })\subset J_{\sigma (\gamma )}$ and 
$\gamma _{1}(\hat{J}_{\gamma ,\G })\subset \hat{J}_{\gamma ,\Gamma }.$ 
\end{lem}
\begin{proof}
Let $\gamma \in \GN. $ Let $y\in Y$ and 
suppose $\gamma _{1}(y)\in F_{\sigma (\g )}.$ Then 
it is easy to see that $y\in F_{\gamma }.$ Hence, 
we have $\gamma _{1}(J_{\gamma })\subset J_{\sigma (\gamma )}.$ 
By the continuity of $f:\GN \times Y\rightarrow \GN \times Y$, 
we obtain $f(\tilde{J}(f))\subset \tilde{J}(f).$ 
Therefore, $\gamma _{1}(\hat{J}_{\gamma ,\Gamma })\subset \hat{J}_{\gamma ,\Gamma }.$  
Thus, we have completed the proof of our lemma. 
\end{proof}

\begin{lem}
\label{l:pctjjg}
Let $\Gamma \in \emCpt (\emRat )$ and let 
$G=\langle \G \rangle $.  
Let $f: \Gamma ^{\NN }\times \CCI \rightarrow 
\Gamma ^{\NN } \times \CCI $ be the skew product associated with 
$\Gamma .$   
Then, 
$\pi _{\CCI }(\tilde{J}(f))=J(G)$  
and 
for each 
$\gamma  =(\gamma  _{1},\gamma  _{2},\ldots )\in \Gamma ^{\NN }$, 
we have $\hat{J}_{\gamma  ,\Gamma }=
\bigcap _{j=1}^{\infty }\gamma  _{1}^{-1}\cdots \gamma  _{j}^{-1}(J(G)).$
\end{lem}
\begin{proof}
We first prove $\pi _{\CCI }(\tilde{J}(f))=J(G).$ 
Since $J_{\gamma }\subset J(G)$ for each $\gamma \in \Gamma ^{\NN }$, 
it is easy to see $\pi _{\CCI }(\tilde{J}(f))\subset J(G).$ 
In order to show the opposite inclusion, we consider the following four cases: 
Case 1: $\sharp (J(G))\geq 3$; Case 2: $J(G)=\emptyset $; Case 3: 
$J(G)=\{ a\} $; and Case 4: $J(G)=\{ a_{1},a_{2}\} ,a_{1}\neq a_{2}.$  

Suppose we have case 1: $\sharp (J(G))\geq 3.$ 
 Then, by \cite[Lemma 2.3 (g)]{S3}, 
$J(G)=\overline{\bigcup _{g\in G}J(g)}.$ Hence, 
$\pi _{\CCI }(\tilde{J}(f))=J(G).$ 

Suppose we have case 2: $J(G)=\emptyset .$ Then it is easy to see $\pi _{\CCI }(\tilde{J}(f))=J(G)=\emptyset .$  

Suppose we have case 3: $J(G)=\{ a\}, a\in \CCI .$ Then $G\subset \mbox{Aut}(\CCI ).$ 
Since $g^{-1}(J(G))\subset J(G)$ for each $g\in G$, it follows that 
$g(a)=a$ for each $g\in G.$ If there exists an element $g\in G$ with $|m(g,a)|<1$, 
then the repelling fixed point $b$ of $g$ is different from $a$ and $b\in J(G).$ 
This is a contradiction. Hence, $|m(g,a)|\geq 1$ for each $g\in G.$ 
If there exists an element $g$ such that $g$ is either loxodromic or parabolic, 
then $a\in J(g)\subset J(G)$ and it implies $\pi _{\CCI }(\tilde{J}(f))=J(G).$ 
Hence, in order to show $\pi _{\CCI }(\tilde{J}(f))=J(G)$, 
we may assume that each $g\in G$ is either an elliptic element or the identity map. 
Under this assumption, we will show the following claim: \\ 
\noindent Claim 1: There exists an element $\gamma \in \Gamma ^{\NN }$ 
such that $J_{\gamma }=\{ a\} .$ 

 In order to prove claim 1, 
since we are assuming $J(G)=\{ a\} \neq \emptyset $, 
there exists an $h_{1}\in \Gamma $ and an $h_{2}\in \Gamma $ such that 
$\sharp (\mbox{Fix}(h_{1}))=2$ and $\sharp (\mbox{Fix}(h_{1})\cap \mbox{Fix}(h_{2}))=1$, 
where $\mbox{Fix}(\cdot )$  denotes the set of all fixed points.  
By \cite[page 12]{Ma}, $h_{1}h_{2}h_{1}^{-1}h_{2}^{-1}$ is parabolic. 
 Hence, there exists a sequence 
$\{ g_{m}\} _{m=1}^{\infty }$ in the semigroup $\langle h_{1},h_{2}\rangle $ 
and a parabolic element $h\in \mbox{Aut}(\CCI )$ such that 
$g_{m}\rightarrow h$ as $m\rightarrow \infty .$ 
We may assume that 
$\mbox{Fix}(h_{1})\cap \mbox{Fix}(h_{2})=\{ a\} $ and $a=\infty .$ 
Then there exists a sequence $\{ n_{m}\} _{m=1}^{\infty }$ in $\NN \cup \{ 0\} $ such that 
$\sup \{ d(\infty , z) \mid z\in g_{m}^{n_{m}}\cdots g_{1}^{n_{1}}(\DD )\} \rightarrow 0$ as $m\rightarrow \infty$, 
where $\DD $ denotes the unit disc and $d$ denotes the spherical distance. 
Let $\gamma \in \{ h_{1},h_{2}\} ^{\NN }$ be an element and $\{ k_{m}\} _{m=1}^{\infty }$ 
a sequence in $\NN$ such that 
$\gamma _{k_{m},1}=g_{m}^{n_{m}}\cdots g_{1}^{n_{1}}$ for each $m\in \NN .$ 
Then $\sup \{ d(\infty ,z)\mid z\in \gamma _{k_{m},1}(\DD )\} \rightarrow 0$ as $m\rightarrow \infty .$ 
Hence, if $J_{\gamma }=\emptyset $, then $\gamma _{k_{m},1}\rightarrow \infty $ as $m\rightarrow \infty $ 
uniformly on $\CCI $. It implies that for each $\epsilon >0$ there exists a $j\in \NN $ such that 
$\gamma _{k_{j},1}(\CCI )\subset B(\infty ,\epsilon ).$ However, this is a contradiction. 
Therefore, we must have that $J_{\gamma }\neq \emptyset .$ 
Hence, we have proved claim 1. 

 By claim 1, $\pi _{\CCI }(\tilde{J}(f))=J(G)=\{ a\} .$ 
 
We now suppose we have case 4: $J(G)=\{ a_{1},a_{2}\} ,a_{1}\neq a_{2}.$ 
Then $G\subset \mbox{Aut}(\CCI )$. Since $g^{-1}(J(G))\subset J(G)$ for each $g\in G$, 
it follows that $g(J(G))=J(G)$ for each $g\in G.$ 
Hence there exists no parabolic element in $G.$ 
Let $\Lambda := \{ g_{1}\circ g_{2}\mid g_{1},g_{2}\in \Gamma \} .$ 
Then $\Lambda $ is a compact subset of $\mbox{Aut}(\CCI ).$ 
It is easy to see that 
$J(\langle \Lambda \rangle )=J(G).$  
Moreover, for each $g\in \Lambda $, $g(a_{i})=a_{i}$ for each $i=1,2.$ 
Since each $a_{i}$ belongs to $J(G)=J(\langle \Lambda \rangle )$, it follows that 
for each $i=1,2,$ there exists an element $g_{i}\in \Lambda $ such that 
$|m(g_{i},a_{i})|>1.$ Hence, $a_{i}\in J(g_{i}).$ 
Therefore, $\pi _{\CCI }(\tilde{J}(f))=J(G)=\{ a_{1},a_{2}\} .$   

 Thus, we have proved that $\pi _{\CCI }(\tilde{J}(f))=J(G).$ 
 
We now prove that for each $\gamma =(\gamma _{1},\gamma _{2},\ldots )\in \Gamma ^{\NN }$, 
$\hat{J}_{\gamma ,\Gamma }=\bigcap _{j=1}^{\infty }\gamma _{j,1}^{-1}(J(G)).$  
 Let $\g =(\g _{1},\g _{2},\ldots )\in \GN .$ 
By \cite[Lemma 2.1]{S7},  
we see that for each $j\in \NN $, 
$\g_{j,1}(\hat{J}_{\g ,\Gamma })=
\hat{J}_{\sigma ^{j}(\g ),\Gamma }
\subset J(G).$ 
Hence, 
$\hat{J}_{\g ,\Gamma}\subset 
\bigcap_{j=1}^{\infty }
\g _{j,1}^{-1}(J(G)).$
 Suppose that there exists a point 
 $(\g ,y)\in \GN \times \CCI $ such that  
 $y\in $ $\left( \bigcap_{j=1}^{\infty }
\g _{j,1}^{-1}(J(G))\right) \setminus 
\hat{J}_{\g ,\Gamma }.$ 
Then, we have 
$(\g ,y)\in (\GN \times \CCI ) 
\setminus \tilde{J}(f).$ 
Hence, there exists a 
neighborhood $U$
of $\g $ in $\GN $ and a 
neighborhood $V$ of $y$ in $\CCI $ such that 
$U\times V\subset \tilde{F}(f).$ 
Then, there exists an $n\in \NN $ such that 
$\{ \rho \in X_{\tau }\mid \rho _{j}=\gamma _{j}, j=1,\ldots ,n\} \subset U.$ 
Combining it with \cite[Lemma 2.1]{S7}, 
we obtain $\tilde{F}(f)\supset 
f^{n}(U\times V)\supset \GN \times 
\{ \gamma_{n,1}(y)\} .$ 
Moreover, 
since we have 
$\gamma_{n,1}(y)\in J(G)=
\pi _{\CCI }(\tilde{J}(f))$, 
we get that there exists an element 
$\g '\in \GN $ such that 
$(\g ', \gamma_{n,1}(y))\in 
\tilde{J}(f).$ However, it contradicts 
$(\g ',\gamma_{n,1}(y))\in 
\GN \times \{ \gamma_{n,1}(y)\} 
\subset \tilde{F}(f).$ 
Hence, we obtain 
$\hat{J}_{\g }(f)= 
\bigcap_{j=1}^{\infty }
\g_{j,1}^{-1}(J(G)).$

Thus, we have proved Lemma~\ref{l:pctjjg}.
\end{proof}

\begin{lem}
\label{Lkerlem}
Let $Y$ be a compact metric space and 
let $\tau \in {\frak M}_{1,c}(\emCMX)$.  
Let $V$ be a non-empty open subset of $Y $ such that 
$G_{\tau }(V)\subset V.$ For each 
$\gamma  =(\gamma  _{1},\gamma  _{2},\ldots )\in X_{\tau }$, 
we set 
$L_{\gamma  }:= \bigcap_{j=1}^{\infty }\gamma  _{j,1}^{-1}
(Y \setminus V).$ Moreover, we set 
$L_{\ker}:= \bigcap_{g\in G_{\tau }}g^{-1}(Y \setminus V).$ 
Let $y\in Y$ be a point. Then, 
we have that 
$$\tilde{\tau }(\{ \gamma  \in X_{\tau }\mid 
y\in L_{\gamma  },\ \liminf_{n\rightarrow \infty }
d(\gamma _{n,1}(y), L_{\ker})>0\})=0.$$ (When 
$L_{\ker }=\emptyset $, we set 
$d(z,L_{\ker}):=\infty $ for each $z\in Y .$) 
\end{lem}
\begin{proof}
For each $c>0$, 
we set 
$E_{c}:= \{ \gamma \in X_{\tau }\mid 
y\in L_{\gamma }, \forall n, d(\gamma _{n, 1}(y), L_{\ker })\geq c\} .$ 
In order to prove our lemma, it is enough to show that 
for each $c>0$, $\tilde{\tau }(E_{c})=0.$ 
It clearly holds when $y\in V.$ Hence, we assume 
$y\in Y\setminus V.$ 
Let $B_{c}:= \{ z\in Y\setminus V\mid d(z, L_{\ker })\geq c\} .$ 
For each $z\in B_{c}$, 
there exists a positive integer $k(z)$, 
an element $(\alpha _{1,z},\ldots ,\alpha _{k(z),z})\in \G _{\tau }^{k(z)}$, 
a neighborhood $U_{z}$ of $ (\alpha _{1,z},\ldots ,\alpha _{k(z),z})$ in 
$ \G _{\tau }^{k(z)}$, and a $\delta _{z}>0$ such that 
for each $\tilde{\alpha }=(\tilde{\alpha }_{1},\ldots ,\tilde{\alpha }_{k(z)})\in U_{z}$, 
$\tilde{\alpha }_{k(z)}\cdots \tilde{\alpha }_{1}(B(z,\delta _{z}))\subset V.$ 
Since $B_{c}$ is compact, 
there exists an $l\in \NN $, a finite sequence $\{ k(j)\} _{j=1}^{l}$ in $\NN $, 
a finite subset $\{ z_{j}\} _{j=1}^{l}$ of $B_{c}$, 
a finite subset $\{ \alpha _{j}=(\alpha _{j,1},\ldots ,\alpha _{j,k(j)})\in 
\G _{\tau }^{k(j)}\} _{j=1}^{l}$, 
a neighborhood  
$ U_{j}$ of $\alpha _{j}$ in $\G _{\tau }^{k(j)} $ 
for each $j=1,\ldots , l$, and a finite sequence $\{ \delta _{j}\} _{j=1}^{l}$, such that 
$\bigcup _{j=1}^{l}B(z_{j},\delta _{j})\supset B_{c}$ and 
such that for each $j=1,\ldots ,l$ and each $\tilde{\alpha }=(\tilde{\alpha }_{1},\ldots, \tilde{\alpha }_{k(j)})\in U_{j}$, 
$\tilde{\alpha }_{k(j)}\cdots \tilde{\alpha }_{1}(B(z_{j},\delta _{j}))\subset V.$   
Since $G_{\tau }(V)\subset V$, we may assume that 
there exists a $k\in \NN $ such that for each $j=1,\ldots ,l$, 
$k(j)=k.$  
For each $n\in \NN $, 
we set 
$E_{c}^{n}:= \{\gamma \in X_{\tau }\mid 
\gamma _{jk,1}(y)\in B_{c}, j=1,\ldots ,n\} .$ 
For each $\gamma =(\gamma _{1},\gamma _{2},\ldots ) \in E_{c}^{n}$, 
there exists a neighborhood $A_{\gamma }$ of 
$(\gamma _{1},\ldots ,\gamma _{nk})$ in $\G _{\tau }^{nk}$ 
and a $j(\gamma )\in \NN $ such that 
for each $\tilde{\alpha }\in A_{\gamma }$, 
$\tilde{\alpha }_{nk}\cdots \tilde{\alpha }_{1}(y)\in B(z_{j(\gamma )}, \delta _{j(\gamma )}).$ 
Hence, there exists a finite sequence $\{ W_{i}\} _{i=1}^{r}$ of subsets 
of $\G _{\tau }^{nk}$ and a finite sequence $\{ p(i)\} _{i=1}^{r}$ of 
positive integers such that 
setting $E_{c}^{n,i}:= \{ \gamma \in E_{c}^{n}\mid (\gamma _{1},\ldots ,\gamma _{nk})\in W_{i}\} $, 
we have that $E_{c}^{n,i}\subset \{ \gamma \in X_{\tau }\mid \gamma _{nk, 1}(y)
\in B(z_{p(i)},\delta _{p(i)})\} $ and $E_{c}^{n}=\coprod _{i=1}^{r}E_{c}^{n,i}.$ 
Let $a:= \max _{j=1,\ldots ,l} \{ (\otimes _{s=1}^{k}\tau )(\G _{\tau }^{k}\setminus U_{j})\} (<1).$ 
Since $E_{c}^{n+1}=\coprod _{i=1}^{r}(E_{c}^{n+1}\cap E_{c}^{n,i})$, 
it follows that for each $n\in \NN $, 
\begin{align*}
\tilde{\tau }(E_{c}^{n+1})
& = \sum _{i=1}^{r}\tilde{\tau }(E_{c}^{n+1}\cap E_{c}^{n,i})
 \leq \sum _{i=1}^{r}\tilde{\tau }(\{ \gamma \in X_{\tau }\mid 
(\gamma _{nk+1},\ldots ,\gamma _{(n+1)k})\not\in U_{p(i)}\} \cap E_{c}^{n,i})\\ 
& =\sum _{i=1}^{r}(\otimes _{s=1}^{k}\tau )(\G _{\tau }^{k}\setminus U_{p(i)})\cdot 
\tilde{\tau }(E_{c}^{n,i}) 
 \leq a\sum _{i=1}^{r}\tilde{\tau }(E_{c}^{n,i})=a \tilde{\tau }(E_{c}^{n}).
\end{align*}  
Combining it with $E_{c}\subset \bigcap_{n=1}^{\infty }E_{c}^{n}$, 
we obtain that $\tilde{\tau }(E_{c})\leq \tilde{\tau }(\bigcap_{n=1}^{\infty }E_{c}^{n})=0.$ 
Thus, we have completed the proof of Lemma~\ref{Lkerlem}. 
\end{proof}
\begin{prop}[Cooperation Principle I] 
\label{Jkeremptygenprop1}
Let $Y$ be a compact metric space and let $\tau \in {\frak M}_{1,c}(\emCMX )$ with 
$\G _{\tau }\subset \emOCMX .$  
Suppose that $J_{\ker }(G_{\tau })=\emptyset .$ 
Then, $F_{meas}(\tau )={\frak M}_{1}(Y)$ and 
for each $y\in Y$, there exists a Borel subset ${\cal A}_{y}$ of 
$X_{\tau }$ with $\tilde{\tau }({\cal A}_{y})=1$ such that 
for each $\gamma \in {\cal A}_{y}$, 
there exists an $n\in \NN $ with $\gamma _{n,1}(y)\in F(G_{\tau }).$ 
\end{prop}
\begin{proof}
Let $V:= F(G_{\tau }).$ 
By Lemma~\ref{ocminvlem}, for each $g\in G_{\tau }$, 
$g(V)\subset V.$ 
By Lemma~\ref{Lkerlem}, 
we obtain that  for each $y\in Y$ and for $\tilde{\tau }$-a.e. $\gamma \in X_{\tau }$,  
there exists an $n\in \NN $ such that $\gamma _{n,1}(y)\in F(G_{\tau }).$ 
From Lemma~\ref{FJmeaslem2} and Lemma~\ref{FJmeaslem1}-\ref{FJmeaslem1-6}, 
it follows that $F_{meas}(\tau )={\frak M}_{1}(Y).$ 
Thus, we have completed the proof of Proposition~\ref{Jkeremptygenprop1}.  
\end{proof}
\begin{prop}
\label{Jkeremptygenprop2}
Let $Y$ be a compact metric space. Let $\lambda $ be a Borel finite measure on $Y.$ 
Let $\tau \in {\frak M}_{1,c}(\emCMX )$ with $\G _{\tau }\subset \emOCMX .$ 
Suppose that $J_{\ker }(G_{\tau })=\emptyset .$ 
Then, for $\tilde{\tau }$-a.e. $\gamma \in X_{\tau }$,  
$\lambda (J_{\gamma })=\lambda (\hat{J}_{\gamma ,\G _{\tau} })=0.$
\end{prop}
\begin{proof}
By Proposition~\ref{Jkeremptygenprop1}, 
for each $y\in Y$, for $\tilde{\tau }$-a.e. $\gamma \in X_{\tau }$, 
there exists an $n\in \NN $ such that 
$\gamma _{n,1}(y)\in F(G_{\tau })\subset (Y\setminus \bigcup _{\gamma \in X_{\tau }}\hat{J}_{\gamma ,\Gamma _{\tau }}).$ 
Combining it with Lemma~\ref{genskewprodinvlem1}, 
we obtain that for each $y\in Y$, $\tilde{\tau }(\{ \gamma \in X_{\tau }\mid (\gamma ,y)\in \tilde{J}(f)\} )=0.$ 
From Fubini's theorem, it follows that 
for $\tilde{\tau }$-a.e. $\gamma \in X_{\tau }$, $\lambda (\hat{J}_{\g, \G _{\tau }})=0.$ 
Since $J_{\gamma }\subset \hat{J}_{\gamma ,\Gamma _{\tau }}$ for each $\gamma \in X_{\tau }$, 
we obtain that for $\tilde{\tau }$-a.e. $\gamma \in X_{\tau }$, $\lambda (J_{\g})=0.$ 
Thus, we have completed the proof of Proposition~\ref{Jkeremptygenprop2}.  
\end{proof}
\begin{lem}
\label{l:mgjpt0}
Let $Y$ be a compact metric space and 
let $\lambda $ be a Borel finite measure on $Y.$ 
Let $\tau \in {\frak M}_{1,c}(\emCMX )$ with $\G _{\tau }\subset \emOCMX.$ 
Suppose that 
for $\tilde{\tau }$-a.e. $\g \in X_{\tau }$,  
$\lambda (\bigcap _{j=1}^{\infty }\g _{1}^{-1}\cdots \g _{j}^{-1}(J(G_{\tau }))) =0.$ 
Then, for $\lambda $-a.e. $y\in Y$, 
there exists a Borel subset ${\cal A}_{y}$ of $X_{\tau }$ with $\tilde{\tau }({\cal A}_{y})=1$ 
such that for each $\gamma \in {\cal A}_{y}$, 
there exists an $n\in \NN $ with  
$\g _{n, 1}(y)\in F(G_{\tau })$. Moreover, $\lambda (J_{pt}^{0}(\tau ))=0.$ 
\end{lem}
\begin{proof}
Let $f:X_{\tau }\times Y\rightarrow X_{\tau }\times Y$ be the skew product associated with 
$\G _{\tau }.$ 
Let $M:= \{ (\g ,y)\in X_{\tau }\times Y \mid \forall n\in \NN , \gamma _{n,1}(y)\in J(G_{\tau })\} .$ 
By the assumption of our lemma and Fubini's theorem, 
we obtain that there exists a measurable subset $Z$ of $Y$ with $\lambda (Z)=\lambda (Y)$ such that 
for each $y\in Z$, $\tilde{\tau }(\{ \g \in X_{\tau }\mid (\g ,y)\in M\} )=0.$ 
For this $Z$, we have that for each $y\in Z$, 
$\tilde{\tau } (\{ \gamma \in X_{\tau }\mid y\in \bigcap_{j=1}^{\infty }\g _{1}^{-1}\cdots \g _{j}^{-1}(J(G_{\tau }))\})=0.$ 
By Lemma~\ref{FJmeaslem2}, we obtain $Z\subset F_{pt}^{0}(\tau ).$ 
Thus, we have completed the proof of our lemma. 
\end{proof}

\section{Proofs of the main results}
\label{Proofs}
In this section, we prove the main results. 
\subsection{Proofs of results in subsection~\ref{Genres}}
\label{pfgenres}
In this subsection, we give the proofs of subsection~\ref{Genres}. 

\ 

\noindent {\bf Proof of Theorem~\ref{kerJthm1}:} 
Since $\NHM (\CPn ) \subset \OCM (\CPn)$, 
the statement of Theorem~\ref{kerJthm1} follows from 
Proposition~\ref{Jkeremptygenprop1} and Proposition~\ref{Jkeremptygenprop2}.
\qed 

\ 

In order to prove Theorem~\ref{t:mtauspec}, we need several lemmas.
%
%
\begin{lem}
\label{l:sjg3}
Under the assumptions of Theorem~\ref{t:mtauspec}, 
$\sharp J(G_{\tau })\geq 3.$ 
\end{lem}
\begin{proof}
Suppose $\sharp J(G_{\tau })\leq 2.$ Then, 
$G_{\tau }\subset \mbox{Aut}(\CCI ).$ 
By Lemma~\ref{ocminvlem}, it follows that $G_{\tau }(J(G_{\tau })) =J(G_{\tau })$. 
This implies $J_{\ker }(G_{\tau })=J(G_{\tau })$, which contradicts our assumption. 
Thus, our lemma holds. 
\end{proof}

\begin{lem}
\label{l:aeconst}
Under the assumption of Theorem~\ref{t:mtauspec}, 
there exists a Borel measurable subset ${\cal A}$ of $X_{\tau }$ with $\tilde{\tau }({\cal A})=1$ such that 
for each $\gamma \in {\cal A}$ and for each $U\in \mbox{{\em Con}}(F(G_{\tau }))$, 
there exists no non-constant limit function of $\{ \gamma _{n,1}|_{U}:U\rightarrow \CCI \} _{n=1}^{\infty }.$   
\end{lem}
\begin{proof}
Since $\sharp \mbox{Con}(F(G_{\tau }))\leq \aleph _{0}$, 
it is enough to show that for each $U\in \mbox{Con}(F(G_{\tau }))$, 
there exists a Borel measurable subset ${\cal A}_{U}$ of $X_{\tau }$ with 
$\tilde{\tau }({\cal A}_{U})=1$ such that for each 
$\gamma \in {\cal A}_{U}$, there exists no non-constant limit function of 
$\{ \gamma _{n,1}|_{U}:U\rightarrow \CCI \} _{n=1}^{\infty }$. 
In order to show this, let $U\in \mbox{Con}(F(G_{\tau }))$ and let $a\in U.$ 
Since $J_{\ker }(G_{\tau })=\emptyset $, for each $z\in \partial J(G_{\tau })$ there 
exists an element $g_{z}\in G_{\tau }$ and a disk neighborhood $V_{z}$ of $z$ in $\CCI $ such that 
$g_{z}(\overline{V_{z}})\subset F(G_{\tau }).$ 
Since $\partial J(G_{\tau })$ is compact, there exists a finite family 
$\{ z_{1},\ldots ,z_{p}\} $ of points in $\partial J(G_{\tau })$ such that 
$\bigcup _{j=1}^{p}V_{z_{j}}\supset \partial J(G_{\tau })$ and 
$g_{z_{j}}(\overline{V_{z_{j}}})\subset F(G_{\tau })$ for each 
$j=1,\ldots ,p.$ For each $j$, there exists a $k(j)\in \NN $ and an element 
$\alpha ^{j}=(\alpha _{1}^{j},\ldots ,\alpha _{k(j)}^{j})\in \Gamma _{\tau }^{k(j)}$ such that 
$g_{z_{j}}=\alpha _{k(j)}^{j}\circ \cdots \circ \alpha _{1}^{j}.$ 
Since $G_{\tau }(F(G_{\tau }))\subset F(G_{\tau })$, 
we may assume that there exists a $k\in \NN $ such that for each $j\in \NN $, 
$k(j)=k.$ For each $j$, let $W_{j}$ be a compact neighborhood of 
$\alpha ^{j}$ in $\Gamma _{\tau }^{k}$ such that 
for each $\beta =(\beta _{1},\ldots ,\beta _{k})\in W_{j}$, 
$\beta _{k}\cdots \beta _{1}(\overline{V_{z_{j}}})\subset F(G_{\tau }).$   
For each $j$, let $B_{j}:= \bigcup _{B\in \mbox{Con}(F(G_{\tau })), B\cap V_{z_{j}}\neq \emptyset }B.$ 
Let $n\in \NN $ and let 
$\{ c_{q}\} _{q\in \NN }$ be a decreasing sequence of positive numbers such that 
$c_{q}\rightarrow 0$ as $q\rightarrow \infty .$  
Let $(i_{1},\ldots ,i_{l})$ be a finite sequence of positive integers with 
$i_{1}< \cdots <i_{l}.$ Let $q>0.$   
We denote by $A_{q,j}(i_{1},\ldots ,i_{l})$ the set of elements 
$\gamma \in X_{\tau }$ which satisfies all of the following (a) and (b). 
\begin{itemize}
\item[(a)]
$\gamma _{kt,1}(a)\in (\CCI \setminus B(\partial J(G_{\tau }), c_{q}))\cap B_{j}$ if 
$t\in \{ i_{1},\ldots ,i_{l}\}.$  
\item[(b)]
$\gamma _{kt,1}(a)\not\in (\CCI \setminus B(\partial J(G_{\tau }), c_{q}))\cap B_{j}$ if 
$t\in \{ 1,\ldots ,i_{l}\} \setminus \{ i_{1},\ldots ,i_{l}\}.$  
\end{itemize}
Moreover, when $l\geq n$,  
we denote by $B_{q,j,n}(i_{1},\ldots ,i_{l})$ the set of elements $\gamma \in X_{\tau }$ 
which satisfies items (a) and (b) above and the following (c). 
\begin{itemize}
\item[(c)]
$(\gamma _{ki_{s}+1}, \ldots ,\gamma _{ki_{s}+k})\not\in W_{j}$ for each $s=n,n+1,\ldots ,l.$
\end{itemize}
 Furthermore, we denote by   
$C_{q,j,n}(i_{1},\ldots ,i_{l})$ the set of elements $\gamma \in X_{\tau }$ 
which satisfies items (a) and (b) above and the following (d). 
 \begin{itemize}
\item[(d)]
$(\gamma _{ki_{s}+1}, \ldots ,\gamma _{ki_{s}+k})\not\in W_{j}$ for each $s=n,n+1,\ldots ,l-1.$
\end{itemize}
Furthermore, 
for each $q$, $j$, $n,l$ with $l\geq n$, 
let $B_{q,j,n,l}:= \bigcup _{i_{1}<\cdots <i_{l}}B_{q,j,n}(i_{1},\ldots ,i_{l}).$ 
Let ${\cal D}:= \bigcup _{q=1}^{\infty }\bigcup _{j=1}^{p}\bigcup _{n\in \NN }\bigcap _{l\geq n}B_{q,j,n,l}.$ 
We show the following claim. 

Claim 1. Let $\gamma \in X_{\tau }$ be such that there exists a non-constant limit function of 
$\{ \gamma _{n,1}|_{U}:U\rightarrow \CCI \} _{n=1}^{\infty }$. Then $\gamma \in {\cal D}.$ 

To show this claim, let $\gamma $ be such an element. Then there exists a $q\in \NN $, a $j\in \{ 1,\ldots ,p\} $,  
and a strictly increasing sequence $\{ i_{l}\} _{l=1}^{\infty }$ in $\NN $
such that $\gamma \in \bigcap _{l=1}^{\infty }A_{q,j}(i_{1},\ldots ,i_{l})$ and 
$\{ \gamma _{ki_{l},1}|_{U}:U\rightarrow \CCI \} _{l=1}^{\infty }$ converges to a non-constant map. 
 Suppose that there exists 
a strictly increasing sequence $\{ l_{p}\} _{p=1}^{\infty }$ in $\NN $ such that 
for each $p\in \NN $, 
$(\gamma _{ki_{l_{p}}+1},\ldots ,\gamma _{ki_{l_{p}}+k})\in W_{j}.$ 
By Lemma~\ref{l:sjg3}, for each $A\in \mbox{Con}(F(G_{\tau }))$, we can take the hyperbolic metric on $A.$ 
From the definition of $W_{j}$, we obtain that there exists a constant 
$0<\alpha <1 $ such that 
for each $p\in \NN $, 
$\| (\gamma _{ ki_{l_{p}}+k}\cdots \gamma _{ki_{l_{p}+1}})'(\g _{ki_{l_{p}},1}(a))\| _{h}\leq \alpha $, 
where for each $g\in G_{\tau }$ and for each $z\in F(G_{\tau })$, $\| g'(z) \| _{h}$ denotes the norm of the derivative of $g$ at $z$ 
measured from the  
hyperbolic metric on the element of $\mbox{Con}(F(G_{\tau }))$ containing 
$z$
to that on the element of $\mbox{Con}(F(G_{\tau }))$ containing 
$g(z).$ 
Hence, 
$\| (\gamma _{ki_{l_{p}},1})'(a)\| _{h}\rightarrow 0$ as $p\rightarrow \infty $. 
However, this is a contradiction, since 
$\{ \gamma _{ki_{l_{p}},1}|_{U}\} _{p=1}^{\infty }$ converges to a non-constant map. 
Therefore, $\gamma \in {\cal D}.$ Thus, we have proved claim 1. 

 Let $\eta := \max _{j=1}^{p}(\otimes _{s=1}^{k}\tau )(\Gamma _{\tau }^{k}\setminus W_{j})\ (<1).$  
Then we have for each $(l,n)$ with $l\geq n,$  
\begin{align*}
\tilde{\tau }(B_{q,j,n}(i_{1},\ldots ,i_{l+1}))
& \leq \tilde{\tau }(C_{q,j,n}(i_{1},\ldots ,i_{l+1})\cap 
\{ \gamma \in X_{\tau }\mid (\gamma _{ki_{l+1}+1},\ldots ,\gamma _{ki_{l+1}+k})\not\in W_{j}\} )\\ 
& \leq \tilde{\tau }(C_{q,j,n}(i_{1},\ldots ,i_{l+1}))\cdot \eta .  
\end{align*}
Hence, for each $l$ with $l\geq n$, 
\begin{align*}
\tilde{\tau }(B_{q,j,n,l+1})
= & \tilde{\tau }(\bigcup _{i_{1}<\cdots <i_{l+1}}B_{q,j,n}(i_{1},\ldots ,i_{l+1})) 
=  \sum _{i_{1}<\cdots <i_{l+1}}\tilde{\tau }(B_{q,j,n}(i_{1},\ldots ,i_{l+1}))\\ 
\leq & \sum _{i_{1}<\cdots <i_{l+1}}\eta \tilde{\tau }(C_{q,j,n}(i_{1},\ldots ,i_{l+1})) 
=  \eta \tilde{\tau }(\bigcup _{i_{1}<\cdots <i_{l+1}}C_{q,j,n}(i_{1},\ldots ,i_{l+1}))\leq 
\eta \tilde{\tau }(B_{q,j,n,l}).
\end{align*}
Therefore $\tilde{\tau }({\cal D})\leq \sum _{q=1}^{\infty }\sum _{j=1}^{p}\sum _{n\in \NN }
\tilde{\tau }(\bigcap _{l\geq n}B_{q,j,n,l})=0.$ 
Thus, we have completed the proof of Lemma~\ref{l:aeconst}.     
\end{proof}
\begin{lem}
\label{l:hypmetcont}
Under the assumptions of Theorem~\ref{t:mtauspec},
 there exists a Borel measurable subset ${\cal V}$ of $X_{\tau }$ with 
 $\tilde{\tau }({\cal V})=1$ such that 
 for each $\gamma \in {\cal V}$ and for each $Q\in \emCpt(F(G_{\tau }))$, 
 $\sup _{a\in Q}\| \gamma _{n,1}'(a)\| _{h}\rightarrow 0$ as $n\rightarrow \infty $, 
 where $\| \gamma _{n,1}'(a)\| _{h}$ denotes the 
 norm of the derivative of $\gamma _{n,1}$ at a point $a$ measured from the hyperbolic metric on 
 the element $U_{0}\in \mbox{{\em Con}}(F(G_{\tau }))$ with $a\in U_{0}$ 
 to that on the element $U_{n}\in \mbox{{\em Con}}(F(G_{\tau }))$ with $\gamma _{n,1}(a)\in U_{n}.$ 
\end{lem}
\begin{proof}
Let ${\cal A}$ be the subset of $X_{\tau }$ in Lemma~\ref{l:aeconst}. 
Let $U\in \mbox{Con}(F(G_{\tau }))$ be an element and let $a_{0}\in U.$ 
Let $\eta (a_{0}):= \{ \gamma \in {\cal A}\mid d(\gamma _{n,1}(a_{0}),J(G_{\tau }))\rightarrow 0 \mbox{ as }
n\rightarrow \infty \} .$ 
Let $\{ z_{j}\} _{j=1}^{p}, \{ V_{z_{j}}\} _{j=1}^{p}$, 
$k, \{ W_{j}\} _{j=1}^{k}$ be as in the proof of Lemma~\ref{l:aeconst}. 
For each $(n,m)\in \NN ^{2}$ with $n\leq m$, let 
$E_{n,m}:=\{ \gamma \in {\cal A}\mid \gamma _{ik,1}(a_{0})\in \bigcup _{j=1}^{p}V_{z_{j}}, i=n,\ldots ,m\} $.  
Let $(n,m)\in \NN^{2}$ with $n\leq m.$ 
Then there exist  mutually disjoint Borel subsets $Z_{1},\ldots ,Z_{r}$ of $\Gamma _{\tau }^{mk}$ and 
a sequence $\{ j(s)\} _{s=1}^{r}\subset \{ 1,\ldots ,p\} $ such that 
setting $E_{n,m,s}:= \{ \gamma \in E_{n,m}\mid (\gamma _{1},\ldots ,\gamma _{mk})\in Z_{s}\} $, 
we have $E_{n,m,s}\subset \{ \gamma \in X_{\tau }\mid \gamma _{mk,1}(a_{0})\in V_{z_{j(s)}}\} $ and 
$E_{n,m}=\amalg _{s=1}^{r}E_{n,m,s}.$ 
Let $\alpha := \max _{j=1}^{p}\{ (\otimes _{i=1}^{k}\tau )(\Gamma _{\tau }^{k}\setminus W_{j}) \} \ (<1).$ 
Since $E_{n,m+1}=\amalg _{s=1}^{r}(E_{n,m+1}\cap E_{n,m,s})$, we have 
\begin{align*}
\tilde{\tau }(E_{n,m+1}) 
=    & \sum _{s=1}^{r}\tilde{\tau }(E_{n,m+1}\cap E_{n,m,s})\\ 
\leq & \sum _{s=1}^{r}\tilde{\tau }(\{ \gamma \in X_{\tau } \mid (\gamma _{mk+1},\ldots ,\gamma _{(m+1)k})\notin W_{j(s)}\} \cap E_{n,m,s})\\ 
=    & \sum _{s=1}^{r}(\otimes _{i=1}^{k}\tau )(\Gamma _{\tau }^{k}\setminus W_{j(s)})\cdot \tilde{\tau }(E_{n,m,s})\leq 
\alpha \sum _{s=1}^{r}\tilde{\tau }(E_{n,m,s})=\alpha \tilde{\tau }(E_{n,m}).
\end{align*} 
Therefore, $\tilde{\tau }(\bigcap _{m\in \NN : m\geq n}E_{n,m})=0$ for each $n\in \NN .$  
Thus, $\tilde{\tau }(\eta (a_{0}))=0.$ 
Let $\gamma \in {\cal A}\setminus \eta (a_{0})$ 
be an element. 
Then for each compact subset $Q_{0}$ of $U$ there exists a compact subset $C$ of 
$F(G_{\tau })$ and a strictly increasing sequence $\{ m_{j}\} _{j=1}^{\infty }$ in $\NN $ 
such that $\gamma _{m_{j},1}(Q_{0})\subset C$ for each $j\in \NN .$ 
Therefore, for each $\gamma \in {\cal A}\setminus \eta (a_{0})$, 
$\sup _{w_{0}\in Q_{0}}\| \gamma _{n,1}'(w_{0})\| _{h}\rightarrow 0$ as $n\rightarrow \infty .$ 
From these arguments, the statement of our lemma follows.   
\end{proof}
\begin{lem}
\label{l:azproof}
Under the assumptions of Theorem~\ref{t:mtauspec}, statement \ref{t:mtauspecaz} of Theorem~\ref{t:mtauspec} holds. 
\end{lem}
\begin{proof}
Let $z\in \CCI .$ By Proposition~\ref{Jkeremptygenprop1}, there exists a Borel subset 
${\cal A}'_{z}$ of $X_{\tau }$ with $\tilde{\tau }({\cal A}'_{z})=1$ such that 
for each $\gamma \in {\cal A}'_{z}$, there exists an $n\in \NN $ such that 
$\gamma _{n,1}(z)\in F(G_{\tau }).$ 
Let ${\cal A}_{z}:={\cal A}'_{z}\cap \bigcap _{n=0}^{\infty }\sigma ^{-n}({\cal A})$, 
where ${\cal A}$ is the set in Lemma~\ref{l:aeconst}. Then ${\cal A}_{z}$ satisfies the desired property. 
\end{proof}

\begin{lem}
\label{l:minfin}
Under the assumptions of Theorem~\ref{t:mtauspec}, 
$\sharp \emMin (G_{\tau },\CCI )<\infty .$ 
\end{lem}
\begin{proof}
Let $\{ z_{j}\} _{j=1}^{p}$, $\{ g_{z_{j}}\} _{j=1}^{p}$ and $\{ V_{j}\} _{j=1}^{p}$ 
be as in the proof of Lemma~\ref{l:aeconst}. 
 Then $\bigcup _{j=1}^{p}g_{z_{j}}(\overline{V_{j}})$ is a compact subset of $F(G_{\tau }).$ 
 Let $A:=\bigcup _{j=1}^{p}g_{z_{j}}(\overline{V_{j}}).$ 
Suppose $\sharp \Min (G_{\tau },\CCI )=\infty .$ 
Then there exists a sequence $\{ K_{n} \} _{n=1}^{\infty }$ of mutually distinct elements of $\Min(G_{\tau },\CCI ).$ 
By Lemma~\ref{l:aeconst}, for each $(n,m)$ with $n\neq m$, 
there exists no $U\in \mbox{Con}(F(G_{\tau }))$ such that $U\cap K_{n}\neq \emptyset $ and 
$U\cap K_{m}\neq \emptyset .$ Hence, there exists a sequence $\{ n_{j}\} _{j=1}^{\infty }$ in $\NN $ 
and a sequence $\{ a_{j}\} _{j=1}^{\infty }$ in $\CCI $ such that 
$a_{j}\in K_{n_{j}}$ for each $j$ and such that $d(a_{j},\partial J(G_{\tau }))\rightarrow 0$ as $j\rightarrow \infty .$ 
Hence, there exists a $j_{0}\in \NN $ such that for each $j\in \NN $ with $j\geq j_{0}$, 
$K_{n_{j}}\cap A\neq \emptyset .$ However, this is a contradiction. Thus, 
we have proved Lemma~\ref{l:minfin}. 
\end{proof}
\begin{lem}
\label{l:kcpt} Under the assumptions of Theorem~\ref{t:mtauspec}, 
statement~\ref{t:mtauspec4} of Theorem~\ref{t:mtauspec} holds.
\end{lem}
\begin{proof}
By Lemma~\ref{l:minfin}, 
$S_{\tau }=\bigcup _{L\in \Min(G_{\tau },\CCI )} L$ is compact. 
Moreover, $G_{\tau }(S_{\tau })\subset S_{\tau }.$ 
Let $W:= \bigcup _{A\in \mbox{Con}(F(G_{\tau })),A\cap S_{\tau }\neq \emptyset } A.$ 
Then $G_{\tau }(W)\subset W.$ 
Let $z_{0}\in \CCI .$ From the definition of $S_{\tau }$, 
$\overline{G_{\tau }(z_{0})}\cap S_{\tau }\neq \emptyset .$ 
Combining this with that $J_{\ker }(G_{\tau })=\emptyset $, 
we obtain that $\overline{G_{\tau }(z_{0})}\cap W\neq \emptyset .$ 
Thus, we have shown that for each $z_{0}\in \CCI $ there exists an element $g\in G_{\tau }$ such that 
$g(z_{0})\in W.$ 
Combining this with $G_{\tau }(W)\subset W$ and Lemma~\ref{Lkerlem}, 
it follows that for each $z_{0}\in \CCI $ there exists a Borel measurable subset 
${\cal V}_{z_{0}}$ of $X_{\tau }$ with $\tilde{\tau }({\cal V}_{z_{0}})=1$ such that 
for each $\gamma \in {\cal V}_{z_{0}}$, 
there exists an $n\in \NN $ with $\gamma _{n,1}(z_{0})\in W.$ 
By Lemma~\ref{l:aeconst}, 
there exists a Borel measurable subset ${\cal A}$ of $X_{\tau }$ with $\tilde{\tau }({\cal A})=1$ 
such that for each $\gamma \in {\cal A}$ and for each $z\in W$, 
$d(\gamma _{n,1}(z),S_{\tau })\rightarrow 0$ as $n\rightarrow \infty .$ 
For each $z_{0}\in \CCI $, let ${\cal C}_{z_{0}}:= {\cal V}_{z_{0}}\cap 
\bigcap _{n=0}^{\infty }\sigma ^{-n}({\cal A}).$ Then $\tilde{\tau }({\cal C}_{z_{0}})=1.$ 
Moreover, for each $\gamma \in {\cal C}_{z_{0}}$, 
$d(\gamma _{n,1}(z_{0}),S_{\tau })\rightarrow 0$ as $n\rightarrow \infty .$ 
Thus, we have proved our lemma. 
\end{proof}
\begin{lem}
\label{l:lscf}
Under the assumptions of Theorem~\ref{t:mtauspec}, 
$\emLSfc \subset C_{F(G_{\tau })}(\CCI ).$ 
\end{lem}
\begin{proof}
Let $\varphi \in C(\CCI )$ be such that $\varphi \neq 0$ and $M_{\tau }(\varphi )=a\varphi $ 
for some $a\in S^{1}.$ By Theorem~\ref{kerJthm1}, 
there exists a sequence $\{ n_{j}\} _{j=1}^{\infty }$ in $\NN $  and 
an element $\psi \in C(\CCI )$ such that 
$M_{\tau }^{n_{j}}(\varphi )\rightarrow \psi $ 
and $a^{n_{j}}\rightarrow 1$ 
as $j\rightarrow \infty $. Thus 
$\varphi =\frac{1}{a^{n_{j}}}M_{\tau }^{n_{j}}(\varphi )\rightarrow \psi $ as $j\rightarrow \infty .$ 
Therefore $\varphi =\psi .$ 
Let $U\in \mbox{Con}(F(G_{\tau }))$ and let $x,y\in U.$ 
By Lemma~\ref{l:aeconst}, 
we have 
$\psi (x)-\psi (y)=\lim _{j\rightarrow \infty }(M_{\tau }^{n_{j}}(\varphi )(x)-M_{\tau }^{n_{j}}(\varphi )(y))
=0.$ Therefore, $\varphi =\psi \in C_{F(G_{\tau })}(\CCI ).$ 
Thus, we have proved our lemma.  
\end{proof}
\begin{lem}
\label{l:lspec}
Under the assumptions of Theorem~\ref{t:mtauspec}, 
statement~\ref{t:mtauspec5} holds.
\end{lem}
\begin{proof}
Let $L\in \Min(G_{\tau },\CCI ).$ 
Let $\varphi \in \Ufl $ be such that $M_{\tau }(\varphi )=a\varphi $ for some $a\in S^{1}$ 
and $\sup _{z\in L}|\varphi (z)|=1.$ 
Let $\Omega := \{ z\in L\mid |\varphi (z)|=1\} .$ 
For each $z\in L$, we have 
$|\varphi (z)|= |M_{\tau }(\varphi )(z)|\leq M_{\tau }(|\varphi |)(z)\leq 1.$ 
Thus, $G_{\tau }(\Omega )\subset \Omega .$ 
Since $L\in \Min(G_{\tau },\CCI )$, $\overline{G_{\tau }(z)}=L$ for each 
$z\in \Omega .$ Hence, we obtain $\Omega =L.$ 
By using the argument of the proof of Lemma~\ref{l:lscf}, 
it is easy to see the following claim. \\ 
Claim 1: 
For each $A\in \mbox{Con}(F(G_{\tau }))$ with $A\cap L\neq \emptyset $, 
$\varphi |_{A\cap L}$ is constant.  

 Let $A_{0}\in \mbox{Con}(F(G_{\tau }))$ be an element with $A_{0}\cap L\neq \emptyset $ and 
 let $z_{0}\in A_{0}\cap L$ be a point. We now show the following claim. \\
Claim 2: The map $h\mapsto \varphi (h(z_{0})), h\in \Gamma _{\tau }$, is constant. 

 To show this claim, by claim 1 and that $\bigcup _{h\in \Gamma _{\tau }}\{ h(z_{0})\} $ is a  
compact subset of $F(G_{\tau })$, we obtain that 
$\varphi (z_{0})=\frac{1}{a}M_{\tau }(\varphi )(z_{0})$ is equal to a finite convex combination 
of elements of $S^{1}.$  Since $|\varphi (z_{0})|=1$, it follows that 
$h\mapsto \varphi (h(z_{0})), h\in \Gamma _{\tau }$ is constant. Thus, claim 2 holds. 

 By claim 2 and $M_{\tau }(\varphi )=a\varphi $, we immediately obtain the following claim.\\ 
Claim 3: For each $h\in \Gamma _{\tau }$, $\varphi (h(z_{0}))=a\varphi (z_{0}).$ 

Since $L\in \Min(G_{\tau },\CCI )$, $\overline{G_{\tau }(z_{0})}=L.$ 
Hence there exists an $l\in \NN $ and an element $\beta =(\beta _{1},\ldots ,\beta _{l})\in 
\Gamma _{\tau }^{l}$ such that $\beta _{l}\cdots \beta _{1}(z_{0})\in A_{0}.$ 
From claim 3, it follows that $a^{l}=1.$ 
Thus, we have shown that $\Uvl\subset \{ a\in S^{1}\mid a^{l}=1\} .$ 
Moreover, by claim 3 and the previous argument, we obtain that if $\varphi _{1},\varphi _{2}\in C(L)$ 
with $\sup _{z\in L}|\varphi _{i}(z)|=1$,  
$a_{1},a_{2}\in S^{1}$, and $M_{\tau }(\varphi _{i})=a_{i}\varphi _{i}$, then 
$|\varphi _{i}|\equiv 1$, 
$M_{\tau }(\varphi _{1}\varphi _{2})=a_{1}a_{2}\varphi _{1}\varphi _{2}$ and 
$M_{\tau }(\varphi _{1}^{-1})=a_{1}^{-1}\varphi _{1}^{-1}.$ 
From these arguments, it follows that $\Ufl$ is a finite subgroup of $S^{1}.$ 
Let $r_{L}:= \sharp \Ufl .$ 
Let $a_{L}\in \Ufl$ be an element such that 
$\{ a_{L}^{j}\} _{j=1}^{r_{L}}=\Ufl .$ 
By claim 3 and $\overline{G_{\tau }(z_{0})}=L$, we obtain that 
any element $\varphi \in C(L)$ satisfying $M_{\tau }(\varphi )=a_{L}^{j}\varphi $ 
is uniquely determined by the constant $\varphi |_{A_{0}\cap L}.$ 
Thus, for each $j=1,\ldots r_{L}$, there exists a unique 
$\psi _{L,j}\in \Ufl $ such that 
$M_{\tau }\psi _{L,j}=a_{L}^{j}\psi _{L,j}$ and $\psi _{L,j}|_{A_{0}\cap L}\equiv 1.$ 
It is easy to see that $\{ \psi _{L,j}\} _{j=1}^{r_{L}}$ is a basis of 
$\LSfl $. Moreover, by the previous argument, we obtain that 
$\psi _{L,j}=(\psi _{L,1})^{j}$ for each $j=1,\ldots ,r_{L}.$ 
Thus, we have proved our lemma.   
\end{proof}
\begin{lem}
\label{l:lsfklsfl}
Under the assumptions and notation of Theorem~\ref{t:mtauspec}, 
the map 
$\alpha : \emLSfk \rightarrow \oplus _{L\in \emMin(G_{\tau },\CCI )}\emLSfl$ 
defined by $\alpha (\varphi )=(\varphi |_{L})_{L\in \emMin (G_{\tau },\CCI )}$ 
is a linear isomorphism.  
\end{lem}
\begin{proof}
By Lemma~\ref{l:minfin}, $\sharp \Min(G_{\tau },\CCI )<\infty .$ 
Moreover, elements of $\Min(G_{\tau },\CCI )$ are mutually disjoint. 
Furthermore, for each $L\in \Min (G_{\tau },\CCI )$ and for each $\varphi \in C(S_{\tau })$, 
$(M_{\tau }(\varphi ))|_{L}=M_{\tau }(\varphi |_{L}).$ 
Thus, we easily see that the statement of our lemma holds. 
\end{proof}
\begin{lem}
\label{l:phiinj}
Under the assumptions and notation of Theorem~\ref{t:mtauspec}, 
$\Psi _{S_{\tau }}(\emLSfc )\subset \emLSfk $ and 
$\Psi _{S_{\tau }}: \emLSfc \rightarrow \emLSfk $ is injective.   
\end{lem}
\begin{proof}
We first prove the following claim. 

\noindent Claim 1: $\Psi _{S_{\tau }}: \LSfc \rightarrow C(S_{\tau })$ is injective. 

To prove this claim, let $\varphi \in \Ufc $ and let $a\in S^{1}$ with 
$M_{\tau }(\varphi )=a\varphi $ and suppose $\varphi |_{S_{\tau }}\equiv 0.$ 
Let $\{ n_{j}\} _{j=1}^{\infty }$ be a sequence in $\NN $ such that 
$a^{n_{j}}\rightarrow 1$ as $j\rightarrow \infty .$ 
By Lemma~\ref{l:kcpt}, it follows that 
$\varphi =\frac{1}{a^{n_{j}}}M_{\tau }^{n_{j}}(\varphi )\rightarrow 0$ as $j\rightarrow \infty .$ 
Thus $\varphi =0$, However, this is a contradiction. Therefore, claim 1 holds. 

 The statement of our lemma easily follows from claim 1. 
 Thus, we have proved our lemma.
%
\end{proof}
\begin{lem}
\label{l:cclsfcb0}
Under the assumptions and notation of Theorem~\ref{t:mtauspec}, 
${\cal B}_{0,\tau }$ is a closed subspace of $C(\CCI )$ and  
there exists a direct sum decomposition 
$C(\CCI )=\emLSfc \oplus {\cal B}_{0,\tau }.$ 
Moreover, $\dim _{\CC }(\emLSfc )<\infty $ and 
the projection $\pi :C(\CCI )\rightarrow \emLSfc $ is continuous.   
Furthermore, setting $r:=\prod _{L\in \emMin (G_{\tau },\CCI )}r_{L} $, we have that  
for each $\varphi \in \emLSfc $, $M_{\tau }^{r}(\varphi )=\varphi .$ 
\end{lem}
\begin{proof}
By Theorem~\ref{kerJthm1}, 
for each $\varphi \in C(\CCI )$, 
$\overline{\bigcup _{n=1}^{\infty }\{ M_{\tau }^{n}(\varphi )\} }$ is 
compact in $C(\CCI ).$ 
By \cite[p.352]{Ly}, 
it follows that there exists a direct sum decomposition 
$C(\CCI ) =\overline{\LSfc} \oplus  {\cal B}_{0,\tau }.$ 
Moreover, combining Lemma~\ref{l:phiinj},  Lemma~\ref{l:lspec} and Lemma~\ref{l:lsfklsfl}, 
we obtain that  
$\dim _{\CC }(\LSfc )<\infty $ and for each $\varphi \in \LSfc $, $M_{\tau }^{r}(\varphi )=\varphi .$ 
Hence there exists a direct sum decomposition 
$C(\CCI )=\LSfc \oplus {\cal B}_{0,\tau }.$ 
Since ${\cal B}_{0,\tau }$ is closed in $C(\CCI )$ and 
$\dim _{\CC }(\LSfc )<\infty $, 
it follows that the projection $\pi : C(\CCI )\rightarrow \LSfc $ is continuous. 
Thus, we have proved our lemma.  
\end{proof}
\begin{lem}
\label{l:lsiso}
Under the assumptions and notation of Theorem~\ref{t:mtauspec}, 
statement~\ref{t:mtauspec6} holds. 
\end{lem}
\begin{proof}
It is easy to see that $\Psi _{S_{\tau }}\circ M_{\tau }=M_{\tau }\circ \Psi _{S_{\tau }}$ 
on $\LSfc .$ To prove our lemma, 
by Lemma~\ref{l:phiinj}, it is enough to show that 
$\Psi _{S_{\tau }}: \LSfc \rightarrow \LSfk \cong \oplus _{L\in \Min(G_{\tau },\CCI )}\LSfl $ 
is surjective. In order to show this, 
let $L\in \Min(G_{\tau },\CCI )$ and let $a_{L},r_{L}$, and $\{ \psi _{L,j}\} _{j=1}^{r_{L}}$ 
be as in Lemma~\ref{l:lspec} (statement~\ref{t:mtauspec5} of Theorem~\ref{t:mtauspec}).  
Let $\tilde{\psi }_{L,j}\in C(\CCI )$ be an element such that 
$\tilde{\psi }_{L,j}|_{L}=\psi _{L,j}$  
and 
$\tilde{\psi }_{L,j}|_{L'}\equiv 0$ for each $L'\in \Min(G_{\tau },\CCI )$ with 
$L'\neq L.$ 
Let $r$ be the number in Lemma~\ref{l:cclsfcb0} and 
let $\pi : C(\CCI )\rightarrow \LSfc $ be the projection. Then 
$M_{\tau }^{rn}(\tilde{\psi }_{L,j})\rightarrow \pi (\tilde{\psi }_{L,j})$ as 
$n\rightarrow \infty .$ Therefore, 
$\pi (\tilde{\psi }_{L,j})|_{L}=\lim _{n\rightarrow \infty }
M_{\tau }^{rn}(\tilde{\psi }_{L,j}|_{L})=\psi _{L,j}.$ 
Similarly, $\pi (\tilde{\psi }_{L,j})|_{L'}\equiv 0$ for each 
$L'\in \Min(G_{\tau },\CCI )$ with $L'\neq L.$ 
Therefore, $\Psi _{S_{\tau }}: \LSfc \rightarrow \LSfk $ is surjective. 
Thus, we have completed the proof of our lemma. 
\end{proof}
\begin{lem}
\label{l:pf7}
Under the assumptions of Theorem~\ref{t:mtauspec}, 
statement \ref{t:mtauspec7} holds.
\end{lem}
\begin{proof}
Statement~\ref{t:mtauspec7} of Theorem~\ref{t:mtauspec} follows 
from Lemma~\ref{l:lsiso}, \ref{l:lsfklsfl}, and Lemma~\ref{l:lspec}.
\end{proof}   
\begin{lem}
\label{l:pf2-1}
Under the assumptions of Theorem~\ref{t:mtauspec}, 
statement~\ref{t:mtauspec2-1} holds. 
\end{lem}
\begin{proof}
Let $\{ \varphi _{j}\} $ and $\{ \alpha _{j}\} $ be as in 
statement~\ref{t:mtauspec2-1} of Theorem~\ref{t:mtauspec}. 
Let $\varphi \in C(\CCI )$. 
Then there exist a unique family $\{ \rho _{j}(\varphi )\} _{j=1}^{q}$ in $\CC $  
such that $\pi (\varphi )=\sum _{j=1}^{q}\rho _{j}(\varphi )\varphi _{j}.$ 
It is easy to see that $\rho _{j}:C(\CCI )\rightarrow \CC $ is a linear 
functional. 
Moreover, since $\pi :C(\CCI )\rightarrow \LSfc $ 
is continuous (Lemma~\ref{l:cclsfcb0}), 
each $\rho _{j}:C(\CCI )\rightarrow \CC $ is continuous. 
By Lemma~\ref{l:cclsfcb0} again, 
it is easy to see that $\rho _{i}(\varphi _{j})=\delta _{ij}.$ 
In order to show $M_{\tau }^{\ast }(\rho _{j})=\alpha _{j}\rho _{j}$, 
let $\varphi \in C(\CCI )$ and let $\zeta := \varphi -\pi (\varphi ).$ 
Then 
$M_{\tau }(\varphi )=\sum _{j=1}^{q}\rho _{j}(\varphi )\alpha _{j}\varphi _{j}+M_{\tau }(\zeta ).$ 
Hence $\rho _{j}(M_{\tau }(\varphi ))=\alpha _{j}\rho _{j}(\varphi ).$ 
Therefore, $M_{\tau }^{\ast }(\rho _{j})=\alpha _{j}\rho _{j}.$ 
In order to prove that $\{ \rho _{j}\} $ is a basis of 
$\LSfac$, let $\rho \in \Ufac $ and $a\in S^{1}$ be such that 
$M_{\tau }^{\ast }(\rho )=a\rho .$ 
Let $r$ be the number in Lemma~\ref{l:cclsfcb0}. 
Let $\{ n_{i}\} _{i=1}^{\infty }$ be a sequence in $\NN $ such that 
$a^{rn_{i}}\rightarrow 1$ as $i\rightarrow \infty .$ 
Let $\varphi \in C(\CCI) $ and let 
$\zeta =\varphi -\pi (\varphi ).$ 
Then $\rho (\varphi )=\frac{1}{a^{rn_{i}}}(M_{\tau }^{\ast })^{rn_{i}}(\rho )(\varphi )
=\frac{1}{a^{rn_{i}}}\rho (\sum _{j=1}^{q}\rho _{j}(\varphi )\varphi _{j}+
M_{\tau }^{rn_{i}}(\zeta ))\rightarrow \sum _{j=1}^{q}\rho _{j}(\varphi )
\rho (\varphi _{j})$ as $i\rightarrow \infty .$ 
Therefore $\rho \in \mbox{LS}(\{ \rho _{j}\} _{j=1}^{q}). $ Thus $\{ \rho _{j}\} _{j=1}^{q}$ is a basis of 
$\LSfac. $ 
In order to prove supp$\, \rho _{j}\subset S_{\tau }$, 
let $\varphi \in C(\CCI )$ be such that supp$\, \varphi \subset \CCI \setminus S_{\tau }.$ 
Let $\zeta =\varphi -\pi (\varphi ).$ Then 
$\varphi =\sum _{j=1}^{q}\rho _{j}(\varphi )\varphi _{j}+\zeta .$ 
Let $r$ be the number in Lemma~\ref{l:cclsfcb0}. 
Then $M_{\tau }^{rn}(\varphi )\rightarrow \sum _{j=1}^{q}\rho _{j}(\varphi )\varphi _{j}$ 
as $n\rightarrow \infty .$ 
Hence $\sum _{j=1}^{q}\rho _{j}(\varphi )\varphi _{j}|_{S_{\tau }}=
\lim _{n\rightarrow \infty }M_{\tau }^{rn}(\varphi |_{S_{\tau }})=0.$ 
By Lemma~\ref{l:phiinj}, we obtain $\rho _{j}(\varphi )=0$ for each $j.$ 
Therefore supp$\, \rho _{j}\subset S_{\tau }$ for each $j.$ 
Thus, we have completed the proof of our lemma. 
\end{proof}
\begin{lem}
\label{l:pf7-1}
Under the assumptions of Theorem~\ref{t:mtauspec}, 
statement~\ref{t:mtauspec7-1} holds. 
\end{lem} 
\begin{proof}
By items (b), (c) of statement~\ref{t:mtauspec2-1} of Theorem~\ref{t:mtauspec} 
(see Lemma~\ref{l:pf2-1}), 
we obtain $\Uvac =\Uvc. $ By using the same method as that in the proof of Lemma~\ref{l:pf2-1},  
we obtain $\Uvak=\Uvk $ and $\Uval=\Uvl $ for each $L\in \Min(G_{\tau },\CCI ).$ 
Thus, we have completed the proof of our lemma. 
\end{proof}
\begin{lem}
\label{l:suppllj}
Under the assumptions of Theorem~\ref{t:mtauspec}, 
statements~\ref{t:mtauspec7-2} and \ref{t:mtauspec8} of Theorem~\ref{t:mtauspec} hold. 
\end{lem}
\begin{proof}
Let $L\in \Min(G_{\tau },\CCI )$ and 
let $z_{0}\in L\cap F(G_{\tau })$ be a point. 
Let $\{ \psi _{L,j}\} _{j=1}^{r_{L}}$ be as in statement \ref{t:mtauspec5} of 
Theorem~\ref{t:mtauspec}. 
We may assume $\psi _{L,1}(z_{0})=1.$ 
For each $j=1,\ldots ,r_{L}$,  
let $L_{j}:= \{ z\in L\mid \psi _{L,1}(z)=a_{L}^{j}\} .$ 
By claim 3 in the proof of Lemma~\ref{l:lspec}, 
$L$ is equal to the disjoint union of 
compact subsets $L_{j}$, and 
for each $h\in \Gamma _{\tau }$ and 
for each $j=1,\ldots ,r_{L}$,  
$h(L_{j})\subset L_{j+1}$ where $L_{r_{L}+1}:=L_{1}.$ 
In particular, 
$G_{\tau }^{r_{L}}(L_{j})\subset L_{j}$ for each 
$j=1,\ldots ,r_{L}.$ 
Since $\overline{G_{\tau }(z)}=L$ for each 
$z\in L$, it follows that $\overline{G_{\tau }^{r_{L}}(z)}=L_{j}$ for each 
$j=1,\ldots ,r_{L}$ and for each $z\in L_{j}.$ 
Therefore, $\{ L_{j}\} _{j=1}^{r_{L}}=
\Min (G_{\tau }^{r_{L}}, L).$ 
Thus, statement~\ref{t:mtauspec7-2} of Theorem~\ref{t:mtauspec} holds. 
Let $j\in \{ 1,\ldots ,r_{L}\} .$ 
Let us consider the argument in the proof of Lemma~\ref{l:lspec}, replacing $L$ by $L_{j}$ and 
$G_{\tau }$ by $G_{\tau }^{r_{L}}$. Then the number $r_{L}$ in the proof of Lemma~\ref{l:lspec}
is equal to $1$ in this case. 
For, if there exists a non-zero element $\psi \in C(L_{j})$ and a $b=e^{\frac{2\pi i}{s}}\neq 1$ with 
$s\in \NN $ such that $M_{\tau }^{r_{L}}(\psi ) =b\psi $, 
then extending $\psi $ to the element $\tilde{\psi }\in C(L)$ by setting 
$\tilde{\psi }|_{L_{i}}=0$ for each $i$ with $i\neq j$, 
and setting $\hat{\psi }:= \sum _{j=1}^{sr_{L}}(e^{\frac{2\pi i}{sr_{L}}})^{-j}M_{\tau }^{j}(\tilde{\psi })\in C(L)$, 
we obtain $\hat{\psi }\neq 0$ and $M_{\tau }(\hat{\psi })=e^{\frac{2\pi i}{sr_{L}}}\hat{\psi }$, 
which is a contradiction.    
Therefore, by using the argument in the proof of Lemmas~\ref{l:lspec} and \ref{l:cclsfcb0}, 
it follows that for each $\varphi \in C(L_{j})$, 
there exists a number $\omega _{L,j}(\varphi )\in \CC $ 
such that 
$M_{\tau }^{nr_{L}}(\varphi )\rightarrow \omega _{L,j}(\varphi )\cdot 1_{L_{j}}$ 
as $n\rightarrow \infty .$ 
It is easy to see that $\omega _{L,j}$ is a positive linear functional. 
Therefore, $\omega _{L,j}\in {\frak M}_{1}(L_{j}).$ 
Thus, $\omega _{L,j}$ is the unique $(M_{\tau }^{\ast })^{r_{L}}$-invariant element 
of ${\frak M}_{1}(L_{j}).$ 
Since $L_{j}\in \Min(G_{\tau }^{r_{L}},L)$, 
it is easy to see that supp$\, \omega _{L,j}=L_{j}.$ 
Since $M_{\tau }^{\ast }(\omega _{L,j})\in {\frak M}_{1}(L_{j+1})$ and 
$(M_{\tau }^{\ast })^{r_{L}}(M_{\tau }^{\ast }(\omega _{L,j}))=
M_{\tau }^{\ast }(\omega _{L,j})$, 
it follows that $M_{\tau }^{\ast }(\omega _{L,j})=\omega _{L,j+1}$ for each 
$j=1,\ldots ,r_{L}$, where $\omega _{L,r_{L}+1}:=\omega _{L,1}.$ 
For each 
$i=1,\ldots, r_{L}$, 
let $\rho _{L,i}:= \frac{1}{r_{L}}\sum _{j=1}^{r_{L}}a_{L}^{-ij}\omega _{L,j}\in 
C(L)^{\ast }$  
and $\tilde{\psi } _{L,i}:=\sum _{j=1}^{r_{L}}a_{L}^{ij}1_{L_{j}}\in C(L).$ 
Then it is easy to see that 
$M_{\tau }^{\ast }(\rho _{L,i})=a_{L}^{i}\rho _{L,i}$,  
$M_{\tau }(\tilde{\psi }_{L,i})=a_{L}^{i}\tilde{\psi }_{L,i}$, and 
$\rho _{L,i}(\tilde{\psi }_{L,j})=\delta _{ij}.$ 
By Lemma~\ref{l:lsiso}, 
there exists a unique element $\varphi _{L,i}\in \LSfc$ such that 
$\varphi _{L,i}|_{L}=\tilde{\psi }_{L,i}$ and $\varphi _{L,i}|_{L'}\equiv 0$ for 
each $L'\in \Min (G_{\tau },\CCI )$ with $L'\neq L.$ 
It is easy to see that $\{ \varphi _{L,i}\} _{L,i}$ and $\{ \rho _{L,i}\} _{L,i}$ 
are the desired families. Thus, we have completed the proof of our lemma. 
\end{proof} 
\begin{lem}
\label{l:mtdualpf}
Under the assumptions of Theorem~\ref{t:mtauspec}, 
statement~\ref{t:mtauspecdual} holds. 
\begin{proof}
Statement~\ref{t:mtauspecdual} follows from Lemma~\ref{l:pf2-1}. 
\end{proof}

\end{lem}
\begin{lem}
\label{l:mt9pf}
Under the assumptions and notation of Theorem~\ref{t:mtauspec}, 
statement~\ref{t:mtauspec9} holds.  
\end{lem}
\begin{proof}
By Lemma~\ref{l:kcpt}, 
we have $\sum _{L\in \Min(G_{\tau },\CCI )}T_{L,\tau }(z)=1$ for each 
$z\in \CCI .$ 
For each $L\in \Min(G_{\tau },\CCI )$, 
let $W_{L}:= \bigcup _{A\in \text{Con}(F(G_{\tau })),A\cap L\neq \emptyset }A.$ 
Then $G(W_{L})\subset W_{L}$ and $W=\bigcup _{L}W_{L}.$ 
By Lemma~\ref{l:lscf} and Lemma~\ref{l:lsiso}, 
we obtain that $\overline{W_{L}}\cap \overline{W_{L'}}=\emptyset $ whenever 
$L, L'\in \Min(G_{\tau },\CCI )$ and $L\neq L'.$ For each $L\in \Min(G_{\tau },\CCI )$, 
let $\varphi _{L}\in C(\CCI )$ be such that $\varphi _{L}|_{W_{L}}\equiv 1$ and 
$\varphi _{L}|_{\bigcup _{L'\neq L}W_{L'}}\equiv 0.$ 
From Lemma~\ref{l:kcpt} and Lemma~\ref{l:aeconst}, 
it follows that 
\begin{equation}
\label{eq:tltinv}
T_{L,\tau }(z)=\int _{X_{\tau }}\lim _{n\rightarrow \infty }\varphi _{L}(\gamma _{n,1}(z))\ 
d\tilde{\tau }(\gamma )=\lim _{n\rightarrow \infty }\int _{X_{\tau }}\varphi _{L}(\gamma _{n,1}(z))\ d\tilde{\tau }(\gamma )
 =\lim _{n\rightarrow \infty }M_{\tau}^{n}(\varphi _{L})(z)
\end{equation}
 for each $z\in \CCI .$ 
Combining (\ref{eq:tltinv}) and Theorem~\ref{kerJthm1}, we obtain that 
$T_{L,\tau }\in C(\CCI )$ for each $L\in \Min(G_{\tau },\CCI )$. 
Moreover, from (\ref{eq:tltinv}) again, we obtain 
$M_{\tau }(T_{L,\tau })=T_{L,\tau }.$  

Thus, we have proved our lemma. 
\end{proof} 
\begin{lem}
\label{l:mt9-2pf} 
Under the assumptions of Theorem~\ref{t:mtauspec}, statement~\ref{t:mtauspec9-2} holds.  
\end{lem} 
\begin{proof} 
We now suppose  $\sharp \Min (G_{\tau },\CCI)\geq 2$. 
Let $L\in \Min(G_{\tau },\CCI ).$ 
Since 
$T_{L,\tau }:\CCI \rightarrow [0,1]$ is continuous, and since 
$T_{L,\tau }|_{L}\equiv 1$ and $T_{L,\tau }|_{L'}\equiv 0$ for each $L'\neq L$, 
it follows that $T_{L,\tau }(\CCI )=[0,1].$ Since 
$T_{L,\tau }$ is continuous on $\CCI $ and since $T_{L,\tau }\in C_{F(G_{\tau })}(\CCI )$, 
we obtain that $T_{L,\tau }(J(G_{\tau }))=[0,1].$ In particular, 
$\dim _{\CC }(\LSfc )>1.$  
Thus, we have proved our lemma. 
\end{proof}
\begin{lem}
\label{l:kattfpt1}
Under the assumptions of Theorem~\ref{t:mtauspec}, 
statement~\ref{t:mtauspec10} holds. 
\end{lem}
\begin{proof}
Let $L\in \Min(G_{\tau },\CCI ).$ 
Since $J_{\ker }(G_{\tau })=\emptyset $, 
$L\cap F(G_{\tau })\neq \emptyset .$ 
Let $a\in L\cap F(G_{\tau }).$ 
Since $\overline{G_{\tau }(a)}=L$, 
we obtain $L=\overline{L\cap F(G_{\tau })}.$ 
Hence, in order to prove our lemma, it suffices to prove the following claim. \\  
Claim: 
$L\cap F(G_{\tau })\subset 
\overline{\{ z\in L\cap F(G_{\tau })\mid \exists g\in G_{\tau } \mbox{ s.t. } g(z)=z, |m(g,z)|<1\} }.$ 

In order to prove the above claim, let $b\in L\cap F(G_{\tau }).$ Let $U\in \mbox{Con}(F(G_{\tau }))$ 
with $b\in U.$ 
We take the hyperbolic metric on each element of $\mbox{Con}(F(G_{\tau })).$ 
For each $\epsilon _{0}>0$ and for each $c\in F(G_{\tau })$, let $B_{h}(c,\epsilon _{0})$ 
be the disc with center $c$ and radius $\epsilon _{0}$ in $F(G_{\tau })$ 
with respect to the hyperbolic distance.  
Let $\epsilon >0.$ By Lemma~\ref{l:hypmetcont}, there exists an element $g_{1}\in G_{\tau }$ 
such that $g_{1}(B_{h}(b,\epsilon ))\subset B_{h}(g_{1}(b),\frac{\epsilon }{2}).$ 
Since $\overline{G_{\tau }(g_{1}(b))}=L$, there exists an element $g_{2}\in G_{\tau }$ 
such that $\overline{g_{2}(B_{h}(g_{1}(b),\frac{\epsilon }{2}))}\subset B_{h}(b,\epsilon ).$ 
Thus $\overline{g_{2}g_{1}(B_{h}(b,\epsilon ))}\subset B_{h}(b,\epsilon ).$ 
Let $g=g_{2}g_{1}.$ 
Then $z_{0}:= \lim _{n\rightarrow \infty }g^{n}(b)\in B_{h}(b,\epsilon )\cap L$ is an attracting fixed point of 
$g.$ Therefore, we have proved the above claim. 
Thus, we have proved our lemma.
\end{proof}
\begin{lem}
\label{l:kattfpt2}
Under the assumptions of Theorem~\ref{t:mtauspec}, 
statement~\ref{t:mtauspec11} holds. 
\end{lem}
\begin{proof}
We will modify the proof of Lemma~\ref{l:kattfpt1}. 
If $\G _{\tau } \cap \Ratp\neq\ \emptyset $, then by Lemma~\ref{l:hypmetcont}, 
we may assume that the element $g_{1}$ in the proof of Lemma~\ref{l:kattfpt1} belongs to 
$\Ratp.$ 
Therefore, 
$S_{\tau }= \{ \overline{z\in F(G)\cap S_{\tau }\mid \exists g\in G_{\tau }\cap \Ratp \mbox{ {\em s.t.} } g(z)=z, |m(g,z)|<1} \}.$ 
Since any attracting fixed point of $g\in G_{\tau }\cap \Ratp$ belongs to $UH(G_{\tau })$, 
our lemma holds. 
\end{proof}
\begin{lem}
\label{l:pfmtauspec12}
Under the assumptions of Theorem~\ref{t:mtauspec}, 
statement~\ref{t:mtauspec12} holds.
\end{lem}
\begin{proof}
Suppose $\dim _{\CC }(\LSfc )>1$ and let 
$\varphi \in \LSfc _{nc}.$ 
Let $A:= \varphi (\CCI )\setminus \varphi (F(G_{\tau })) .$ 
Since $\varphi \in C_{F(G_{\tau })}(\CCI )$ and 
since $\sharp \mbox{Con}(F(G_{\tau }))\leq \aleph _{0}$, 
we have $\sharp A>\aleph _{0}.$ 
Moreover, since $\varphi $ is continuous on $\CCI $, 
it is easy to see that for each $t\in A$, 
$\emptyset \neq \varphi ^{-1}(\{ t\} )\subset J_{res}(G_{\tau }).$ 
Thus we have proved our lemma.  
\end{proof}
\begin{lem}
\label{l:pfmtauspeccfi}
Under the assumptions of Theorem~\ref{t:mtauspec}, 
statement~\ref{t:mtauspeccfi} holds.
\end{lem}
\begin{proof}
Suppose $\dim _{\CC }(\LSfc )>1$ and int$(J(G_{\tau }))=\emptyset .$ 
Let $\varphi \in \LSfc _{nc}.$ 
Then $\sharp \varphi (\CCI )>\aleph _{0}.$ 
Since  int$(J(G_{\tau }))=\emptyset $ and 
$\varphi $ is continuous on $\CCI $, 
we have $\varphi (\CCI )=\overline{\varphi (F(G_{\tau }))}.$ 
Therefore, $\sharp \mbox{Con}(F(G_{\tau }))=\infty .$ 
Thus, we have proved our lemma. 
\end{proof}

We now prove Theorem~\ref{t:mtauspec}. \\ 

\noindent {\bf Proof of Theorem~\ref{t:mtauspec}:}
Combining Lemma~\ref{l:sjg3}--Lemma~\ref{l:pfmtauspeccfi}, 
we easily see that all of the statements 1--20 of Theorem~\ref{t:mtauspec} hold. 
Statement~\ref{t:mtauspecdeg1} follows from statements~\ref{t:mtauspec10} and \ref{t:mtauspec7-2}.   
\qed 
\subsection{Proofs of results in subsection~\ref{ProT}}
\label{pfProT}
In this subsection, we give the proofs of the results in subsection~\ref{ProT}.
\begin{lem}
\label{tinftyphilem1}
Let $\tau \in {\frak M}_{1}({\cal P})$ and suppose that 
$\infty \in F(G_{\tau }).$ Let $\phi \in C(\CCI )$ be such that 
$\phi $ is equal to constant function $1$ around $\infty $ and such that 
{\em supp}$\, \phi \subset F_{\infty }(G_{\tau }).$ 
Then, for each $\gamma \in X_{\tau }$, 
$\gamma _{n,1}\rightarrow \infty $ as $n\rightarrow \infty $ locally uniformly 
on $F_{\infty }(G_{\tau })$ and for each $y\in \CCI $, 
\begin{align*}
T_{\infty ,\tau }(y)
& = \tilde{\tau }(\{ \gamma \in X_{\tau }\mid \phi (\gamma _{n,1}(y))\rightarrow \infty \  
(n\rightarrow \infty )\} )\\ 
& = \tilde{\tau }(\{ \gamma \in X_{\tau }\mid \exists n\in \NN \ \phi (\gamma _{n,1}(y))=1  \} )=
\lim _{n\rightarrow \infty }M_{\tau }^{n}(\phi )(y).
\end{align*}
In particular, $M_{\tau }(T_{\infty ,\tau })=T_{\infty ,\tau }.$ 
\end{lem}
\begin{proof}
First, we show the following claim.\\ 
Claim. For each $\gamma \in X_{\tau }$, 
$\gamma _{n,1}\rightarrow \infty $ as $n\rightarrow \infty $ locally uniformly 
on $F_{\infty }(G_{\tau }).$ 

 To prove the claim, let $\gamma \in X_{\tau }.$ 
Then $\{ \gamma _{n,1}\} _{n=1}^{\infty }$ is normal in 
$F_{\infty }(G_{\tau }).$ Let $\{ n_{j}\} _{j\in \NN }$ be a 
sequence in $\NN $ such that 
$\gamma _{n_{j},1}$ converges to some $\alpha $ as $j\rightarrow \infty $ 
locally uniformly on $F_{\infty }(G_{\tau }).$  
Since 
 the local degree of $\gamma _{n_{j},1}$ at $\infty $ tends to 
 $\infty $, $\alpha $ should be the constant $\infty .$ Thus, the above claim holds. 

 Let $\gamma \in X_{\tau }$ and let $y\in \CCI .$  By the above claim, the following (1),(2) and (3) are equivalent:
(1) $\gamma _{n,1}(y)\rightarrow \infty $ as $n\rightarrow \infty .$ 
(2) $\phi (\gamma _{n,1}(y))\rightarrow 1$ as $n\rightarrow \infty .$ 
(3) There exists an $n\in \NN $ such that $\phi (\gamma _{n,1}(y))=1.$

Moreover, by the claim, for a point $y\in \CCI $, either 
$\phi (\g _{n,1}(y))\rightarrow 1$ or $\phi (\g _{n,1}(y))\rightarrow 0.$ 
Hence 
\begin{align*}
T_{\infty ,\tau }(y)=\tilde{\tau }(\{ \gamma \in X_{\tau } \mid 
\phi (\gamma _{n,1}(y))\rightarrow 1\} )
& = \int _{X_{\tau }}\lim _{n\rightarrow \infty }\phi (\g _{n,1}(y))\ d\tilde{\tau }(\g )\\ 
& = \lim _{n\rightarrow \infty }\int _{X_{\tau }}\phi (\gamma _{n,1}(y))\ d\tilde{\tau }(\g )=
\lim _{n\rightarrow \infty }M_{\tau }^{n}(\phi )(y).
\end{align*}
From these arguments, the statement of the lemma follows. 
\end{proof}
\begin{lem}
\label{lem:fpt0tconti}
Let $\tau \in {\frak M}_{1}({\cal P})$ and suppose that 
$\infty \in F(G_{\tau }).$ Let $y\in F_{pt}^{0}(\tau )$ be a point. 
Then, $T_{\infty ,\tau }$ is continuous at $y.$ 
\end{lem}
\begin{proof}
Let $\phi \in C(\CCI )$ be as in Lemma~\ref{tinftyphilem1}. 
By Lemma~\ref{tinftyphilem1}, we have that for each $y\in \CCI $, 
$T_{\infty ,\tau }(y)=\lim _{n\rightarrow \infty }M_{\tau }^{n}(\phi )(y).$ 
From Lemma~\ref{FJmeaslem1}-\ref{FJmeaslem1-2}, 
it follows that $T_{\infty ,\tau }$ is continuous at $y.$ 
Thus, we have completed the proof of our lemma.
\end{proof}

\begin{lem}
\label{fmeasytconti}
Let $\tau \in {\frak M}_{1}({\cal P}).$ Suppose that 
$\infty \in F(G_{\tau }) $ and $F_{meas}(\tau )={\frak M}_{1}(\CCI ).$ 
Then, $T_{\infty ,\tau }: \CCI \rightarrow [0,1]$ is continuous on $\CCI .$  
\end{lem}
\begin{proof}
The statement of our lemma easily follows from Lemma~\ref{lem:fpt0tconti}.
\end{proof}
We now prove Theorem~\ref{kerJthm2}. 

\noindent {\bf Proof of Theorem~\ref{kerJthm2}:} 
Since supp$\,\tau $ is compact, $\infty \in F(G_{\tau }).$ 
Combining Theorem~\ref{kerJthm1} and Lemma~\ref{fmeasytconti}, 
the statement of Theorem~\ref{kerJthm2} follows.  
\qed 

\begin{lem}
\label{tinftyconstlem1}
Let $\tau \in {\frak M}_{1}({\cal P}).$ Suppose $\infty \in F(G_{\tau }).$ 
Then, for each $U\in \mbox{{\em Con}}(F(G_{\tau }))$,  
there exists a constant $C_{U}\in [0,1]$ such that 
$T_{\infty ,\tau }|_{U}\equiv C_{U}.$  
\end{lem}
\begin{proof}
Let $U\in \mbox{Con}(F(G_{\tau }))$ and let $y\in U.$ 
Moreover, let $\gamma \in X_{\tau }.$  
By Lemma~\ref{tinftyphilem1} and Lemma~\ref{ocminvlem}, 
if $\gamma _{n,1}(y)\rightarrow \infty $ as 
$n\rightarrow \infty $, then 
for each $y'\in U$, 
$\gamma _{n,1}(y')\rightarrow \infty $ as $n\rightarrow \infty .$ 
Thus, there exists a constant $C_{U}\in [0,1]$ such that 
$T_{\infty ,\tau }|_{U}\equiv C_{U}.$ 
\end{proof}
We now prove Lemma~\ref{Tlem1}. 

\noindent {\bf Proof of Lemma~\ref{Tlem1}:}
Since supp$\, \tau $ is compact, it follows that $\infty \in F(G_{\tau })$, and  
the statement of Lemma~\ref{Tlem1} follows from 
Lemma~\ref{tinftyconstlem1}. 
\qed 

\begin{lem}
\label{ksetfundlem1}
Let $G$ be a polynomial semigroup generated by a family of ${\cal P}.$ 
Then, $\hat{K}(G)$ is a compact subset of $\CC $, 
$g(\hat{K}(G))\subset \hat{K}(G)$ for each $g\in G$,  
 $\partial \hat{K}(G)\subset J(G)$, and $F(G)\cap \hat{K}(G)= \mbox{{\em int}}(\hat{K}(G)).$  
\end{lem}
\begin{proof}
Let $h\in G$ be an element. Then $\hat{K}(G)=\bigcap _{g\in G}g^{-1}(K(h)).$ 
Thus, $\hat{K}(G)$ is a compact subset of $\CC $ and for each $g\in G$, 
$g(\hat{K}(G))\subset \hat{K}(G).$ Hence, we obtain that 
$\partial \hat{K}(G)\cap F(G)=\emptyset .$ Therefore, 
$\partial \hat{K}(G)\subset J(G)$ and $F(G)\cap \hat{K}(G)\subset \mbox{int}(\hat{K}(G)).$ 
Moreover, it is easy to see int$(\hat{K}(G))\subset F(G)\cap \hat{K}(G).$ 
Thus we have completed the proof of our lemma.
\end{proof}
We now prove Proposition~\ref{p:chtinfty}. \\ 
{\bf Proof of Proposition~\ref{p:chtinfty}:}
It is easy to see that 
$M_{\tau }(T_{\infty ,\tau })=T_{\infty ,\tau }$, 
$T_{\infty ,\tau }|_{F_{\infty }(G_{\tau })}\equiv 1$ and 
$T_{\infty ,\tau }|_{\hat{K}(G_{\tau })}\equiv 0.$ 
Let $\varphi :\CCI \rightarrow \RR $ be a bounded Borel measurable function such that 
$\varphi =M_{\tau }(\varphi )$, 
$\varphi |_{F_{\infty }(G_{\tau })}\equiv 1$ and $\varphi |_{\hat{K}(G_{\tau })}\equiv 0.$ 
For each $L\in \Min(G_{\tau },\CCI )$ with $L\neq \{ \infty \} $, 
$L\subset \hat{K}(G_{\tau }).$ 
Hence, by Theorem~\ref{t:mtauspec}-\ref{t:mtauspec4} and Lemma~\ref{ksetfundlem1}, 
we obtain that $\varphi (z)=\lim _{n\rightarrow \infty }M_{\tau }^{n}(\varphi )(z)=T_{\infty ,\tau }(z)$ for each 
$z\in \CCI .$ 
Thus, we have proved Proposition~\ref{p:chtinfty}. 
\qed

We now prove Lemma~\ref{kemptylem1}.\\ 
\noindent {\bf Proof of Lemma~\ref{kemptylem1}:} 

It is easy to see that (1) $\Rightarrow $ (2). 

We now show (2) $\Rightarrow $ (3).  
Suppose $T_{\infty ,\tau }|_{J(G_{\tau })} \equiv 1$ and 
$\hat{K}(G_{\tau })\neq \emptyset .$ 
Let $y\in \partial \hat{K}(G_{\tau })\subset J(G_{\tau }).$  
Since we are assuming $T_{\infty ,\tau }|_{J(G_{\tau })} \equiv 1$, 
there exists a $\gamma \in X_{\tau }$ such that 
$\gamma _{n,1}(y)\rightarrow \infty .$ 
However, this contradicts $y\in \hat{K}(G_{\tau }).$ 
Thus, we have proved (2) $\Rightarrow $ (3). 

 We now prove (3) $\Rightarrow $ (1). 
Since supp$\,\tau $ is compact, $\infty \in F(G_{\tau }).$ 
Let $V:= F_{\infty }(G_{\tau }).$ By Lemma~\ref{ocminvlem}, 
for each $g\in G_{\tau }$, $g(V)\subset V.$ 
Moreover, 
we have $\bigcap _{g\in G_{\tau }}g^{-1}(\CCI \setminus V)=\hat{K}(G_{\tau }).$ 
Hence, from Lemma~\ref{Lkerlem} and Lemma~\ref{tinftyphilem1}, 
it follows that if $\hat{K}(G_{\tau })=\emptyset $, 
then for each $y\in \CCI $, for $\tilde{\tau }$-a.e. $\gamma \in X_{\tau }$,  
$\gamma _{n,1}(y)\rightarrow \infty $ as $n\rightarrow \infty .$ 
Hence, for each $y\in \CCI $, $T_{\infty ,\tau }(y)=1.$ 
Therefore, we have proved  (3) $\Rightarrow $ (1). 

 Thus, we have proved Lemma~\ref{kemptylem1}. 
\qed 

We now prove Theorem~\ref{kerJthm3}.\\ 
{\bf Proof of Theorem~\ref{kerJthm3}:} 
We first prove statement \ref{kerJthm3-1}. 
Let $y\in \partial \hat{K}(G_{\tau })$ be a point. 
By Lemma~\ref{ksetfundlem1}, $y\in J(G_{\tau }).$ 
Since $J_{\ker }(G_{\tau })=\emptyset $, 
there exists an element $g\in G_{\tau }$ such that 
$g(z)\in F(G_{\tau }).$ By Lemma~\ref{ksetfundlem1} again, 
we obtain $g(z)\in \mbox{int}(\hat{K}(G_{\tau })).$ 
Therefore, int$(\hat{K}(G_{\tau }))\neq \emptyset .$ 

 We next show statement~\ref{kerJthm3-2}. 
By Theorem~\ref{kerJthm2}, $T_{\infty ,\tau }:\CCI \rightarrow [0,1]$ 
is continuous. 
Furthermore, 
since supp$\, \tau $ is compact, $\infty \in F(G_{\tau }).$ 
Since $T_{\infty,\tau }|_{\hat{K}(G_{\tau })}\equiv 0$ and 
$T_{\infty ,\tau }|_{F_{\infty }(G_{\tau })}\equiv 1$, 
it follows that $T_{\infty ,\tau }(\CCI )=[0,1].$ 
Let $t\in [0,1]$ be any number.    
From the above argument, there exists a point $z_{0}\in \CCI $ such that 
$T_{\infty ,\tau }(z_{0})=t.$ 
Suppose $z_{0}\in F(G_{\tau }).$ Then 
denoting by $U$ the connected component of $F(G_{\tau })$ containing $z_{0}$, 
Theorem~\ref{kerJthm2} and Lemma~\ref{Tlem1} imply that 
$T_{\infty ,\tau }|_{\overline{U}}\equiv t.$ 
Since $\partial U\subset J(G_{\tau })$, 
it follows that there exists a point $z_{1}\in J(G_{\tau })$ such that 
$T_{\infty ,\tau }(z_{1})=t.$ 
This argument shows that $T_{\infty ,\tau }(J(G_{\tau }))=[0,1].$ 
Therefore, we have proved statement~\ref{kerJthm3-2}.  

 We next show statement~\ref{kerJthm3-3}. 
Suppose that the statement is false. Then, there exist 
$t_{1}$ and $t_{2}$ in $[0,1]$ with $t_{1}<t_{2}$ such that  
denoting by $A$ the unbounded component of 
$\CC \setminus (T_{\infty, \tau }^{-1}(\{ t_{2}\} )\cap J(G_{\tau }))$, 
$T_{\infty ,\tau }^{-1}(\{ t_{1}\} )\cap A\neq \emptyset .$ 
 
Let $w_{0}\in  T_{\infty ,\tau }^{-1}(\{ t_{1}\} )\cap A$ be a point. 
Let $\zeta :[0,1]\rightarrow A$ be a curve such that 
$\zeta (0)=\infty \in T_{\infty ,\tau }^{-1}(\{ 1\} )$ and 
$\zeta (1)=w_{0}\in T_{\infty ,\tau }^{-1}(t_{1}).$ 
Since $t_{1}<t_{2}\leq 1$, there exists an $s\in [0,1)$ such that 
$\zeta (s)\in T_{\infty ,\tau }^{-1}(t_{2}).$ 
Then, we have $\zeta (s)\in A\cap F(G_{\tau }).$ 
Let $U$ be the connected component of $F(G_{\tau })$ containing $\zeta (s).$ 
By Theorem~\ref{kerJthm2} and Lemma~\ref{Tlem1}, 
we have $T_{\infty ,\tau }|_{\overline{U}}\equiv t_{2}.$ 
Since $\zeta (1)\in T_{\infty ,\tau }^{-1}(\{ t_{1}\} )$, 
$\zeta (s)\in U$ and $T_{\infty ,\tau }|_{U}\equiv t_{2}$, 
we obtain that there exists an $s'\in (s,1)$ such that $\zeta (s')\in \partial U\subset 
J(G_{\tau })\cap T_{\infty ,\tau }^{-1}(\{ t_{2}\} ).$ 
However, this is a contradiction since $\zeta (s')\in A$ and 
$A\cap (J(G_{\tau })\cap T_{\infty ,\tau }^{-1}(\{ t_{2}\} ))=\emptyset .$ 
Therefore, statement~\ref{kerJthm3-3} holds. 

 We now prove statement~\ref{kerJthm3-4}. 
Let $t\in (0,1).$ 
Since $\hat{K}(G_{\tau })\subset T_{\infty ,\tau }^{-1}(\{ 0\} )$, 
statement~\ref{kerJthm3-3} implies that 
$\hat{K}(G_{\tau })<_{s}T_{\infty ,\tau }^{-1}(\{ t\} )\cap J(G_{\tau }).$ 
 By Lemma~\ref{tinftyphilem1} and Theorem~\ref{kerJthm2}, 
$\overline{F_{\infty }(G_{\tau })} \subset T_{\infty ,\tau }^{-1}(\{ 1\} ).$ 
Hence, $T_{\infty ,\tau }^{-1}(\{ t\} )\cap J(G_{\tau })<_{s}\overline{F_{\infty }(G_{\tau })}.$ 
Therefore, we have proved statement~\ref{kerJthm3-4}. 

 We now prove statement~\ref{kerJthm3-5}. 
 Let $A:= [0,1]\setminus T_{\infty ,\tau }(F(G_{\tau })).$ 
Since  
$T_{\infty ,\tau }\in C_{F(G_{\tau })}(\CCI )$, 
we have $\sharp ([0,1]\setminus A)\leq \aleph _{0}.$ Let $t\in A.$ 
Since $T_{\infty ,\tau }(J(G_{\tau }))=[0,1]$ and   
$T_{\infty ,\tau }\in C_{F(G_{\tau })}(\CCI )$, 
it follows that $\emptyset \neq T_{\infty ,\tau }^{-1}(\{ t\} )\cap J(G_{\tau })\subset J_{res}(G_{\tau }).$ 
Therefore, we have proved statement~\ref{kerJthm3-5}.  
 
Thus, we have completed the proof of Theorem~\ref{kerJthm3}.  
\qed

We now prove Theorem~\ref{thm:jkreyintj}. \\ 
\noindent {\bf Proof of Theorem~\ref{thm:jkreyintj}:} 
Let $\tau \in {\frak M}_{1}({\cal P})$ be an element such that 
$\G _{\tau }=\G .$ By Theorem~\ref{kerJthm2}, $T_{\infty ,\tau }:\CCI \rightarrow [0,1]$ is 
continuous. 
By Theorem~\ref{kerJthm3}-\ref{kerJthm3-2},  
$T_{\infty ,\tau }(\CCI )=[0,1].$   
Suppose that each of statements (a) and (b) of the theorem does not hold. 
Since statement (b) does not hold, there exists 
a finite set $C=\{ c_{1},\ldots ,c_{n}\} $ of $[0,1]$ such that 
$T_{\infty ,\tau }(F(G))\subset C.$  
Since int$(J(G))=\emptyset $ and $T_{\infty ,\tau }:\CCI \rightarrow [0,1]$ is continuous, 
it follows that $T_{\infty ,\tau }(\CCI )\subset C.$ However, 
this is a contradiction. Therefore, 
at least one of the statements (a) and (b) holds. 
Thus, we have completed the proof of Theorem~\ref{thm:jkreyintj}. 
\qed 

\subsection{Proofs of results in subsection~\ref{Planar}}
\label{pfPlanar} 
In this subsection, we give the proofs of results in subsection~\ref{Planar}.

In order to prove Theorem~\ref{Punbddthm1}, we need several lemmas.

\begin{lem}
\label{finjgcompolem}
Let $\G \in \emCpt ({\cal P})$  
and let $f:\GN \times \CCI \rightarrow \GN \times \CCI $ be 
the skew product associated with $\Gamma .$ 
Let $\g \in \GN $ be an element such that 
$\sharp \mbox{{\em Con}}(J_{\g })<\infty .$ Then, 
there exists an $n\in \NN $ such that 
$J_{\sigma ^{n}(\g )}$ is connected.  
\end{lem}
\begin{proof}
Let $B\in \mbox{Con}(J_{\g }).$ Since 
$\sharp \mbox{Con}(J_{\g })<\infty $, 
$B$ is an open subset of $J_{\g }.$ 
By the self-similarity of $J_{\g }$ (see \cite{Bu1}), 
there exists an $n\in \NN $ such that 
$f_{\g ,n}(B)=J_{\sigma ^{n}(\g )}.$ 
It follows that for this $n$, $J_{\sigma ^{n}(\g )}$ is connected. 
Thus, we have completed the proof of our lemma. 
\end{proof}
\begin{lem}
\label{infinjgcompolem}
Let $\G \in \emCpt ({\cal P})$  
and let $f:\GN \times \CCI \rightarrow \GN \times \CCI $ be 
the skew product associated with $\Gamma .$ 
Let $\g \in \GN $ be an element  such that 
$\sharp \mbox{{\em Con}}(J_{\g })\geq \aleph _{0}.$ 
Then, $\sharp \mbox{{\em Con}}(J_{\g })> \aleph _{0}.$
\end{lem}
\begin{proof}
We first show the following claim. \\ 
Claim 1: Let $B\in \mbox{Con}(J_{\g }).$ Then there exists an 
sequence $\{ B_{j}\} _{j\in \NN }$ in Con$(J_{\g })\setminus \{ B\} $ such that 
$\min _{a\in B, b\in B_{j}} d(a,b)\rightarrow 0$ as $j\rightarrow \infty .$ 

 To prove claim 1,  suppose that there exists no such sequence $\{ B_{j}\} _{j\in \NN }.$ 
Then, $B$ is an open subset of $J_{\g }.$ By the self-similarity of $J_{\g }$ (see \cite{Bu1}), 
there exists an $n\in \NN $ such that $f_{\g, n}(B)=J_{\sigma ^{n}(\g )} $ and 
$J_{\sigma ^{n}(\g )}$ is connected. 
Since $f_{\g ,n}^{-1}(J_{\sigma ^{n}(\g )})=J_{\g }$, 
\cite[Lemma 5.7.2]{Be} implies that $\sharp \mbox{Con}(J_{\g })\leq \deg (f_{\g,n})<\infty .$ 
However, this contradicts the assumption of our lemma. Therefore, we have proved claim 1.

 Let $Z$ be the space obtained by making each element of Con$(J_{\g })$ into one point, 
endowed with the quotient topology. 
Then, by the cut wire theorem (see \cite{N}), 
$Z$ is a compact normal Hausdorff space. 
Suppose that $\sharp \mbox{Con}(J_{\g })=\aleph _{0}.$ 
Then there exists a sequence $\{ C_{j}\} _{j\in \NN }$ of mutually distinct elements 
of Con$(J_{\g })$ such that Con$(J_{\g })=\bigcup _{j=1}^{\infty }\{ C_{j}\}.$  
Let $\pi :J_{\g }\rightarrow Z$ be the canonical projection and 
let $Z_{j}:=\pi (C_{j})$ for each $j\in \NN.$   
Then  
 $Z=\bigcup _{j=1}^{\infty }\{ Z_{j}\} .$ 
We now prove the following claim.\\ 
Claim 2: For each $j\in \NN $, $Z\setminus \{ Z_{j}\} $ is 
dense in $Z.$ 

 To prove claim 2, let $j\in \NN .$ By claim 1, there exists a 
sequence $\{ k_{n}\} _{n\in \NN }$ in $\NN \setminus \{ j\} $ such that  
$\min _{a\in C_{j},b\in C_{k_{n}}}d(a,b)\rightarrow 0$ as $n\rightarrow \infty . $ 
Let $V$ be an open set in $Z$ with $Z_{j}\in V.$ Then $\pi ^{-1}(V)$ is an open set 
in $J_{\gamma }$ with $C_{j}\subset \pi ^{-1}(V).$ Therefore there exists an $n\in \NN $ with 
$\pi ^{-1}(V)\cap C_{k_{n}}\neq \emptyset .$ Let $x\in \pi ^{-1}(V)\cap C_{k_{n}}.$ 
Then $Z_{k_{n}}=\pi (x)\in V.$ 
Therefore, $Z\setminus \{ Z_{j}\} $ is dense in $Z.$ Thus, we have proved claim 2. 

 Since $\emptyset =Z\setminus \bigcup _{j=1}^{\infty }\{ Z_{j}\} =
\bigcap _{j=1}^{\infty }(Z\setminus \{ Z_{j}\} )$, 
 claim 2 and the Baire category theorem imply a contradiction.   
 Therefore, $\sharp \mbox{Con}(J_{\g })>\aleph _{0}.$ 
Thus, we have completed the proof of our lemma. 
\end{proof}

We now prove Theorem~\ref{Punbddthm1}.\\ 
\noindent {\bf Proof of Theorem~\ref{Punbddthm1}:}
Since $P^{\ast }(G)$ is not bounded in $\CC $, there exists an element 
$g\in G$ and a critical value $c$ of $g$ such that 
$c \in F_{\infty }(G).$ 
Then $g^{l}(c)\rightarrow \infty $ as $l\rightarrow \infty .$ 
We write $g$ as $g=h_{n}\circ \cdots \circ h_{1}$, where $h_{j}\in \G $ for each $j=1,\ldots ,n.$ 
For each $j=1,\ldots ,n$, let $B_{j}$ be the small neighborhood of $h_{j}$ in ${\cal P}$ 
such that for each $\alpha =(\alpha _{1},\ldots ,\alpha _{n})\in B_{1}\times \cdots \times B_{n}$, 
there exists a critical value $c_{\alpha }\in F_{\infty }(G)$ of 
$\alpha _{n}\circ \cdots \circ \alpha _{1}.$ 
We set ${\cal U}:= \{ \gamma \in \GN \mid 
\exists \{ j_{k}\} _{k=1}^{\infty }\rightarrow \infty , 
\forall k, \g _{j_{k}}\in B_{1},\ldots ,\g _{j_{k}+n-1}\in B_{n}\} .$ 
Then, ${\cal U}$ is a residual subset of $\GN $ and 
for each $\tau \in {\frak M}_{1}({\cal P})$ with $\G _{\tau }=\G $, 
 $\tilde{\tau }({\cal U})=1.$ 
We now prove the following claim.\\ 
Claim: For each $\g \in {\cal U}$, $\sharp \mbox{Con}(J_{\g })>\aleph _{0}.$ 

 To prove the claim, by Lemma~\ref{infinjgcompolem}, 
it is enough to show that for each $\g \in {\cal U}$, $\sharp \mbox{Con}(J_{\g })\geq \aleph _{0}.$ 
Suppose that there exists an element $\g \in {\cal U}$ such that 
 $\sharp \mbox{Con}(J_{\g })< \infty .$ 
Then, Lemma~\ref{finjgcompolem} implies that 
there exists an $s\in \NN $ such that 
$J_{\sigma ^{s}(\g )}$ is connected. 
Since $\g \in {\cal U}$, there exists an $m\in \NN $ such that 
$\g _{s+m, s+1}$ 
has a critical point in $A_{\infty ,\sigma ^{s}(\g )}.$ 
For this $m\in \NN $, 
$J_{\sigma ^{s+m}(\g )}=\g _{s+m, s+1}(J_{\sigma ^{s}(\g )}) $ is connected. 
Hence,  $A_{\infty ,\sigma ^{s}(\g )}$ and $A_{\infty ,\sigma ^{s+m}(\g )}$ are simply connected. 
Applying the Riemann-Hurwitz formula to 
$\g _{s+m, s+1}: A_{\infty ,\sigma ^{s}(\g )}\rightarrow A_{\infty ,\sigma ^{s+m}(\g )}$, 
we obtain a contradiction. Therefore, the above claim holds. 

 By the above claim, the statement of Theorem~\ref{Punbddthm1} holds. 
\qed 

\

We now prove Theorem~\ref{Tequals1thm1}.

\noindent {\bf Proof of Theorem~\ref{Tequals1thm1}:} 
By Lemma~\ref{kemptylem1}, $T_{\infty, \tau }\equiv 1 $.   
Moreover, since supp$\,\tau $ is compact, 
$\infty \in F(G_{\tau }).$ Hence for each $z\in \CCI $, 
there exists an element $g\in G_{\tau }$ such that $g(z)\in F_{\infty }(G_{\tau })\subset F(G_{\tau }).$ 
Therefore, $J_{\ker }(G_{\tau })=\emptyset .$ By Theorem~\ref{kerJthm1}, we obtain that 
$F_{meas}(\tau )={\frak M}_{1}(\CCI ).$ 
Moreover, since $T_{\infty, \tau }\equiv 1$, we obtain 
that $(M_{\tau }^{\ast })^{n}(\nu )\rightarrow \delta _{\infty }$ 
as $n\rightarrow \infty $ uniformly on $\nu \in {\frak M}_{1}(\CCI ).$   
Let $\tilde{K}:=\bigcup _{\rho \in X_{\tau }}(\{ \rho \} \times K_{\rho })\ (\subset X_{\tau }\times \CCI ).$ 
Since $T_{\infty ,\tau }\equiv 1$, 
it follows that 
for each $y\in \CCI $, 
$\tilde{\tau }(\{ \gamma \in X_{\tau }\mid (\gamma ,y)\in \tilde{K}\} )=0.$ 
Hence, by Fubini's theorem, 
we obtain that there exists a subset ${\cal V}$ of  $X_{\tau }$ with $\tilde{\tau }({\cal V})=1$ 
such that for each $\g \in {\cal V}$, Leb$_{2}(K_{\g })=0.$ 
Since $\partial K_{\g }=J_{\g }$ for each $\g \in X_{\tau }$, 
we get that for each $\g \in {\cal V}$, $K_{\g }=J_{\g }.$ 
Moreover, since $T_{\infty, \tau }\equiv 1 $, we have that $P^{\ast }(G_{\tau })$ is not bounded in $\CC .$ 
By Theorem~\ref{Punbddthm1}, we obtain that for $\tilde{\tau }$-a.e. $\g \in X_{\tau }$,  
$J_{\g }$ has uncountably many connected components. 
 Thus, we have completed the proof of Theorem~\ref{Tequals1thm1}.  
\qed 

\subsection{Proofs of results in subsection~\ref{Conleb}}
\label{pfConleb}
In this subsection, we give the proofs of the results in subsection~\ref{Conleb}. 

In order to prove Theorem~\ref{t:jkucuh}, 
we need some lemmas.
\begin{df}
Let $X$ and $Y$ be two topological space and let $g:X\rightarrow Y$ be a map. 
For each subset $Z$ of $Y$, we denote by $c(Z,g)$ the set of all 
connected components of $g^{-1}(Z).$ 
\end{df}

\begin{lem}
\label{l:jkejgcuh}
Let $\tau \in {\frak M}_{1,c}({\cal P})$.  
Suppose that for each $\g \in X_{\tau }$, $J_{\ker }(G_{\tau })\subset J_{\g }$, 
and that $J_{\ker }(G_{\tau })\cap UH(G_{\tau })=\emptyset .$ 
Then, for $\tilde{\tau }$-a.e. $\g \in X_{\tau }$,   
{\em Leb}$_{2}(J_{\g })=\mbox{{\em Leb}}_{2}(\hat{J}_{\g ,\G _{\tau }})=0.$  
\end{lem}
\begin{proof}
Let $f:X_{\tau } \times \CCI \rightarrow X_{\tau } \times \CCI $ be the skew product associated with 
$\G _{\tau }.$ 
By Proposition~\ref{Jkeremptygenprop2},  
we may assume that 
$J_{\ker }(G_{\tau })\neq \emptyset .$ 
Combining Lemma~\ref{ocminvlem}, Lemma~\ref{Lkerlem}, Lemma~\ref{l:pctjjg} and Fubini's theorem, 
we obtain that there exists a measurable subset ${\cal U}$ of $X_{\tau }$ with 
$\tilde{\tau }({\cal U})=1$ such that 
for each $\g \in {\cal U}$,  for Leb$_{2}$-a.e. $y\in \hat{J}_{\g ,\G _{\tau }}$, 
$\liminf _{n\rightarrow \infty }d(f_{\g, n}(y), J_{\ker}(G_{\tau }))=0.$ 
In order to prove our lemma, it is enough to show that 
for each $\g \in {\cal U}, \mbox{Leb}_{2}(\hat{J}_{\g, \G _{\tau }})=0.$ 
 For this purpose, suppose that 
there exists an element $\rho \in {\cal U}$ and a point $y_{0}\in \hat{J}_{\rho ,\G _{\tau }}$ 
such that $y_{0}$ is a Lebesgue density point of $\hat{J}_{\rho ,\G _{\tau }}.$ 
We will deduce a contradiction. 
We may assume that $\liminf _{n\rightarrow \infty }d(f_{\rho, n}(y_{0}), J_{\ker }(G_{\tau }))=0.$ 
We show the following claim:\\ 
Claim 1: $y_{0}\in J_{\rho }.$   

 To show claim 1, suppose that $y_{0}\in F_{\rho }.$ 
Then there exists a strictly increasing sequence $\{ n_{j}\} _{j\in \NN }$ in $\NN $ 
and a $\delta >0$ such that $f_{\rho ,n_{j}}|_{B(y_{0},2\delta )}$ tends to 
a holomorphic function $\phi :B(y_{0},2\delta )\rightarrow \CC $ as $j\rightarrow \infty $ 
locally uniformly on $B(y_{0},2\delta ).$ We may assume that there exists a 
point $(\alpha ,a)\in X_{\tau }\times J_{\ker }(G_{\tau })$ such that $f^{n_{j}}(\rho , y_{0})\rightarrow (\alpha ,a)$ as 
$j\rightarrow \infty .$ Since $J_{\ker }(G_{\tau })\cap UH(G_{\tau })=\emptyset $, 
\cite[Lemma 1.10]{S4} implies that 
$\phi :B(y_{0}, 2\delta )\rightarrow \CC $ is non-constant. 
Hence $\phi (B(y_{0},\delta ))$ is a bounded open neighborhood of $a.$ 
Let $D$ be a neighborhood of $\infty $ such that $\overline{D}\cap \phi (B(y_{0},\delta ))=\emptyset $ and 
$h(D)\subset D$ for each $h\in G_{\tau }.$   
Since $a \in J_{\ker }(G_{\tau })\subset J_{\alpha }$, 
there exists a point $b\in B(y_{0}, \delta )$ and a $q\in \NN $  such that 
$\alpha _{q,1}(\phi (b))\in D.$ 
Then there exists a  neighborhood ${\cal V}$ of $(\alpha _{1},\ldots ,\alpha _{q})$ in $\G_{\tau }^{q}$ and 
 a neighborhood $\Omega $ of $\phi (b)$ such that for each $\beta =(\beta _{1},\ldots ,\beta _{q})\in 
 {\cal V}$, $\beta _{q}\cdots \beta _{1}(\Omega )\subset D.$ 
Let $k\in \NN $ be a large number. Then, $f_{\rho ,n_{k}}(b)\in \Omega $ and 
 $(\rho _{n_{k}+1},\rho _{n_{k}+2},\ldots ,\rho _{n_{k}+q})\in {\cal V}.$ 
 This implies that $\phi (b)=\lim _{j\rightarrow \infty }f_{\rho ,n_{j}}(b)\in \overline{D}\subset 
\CCI \setminus \phi (B(y_{0}, \delta ))$, 
which is a contradiction. Therefore, $y_{0}\in J_{\rho }.$ Thus, we have shown Claim 1.  

Let $\{ n_{j}\} _{j\in \NN }$ be a strictly increasing sequence in $\NN $ such that 
$f^{n_{j}}(\rho ,y_{0})$ tends to some $(\eta ,y_{\infty })\in X_{\tau }\times J_{\ker }(G_{\tau }).$  
Let $g_{j}:=f_{\rho ,n_{j}}$ for each $j\in \NN .$ 
 Since $J_{\ker }(G_{\tau })\cap UH(G_{\tau })=\emptyset $, 
 there exists an $0<r$ and an $N\in \NN $ such that 
for each $z\in J_{\ker }(G_{\tau })$, each $g\in G_{\tau }$, 
and each $V\in c(D(z,3r), g)$, 
$\deg (g:V\rightarrow D(z,3r))\leq N.$ 
We may assume that for each $j\in \NN $, 
$g_{j}(y_{0})\in D(J_{\ker }(G_{\tau }),r).$ 
For each $j\in \NN $, let $U_{j}$ (resp. $U_{j}'$) be the element of 
$c(D(g_{j}(y_{0}),r),g_{j})$ (resp. $c(D(g_{j}(y_{0}),2r),g_{j}))$ containing 
$y_{0}$. 
Then, $U_{j}$ and $U_{j}'$ are simply connected. 
Moreover, since $y_{0}\in J_{\rho }$, \cite[Corollary 1.9]{S4} implies that 
diam $U_{j}\rightarrow 0$ as $j\rightarrow \infty .$ 
Since $y_{0}$ is a Lebesgue density point of $\hat{J}_{\rho ,\G _{\tau }}$, 
 \cite[Corollary 1.9]{S4} again implies that 
$
\lim _{j\rightarrow \infty }\frac{\mbox{Leb}_{2}(U_{j}\cap \hat{J}_{\rho ,\G _{\tau }})}{\mbox{Leb}_{2}(U_{j})}=1.
$
Hence, 
\begin{equation}
\label{eq:l:jkejgcuhpf1}
\lim _{j\rightarrow \infty }\frac{\mbox{Leb}_{2}(U_{j}\cap \hat{F}_{\rho ,\G _{\tau }})}{\mbox{Leb}_{2}(U_{j})}=0.
\end{equation}
For each $j\in \NN $, let $\phi _{j}:D(0,1)\rightarrow U_{j}'$ be a conformal map such that 
$\phi _{j}(0)=y_{0}.$ By (\ref{eq:l:jkejgcuhpf1}) and the Koebe distortion theorem, 
we obtain 
\begin{equation}
\label{eq:l:jkejgcuhpf2}
\lim _{j\rightarrow \infty }\frac{\mbox{Leb}_{2}(\phi _{j}^{-1}(U_{j}\cap \hat{F}_{\rho ,\G _{\tau }}))}
{\mbox{Leb}_{2}(\phi _{j}^{-1}(U_{j}))}=0.
\end{equation}
By \cite[Corollary 1.8]{S4}, there exists a constant $0<c_{1}<1$ such that 
for each $j\in \NN $, $\phi _{j}^{-1}(U_{j})\subset D(0,c_{1}).$ 
Combining it with Cauchy's formula, we obtain that there exists a constant 
$c_{2}>0$ such that for each $j\in \NN $, 
\begin{equation}
\label{eq:l:jkejgcuhpf3}
|(g_{j}\circ \phi _{j})'(z)|\leq c_{2} \mbox{ on } \phi _{j}^{-1}(U_{j}).
\end{equation}
By (\ref{eq:l:jkejgcuhpf2}) and (\ref{eq:l:jkejgcuhpf3}), we obtain 
\begin{align*}
\frac{\mbox{Leb}_{2}(D(g_{j}(y_{0}),r)\cap \hat{F}_{\sigma ^{n_{j}}(\rho ),\G _{\tau }})}
{\mbox{Leb}_{2}(D(g_{j}(y_{0}),r))}
& =\frac{\mbox{Leb}_{2}((g_{j}\phi _{j})(\phi _{j}^{-1}(U_{j}\cap \hat{F}_{\rho ,\G _{\tau }})))}
{\mbox{Leb}_{2}(D(g_{j}(y_{0}),r))}\\ 
& \leq \frac{\int _{\phi _{j}^{-1}(U_{j}\cap \hat{F}_{\rho, \G _{\tau }})}|(g_{j}\circ \phi _{j})'(z)|^{2}\ d\mbox{Leb}_{2}(z)}
{\mbox{Leb}_{2}(\phi _{j}^{-1}(U_{j}))}\cdot 
\frac{\mbox{Leb}_{2}(\phi _{j}^{-1}(U_{j}))}{\mbox{Leb}_{2}(D(g_{j}(y_{0}),r))}\\ 
& \rightarrow 0, \mbox{ as }j\rightarrow \infty .
\end{align*}
Hence, $\frac{\mbox{Leb}_{2}(D(g_{j}(y_{0}),r)\cap \hat{J}_{\sigma ^{n_{j}}(\rho ),\G _{\tau }})}
{\mbox{Leb}_{2}(D(g_{j}(y_{0}),r))}\rightarrow 1 $ as $j\rightarrow \infty .$ 
Thus, $D(y_{\infty },r)\subset \hat{J}_{\eta ,\G _{\tau }}.$ In particular, 
$f_{\eta ,n}(D(y_{\infty ,}r))$ $\subset J(G_{\tau })$ for each $n\in \NN .$ 
Hence, $y_{\infty }\in F_{\eta }.$ However, since $y_{\infty }\in J_{\ker }(G_{\tau })\subset J_{\eta }$, 
this is a contradiction. Thus, we have completed the proof of our lemma. 
\end{proof}
\begin{lem}
\label{l:jkucuhjkj}
Under the assumptions of Theorem~\ref{t:jkucuh}, 
we have that for each $\g \in X_{\tau }$, $J_{\ker }(G_{\tau })\subset J_{\g }.$ 
\end{lem}
\begin{proof}
Under the assumptions of Theorem~\ref{t:jkucuh}, suppose that 
there exists a $\g \in X_{\tau }$ such that $J_{\ker }(G_{\tau })\cap F_{\g }\neq \emptyset .$ 
Let $y_{0}\in J_{\ker }(G_{\tau })\cap F_{\g }$ be a point.   
Let $f:X_{\tau }\times \CCI \rightarrow X_{\tau }\times \CCI $ be the skew product associated with 
$\G _{\tau }.$ 
Then there exists a strictly increasing sequence $\{ n_{j}\} _{j\in \NN }$ in $\NN $, 
an open connected neighborhood $U$ of $y_{0}$,   
and a holomorphic map $\varphi :U\rightarrow \CCI $ such that 
$f_{\g ,n_{j}}\rightarrow \varphi $ as $j\rightarrow \infty $ uniformly on $U.$ 
Let $(\rho ^{j},y_{j})=f^{n_{j}}(\g ,y). $ We may assume that there exists a 
point $(\rho ^{\infty }, y_{\infty })\in X_{\tau }\times J_{\ker }(G_{\tau })$ 
such that $(\rho ^{j}, y_{j})\rightarrow (\rho ^{\infty },y_{\infty })$ as $j\rightarrow \infty .$ 
If $\varphi $ is constant, then \cite[Lemma 1.10]{S4} implies that $y_{\infty }\in UH(G_{\tau })$. However, 
this is a contradiction, since $J_{\ker }(G_{\tau })\cap UH(G_{\tau })=\emptyset .$ 
Hence, we obtain that $\varphi $ is not constant. 
We set 
$$V:= \{ y\in \CCI \mid \exists \epsilon >0, \lim _{i\rightarrow \infty }\sup _{j>i}\sup _{d(\xi ,y)\leq \epsilon }
d(f_{\rho ^{i},n_{j}-n_{i}}(\xi ) ,\xi )=0\} .$$ 
Then, $V$ is an open subset of $\CCI .$ 
By Lemma~\ref{tinftyphilem1}, $V\cap F_{\infty }(G_{\tau })=\emptyset .$ 
Moreover, $y_{\infty }\in V$. For, since $\varphi $ is non-constant, 
there exists a number $a>0$ and an $s\in \NN $ such that 
for each $j\in \NN $ with $j\geq s$, $f_{\g ,n_{j}}(U)\supset B(y_{\infty },a).$ 
If $y\in B(y_{\infty },a)$ then $y=f_{\g ,n_{i}}(\xi _{i})$ for some $\xi _{i}\in U$ and so
$d(f_{\rho ^{i},n_{j}-n_{i}}(y),y)=d(f_{\g ,n_{j}}(\xi _{i}), f_{\g ,n_{i}}(\xi _{i}))$ which is small 
if $i$ is large. Hence $y_{\infty }\in V.$ 

 Furthermore, by \cite[Lemma 2.13]{S4}, $\partial V\subset J(G_{\tau })\cap UH(G_{\tau }).$ 
These arguments imply that $y_{\infty }$ belongs to a bounded connected component of 
$\CC \setminus (J(G_{\tau })\cap UH(G_{\tau })).$ However, this contradicts the assumption of our lemma. 
Therefore, for each $\g \in X_{\tau }$, $J_{\ker }(G_{\tau })\subset J_{\g }.$ 
Thus, we have completed the proof of our lemma. 
\end{proof}

We now prove Theorem~\ref{t:jkucuh}.\\ 
\noindent {\bf Proof of Theorem~\ref{t:jkucuh}:} 
Combining Lemma~\ref{l:jkucuhjkj}, Lemma~\ref{l:jkejgcuh},  
Lemma~\ref{l:pctjjg}, Lemma~\ref{l:mgjpt0}, and Lemma~\ref{lem:fpt0tconti}, 
we obtain all statements of Theorem~\ref{t:jkucuh}. 
Thus, we have completed the proof of Theorem~\ref{t:jkucuh}.
\qed
 
\subsection{Proofs of results in subsection~\ref{Conjkeremp}}
\label{pfConjkeremp}
In this subsection, we give the proofs of the results in subsection~\ref{Conjkeremp}. 
Moreover, we show several related results. 
\begin{lem}
\label{l:gignejke}
Let $\G $ be a non-empty subset of {\em Rat} and let 
$G=\langle \G \rangle .$ 
Suppose that $F(G)\neq \emptyset $, and 
that for each $z\in J(G)$, 
there exists a holomorphic family $\{ g_{\lambda }\} _{\lambda \in \Lambda }$ 
of rational maps 
 such that $\{ g_{\lambda }\mid \lambda \in \Lambda \} \subset \G $ and 
the map 
$\lambda \mapsto g_{\lambda }(z)$ is nonconstant on $\Lambda .$ Then, 
$J_{\ker }(G)=\emptyset .$ 
\end{lem}
\begin{proof}
Suppose that $J_{\ker }(G)\neq \emptyset .$
Let $z_{0}\in J_{\ker }(G)$ be a point. 
Then there exists a holomorphic family $\{ g_{\lambda }\} _{\lambda \in \Lambda }$ 
of rational maps 
such that the map 
$\Theta : \lambda \mapsto g_{\lambda }(z_{0})$ is nonconstant on $\Lambda $ and 
$\{ g_{\lambda }(z_{0})\mid \lambda \in \Lambda \} \subset \Gamma .$ 
Hence, $J_{\ker}(G)$ contains a non-empty open subset $\Theta (\Lambda )$ of $\CCI .$ 
However, this contradicts 
 Remark~\ref{r:kjulia}.  
Therefore, $J_{\ker }(G)=\emptyset .$ Thus, 
we have completed the proof of our lemma.
\end{proof}

 We now prove Lemma~\ref{l:ignejke}.\\ 
\noindent {\bf Proof of Lemma~\ref{l:ignejke}:}
The statement of our lemma immediately follows from Lemma~\ref{l:gignejke}. 
\qed 

We now prove Lemma~\ref{l:aigpjke}.\\ 
\noindent {\bf Proof of Lemma~\ref{l:aigpjke}:}
Since $\G $ is relative compact, $\infty \in F(G).$ 
From Lemma~\ref{l:gignejke}, it follows that $J_{\ker }(G)=\emptyset .$ 
Thus, we have completed the proof of our lemma.
\qed  
\begin{lem}
\label{l:gnbdhjke}
Let ${\cal Y}$ be a closed subset of an open subset of ${\cal P}.$  
Suppose that ${\cal Y}$ is strongly admissible.  
Let $\G \in \mbox{{\em Cpt}}({\cal Y})$ and let $V$ be a neighborhood 
of $\G $ in $\mbox{{\em Cpt}}({\cal Y}).$ 
Then, there exists a $\G '\in V$ such that $J_{\ker }(\langle \G '\rangle )=\emptyset .$ 
\end{lem}
\begin{proof}
Take a small $\epsilon >0$ such that the element  
 $\G ':=\{ h\in {\cal Y}\mid \kappa (h,\G )\leq \epsilon \} \in \mbox{Cpt}({\cal Y})$ belongs to 
$V$, where $\kappa $ denotes the relative distance in ${\cal P}$ from Rat. By Lemma~\ref{l:gignejke}, 
$J_{\ker }(\langle \G '\rangle )=\emptyset .$ 
Hence, we have completed the proof of our lemma.   
\end{proof}
\begin{lem}
\label{l:rgnbdhjkef}
Let ${\cal Y}$ be a subset of {\em Rat} endowed with the relative distance from 
{\em Rat}. 
Let $\G \in \mbox{{\em Cpt}}({\cal Y})$ be an element such that 
$J_{\ker }(\langle \G \rangle )=\emptyset .$ 
Let $V$ be a neighborhood of $\G $ in $\mbox{{\em Cpt}}({\cal Y}).$ 
Then, there exists an element $\G '\in V$ such that 
$\G' \subset \G , \sharp \G ' <\infty , $ and 
$J_{\ker }(\langle \G '\rangle )=\emptyset .$ 
\end{lem}
\begin{proof}
Since $J_{\ker }(\langle \G \rangle )=\emptyset $, 
there exist finitely many elements $g_{1},\ldots ,g_{r}\in \langle \G \rangle $ and 
finitely many open subsets $U_{1},\ldots ,U_{r}$ of $\CCI $ such that 
$J(\langle \G \rangle )\subset \bigcup _{j=1}^{r}U_{j}$ and 
$\bigcup _{j=1}^{r}g_{j}(U_{j})\subset F(\langle \G \rangle ).$ 
In particular, $\bigcap _{j=1}^{r}g_{j}^{-1}(J(\langle \G \rangle ))=\emptyset .$ 
For each $j=1.\ldots ,r$, we write $g_{j}$ as 
$g_{j}=h_{j,1}\circ \cdots \circ h_{j,t_{j}}$, where 
$h_{j,k}\in \G $ for each $k=1,\ldots ,t_{j}.$ Take an element 
$\G '\in V$ such that $\G '\subset \G $, $\sharp \G '<\infty $ and 
$\G '\supset \bigcup _{j=1}^{r}\{ h_{j,1},\ldots ,h_{j,t_{j}}\} .$ 
Then, $J_{\ker }(\langle \G '\rangle )=\bigcap _{h\in \langle \G '\rangle }h^{-1}(J(\langle \G '\rangle ))
\subset \bigcap _{j=1}^{r}g_{j}^{-1}(J(\langle \G \rangle ))=\emptyset .$ 
Thus, we have completed the proof of our lemma. 
\end{proof} 
\begin{lem}
\label{l:gnbdfhjke}
Let ${\cal Y}$ be a closed subset of an open subset of ${\cal P}.$ 
Suppose that ${\cal Y}$ is strongly admissible.  
Let $\G \in \mbox{{\em Cpt}}({\cal Y})$ and 
let $V$ be a neighborhood of $\G $ in $\mbox{{\em Cpt}}({\cal Y}).$ 
Then, there exists an element $\G '\in V$ such that 
$\sharp \G '<\infty $ and $J_{\ker }(\langle \G '\rangle )=\emptyset .$ 
\end{lem}
\begin{proof}
Combining Lemma~\ref{l:gnbdhjke} and Lemma~\ref{l:rgnbdhjkef}, 
the statement of our lemma holds. 
\end{proof}

We now prove Proposition~\ref{p:v1v2rho}.\\ 
\noindent {\bf Proof of Proposition~\ref{p:v1v2rho}:}
Let $\rho _{0}\in {\frak M}_{1}({\cal Y})$ be an element 
such that $\sharp \G _{\rho _{0}} <\infty $, $\rho _{0}\in V_{1}$, 
and 
$\G _{\rho _{0}}\in V_{2}.$ 
We write $\rho _{0}$ as $\rho _{0}=\sum _{j=1}^{r}p_{j}\delta _{h_{j}}$, where 
$p_{j}>0 $ for each $j$, $\sum _{j=1}^{r}p_{j}=1,$,  and 
$h_{1},\ldots ,h_{r}$ are mutually distinct elements of ${\cal Y}.$   
Let $U_{1}$ be a small compact neighborhood of $h_{1}$ in ${\cal Y}$ such that 
the compact set $\Lambda  _{1}:= U_{1}\cup \{ h_{2},\ldots ,h_{r}\} $ belongs to $V_{2}.$ 
By Lemma~\ref{l:gignejke}, $J_{\ker }(\langle \Lambda  _{1}\rangle )=\emptyset .$ 
Hence, Lemma~\ref{l:rgnbdhjkef} implies that 
there exists a finitely many elements $g_{1},\ldots ,g_{s}$ of $U_{1}$ 
such that setting $\Lambda  _{2}:= \{ g_{1},\ldots ,g_{s}\} \cup \{ h_{2},\ldots ,h_{r}\} $, 
we have $\Lambda  _{2}\in V_{2}$ and $J_{\ker }(\langle \Lambda _{2}\rangle )=\emptyset .$ 
Let $q_{1},\ldots ,q_{s}$ be positive numbers such that 
$\sum _{j=1}^{s}q_{j}=p_{1}.$ Let $\rho := \sum _{j=1}^{s}q_{j}\delta _{g_{j}}+
\sum _{j=2}^{r}p_{j}\delta _{h_{j}}.$ 
Then $\G _{\rho }\in V_{2}$ and $J_{\ker }(G_{\rho })=\emptyset .$ 
Moreover, if we take $U_{1}$ so small, then $\rho \in V_{1}.$ 
Thus, we have completed the proof of Proposition~\ref{p:v1v2rho}. 
\qed 
\subsection{Proofs of results in subsection~\ref{Almostst}}
\label{pfAlmostst}
In this subsection, we give the proofs of the results in subsection~\ref{Almostst}.

In order to prove Proposition~\ref{p:hyppjke}, 
we need some notations and lemmas. 
\begin{df}
Let $Y$ be a compact metric space and let $U$ be an open subset of $Y.$ 
Let $\G $ be a subset of $\CMX $ and let $G=\langle \G \rangle .$ 
Let $K$ be a non-empty compact subset of $U$. 
\begin{enumerate}
\item 
We say that 
$K$ is a weak attractor for $(G,\G ,U)$ if 
for each $\g \in \GN $ and each $y\in U$, 
$d(\g _{n,1}(y),K)\rightarrow 0$ as $n\rightarrow \infty .$ 
\item We say that $K$ is an attractor for $(G, \G ,U)$ if 
$K$ is a weak attractor for $(G,\G ,U)$ and 
$g(K)\subset K$ for each $g\in \G .$    
\end{enumerate}
\end{df}
\begin{lem}
\label{l:att1}
Let $\G \in \mbox{{\em Cpt}}(\emRat)$ with $\sharp J(\langle \Gamma \rangle )\geq 3.$  Let $G=\langle \G \rangle .$ 
Suppose that there exists an attractor $K$ for $(G, \G , F(G)).$  
Then, for each $L\in \mbox{{\em Cpt}}(F(G))$, 
\begin{equation}
\label{eq:l:att1-1}
\sup \{ d(\g _{n}\cdots \g _{1}(z), K)\mid z\in L, (\g_{1},\ldots ,\g _{n})\in \G ^{n}\} 
\rightarrow 0 \mbox{ as } n\rightarrow \infty 
\end{equation}
 and there exists a constant $C>0$ and an $0<\eta  <1$ 
such that 
\begin{equation}
\label{eq:l:att1-2}
\sup \{ \| (\g _{n}\cdots \g _{1})'(z)\| _{s}\mid z\in L, (\g_{1},\ldots ,\g _{n})\in \G ^{n}\} \leq C\eta ^{n}
 \mbox{ for each } n\in \NN ,
\end{equation}   
where $\| \cdot \| _{s}$ denotes the norm of the derivative with respect to the spherical 
metric of $\CCI .$  
\end{lem}
\begin{proof}
Let $V_{1},\ldots ,V_{s}$ be finitely many connected components of 
$F(G)$ such that $K\subset \bigcup _{j=1}^{s}V_{j}$ and 
$V_{j}\cap K\neq \emptyset $ for each $j=1,\ldots ,s.$ 
We set $V=\bigcup _{j=1}^{s}V_{j}.$  
In each $j=1,\ldots ,s$, 
we take the hyperbolic metric $\rho _{j}$ on $V_{j}$ and 
let $W_{j}$ be the $\epsilon $-neighborhood of $K\cap V_{j}$ 
with respect to $\rho _{j}.$ 
Let $W=\bigcup _{j=1}^{s}W_{j}.$ 
Then $G(W)\subset W.$ 
Let $L\in \Cpt (F(G)).$ 
Since $K$ is an attractor for $(G,\G ,F(G))$ and 
$\GN $ is compact, it follows that 
there exists an $n\in \NN$ such that 
for each $\g \in \GN $, 
$\g _{n,1}(L)\subset W.$ 
Let $j\in \{ 1,\ldots ,s\} $ and let $g\in G$. 
Since $K$ is an attractor for $(G,\G,F(G))$, 
we obtain that if $g(V_{j})\subset V_{j}$, 
then $\| g'(z)\| _{h}<1$ for each $z\in V_{j}$, where 
$\| \cdot \| _{h}$ denotes the norm of the derivative 
with respect to $\rho _{j}.$ 
Moreover, for each $\g \in \GN $ and each $z\in V$, 
there exist $p,q\in \NN $ with $1\leq p,q\leq s$ and an 
$i\in \{ 1,\ldots ,s\} $ such that 
$\g _{q,1}(z)\in V_{i}$ and 
$\g _{p+q,1}(z) \in V_{i}.$ 
From these arguments, the statement of our lemma easily follows. 
\end{proof}
\begin{lem}
\label{l:attas}
Let $\G \in \mbox{{\em Cpt}}(\emRat)$ with $\sharp J(\langle \Gamma \rangle )\geq 3.$ 
Let $G=\langle \G \rangle .$ 
Suppose that there exists an attractor $K$ for $(G, \G , F(G)).$  
Moreover, suppose that $J_{\ker }(G)=\emptyset .$ 
Then, there exists a neighborhood ${\cal U}$ of $\Gamma $ 
in $\emCpt(\emRat)$ such that for each $\Gamma '\in {\cal U}$, 
$\Gamma '$ is mean stable and $J_{\ker }(\langle \G'\rangle )=\emptyset .$ 
\end{lem} 
\begin{proof}
Since $J_{\ker }(G)=\emptyset $, for each point $z\in \CCI $, there exists an element $g\in G$ 
such that $g(z)\in F(G).$ 
From Lemma~\ref{l:att1}, it follows that $G$ is mean stable. 
The rest of the statement of our lemma easily follows from Lemma~\ref{l:asnbd} and Remark~\ref{r:asjkere}.  
\end{proof}
\begin{df}
Let $G$ be a rational semigroup. 
We set 
$$A(G):= \overline{G(\{ z\in \CCI \mid \exists g\in G \mbox{ s.t. } 
g(z)=z, |m(g,z)|<1\} )}.$$ 
\end{df} 
\begin{lem}
\label{l:shatt1}
Let $\G \in \mbox{{\em Cpt}}(\emRatp)$. Let $G=\langle \G \rangle .$ 
Suppose that $G$ is semi-hyperbolic and $F(G)\neq \emptyset .$ 
Then, 
 $A(G)$ is an attractor for $(G, \G , F(G))$   
and for each $L\in \mbox{{\em Cpt}}(F(G))$, 
\begin{equation}
\label{eq:l:shatt1-1}
\sup \{ d(\g _{n}\cdots ,\g _{1}(z), A(G))\mid z\in L, (\g_{1},\ldots ,\g _{n})\in \G ^{n}\} 
\rightarrow 0 \mbox{ as } n\rightarrow \infty 
\end{equation}
 and there exists a constant $C>0$ and an $0<\eta  <1$ 
such that 
\begin{equation}
\label{eq:l:shatt1-2}
\sup \{ \| (\g _{n}\cdots \g _{1})'(z)\| _{s}\mid z\in L, (\g_{1},\ldots ,\g _{n})\in \G ^{n}\} \leq C\eta ^{n}
 \mbox{ for each } n\in \NN ,
\end{equation}   
where $\| \cdot \| _{s}$ denotes the norm of the derivative with respect to the spherical 
metric of $\CCI .$  
\end{lem}
\begin{proof}
Since $G$ is semi-hyperbolic and $F(G)\neq \emptyset$, 
\cite[Theorem 1.26]{S4} implies that 
$A(G)$ is a non-empty compact subset of $F(G).$  
Moreover, by the definition of $A(G)$, we have that 
$h(A(G))\subset A(G)$ for each $h\in G.$ 
Let $V_{1},\ldots ,V_{s}$ be finitely many connected components of 
$F(G)$ such that $A(G)\subset \bigcup _{j=1}^{s}V_{j}$ and 
$V_{j}\cap A(G)\neq \emptyset $ for each $j=1,\ldots ,s.$ 
We set $V=\bigcup _{j=1}^{s}V_{j}.$  
In each $j=1,\ldots ,s$, 
we take the hyperbolic metric $\rho _{j}$ on $V_{j}$.  
Let $j\in \{ 1,\ldots ,s\} $ and let $g\in G$.
Since $G$ is semi-hyperbolic, 
we obtain that if $g(V_{j})\subset V_{j}$, 
then $\| g'(z)\| _{h}<1$ for each $z\in V_{j}$, where 
$\| \cdot \| _{h}$ denotes the norm of the derivative 
with respect to $\rho _{j}.$ 
Moreover, for each $\g \in \GN $ and each $z\in V$, 
there exist $p,q\in \NN $ with $1\leq p,q\leq s$ and an 
$i\in \{ 1,\ldots ,s\} $ such that 
$\g _{q,1}(z)\in V_{i}$ and 
$\g _{p+q,1}(z) \in V_{i}.$ 
From these arguments, it follows that if $L$ is a compact 
neighborhood of $A(G)$ in $V$, then 
there exists a constant $C>0$ and a $0<\eta <1$ such that 
the inequality (\ref{eq:l:att1-2}) holds. 
In particular, 
for any $z\in V$ and any $\gamma \in \Gamma ^{\NN }$, 
$d(\gamma _{n,1}(z), A(G))\rightarrow 0$ as $n\rightarrow 0.$ 
We now take arbitrary point $w\in F(G).$ Let $\rho \in \Gamma ^{\NN }$ be arbitrary element. 
By \cite[Theorem 1.26]{S4} again, we have $\overline{\bigcup _{g\in G}g(w)}$ is a compact subset of 
$F(G).$ Hence, there exist $r,s\in \NN $ with $r<s$ and a 
$U\in\mbox{Con}(F(G))$ such that the two points  
$\rho _{s,1}(w)$ and $\rho _{r,1}(w)$ belong to $U.$ 
Then $\rho _{s,r+1}(U)\subset U.$ Since $G$ is semi-hyperbolic, 
it follows that $U\cap A(G)\neq \emptyset .$ 
Therefore, $d(\rho _{n,1}(w), A(G))\rightarrow 0$ as $n\rightarrow \infty .$    
From these argument, we obtain that $A(G)$ is an attractor for 
$(G,\Gamma ,F(G)).$ 
By Lemma~\ref{l:att1}, 
the statement of Lemma~\ref{l:shatt1} holds. 
\end{proof}

We now prove Proposition~\ref{p:hyppjke}.\\ 
{\bf Proof of Proposition~\ref{p:hyppjke}:}
Combining Lemma~\ref{l:shatt1} and Lemma~\ref{l:attas}, the statement of our proposition holds. 
\qed 

We now prove Proposition~\ref{p:asmt}. 

\noindent {\bf Proof of Proposition~\ref{p:asmt}:}
From the definition of mean stability, it is easy to see that 
$S_{\tau }\subset \overline{G_{\tau }^{\ast }(\overline{V})}\subset F(G_{\tau }).$ 
Combining this with Theorem~\ref{t:mtauspec}-\ref{t:mtauspec6} and Theorem~\ref{t:mtauspec}-\ref{t:mtauspec4}, 
we easily obtain that  
statement~\ref{p:asmt2} and statement~\ref{p:asmt3} hold. 
\qed 

\begin{rem}
\label{r:shsimilar}
Let $\Gamma \in \Cpt(\Ratp)$ and let $G=\langle \Gamma \rangle .$ 
Let $\tau \in {\frak M}_{1,c}(\Ratp)$ with $\Gamma _{\tau }=\Gamma .$ 
Suppose that 
$G$ is semi-hyperbolic and $F(G)\neq \emptyset .$ 
Then, by Lemma~\ref{l:shatt1} and the arguments in the proof of Theorem~\ref{t:mtauspec}, 
regarding $M_{\tau }:C(A(G))\rightarrow C(A(G))$, statements which are similar to 
statements \ref{t:mtauspec2},\ref{t:mtauspec2-1}, \ref{t:mtauspec3}--\ref{t:mtauspecdual} in Theorem~\ref{t:mtauspec} hold.
\end{rem}
\subsection{Proofs of results in subsection~\ref{Suffnec}}
\label{pfSuffnec}
In this subsection, we give the proofs of the results in subsection~\ref{Suffnec}. 
We need some lemmas. 
\begin{lem}
\label{l:dimjpt0}
Let $\tau \in {\frak M}_{1,c}(\emRat)$. 
Then, 
$\dim _{H}(J_{pt}^{0}(\tau ))\leq \emMHDt.$ 
\end{lem}
\begin{proof}
Let $f:X_{\tau }\times \CCI \rightarrow X_{\tau }\times \CCI $ be the 
skew product associated with $\Gamma _{\tau }.$ 
Suppose that $\MHDt<\dim _{H}(J_{pt}^{0}(\tau )).$ 
Let $t\in \RR $ be a number such that $\MHDt<t<\dim _{H}(J_{pt}^{0}(\tau )).$ 
Then $H^{t}(J_{pt}^{0}(\tau ))=\infty $, where $H^{t}$ denotes the 
$t$-dimensional Hausdorff measure. 
By \cite[Theorem 5.6]{Fa}, 
there exists a compact subset $F$ of $J_{pt}^{0}(\tau )$ such that 
$0<H^{t}(F)<\infty .$ 
Let $\nu =H^{t}|_{F}.$ 
Since $\MHDt<t$, for $\tilde{\tau }$-a.e. $\g \in X_{\tau }$,  
$\nu (\hat{J}_{\g ,\Gamma _{\tau }})=0.$ 
From Lemma~\ref{l:pctjjg} and Lemma~\ref{l:mgjpt0}, 
it follows that for $\nu $-a.e. $y\in F$, $y\in F_{pt}^{0}(\tau ).$ However, 
this is a contradiction. Thus, 
$\dim _{H}(J_{pt}^{0}(\tau ))\leq \MHDt.$ 
\end{proof}
\begin{df}[\cite{HM}]
\label{d:E(G)}
Let $G$ be a rational semigroup. 
We set $E(G):= \{ z\in \CCI \mid \sharp G^{-1}(z)<\infty \} .$ 
This is called the exceptional set of $G.$ 
\end{df}
\begin{rem}\label{l:egfg}
Let $\Lambda \in \Cpt (\Ratp)$ and let $G=\langle \Lambda \rangle .$ Then 
by \cite[Theorem 4.1.2]{Be}, 
$E(G)\subset F(G).$ 
\end{rem}
\begin{lem}
\label{l:atdimjpt0}
Let $\tau \in {\frak M}_{1,c}(\emRat)$. 
Suppose that $\mbox{{\em Leb}}_{2}(\hat{J}_{\g ,\Gamma _{\tau }})=0$ 
for $\tilde{\tau }$-a.e. $\gamma \in X_{\tau }$, and that 
there exists a weak attractor $A$ for $(G_{\tau },\Gamma _{\tau },F(G_{\tau })).$  
Then, we have the following.
\begin{enumerate}
\item \label{l:atdimjpt01}
For $\mbox{{\em Leb}}_{2}$-a.e. $z\in \CCI $, 
$\tilde{\tau }(\{ \gamma \in X_{\tau }\mid z\in \hat{J}_{\g, \Gamma _{\tau }}\})
=
\tilde{\tau }(\{ \gamma \in X_{\tau }\mid z\in \bigcap _{j=1}^{\infty }\gamma _{1}^{-1}
\cdots \gamma _{j}^{-1}(J(G_{\tau }))\} )=0$. Moreover, 
$\mbox{{\em Leb}}_{2}(J_{pt}^{0}(\tau ))=0.$ 
\item \label{l:atdimjpt02}
$J_{\ker}(G_{\tau })\subset J_{pt}^{0}(\tau ).$ 
\item \label{l:atdimjpt02-1} 
$F_{meas}(\tau )={\frak M}_{1}(\CCI )$ if and only if 
$J_{\ker }(G_{\tau })=\emptyset .$ If $J_{\ker }(G_{\tau })\neq \emptyset $, then 
$J_{meas}(\tau )={\frak M}_{1}(\CCI ).$ 
\item \label{l:atdimjpt03}
If, in addition to the assumption, $\sharp \Gamma _{\tau } <\infty $, 
then we have the following.
\begin{enumerate}
\item \label{l:atdimjpt03a}
$G_{\tau }^{-1}(J_{\ker }(G_{\tau }))\subset J_{pt}^{0}(\tau )$.
\item \label{l:atdimjpt03b}If $\sharp (J(G_{\tau }))\geq 3$ and $J_{\ker}(G_{\tau })\setminus E(G_{\tau })\neq \emptyset $, 
then $J_{pt }(\tau )=J(G_{\tau }).$ 
\end{enumerate}
\end{enumerate}
\end{lem}
\begin{proof}
Statement \ref{l:atdimjpt01} follows from Lemma~\ref{l:pctjjg} and Lemma~\ref{l:mgjpt0}. 
We now show statement \ref{l:atdimjpt02}. 
Let $z_{0}\in J_{\ker }(G_{\tau }).$ 
Let $\varphi \in C(\CCI )$ be an element such that 
supp$\,\varphi \subset \CCI \setminus A$ and $\varphi \equiv 1$ around a 
neighborhood of $J_{\ker }(G_{\tau }).$ Then for each $m\in \NN $, 
$M_{\tau }^{m}(\varphi )(z_{0})=1.$ Moreover, by statement~\ref{l:atdimjpt01}, 
there exists a sequence $\{ z_{n}\} _{n=1}^{\infty }$ in $\CCI $ such that 
$z_{n}\rightarrow z_{0}$ as $n\rightarrow \infty $ and such that 
for each $n\in \NN $, 
$\tilde{\tau }(\{ \gamma \in X_{\tau }\mid z_{n}\in 
\bigcap _{j=1}^{\infty }\gamma _{1}^{-1}\cdots \gamma _{j}^{-1}(J(G_{\tau }))\} )=0.$ 
Hence, for each $n\in \NN $, 
$M_{\tau }^{m}(\varphi )(z_{n})=\int _{X_{\tau }}\varphi (\gamma _{m,1}(z_{n}))\ d\tilde{\tau } (\gamma )
\rightarrow 0$ as $m\rightarrow \infty .$ It implies that $z_{0}\in J_{pt}^{0}(\tau ).$ 
Thus, we have proved statement~\ref{l:atdimjpt02}. 

 We now prove statement~\ref{l:atdimjpt02-1}. 
Combining statement~\ref{l:atdimjpt02} with Theorem~\ref{kerJthm1}, 
we obtain that $F_{meas}(\tau )={\frak M}_{1}(\CCI )$ if and only if 
$J_{\ker }(G_{\tau })=\emptyset .$ 
Suppose that $J_{\ker}(G_{\tau })\neq \emptyset .$ 
Let $\varphi \in C(\CCI )$ be an element such that 
$\varphi \equiv 1 $ in a neighborhood of $J_{\ker}(G_{\tau })$ and 
$\varphi \equiv 0$ in a neighborhood of $A.$ 
Let $\rho \in {\frak M}_{1}(\CCI )$ be an element and let 
$B$ be a neighborhood of $\rho $ in ${\frak M}_{1}(\CCI ).$ 
By statement~\ref{l:atdimjpt01}, there exists an element 
$\rho _{0}\in B$ such that 
$\rho _{0}=\sum _{j=1}^{r}p_{j}\delta _{a_{j}}$, 
where $a_{1}\in J_{\ker }(G_{\tau })$, 
$\tilde{\tau }(\{ \gamma \in X_{\tau }\mid a_{j}\in \hat{J}_{\gamma ,\Gamma _{\tau }}\} )=0$ 
for each $j=2,\ldots ,r$, and $p_{j}>0$ for each $j=1,\ldots ,r.$ 
Then 
$(M_{\tau }^{\ast })^{n}(\rho _{0})(\varphi )=\sum _{j=1}^{r}p_{j}\delta _{z_{j}}(M_{\tau }^{n}(\varphi ))
\rightarrow p_{1}>0$ as $n\rightarrow \infty .$ Moreover, 
by statement~\ref{l:atdimjpt01}, for any neighborhood $B_{0}$ of $\rho _{0}$ in ${\frak M}_{1}(\CCI )$, 
there exists an element
$\rho _{1}\in B_{0}$ such that $\rho _{1}=\sum _{j=1}^{t}q_{j}\delta _{b_{j}}$,
where $\tilde{\tau }(\{ \gamma \in X_{\tau }\mid b_{j}\in \hat{J}_{\gamma , \Gamma _{\tau }}\} )=0$ 
and $q_{j}>0$ for each $j=1,\ldots ,t.$ 
Then $(M_{\tau }^{\ast })^{n}(\rho _{1})(\varphi )\rightarrow 0$ as $n\rightarrow \infty .$ 
Hence, $\rho _{0}\in J_{meas }(\tau ).$ Since $B$ is an arbitrary neighborhood of $\rho $, 
it follows that $\rho \in J_{meas}(\tau ).$  
Thus, we have proved statement~\ref{l:atdimjpt02-1}. 

 We now prove statement~\ref{l:atdimjpt03a}. 
We write $\tau $ as $\sum _{j=1}^{t}p_{j}\delta _{h_{j}}$, where 
$0<p_{j}<1$ and $h_{j}\in \Rat$ for each $j=1,\ldots ,t$. 
Let $z_{0}\in (h_{i_{r}}\cdots h_{i_{1}})^{-1}(J_{\ker}(G_{\tau })).$ 
Let $\varphi \in C(\CCI )$ be an element such that 
$\varphi \geq 0$, 
$\varphi \equiv 1 $ in a neighborhood of $J_{\ker}(G_{\tau })$ and 
$\varphi \equiv 0$ in a neighborhood of $A.$ 
Then for each $m\in \NN $ with $m\geq r+1$, 
$$M_{\tau }^{m}(\varphi )(z_{0})
 \geq p_{i_{r}}\cdots p_{i_{1}}\int \varphi (\gamma _{m}\cdots \gamma _{r+1}h_{i_{r}}\cdots h_{i_{1}}(z_{0}))\ 
d\tau (\gamma _{m})\cdots d\tau (\gamma _{r+1})  
 \geq  p_{i_{r}}\cdots p_{i_{1}}>0.$$
Moreover, by statement~\ref{l:atdimjpt01}, 
there exists a sequence $\{ z_{n}\} _{n=1}^{\infty }$ in $\CCI $ such that 
$z_{n}\rightarrow z_{0}$ as $n\rightarrow \infty $ and such that 
for each $n\in \NN $, 
$\tilde{\tau }(\{ \gamma \in X_{\tau }\mid z_{n}\in \bigcap _{j=1}^{\infty }
\gamma _{1}^{-1}\cdots \gamma _{j}^{-1}(J(G_{\tau }))\} )=0.$ 
Hence, for each $n\in \NN $, 
$M_{\tau }^{m}(\varphi )(z_{n})\rightarrow 0$ as $m\rightarrow \infty .$ 
It implies that $z_{0}\in J_{pt}^{0}(\tau ).$ Thus, 
we have proved statement~\ref{l:atdimjpt03a}. 

 We now prove statement~\ref{l:atdimjpt03b}. 
Under the assumptions of statement~\ref{l:atdimjpt03b}, 
by statement~\ref{l:atdimjpt03a} and \cite[Lemma 2.3 (e)]{S3}, 
we obtain that $J(G_{\tau })=\overline{G_{\tau }^{-1}(J_{\ker }(G_{\tau }))}\subset 
J_{pt}(\tau ).$ Combining this with Lemma~\ref{FJmeaslem1}-\ref{FJmeaslem1-5}, 
we get that $J_{pt}(\tau )=J(G_{\tau }).$ 
Therefore, we have proved statement~\ref{l:atdimjpt03b}. 
%
\end{proof}

We now prove Theorem~\ref{t:dimjpt0jkn}. 

\noindent {\bf Proof of Theorem~\ref{t:dimjpt0jkn}:} 
Combining Lemmas~\ref{l:shatt1}, \ref{l:dimjpt0}, , \ref{l:atdimjpt0}, and 
Remarks~\ref{r:sjgg3}, \ref{r:mhdt}, \ref{l:egfg}, the statement of  
Theorem~\ref{t:dimjpt0jkn} holds.   
\qed 
\subsection{Proofs of results in subsection~\ref{Singular}}
\label{pfSingular}
In this subsection, we give the proofs of the results in 
subsection~\ref{Singular}. We need some lemmas.  

We now give  proofs of Lemmas~\ref{l:disjker} and \ref{l:lsncnonc}.\\ 
\noindent {\bf Proof of Lemma~\ref{l:disjker}:} 
By Lemma~\ref{l:bss}, $J(G)=\bigcup _{j=1}^{m}h_{j}^{-1}(J(G)).$ 
Hence, the statement of our lemma holds.
\qed 
%

\noindent {\bf Proof of Lemma~\ref{l:lsncnonc}:}
By \cite[Theorem 2.3]{S2}, int$(J(G))=\emptyset .$  
By Lemma~\ref{l:disjker}, $J_{\ker}(G)=\emptyset .$ 
Let $\varphi \in (\mbox{LS}({\cal U}_{f,\tau }(\CCI )))_{nc}.$ 
By Theorem~\ref{t:mtauspec}-\ref{t:mtauspec2}, $\varphi \in C_{F(G)}(\CCI ).$ 
Moreover, by Theorem~\ref{t:mtauspec}-\ref{t:mtauspec7}, there exists 
an $l\in \NN $ such that $M_{\tau }^{l}(\varphi )=\varphi .$ 
By Theorem~\ref{t:mtauspec}-\ref{t:mtauspecj3}, 
$\sharp J(G)\geq 3$. 
By \cite[Lemma 2.3 (d)]{S3}, it follows that 
$\sharp E(G)\leq 2.$ Moreover, since $G^{-1}(E(G)\cap J(G))\subset E(G)\cap J(G)$ and 
$h_{i}^{-1}(J(G))\cap h_{j}^{-1}(J(G))=\emptyset $ for each $(i,j)$ with $i\neq j$, 
we obtain that $E(G)\cap J(G)=\emptyset .$ 

Suppose that there exists an open subset $V$ of $\CCI $ such that 
$V\cap J(G)\neq \emptyset $ and $\varphi |_{V}$ is constant. 
We will deduce a contradiction. Let $z_{0}\in J(G)$ be any point. Then $z_{0}\cap J(G)\setminus E(G).$ 
By \cite[Lemma 2.3 (b) (e)]{S3}, 
there exists an $n\in \NN $, an element $(j_{1},\ldots ,j_{nl})\in \{ 1,\ldots ,m\} ^{nl}$, 
and a point $z_{1}\in J(G)\cap V$ such that 
$h_{j_{nl}}\cdots h_{j_{1}}(z_{1})=z_{0}.$ 
Then for each $(i_{1},\ldots ,i_{nl})\in \{ 1,\ldots ,m\} ^{nl}\setminus \{ (j_{1},\ldots ,j_{nl})\} $, 
$h_{i_{nl}}\cdots h_{i_{1}}(z_{1})\in F(G).$ 
Combining this with 
$M_{\tau }^{l}(\varphi )=\varphi $ and  $\varphi \in C_{F(G)}(\CCI )$, 
we obtain that there exists a neighborhood $W$ of 
$z_{1}$ such that 
$\varphi |_{g(W)}$ is constant, where $g=h_{j_{nl}}\cdots h_{j_{1}}.$ 
Therefore $\varphi $ is constant in a neighborhood of $z_{0}.$ 
From this argument and that $\varphi \in C_{F(G)}(\CCI ),$  
it follows that $\varphi :\CCI \rightarrow \CC $ is locally constant on $\CCI $, 
thus $\varphi :\CCI \rightarrow \CC $ is constant. However, this is a contradiction. 

Thus, we have proved Lemma~\ref{l:lsncnonc}. 
\qed 
\begin{lem}
\label{l:nondiff1}
Let $m\in \NN $ with $m\geq 2.$ 
Let $h=(h_{1},\ldots ,h_{m})\in (\emRatp)^{m}$ and we set 
$\Gamma := \{ h_{1},h_{2},\ldots ,h_{m}\} .$ 
Let $G=\langle h_{1},\ldots ,h_{m}\rangle .$ 
Let $f:\Gamma ^{\NN }\times \CCI \rightarrow \Gamma ^{\NN }\times \CCI $ be the  
skew product associated with $\Gamma .$  
Let $p=(p_{1},\ldots ,p_{m})\in {\cal W}_{m}.$  
Let 
$\tau := \sum _{j=1}^{m}p_{j}\delta _{h_{j}}\in {\frak M}_{1}(\Gamma )
\subset {\frak M}_{1}(\emRatp ).$
Suppose that 
$h_{i}^{-1}(J(G))\cap h_{j}^{-1}(J(G))=\emptyset $ for each 
$(i,j)$ with $i\neq j$. 
Then, we have all of the following.
\begin{enumerate}
\item \label{l:nondiff1-1}
Let $(\gamma ,z_{0})\in \tilde{J}(f)$ and let $t\geq 0.$ 
Suppose that there exists a point $z_{1}\in J(G)\setminus P(G)$ and a 
sequence $\{ n_{j}\} _{j=1}^{\infty }$ in $\NN $ such that 
$\gamma _{n_{j},1}(z_{0})\rightarrow z_{1}$ and 
$\tilde{p}(f^{n_{j}-1}(\gamma ,z_{0}))\cdots \tilde{p}(\gamma ,z_{0})\| \gamma _{n_{j},1}'(z_{0})\| _{s}^{t}
\rightarrow \infty $ as $j\rightarrow \infty .$ 
Then, for any $\varphi \in (\mbox{{\em LS}}({\cal U}_{f,\tau }(\CCI )))_{nc}$, 
 $\limsup _{z\rightarrow z_{0}} 
\frac{|\varphi (z)-\varphi (z_{0})|}
{d(z, z_{0})^{t}}=\infty $, where $d$ denotes the spherical distance.  
\item \label{l:nondiff1-2}
Suppose that for each $j=1,\ldots ,m$, 
$1<p_{j}\min \{ \| h_{j}'(z)\| _{s}\mid  z\in h_{j}^{-1}(J(G))\} .$ 
Then, for each $z_{0}\in J(G)$ and for each $\varphi \in (\mbox{{\em LS}}({\cal U}_{f,\tau }(\CCI )))_{nc},$  
$\limsup _{z\rightarrow z_{0}}
\frac{|\varphi (z)-\varphi (z_{0})|}
{d(z,z_{0})}=\infty $ and $\varphi $ is 
not differentiable at $z_{0}.$ 
\end{enumerate}
\end{lem}
\begin{proof}
We first show statement \ref{l:nondiff1-1}. 
Let $\delta := \min _{x\in P(G)}d(x,z_{1})>0.$ 
We may assume that $\gamma _{n_{j},1}(z_{0})\in B(z_{1},\frac{\delta }{4})$ for each $j\in \NN .$ 
By Theorem~\ref{t:mtauspec}-\ref{t:mtauspec7}, there exists an $l\in \NN $ such that 
for each $\varphi \in \mbox{LS}({\cal U}_{f,\tau }(\CCI ))$, $M_{\tau }^{l}(\varphi )=\varphi .$ 
We may assume that for each $j\in \NN $, $l|n_{j}.$ 
Since $h_{i}^{-1}(J(G))\cap h_{j}^{-1}(J(G))=\emptyset $ for each 
$(i,j)$ with $i\neq j$, there exists a number $r_{0}>0$ such that 
for each $(i,j)$ with $i\neq j$, if $z\in h_{i}^{-1}(J(G))$, then 
$h_{j}(B(z,r_{0}))\subset F(G).$  Let $r$ be a number such that 
$0<4r<\delta .$ For each $j\in \NN $, 
let $\alpha _{j} : B(z_{1}, \delta )\rightarrow \CCI $ be 
the well-defined inverse branch of $\gamma _{n_{j},1}$ such that 
$\alpha _{j}(\gamma _{n_{j},1}(z_{0}))=z_{0}.$ 
By the normality of $\{ \alpha _{j}:B(z_{1},\delta )\rightarrow \CCI \} _{j\in \NN }$ 
(see \cite{HM3}), 
taking $r$ so small, we obtain that for each $j\in \NN $, 
the set $B_{j}:= \alpha _{j}(B(\gamma _{n_{j},1}(z_{0}),r))$ satisfies that 
$\mbox{diam }(\gamma _{k,1}(B_{j}))\leq \frac{r_{0}}{2}$ for each 
$k=1,\ldots ,n_{j}.$ Let $(w_{1},w_{2},\ldots )\in \{ 1,\ldots ,m\} ^{\NN } $ be the sequence 
such that $\gamma _{j}=h_{w_{j}}$ for each $j\in \NN .$ 
It follows that for each $j\in \NN $ and each 
$(u_{1},\ldots ,u_{n_{j}})\in \{ 1,\ldots ,m\} ^{n_{j}}\setminus 
\{ (w_{1},\ldots ,w_{n_{j}})\} $, 
$h_{u_{n_{j}}}\cdots h_{u_{1}}(B_{j})\subset F(G).$ 
Hence, for each $j\in \NN $, each $a,b\in B_{j}$, and 
each $\varphi \in (\mbox{LS}({\cal U}_{f,\tau })(\CCI ))_{nc}\subset C_{F(G_{\tau })}(\CCI )$, 
\begin{equation}
\label{eq:l:nondiff1-1}
\varphi (a)-\varphi (b)=p_{w_{n_{j}}}\cdots p_{w_{1}}
(\varphi (\gamma _{n_{j},1}(a))-\varphi (\gamma _{n_{j},1}(b))).
\end{equation}
Let $\varphi \in (\mbox{LS}({\cal U}_{f,\tau })(\CCI ))_{nc}.$ 
By Lemma~\ref{l:lsncnonc}, 
there exists a point $v\in B(z_{1},\frac{r}{2})$ such that 
$\varphi (z_{1})\neq \varphi (v).$ 
Let $j_{0}\in \NN $ be such that for each $j\in \NN $ with $j\geq j_{0}$, 
$B(z_{1},\frac{r}{2})\subset B(\gamma _{n_{j},1}(z_{0}),r).$   
For each $j\in \NN $ with $j\geq j_{0}$, 
let $b_{j}:=\alpha _{j}(v)\in B_{j}.$ 
By the Koebe distortion theorem, 
there exists a constant $C>0$ such that 
for each $j\in \NN $ with $j\geq j_{0}$, 
$d(z_{0},b_{j})\leq C\| \gamma _{n_{j},1}'(z_{0})\| _{s}^{-1}.$  
Furthermore, 
since $\varphi \in C(\CCI )$, 
there exists a number $j_{1}\in \NN $ with $j_{1}\geq j_{0} $ such that 
for each $j\in \NN $ with $j\geq j_{1}$, 
$|\varphi (\gamma _{n_{j},1}(z_{0}))-\varphi (v)|\geq 
\frac{1}{2} |\varphi (z_{1})-\varphi (v)|.$ 
From these arguments, it follows that 
for each $j\in \NN $ with $j\geq j_{1}$, 
\begin{align*}
\frac{|\varphi (z_{0})-\varphi (b_{j})|}{d(z_{0},b_{j})^{t}} 
& = \frac{p_{w_{n_{j}}}\cdots p_{w_{1}}}{d(z_{0},b_{j})^{t}}
|\varphi (\gamma _{n_{j},1}(z_{0}))-\varphi (\gamma _{n_{j},1}(b_{j}))|\\ 
& \geq \frac{1}{2C^{t}}p_{w_{n_{j}}}\cdots p_{w_{1}}\| \gamma _{n_{j},1}'(z_{0})\| _{s}^{t}
|\varphi (z_{1})-\varphi (v)|\rightarrow \infty \ (j\rightarrow \infty ). 
\end{align*} 
Thus, we have proved statement~\ref{l:nondiff1-1}. 

Statement~\ref{l:nondiff1-2} easily follows from statement~\ref{l:nondiff1-1}. 

Thus, we have proved our lemma. 
\end{proof}
\begin{lem}
\label{l:diff1}
Let $m\in \NN $ with $m\geq 2.$ 
Let $h=(h_{1},\ldots ,h_{m})\in (\emRatp)^{m}$ and we set 
$\Gamma := \{ h_{1},h_{2},\ldots ,h_{m}\} .$ 
Let $G=\langle h_{1},\ldots ,h_{m}\rangle .$ 
Let $p=(p_{1},\ldots ,p_{m})\in {\cal W}_{m}.$  
Let $f:\Gamma ^{\NN }\times \CCI \rightarrow \Gamma ^{\NN }\times \CCI $ be the  
skew product associated with $\Gamma .$  
Let 
$\tau := \sum _{j=1}^{m}p_{j}\delta _{h_{j}}\in {\frak M}_{1}(\Gamma )
\subset {\frak M}_{1}({\cal P }).$
Let $(\gamma ,z_{0})\in \tilde{J}(f)$ and let $t\geq 0.$ 
Suppose that 
$G$ is hyperbolic and 
$h_{i}^{-1}(J(G))\cap h_{j}^{-1}(J(G))=\emptyset $ for each 
$(i,j)$ with $i\neq j$. 
Moreover, suppose that 
$\tilde{p}(f^{n-1}(\gamma ,z_{0}))\cdots \tilde{p}(\gamma ,z_{0})\| \gamma _{n,1}'(z_{0})\| _{s}^{t}
\rightarrow 0 $ as $n\rightarrow \infty .$ 
Then for each $\varphi \in \mbox{{\em LS}}({\cal U}_{f,\tau }(\CCI ))$, 
$\limsup _{z\rightarrow z_{0}} 
\frac{|\varphi (z)-\varphi (z_{0})|}
{d(z,z_{0})^{t}}=0 .$ 
\end{lem}
\begin{proof}
By Lemma~\ref{l:disjker}, $J_{\ker }(G)=\emptyset .$ 
By Theorem~\ref{t:mtauspec}-\ref{t:mtauspec7},  
there exists an $l\in \NN $ such that for each $\varphi \in \mbox{LS}({\cal U}_{f,\tau }(\CCI ))$, 
$M_{\tau }^{l}(\varphi )=\varphi .$ 
For each $w=(w_{1},\ldots ,w_{l})\in \{ 1,\ldots ,m\} ^{l}$, 
we set $h_{w}:= h_{w_{l}}\circ \cdots \circ h_{w_{1}}.$ 
Then for each $\alpha ,\beta \in \{ 1,\ldots ,m\} ^{l}$ with $\alpha \neq \beta $, 
$h_{\alpha }^{-1}(J(G))\cap h_{\beta }^{-1}(J(G))=\emptyset .$ 
Let $r_{0}>0$ be a number such that 
for each $\alpha ,\beta \in \{ 1,\ldots ,m\} ^{l}$ with $\alpha \neq \beta $, 
if $z\in h_{\alpha }^{-1}(J(G))$, then 
$h_{\beta }(B(z,r_{0}))\subset F(G).$ 
Since $G$ is hyperbolic, we may assume that 
$B(J(G),r_{0})\subset \CCI \setminus P(G).$ 
We may assume that $2r_{0}<\min \{ d(a,b)\mid a\in J(G), b\in P(G)\} .$ 
We may also assume that for each $\zeta \in \{ 1,\ldots ,m\} ^{l}$ and 
for each $z\in h_{\zeta }^{-1}(J(G))$,  
$h_{\zeta}: B(z,r_{0})\rightarrow \CCI $ is injective. 
We set 
$$c_{1}:= \frac{1}{100}\min \{ \min \{ \| h_{w}'(z)\| _{s}^{-1}\mid w\in \{ 1,\ldots, m\} ^{l},\ z\in B(h_{w}^{-1}(J(G)),r_{0})\} , \frac{1}{2}\} .$$
By \cite[Theorem 2.14 (2)]{S4}, $z_{0}\in J_{\gamma }.$ 
Hence, for each $s>0$, there exists an $n\in \NN $ 
such that 
$\mbox{diam}(\gamma _{nl,1}(B(z_{0},s)))\geq c_{1}r_{0}.$   
Let $n(s)$ be the minimal number of the set of elements $n$ which satisfies the above. 
Then $\mbox{diam}(\gamma _{n(s)l,1}(B(z_{0},s)))\leq \frac{r_{0}}{2}.$ 
Let $\epsilon >0$ be a number. 
Let $(w_{1},w_{2},\ldots )\in \{ 1,\ldots ,m\} ^{\NN }$ be the 
sequence such that $\gamma _{j}=h_{w_{j}}$ for each $j\in \NN .$ 
There exists a positive integer $n_{0}$ such that 
for each $n\in \NN $ with $n\geq n_{0}$, 
$p_{w_{nl}}\cdots p_{w_{1}}\| \gamma _{nl,1}'(z_{0})\| _{s}^{t}\leq \epsilon .$ 
For this $n_{0}$, there exists an $s_{0}>0$ such that 
for each $s$ with $0<s\leq s_{0}$, $n(s)\geq n_{0}.$ 
Let $s$ be such that $0<s\leq s_{0}$. 
Let $\alpha _{n(s)}:B(\gamma _{n(s)l,1}(z_{0}),r_{0})\rightarrow \CCI $ 
be the well-defined inverse branch of $\gamma _{n(s)l,1}$ such that 
$\alpha _{n(s)}(\gamma _{n(s)l,1}(z_{0}))=z_{0}.$  
We have $\alpha _{n(s)}(B(\gamma _{n(s)l,1}(z_{0}),r_{0}))\supset \overline{B(z_{0},s)}.$ 
Since $\mbox{diam}(\gamma _{n(s)l,1}(B(z_{0},s)))\geq c_{1}r_{0}$, by 
\cite[Theorem 2.4]{Mc}, we obtain 
$\mbox{mod}(B(\gamma _{n(s)l,1}(z_{0}),r_{0})\setminus \overline{\gamma _{n(s)l,1}(B(z_{0},s))})\leq c_{1}'$, 
where mod$(\cdot )$  denotes the modulus of the annulus, and  $c_{1}'$ is a positive constant which depends only on $c_{1}.$ 
Thus we obtain $\mbox{mod}(\alpha _{n(s)}(B(\gamma _{n(s)l,1}(z_{0}),r_{0}))\setminus \overline{B(z_{0},s)})\leq c'_{1}.$ 
Hence, by the Koebe distortion theorem, 
there exists a constant $c_{2}>0$, which is independent of $s$, 
such that $\frac{1}{s}\| \alpha _{n(s)}'(\gamma _{n(s)l,1}(z_{0}))\| _{s}\leq c_{2}.$ 
Hence $\frac{1}{s}\leq \|\gamma _{n(s),1}'(z_{0})\| _{s}c_{2}.$  
Combining these arguments and that 
$\mbox{LS}({\cal U}_{f,\tau }(\CCI ))\subset C_{F(G)}(\CCI )$ (Theorem~\ref{t:mtauspec}-\ref{t:mtauspec2}), 
it follows that for each $z\in B(z_{0},s)\setminus B(z_{0},\frac{s}{2})$ and each $\varphi \in \mbox{LS}({\cal U}_{f,\tau }(\CCI ))$,  
\begin{align*}
\frac{|\varphi (z)-\varphi (z_{0})|}
{d(z,z_{0})^{t}}
& = \frac{1}{d(z, z_{0})^{t}}p_{w_{n(s)l}}\cdots p_{w_{1}}
|\varphi (\gamma _{n(s)l,1}(z))-\varphi (\gamma _{n(s)l,1}(z_{0}))|\\ 
& \leq \frac{2^{t}}{s^{t}}p_{w_{n(s)}}\cdots p_{w_{1}}2\| \varphi \| _{\infty }
\leq 2^{1+t}\| \varphi \| _{\infty }c_{2}^{t}\| \gamma _{n(s)l,1}'(z_{0})\| _{s}^{t}p_{w_{n(s)l}}\cdots p_{w_{1}}
\leq 2^{1+t}\| \varphi \| _{\infty }c_{2}^{t}\epsilon .
\end{align*}
Since $2^{1+t}c_{2}^{t}$ is independent of $s$ with $0<s\leq s_{0}$, we obtain that 
for each $a\in \NN $, for each $z\in B(z_{0},\frac{s_{0}}{2^{a}})\setminus 
B(z_{0},\frac{s_{0}}{2^{a+1}})$, and for each  $\varphi \in \mbox{LS}({\cal U}_{f,\tau }(\CCI ))$,  
$\frac{|\varphi (z)-\varphi (z_{0})|}{d(z,z_{0})^{t}}\leq 
2^{1+t}\| \varphi \| _{\infty }c_{2}^{t}\epsilon .$ 
Hence, for each $z\in B(z_{0},s_{0})\setminus \{ z_{0}\} $ and each $\varphi \in \mbox{LS}({\cal U}_{f,\tau }(\CCI ))$,  
$\frac{|\varphi (z)-\varphi (z_{0})|}{d(z,z_{0})^{t}}\leq 
2^{1+t}\| \varphi \| _{\infty }c_{2}^{t}\epsilon .$ 
Thus, we have proved our lemma. 
\end{proof}
We now prove Theorem~\ref{t:hholder}. 

\noindent {\bf Proof of Theorem~\ref{t:hholder}:} 
By Lemma~\ref{l:disjker} and Proposition~\ref{p:hyppjke}, $G$ is mean stable 
and $J_{\ker }(G)=\emptyset . $ 
By Theorem~\ref{t:mtauspec}-\ref{t:mtauspec7}, there exists an $l\in \NN $ 
such that for each $\varphi \in \mbox{LS}({\cal U}_{f,\tau }(\CCI ))$, 
$M_{\tau }^{l}(\varphi )=\varphi .$ 
Let $f:\Gamma ^{\NN }\times \CCI \rightarrow \Gamma ^{\NN }\times \CCI $ be the 
skew product associated with $\Gamma .$ 
Let $t>0$ be a number such that 
$(\max _{j=1}^{m}p_{j})\cdot (\max \{ \| h_{j}'(z)\| _{s}\mid j=1,\ldots ,m, z\in h_{j}^{-1}(J(G))\} )^{t}<1.$ 
Then for any $\epsilon _{1}>0$ there exists a number $n_{0}$ such that 
for each $(\gamma ,z_{0})\in \tilde{J}(f)$ and each $n\in \NN $ with $n\geq n_{0}$,  
$\tilde{p}(f^{nl-1}(\gamma ,z_{0}))\cdots \tilde{p}(\gamma ,z_{0})\| \gamma _{nl,1}'(z_{0})\| _{s}^{t}<\epsilon _{1} .$  
By using the argument in the proof of Lemma~\ref{l:diff1}, 
we obtain that for any $\epsilon _{2}>0$, there exists a number $s_{0}>0$ such that 
for each $(\gamma ,z_{0})\in \tilde{J}(f)$, for each $z\in B(z_{0},s_{0})\setminus \{ z_{0}\} $, and 
for each $\varphi \in \mbox{LS}({\cal U}_{f,\tau }(\CCI ))$,  
$\frac{|\varphi (z)-\varphi (z_{0})|}{d(z,z_{0})^{t}}\leq \epsilon _{2}\| \varphi \| _{\infty }.$  
Combining this with 
 the fact $\pi _{\CCI }(\tilde{J}(f))=J(G)$ (see Lemma~\ref{l:pctjjg}), 
it follows that 
there exists a constant $C>0$ such that 
for each $z_{1},z_{2}\in J(G)$ and each $\varphi \in \mbox{LS}({\cal U}_{f,\tau }(\CCI ))$,  
$|\varphi (z_{1})-\varphi (z_{2})|\leq C\| \varphi \| _{\infty }d(z_{1},z_{2})^{t}.$ 
Take any two points $w_{1}, w_{2}\in \CCI .$ 
For any two points $a,b\in \CCI $, 
let $\overline{ab}$ be the geodesic arc from $a$ to $b$ with respect to 
the spherical metric.  
If $\overline{w_{1}w_{2}}$ is included in $F(G)$, 
then by Theorem~\ref{t:mtauspec}-\ref{t:mtauspec2}, $\varphi (w_{1})=\varphi (w_{2})$ for each 
$\varphi \in \mbox{LS}({\cal U}_{f,\tau }(\CCI )).$  
Suppose that $\overline{w_{1}w_{2}}$ is not included in $F(G).$ 
Then there exists a point $w_{3}\in \overline{w_{1}w_{2}}\cap J(G)$ such that 
$\overline{w_{1}w_{3}}\setminus \{ w_{3}\} \subset F(G)$, and there exists a point 
$w_{4}\in \overline{w_{1}w_{2}}\cap J(G)$ such that 
$\overline{w_{4}w_{2}}\setminus \{ w_{4}\} \subset F(G)$. 
By Theorem~\ref{t:mtauspec}-\ref{t:mtauspec2}, $\varphi (w_{1})=\varphi (w_{3})$ and 
$\varphi (w_{4})=\varphi (w_{2})$, for each $\varphi \in \mbox{LS}({\cal U}_{f,\tau }(\CCI )).$ 
Therefore, for each $\varphi \in \mbox{LS}({\cal U}_{f,\tau }(\CCI ))$, 
$|\varphi (w_{1})-\varphi (w_{2})| 
 = |\varphi (w_{3})-\varphi (w_{4})|  
 \leq C\| \varphi \| _{\infty }d(w_{3},w_{4})^{t}\leq C\| \varphi \| _{\infty }d(w_{1},w_{2})^{t}.$ 
Therefore, for any $\varphi \in \mbox{LS}({\cal U}_{f,\tau }(\CCI ))$, 
$\varphi :\CCI \rightarrow \CC $ is $t$-H\"{o}lder continuous on $\CCI $.  
 
 Thus, we have proved Theorem~\ref{t:hholder}. 
\qed 

\ 

In order to prove Theorems~\ref{t:hnondiff},\ref{t:hdiffornd}, we need a proposition and some lemmas. 
\begin{prop}
\label{p:erghol}
Let $m\in \NN $ with $m\geq 2.$ 
Let $h=(h_{1},\ldots ,h_{m})\in (\emRatp)^{m}$ and we set 
$\Gamma := \{ h_{1},h_{2},\ldots ,h_{m}\} .$ 
Let $G=\langle h_{1},\ldots ,h_{m}\rangle .$ 
Let $p=(p_{1},\ldots ,p_{m})\in {\cal W}_{m}.$  
Let $f:\Gamma ^{\NN }\times \CCI \rightarrow \Gamma ^{\NN }\times \CCI $ be the  
skew product associated with $\Gamma .$  
Let 
$\tau := \sum _{j=1}^{m}p_{j}\delta _{h_{j}}\in {\frak M}_{1}(\Gamma )
\subset {\frak M}_{1}(\emRatp).$ 
Let $\tilde{\nu } \in {\frak M}_{1}(\tilde{J}(f))$ 
be an $f$-invariant ergodic Borel probability measure. 
Let $\nu := (\pi _{\CCI })_{\ast }(\tilde{\nu }).$  
Suppose that 
$G$ is hyperbolic and 
$h_{i}^{-1}(J(G))\cap h_{j}^{-1}(J(G))=\emptyset $ for each 
$(i,j)$ with $i\neq j$. 
Then, there exists a Borel subset $A$ of $J(G)$ with $\nu (A)=1$ such that 
for each $z_{0}\in A$ and each $\varphi \in (\mbox{{\em LS}}({\cal U}_{f,\tau }(\CCI )))_{nc}$,  
$\emHol(\varphi ,z_{0})=u(h,p,\tilde{\nu }).$ 
\end{prop}
\begin{proof}
Let $t_{0}:=u(h,p,\tilde{\nu }).$ 
Let $t<t_{0}.$ 
Then $\int _{\tilde{J}(f)}\log (\tilde{p}(z) \|f'(z)\|_{s}^{t}) d\tilde{\nu }(z)<0.$ 
By Birkhoff's ergodic theorem, there exists a Borel subset $\tilde{A}_{t}$ of 
$\tilde{J}(f)$ with $\tilde{\nu }(\tilde{A}_{t})=1$ such that for 
each $z\in \tilde{A}_{t}$, 
$$\frac{1}{n}\log \left(\tilde{p}(f^{n-1}(z))\cdots \tilde{p}(z)\|(f^{n})'(z)\|_{s}^{t}\right) 
\rightarrow \int _{\tilde{J}(f)}\log (\tilde{p}(z)\|f'(z)\|_{s}^{t})\ d\tilde{\nu }(z) \ \mbox{ as }n\rightarrow \infty .$$ 
Hence, 
for each $z\in \tilde{A}_{t}$, 
$\tilde{p}(f^{n-1}(z))\cdots \tilde{p}(z)\|(f^{n})'(z)\|_{s}^{t}\rightarrow 0$ as 
$n\rightarrow \infty .$ Let $A_{t}:= \pi _{\CCI }(\tilde{A}_{t}).$ 
From Lemma~\ref{l:diff1}, 
it follows that for each $z_{0}\in A_{t}$ and for each $\varphi \in (\mbox{LS}({\cal U}_{f,\tau }(\CCI )))_{nc}$,   
$\lim _{z\rightarrow z_{0}}\frac{|\varphi (z)-\varphi (z_{0})|}{d(z,z_{0})^{t}}=0.$ 

 We now let $s>t_{0}.$ By using an argument similar to that of the above, 
we obtain that 
there exists a subset $\tilde{B}_{s}$ of $\tilde{J}(f)$ with $\tilde{\nu }(\tilde{B}_{s})=1$ such that 
for each $z\in \tilde{B}_{s}$, 
$\tilde{p}(f^{n-1}(z))\cdots \tilde{p}(z)\|(f^{n})'(z)\|_{s}^{s}\rightarrow \infty  $ as 
$n\rightarrow \infty .$ Let $B_{s}=\pi _{\CCI }(\tilde{B}_{s}).$ 
From Lemma~\ref{l:nondiff1}, it follows that 
for each $z_{0}\in B_{s}$ and for each $\varphi \in (\mbox{LS}({\cal U}_{f,\tau }(\CCI )))_{nc}$, 
$\limsup _{z\rightarrow z_{0}}\frac{|\varphi (z)-\varphi (z_{0})|}{d(z,z_{0})^{s}}=\infty .$ 
Let $\{ t_{n}\} _{n=1}^{\infty }$ be a strictly increasing sequence in $\RR $ such that 
$t_{n}\rightarrow t_{0}$ as $n\rightarrow \infty $, and 
let $\{ s_{n}\} _{n=1}^{\infty }$ be a strictly decreasing sequence in $\RR $ such that 
$s_{n}\rightarrow t_{0}$ as $n\rightarrow \infty .$ 
Let $A:= \bigcap _{n=1}^{\infty }A_{t_{n}}\cap \bigcap _{n=1}^{\infty }B_{s_{n}}.$ 
From the above arguments, it follows that $\nu (A)=1$ and 
for each $z_{0}\in A$ and for each $\varphi \in (\mbox{LS}({\cal U}_{f,\tau }(\CCI )))_{nc}$,  
$\Hol(\varphi ,z_{0})=u(h,p,\tilde{\nu }).$ 
Thus, we have proved our proposition.  
\end{proof}
\begin{lem}
\label{l:maxentmeas1}
Let $h=(h_{1},\ldots ,h_{m})\in {\cal P}^{m}$ and we set 
$\Gamma := \{ h_{1},h_{2},\ldots ,h_{m}\} .$ 
Suppose that $h_{i}\neq h_{j}$ for each $(i,j)$ with $i\neq j.$ 
Let $G=\langle h_{1},\ldots ,h_{m}\rangle .$ 
Let $p=(p_{1},\ldots ,p_{m})\in {\cal W}_{m}.$  
Let $f:\Gamma ^{\NN }\times \CCI \rightarrow \Gamma ^{\NN }\times \CCI $ be the  
skew product associated with $\Gamma .$  
Let 
$\tau := \sum _{j=1}^{m}p_{j}\delta _{h_{j}}\in {\frak M}_{1}(\Gamma )
\subset {\frak M}_{1}({\cal P }).$ 
Let $\mu \in {\frak M}_{1}(\tilde{J}(f))$ be the measure defined by 
$\langle \mu ,\varphi \rangle := \int _{\Gamma ^{\NN }}
(\int _{\CCI }\varphi (\gamma ,z) d\mu _{\gamma }(z))d\tilde{\tau }(\gamma )$ 
for any $\varphi \in C(\Gamma ^{\NN }\times \CCI )$, where 
$\mu _{\gamma }$ is the measure coming from Definition~\ref{d:green}. 
Then, $\mu $ is an $f$-invariant ergodic measure, 
$\pi _{\ast }(\mu )=\tilde{\tau }$,  and 
$\mu $ is the maximal relative entropy measure for $f$ with respect to $(\sigma ,\tilde{\tau })$ 
(see Remark~\ref{r:maxrelent}). 
\end{lem}
\begin{proof}
By the argument of the proof of \cite[Theorem 4.2(i)]{J2}, 
$\mu $ is $f$-invariant and ergodic, and 
$\pi _{\ast }(\mu )=\tilde{\tau }.$ 
 Moreover, 
by the argument of the proof of \cite[Theorem 5.2(i)]{J2}, 
we obtain $h_{\mu }(f|\sigma )\geq \int \log \deg (\gamma _{1}) d\tilde{\tau }(\gamma )=
\sum _{j=1}^{m}p_{j}\log \deg (h_{j}).$ Combining this with 
\cite[Theorem 1.3(e)(f)]{S3}, it follows that 
$\mu $ is the unique maximal relative entropy measure for $f$ with respect to $(\sigma, \tilde{\tau }).$
\end{proof}
\begin{lem}
\label{l:lyapcalc}
Let $h=(h_{1},\ldots ,h_{m})\in {\cal P}^{m}$ and we set 
$\Gamma := \{ h_{1},h_{2},\ldots ,h_{m}\} .$ 
Suppose that $h_{i}\neq h_{j}$ for each $(i,j)$ with $i\neq j.$ 
Let $p=(p_{1},\ldots ,p_{m})\in {\cal W}_{m}.$  
Let 
$\tau := \sum _{j=1}^{m}p_{j}\delta _{h_{j}}\in {\frak M}_{1}(\Gamma )
\subset {\frak M}_{1}({\cal P }).$ 
Let $f:\Gamma ^{\NN }\times \CCI \rightarrow \Gamma ^{\NN }\times \CCI $ be the  
skew product associated with $\Gamma .$  
Let $\mu $ be the maximal relative entropy measure for $f$ with respect to 
$(\sigma ,\tilde{\tau }).$ 
Then $\int _{\Gamma ^{\NN }\times \CCI }\log \|f'\| _{s} d\mu =
\sum _{j=1}^{m}p_{j}\log \deg (h_{j})+\int _{\Gamma ^{\NN }}\Omega (\gamma ) 
d\tilde{\tau }(\gamma ) .$ 
\end{lem}
\begin{proof}
For each $\gamma \in \Gamma ^{\NN }$, let 
$d(\gamma )=\deg (\gamma _{1})$ and 
$R(\gamma ):= \lim _{z\rightarrow \infty }(G_{\gamma }(z)-\log |z|).$ 
Moreover, we denote by $a(\gamma )$ the coefficient of highest order term of $\gamma _{1}.$ 
Since $\frac{1}{d(\gamma )}G_{\sigma (\gamma )}(\gamma _{1}(z))=G_{\gamma }(z)$, 
we obtain that $R(\sigma (\gamma ))+\log |a(\gamma )|=d(\gamma )R(\gamma )$ for each 
$\gamma \in \Gamma ^{\NN }.$ 
Moreover, 
since $dd^{c}(\int _{\CC }\log |w-z|d\mu _{\gamma }(w))=\mu _{\gamma }$ and 
$\int _{\CC }\log |w-z|d\mu _{\gamma }(w)=\log |z|+o(1)$ as $z\rightarrow \infty $ (see \cite{Ra}), 
we have $\int _{\CC }\log |w-z|d\mu _{\gamma }(w)=G_{\gamma }(z)-R(\gamma )$ for each $\gamma \in \Gamma ^{\NN }$ 
and $z\in \CC .$ In particular, $\gamma \mapsto R(\gamma )$ is continuous on $\Gamma ^{\NN }.$ 
By using the above formula, we obtain  
$\int _{\CCI }\log |\gamma _{1}'(z)| d\mu _{\gamma }(z)=\log |a(\gamma )|+\log d(\gamma )-(d(\gamma )-1)R(\gamma )+\Omega (\gamma )$ 
for each $\gamma \in \Gamma ^{\NN }.$ 
In particular, $\gamma \mapsto \int _{\CCI }\log |\gamma _{1}'(z)| d\mu _{\gamma }(z)$ is continuous on 
$\Gamma ^{\NN }.$ 
Furthermore, $\sigma _{\ast }(\tilde{\tau })=\tilde{\tau }.$ 
From these arguments and Lemma~\ref{l:maxentmeas1}, 
we obtain 
\begin{align*}
\int _{\Gamma ^{\NN }\times \CCI }\log |f'|d\mu 
& = \int _{\Gamma ^{\NN }}d\tilde{\tau }(\gamma )\int _{\CCI }\log |\gamma _{1}'(z)| d\mu _{\gamma }(z)\\ 
&=\int _{\Gamma ^{\NN }}\left(\log |a(\gamma )|+\log d(\gamma )-(d(\gamma )-1)R(\gamma )+\Omega (\gamma )\right) d\tilde{\tau }(\gamma )\\ 
&=\int _{\Gamma ^{\NN }}\left(R(\gamma )-R(\sigma (\gamma ))+\log d(\gamma )+\Omega (\gamma )\right) d\tilde{\tau }(\gamma )\\ 
&=\int _{\Gamma ^{\NN }}(\log d(\gamma )+\Omega (\gamma )) d\tilde{\tau }(\gamma )
=\sum _{j=1}^{m}p_{j}\log \deg (h_{j})+\int _{\Gamma ^{\NN }}\Omega (\gamma ) d\tilde{\tau }(\gamma ). 
\end{align*}
Moreover, since $\mu $ is $f$-invariant, and since the Euclidian metric and the spherical metric 
are comparable on the compact subset $J(G)$ of $\CC $, we have 
$\int _{\Gamma ^{\NN }\times \CCI }\log |f'|d\mu= \int _{\Gamma ^{\NN }\times \CCI }\log \|f'\| _{s}d\mu .$ 

Thus, we have proved our lemma. 
\end{proof}

We now prove Theorem~\ref{t:hnondiff}. 

\noindent {\bf Proof of Theorem~\ref{t:hnondiff}:}
By Lemma~\ref{l:disjker} and Proposition~\ref{p:hyppjke}, 
$G_{\tau }=G$ is mean stable and $J_{\ker }(G)=\emptyset .$  
Since $G$ is hyperbolic and $h_{i}^{-1}(J(G))\cap h_{j}^{-1}(J(G))=\emptyset $ for each 
$(i,j)$ with $i\neq j$, \cite[Corollary 3.6]{S2} implies that 
$0<\dim _{H}(J(G))<2.$  By \cite[Theorem 4.3]{S3}, 
we obtain $\mbox{supp}\, \lambda =J(G).$ Moreover, by \cite[Lemma 5.1]{S3}, 
$\lambda (\{ z\} )=0$ for each $z\in J(G).$ 
Thus we have proved statements \ref{t:hnondiff0}--\ref{t:hnondiff3}.

Statement \ref{t:hnondiff4} follows from 
Proposition~\ref{p:erghol}. 
%

We now prove statement~\ref{t:hnondiff4-1}. 
Since $\pi _{\ast }(\mu )=\tilde{\tau }$, 
$\int _{\Gamma ^{\NN }\times \CCI }\log \tilde{p}\ d\mu =\sum _{j=1}^{m}p_{j}\log p_{j}.$ 
Combining this with Lemma~\ref{l:lyapcalc}, 
it follows that 
$$u(h,p,\mu )=\frac{-(\sum _{j=1}^{m}p_{j}\log p_{j})}
{\sum _{j=1}^{m}p_{j}\log \deg (h_{j})+\int _{\Gamma ^{\NN }}\Omega (\gamma )\ d\tilde{\tau }(\gamma )}.
$$
Moreover, by \cite[Theorem 1.3 (f)]{S3}, 
$h_{\mu }(f|\sigma )=\sum _{j=1}^{m}p_{j}\log \deg (h_{j}).$ 
Hence, $h_{\mu }(f)=h_{\mu }(f|\sigma )+h_{\pi _{\ast }(\mu )}(\sigma )=
\sum _{j=1}^{m}p_{j}\log \deg (h_{j})-\sum _{j=1}^{m}p_{j}\log p_{j}$, 
where $h_{\mu }(f)$ denotes the metric entropy of $(f, \mu ).$  
Combining this with \cite[Lemma 7.1]{S3}, Lemma~\ref{l:lyapcalc}, and 
that $\pi _{\CCI }:\tilde{J}(f)\rightarrow J(G)$ is a homeomorphism, 
we obtain that 
\begin{equation}
\label{eq:dimhmu}
\dim _{H}(\lambda )=\frac{\sum _{j=1}^{m}p_{j}\log \deg (h_{j})-\sum _{j=1}^{m}p_{j}\log p_{j}}
{\sum _{j=1}^{m}p_{j}\log \deg (h_{j})+\int _{\Gamma ^{\NN }}\Omega (\gamma )\ d\tilde{\tau }(\gamma )},
\end{equation}
where $\dim _{H}(\lambda ):= \inf \{ \dim _{H}(A)\mid A\mbox{ is a Borel subset of } J(G), \lambda (A)=1\} .$ 
Hence, we have proved statement \ref{t:hnondiff4-1}. 

We now prove statement \ref{t:hnondiff5}. 
Suppose that at least one of items (a),(b), and (c) in statement~\ref{t:hnondiff5} holds. 
We will show the following claim. 

Claim: $u(h,p,\mu )<1.$  

To prove this claim, let $d_{j}:=\deg h_{j}$ for each $j.$ 
Suppose $\sum _{j=1}^{m}p_{j}\log (p_{j}d_{j})>0$. 
From statement~\ref{t:hnondiff4-1}, it follows that 
$u(h,p,\mu )\leq (- \sum _{j=1}^{m}p_{j}\log p_{j})/(\sum _{j=1}^{m}p_{j}\log d_{j})<1.$ 
We now suppose $P^{\ast }(G)$ is bounded. 
Then, $\Omega (\gamma )=0$ for each $\gamma \in \Gamma ^{\NN }.$ 
Combining this with the second inequality in statement~\ref{t:hnondiff4-1}, 
we obtain $\sum _{j=1}^{m}p_{j}\log (p_{j}d_{j})>0.$ Thus, 
$u(h,p,\mu )<1.$ 
We now suppose that $m=2$ and  
$P^{\ast }(G)$ is not bounded. Then 
there exists an element $\alpha \in \Gamma ^{\NN }$ such that 
$\Omega (\alpha )>0.$ Since $\gamma \mapsto \Omega (\gamma )$ is continuous on $\Gamma ^{\NN }$, 
statement~\ref{t:hnondiff4-1} implies that 
$u(h,p,\mu )\leq \frac{\log 2}{\log 2+\int _{\Gamma ^{\NN }}\Omega (\gamma ) 
d\tilde{\tau }(\gamma )}<1.$ 
Hence, the above claim holds. 
From the above claim and statements ~\ref{t:hnondiff2}--\ref{t:hnondiff4}, 
we easily obtain that statement~\ref{t:hnondiff5} holds. 

Thus, we have proved Theorem~\ref{t:hnondiff}. 
\qed 

\ 

We now prove Theorem~\ref{t:hdiffornd}. 
We use the following notation.

\noindent {\bf Notation.} 
Let $(h_{1},\ldots ,h_{m})\in (\Rat)^{m}.$ 
We set $\Sigma _{m}:= \{ 1,\ldots ,m\} ^{\NN }$ and  
$\Sigma _{m}^{\ast }:= \bigcup _{j=1}^{\NN }\{ 1,\ldots ,m\} ^{j}$,. 
Moreover, for each $w=(w_{1},\ldots ,w_{k})\in \Sigma _{m}^{\ast }$, 
we set $|w|=k$ and $h_{w}:=h_{w_{k}}\circ \cdots \circ h_{w_{1}}.$

\noindent {\bf Proof of Theorem~\ref{t:hdiffornd}:} 
By Theorem~\ref{t:hnondiff}-\ref{t:hnondiff0}, $G$ is mean stable and $J_{\ker }(G)=\emptyset .$  
Since $G$ is hyperbolic, \cite[Theorem 2.17]{S4} implies that 
$G$ is expanding in the sense of \cite[Definition 3.1]{S6}. 
We use the arguments in \cite{S6}. 
We now prove the following claim. 

Claim 1: Under the assumptions of Theorem~\ref{t:hdiffornd}, 
there exists a $k\in \NN $ and a non-empty open subset $U$ of $\CCI $ 
such that 
$\bigcup _{w:|w|=k}h_{w}^{-1}(U)\subset U$ and 
for each $w,w'\in \{ 1,\ldots, m\} ^{k}$ with $w\neq w'$, 
$h_{w}^{-1}(U)\cap h_{w'}^{-1}(U)=\emptyset .$  

To prove this claim, since $G$ is expanding, there exists a $k\in \NN $ such that 
$\inf _{z\in \tilde{J}(f)}\|(f^{k})'(z)\| _{s}\geq 4.$
By \cite[Theorem 2.4]{HM}, 
we have $J(G)=J(\langle h_{w}\mid |w|=k\rangle )$. Moreover, 
by Lemma~\ref{l:pctjjg}, $\pi _{\CCI }(\tilde{J}(f))=J(G).$ 
Take a number $a>0$ such that 
for each $w\in \{ 1,\ldots ,m\} ^{k}$, 
for each $z\in J(G)$, and for each well-defined inverse branch 
$\zeta : B(z,a)\rightarrow \CCI $ 
of $h_{w}$, $\| \zeta '(x)\| _{s}\leq 1/3$ for each $x\in B(z,a).$    
Let $b>0$ be a number such that 
$$b<\frac{1}{2}\min \{ d(z,z') \mid z\in h_{w}^{-1}(J(G)), z'\in h_{w'}^{-1}(J(G)), 
w,w'\in \{ 1,\ldots ,m\} ^{k}, w\neq w'\} , \mbox{ and } b<a.$$ 
Then $B(h_{w}^{-1}(J(G)),b)\cap B(h_{w'}^{-1}(J(G)),b)=\emptyset $ if 
$|w|=|w'|=k$ and $w\neq w'.$ 
Let $U:=B(J(G),b).$ 
since $h_{w}^{-1}(J(G))\subset J(G)$, 
the above arguments imply that 
$\bigcup _{w:|w|=k}h_{w}^{-1}(U)\subset U$, and 
for each $w,w'\in \{ 1,\ldots m\} ^{k}$ with $w\neq w'$, 
$h_{w}^{-1}(U)\cap h_{w'}^{-1}(U)=\emptyset .$ 
Thus, we have prove Claim 1.

 Let $\Lambda := \{ h_{w}\mid |w|=k\} .$ 
Then $J(\langle \Lambda \rangle )=J(G).$ 
  Let $\overline{f}: \Lambda ^{\NN }\times \CCI \rightarrow 
  \Lambda ^{\NN }\times \CCI $ be the skew product map associated with 
  $\Lambda .$ For each $t\geq 0$, 
  let $L_{t}':C(J(G))\rightarrow C(J(G))$ be the operator 
  defined by $L_{t}'(\varphi )(z)=\sum _{|w|=k}\sum _{h_{w}(y)=z}\varphi (y)\|h_{w}'(y)\| _{s}^{-t}.$ 
By \cite[Theorems 1.1, 1,2, Lemma 4.9]{S6}, 
there exists a unique element $\nu \in {\frak M}_{1}(J(G))$ such that 
$L_{\delta }'^{\ast }\nu =\nu .$ Moreover, 
by \cite[Theorem 1.2]{S6}, $0<H^{\delta }(J(G))<\infty $ and 
$\nu =H^{\delta }/(H^{\delta }(J(G))).$ 
Furthermore, by \cite[Corollary 3.6]{S2}, 
$0<\delta <2.$   
For each $t\geq 0$, 
let $\tilde{L}_{t}: C(\tilde{J}(f))\rightarrow C(\tilde{J}(f))$ be 
the operator defined by 
$\tilde{L}_{t}(\varphi )(z)=\sum _{f(y)=z}\varphi (y)\|f'(y)\|_{s}^{-t}$, and 
let $L_{t}:C(J(G))\rightarrow C(J(G))$ be the operator 
defined by $L_{t}(\varphi )(z)=\sum _{j=1}^{m}\sum _{h_{j}(y)=z}\varphi (y)\|h_{j}'(y)\| _{s}^{-t}.$ 
By \cite[Theorem 1.1, Lemma 3.6, Lemma 4.7]{S6}, 
there exists a $t_{0}\geq 0$ satisfying the following: 
\begin{itemize}
\item[(a)] there exists a unique $\tilde{\nu }_{0}\in {\frak M}_{1}(\tilde{J}(f))$ such that 
$\tilde{L}_{t_{0}}^{\ast }(\tilde{\nu }_{0})=\tilde{\nu }_{0}$; 
\item[(b)] the limits $\tilde{\alpha}_{0}:= \lim _{l\rightarrow \infty }\tilde{L}_{t_{0}}^{l}(1 )\in C(\tilde{J}(f))$ 
and $\alpha _{0} := \lim _{l\rightarrow \infty }L_{t_{0}}^{l}(1)\in C(J(G))$ exist; and 
\item[(c)]$\tilde{\rho }_{0}:= \tilde{\alpha }_{0}\tilde{\nu }_{0}\in {\frak M}_{1}(\tilde{J}(f))$ is $f$-invariant and 
ergodic, and $\min _{z\in J(G)}\alpha _{0} (z)>0.$ 
\end{itemize}
Let $\nu _{0}:= (\pi _{\CCI })_{\ast }(\tilde{\nu }_{0})\in {\frak M}_{1}(J(G)).$ 
Since $(L_{t_{0}})^{k}=L_{t_{0}}'$, 
$L_{t_{0}}'^{\ast }(\nu _{0})=\nu _{0}.$ 
Hence, by \cite[Lemma 4.9]{S6}, we obtain that 
$t_{0}=\delta ,\nu _{0}=\nu $.  
From these arguments, statements~\ref{t:hdiffornd1}--\ref{t:hdiffornd3} hold. 

 We now prove statement~\ref{t:hdiffornd4}. 
From the above argument, $\tilde{\alpha }_{0}=\tilde{\alpha }$ and 
$\alpha _{0}=\alpha $. Moreover, $\tilde{\rho }:=\tilde{\rho }_{0}$ is $f$-invariant and ergodic. 
From Proposition~\ref{p:erghol}, it follows that 
there exists a Borel subset $A$ of $J(G)$ with $H^{\delta }(A)=H^{\delta }(J(G))$ such that 
for each $z_{0}\in A$ and for each $\varphi \in (\mbox{LS}({\cal U}_{f,\tau }(\CCI )))_{nc}$, 
$\Hol(\varphi , z_{0})=u(h,p,\tilde{\rho }).$ 
Moreover, by \cite[Lemma 4.7]{S6}, 
$\alpha \circ \pi _{\CCI }=\tilde{\alpha }.$ Therefore, statement~\ref{t:hdiffornd4} holds. 

Thus, we have proved Theorem~\ref{t:hdiffornd}.   
\qed  

\section{Examples} 
\label{Examples} 
We give some examples to which we can apply Theorem~\ref{kerJthm1}, 
Theorem~\ref{t:mtauspec}, Theorem~\ref{kerJthm2}, Proposition~\ref{p:chtinfty}, 
Theorem~\ref{kerJthm3}, Theorem~\ref{thm:jkreyintj}, Proposition~\ref{p:hyppjke}, 
Theorem~\ref{t:hnondiff}, Theorem~\ref{t:hdiffornd}, 
Corollary~\ref{c:hypdisjc}, and Theorem~\ref{t:hholder}.  
\begin{prop} 
\label{semihyposcexprop}
Let $f_{1}\in {\cal P}.$ 
Suppose that {\em int}$(K(f_{1}))$ is not empty. 
Let $b\in \mbox{{\em int}}(K(f_{1}))$ be a point. 
Let $d$ be a positive integer such that 
$d\geq 2.$ Suppose that $(\deg (f_{1}),d)\neq (2,2).$ 
Then, there exists a number $c>0$ such that 
for each $\l \in \{ \l\in \Bbb{C}: 0<|\l |<c\} $, 
setting $f_{\l }=(f_{\l ,1},f_{\l ,2})=
(f_{1},\l (z-b)^{d}+b )$ and $G_{\l }:= \langle f_{1},f_{\l, 2}\rangle $, 
 we have all of the following.
\begin{itemize}
\item[{\em (a)}] 
$f_{\l }$ satisfies the open set condition with 
an open subset $U_{\l }$ of $\CCI $ (i.e., $f_{\l ,1}^{-1}(U_{\l })\cup f_{\l, 2}^{-1}(U_{\l })\subset U_{\l }$ and 
$f_{\l ,1}^{-1}(U_{\l })\cap f_{\l ,2}^{-1}(U_{\l })=\emptyset $), 
$f_{\l ,1}^{-1}(J(G_{\l }))\cap f_{\l, 2}^{-1}(J(G_{\l }))=\emptyset $, 
{\em int}$(J(G_{\l }))=\emptyset $, 
$J_{\ker }(G_{\l })=\emptyset $, 
$G_{\l }(K(f_{1}))\subset K(f_{1})\subset \mbox{{\em int}}(K(f_{\lambda ,2}))$  
and 
$\emptyset \neq K(f_{1})\subset \hat{K}(G_{\l }).$ 
\item[{\em (b)}]
If $K(f_{1})$ is connected, then  
$P^{\ast }(G_{\l })
$ is bounded in $\Bbb{C}$.
\item[{\em (c)}]
If $f_{1}$ is semi-hyperbolic (resp. hyperbolic) and $K(f_{1})$ is connected, 
then 
$G_{\l }$ is semi-hyperbolic (resp. hyperbolic), 
$J(G_{\l } )$ is porous (for the definition of porosity, see \cite{S7}),  and 
$\dim _{H}(J(G_{\l }))<2$.  
\end{itemize} 
\end{prop}
\begin{proof} 
Conjugating $f_{1}$ by a M\"{o}bius transformation, 
we may assume that $b=0$ and the coefficient 
of the highest degree term of $f_{1}$ is equal to $1.$ 
Let $r>0$ be a 
 number such that $\overline{B(0,r)}\subset \mbox{int}(K(f_{1})).$ 
We set $d_{1}:=\deg (f_{1}).$ 
 Let $\alpha >0$ be a number. 
Since $d\geq 2$ and $(d,d_{1})\neq (2,2)$, 
it is easy to see that 
$(\frac{r}{\alpha })^{\frac{1}{d}}>
2\left(2(\frac{1}{\alpha })
^{\frac{1}{d-1}}\right)^{\frac{1}{d_{1}}}
$ if and only if 
\begin{equation}
\label{Contproppfeq1}
\log \alpha <
\frac{d(d-1)d_{1}}{d+d_{1}-d_{1}d}
( \log 2-\frac{1}{d_{1}}\log \frac{1}{2}-\frac{1}{d}\log r) .
\end{equation} 
We set 
\begin{equation}
\label{Contproppfeq2}
c_{0}:=\exp \left(\frac{d(d-1)d_{1}}{d+d_{1}-d_{1}d}
( \log 2-\frac{1}{d_{1}}\log \frac{1}{2}-\frac{1}{d}\log r) \right)
\in (0,\infty ).
\end{equation}
 
Let $0<c<c_{0}$ be a small number and let $\l \in \Bbb{C} $ 
be a number with $0<|\l |<c.$ 
Put $f_{\l ,2}(z)=\l z^{d}.$ 
Then, we obtain $K(f_{\l ,2})=\{ z\in \Bbb{C} \mid 
|z|\leq (\frac{1}{|\l |})^{\frac{1}{d-1}}\} $ and 
$$f_{\l,2}^{-1}(\{ z\in \Bbb{C} \mid  |z|=r\} )=
\{ z\in \Bbb{C} \mid |z|=(\frac{r}{|\l |})^{\frac{1}{d}}\} .$$
Let 
$D_{\l }:=\overline{B(0,2(\frac{1}{|\l |})^{\frac{1}{d-1}})}.$ 
Since $f_{1}(z)=z^{d_{1}}(1+o(1))\ (z\rightarrow \infty )$,  
it follows that if $c$ is small enough, then 
for any $\l \in \Bbb{C} $ with $0<|\l |<c$, 
$$f_{1}^{-1}(D_{\l })\subset 
\left\{ z\in \Bbb{C} \mid 
|z|\leq 2\left( 2(\frac{1}{|\l |})^{\frac{1}{d-1}}\right) 
^{\frac{1}{d_{1}}}\right\} .$$  
This implies that  
\begin{equation}
\label{Contproppfeq3}
f_{1}^{-1}(D_{\l })\subset f_{\l ,2}^{-1}(\{ z\in \Bbb{C} \mid |z|<r\} ).
\end{equation} 
Hence, 
setting $U_{\l }:=(\mbox{int}(K(f_{\l ,2})))\setminus K(f_{1})$, 
$f_{1}^{-1}(U_{\l })\cup f_{\l ,2}^{-1}(U_{\l })\subset U_{\l }$ and 
$f_{1}^{-1}(\overline{U_{\l }})\cap f_{\l, 2}^{-1}(\overline{U_{\l }})=\emptyset $. 
We have $J(G_{\l })\subset \overline{U_{\l }}
\subset K(f_{\l ,2})\setminus \mbox{int}(K(f_{1})). $ 
In particular, 
$f_{\l, 1}^{-1}(J(G_{\l }))\cap f_{\l, 2}^{-1}(J(G_{\l }))=\emptyset $ and 
$(\mbox{int}(K(f_{1})))\cup (\CCI \setminus K(f_{\l ,2}))\subset 
F(G_{\l }).$ By \cite[Theorem 2.3]{S2}, int$(J(G_{\l }))=\emptyset .$ 
Moreover, by Lemma~\ref{l:disjker}, we obtain that $J_{\ker }(G_{\l })=\emptyset .$ 
Furthermore, (\ref{Contproppfeq3}) implies that 
$f_{\l ,2}(K(f_{1}))\subset \mbox{int}(K(f_{1})).$ 
Thus, $G_{\l }(K(f_{1}))\subset K(f_{1})\subset \mbox{int}(K(f_{\lambda ,2}))$ and 
$\emptyset \neq K(f_{1})\subset \hat{K}(G_{\l }).$ 

We now assume that $K(f_{1})$ is connected. 
Then we have $P^{\ast }(G_{\l })  
= \bigcup _{g\in G_{\l }^{\ast }}
g(CV^{\ast }(f_{1})\cup CV^{\ast }(f_{\l, 2}))
\subset 
K(f_{1})$, 
where $CV^{\ast }(\cdot )$ denotes the set of 
all critical values in $\Bbb{C}.$ Hence, 
$P^{\ast }(G_{\l })$ is bounded in $\Bbb{C}.$ 

We now suppose that $f_{1}$ is semi-hyperbolic and $K(f_{1})$ is connected. 
Then  
there exist an $N\in \NN $ and a $\delta _{1}>0$ such that 
for each $x\in J(f_{1})$ and for each $n\in \NN $, 
$\deg (f_{1}^{n}:V\rightarrow B(x,\delta _{1}))\leq N$ for each 
connected component $V$ of $f_{1}^{-n}(B(x,\delta _{1})).$ 
Moreover, $f_{\l, 2}^{-1}(J(f_{1}))\cap K(h_{1})=\emptyset $ and 
so $f_{\l ,2}^{-1}(J(f_{1}))\subset \CCI \setminus P(G_{\l }). $ 
From these arguments and \cite[Lemma 1.10]{S4}, 
it follows that there exists a $0<\delta _{2} (<\delta _{1})$ such that 
for each $x\in J(f_{1})$ and each $g\in G_{\l }$, 
$\deg (g:V\rightarrow B(x,\delta _{2}))\leq N$ for each connected component $V$ of 
$g^{-1}(B(x,\delta _{2})).$ 
Since $P^{\ast }(G_{\l })\subset K(f_{1})$ again, 
we obtain that there exists a $0<\delta _{3} (<\delta _{2} )$ such that 
for each $x\in J(G_{\l })$ and each $g\in G_{\l }$, 
$\deg (g:V\rightarrow B(x,\delta _{3}))\leq N$ for each connected component 
$V$ of $g^{-1}(B(x,\delta _{3})).$ Thus, $G_{\l }$ is semi-hyperbolic. 
Since $J(G_{\l })\subset f_{1}^{-1}(\ov{U_{\l }})\cup f_{\l ,2}^{-1}(\overline{U_{\l }})
\subsetneqq \ov{U_{\l }}$, 
\cite[Theorem 1.25]{S7} implies that $J(G_{\l })$ is porous and 
$\dim _{H} (J(G_{\l }))<2.$ 

 We now suppose that $f_{1}$ is hyperbolic and $K(f_{1})$ is connected. 
Then we may assume that the above $N$ is equal to $1.$ Therefore, 
$G_{\l }$ is hyperbolic. 
 
Thus we have proved our proposition. 
\end{proof}

\begin{ex}[Devil's coliseum] 
\label{ex:dc1}
Let $g_{1}(z):=z^{2}-1, g_{2}(z):=z^{2}/4, h_{1}:=g_{1}^{2},$ and  $h_{2}:=
g_{2}^{2}.$ Let $G=\langle h_{1},h_{2}\rangle $ and $\tau := \sum _{i=1}^{2}\frac{1}{2}\delta _{h_{i}}.$ 
Then it is easy to see that 
setting $A:= K(h_{2})\setminus D(0,0.4)$, 
we have 
$\overline{D(0,0.4)}\subset \mbox{int}(K(h_{1}))$, $h_{2}(K(h_{1}))\subset \mbox{int}(K(h_{1}))$,  
$h_{1}^{-1}(A)\cup h_{2}^{-1}(A)\subset A$, and $h_{1}^{-1}(A)\cap h_{2}^{-1}(A)=\emptyset .$ 
Therefore $h_{1}^{-1}(J(G))\cap h_{2}^{-1}(J(G))=\emptyset $ and $\emptyset \neq K(h_{1})\subset \hat{K}(G).$ 
 Moreover, using the argument in the proof of 
Proposition~\ref{semihyposcexprop}, we obtain that $G$ is hyperbolic. 
By Lemma~\ref{l:disjker}, $J_{\ker}(G)=\emptyset .$ 
By Theorem~\ref{kerJthm2} and Lemma~\ref{l:lsncnonc}, 
we obtain that $T_{\infty ,\tau }$ is continuous on $\CCI $ and the set of varying points of 
$T_{\infty ,\tau }$ is equal to $J(G).$ Moreover, by Theorem~\ref{t:hnondiff}, 
$\dim _{H}(J(G))<2$ and for each non-empty open subset $U$ of $J(G)$ there exists an uncountable dense 
subset $A_{U}$ of $U$ such that for each $z\in A_{U}$, 
$T_{\infty ,\tau }$ is not differentiable at $z.$ See Figures~\ref{fig:dcjulia}, \ref{fig:dcgraphgrey2}, 
and \ref{fig:dcgraphudgrey2}.  $T_{\infty ,\tau }$ is called a devil's coliseum. 
It is a complex analogue of the devil's staircase. 
\end{ex}

\begin{figure}[htbp]
\caption{The Julia set of $G=\langle h_{1}, h_{2}\rangle $, where  
$g_{1}(z):=z^{2}-1, g_{2}(z):=z^{2}/4, h_{1}:=g_{1}^{2}, h_{2}:=
g_{2}^{2}.$ 
We have $J_{\ker }(G)=\emptyset $ and $\dim _{H}(J(G))<2.$\ }
\ \ \ \ \ \ \ \ \ \ \ \ \ \ \ \ \ \ \ \ \ \ \ \ \ \ \ \ \ \ \ 
\ \ \ \ \ \ \ \ \ \ \ \ \ \ \ \ \ \ \ \ 
\includegraphics[width=3cm,width=3cm]{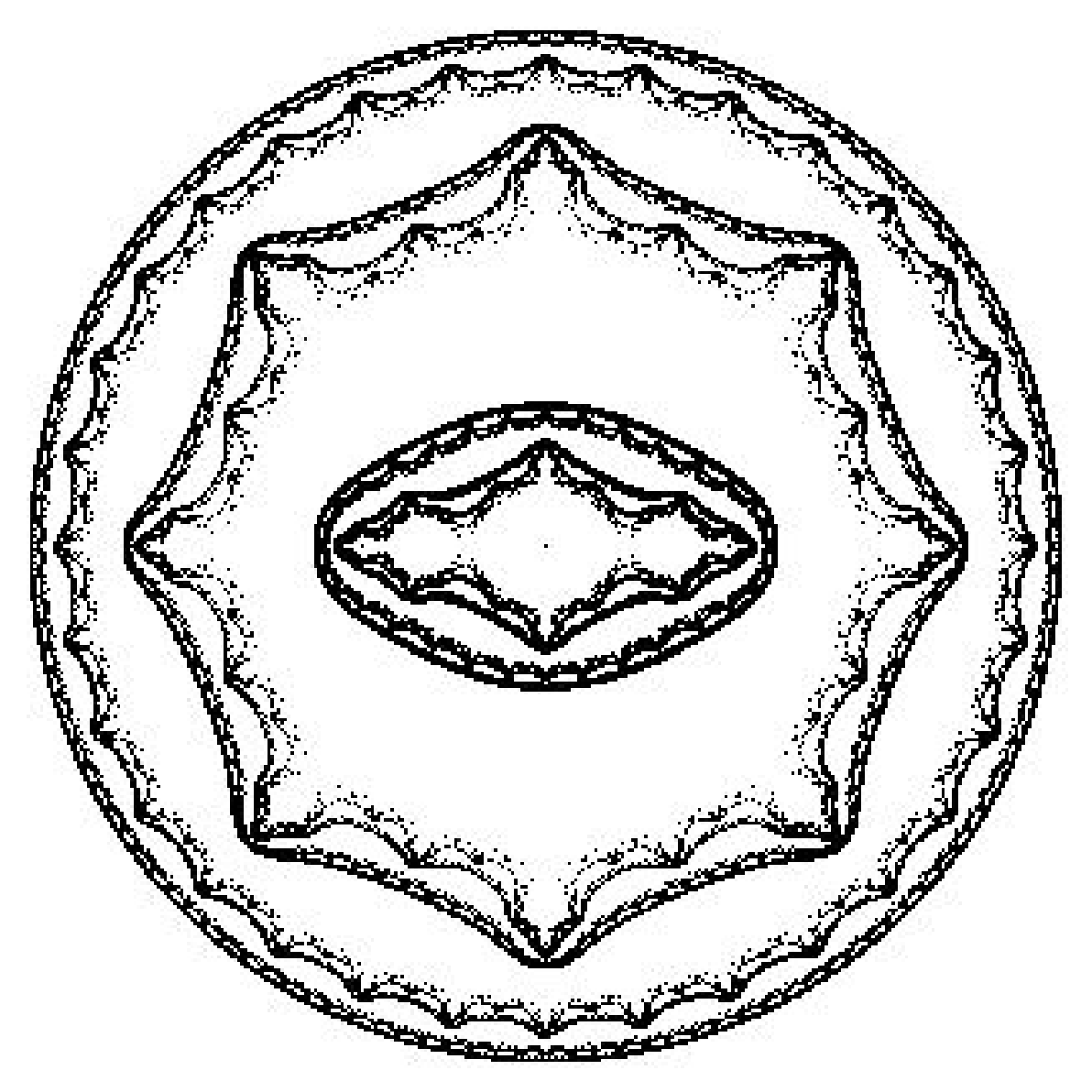}
\label{fig:dcjulia}
\end{figure}
\begin{figure}[htbp]
\caption{The graph of $T_{\infty ,\tau }$, where $\tau =\sum _{i=1}^{2}\frac{1}{2}\delta _{h_{i}}$ with the same 
$h_{i}$ as in Figure~\ref{fig:dcjulia}. 
$T_{\infty ,\tau }$ is continuous on $\CCI .$ The set of varying points of $T_{\infty, \tau }$ is 
equal to $J(G)$ in Figure~\ref{fig:dcjulia}. 
 A``devil's coliseum'' (A complex analogue of 
the devil's staircase). }
\ \ \ \ \ \ \ \ \ \ \ \ \ \ \ \ \ \ \ \ \ \ \ \ \ \ \ \ \ \ \ \ 
\ \ \ \ \ \ \ \ \ \ \ \ \ \ 
\includegraphics[width=4cm,width=4cm]{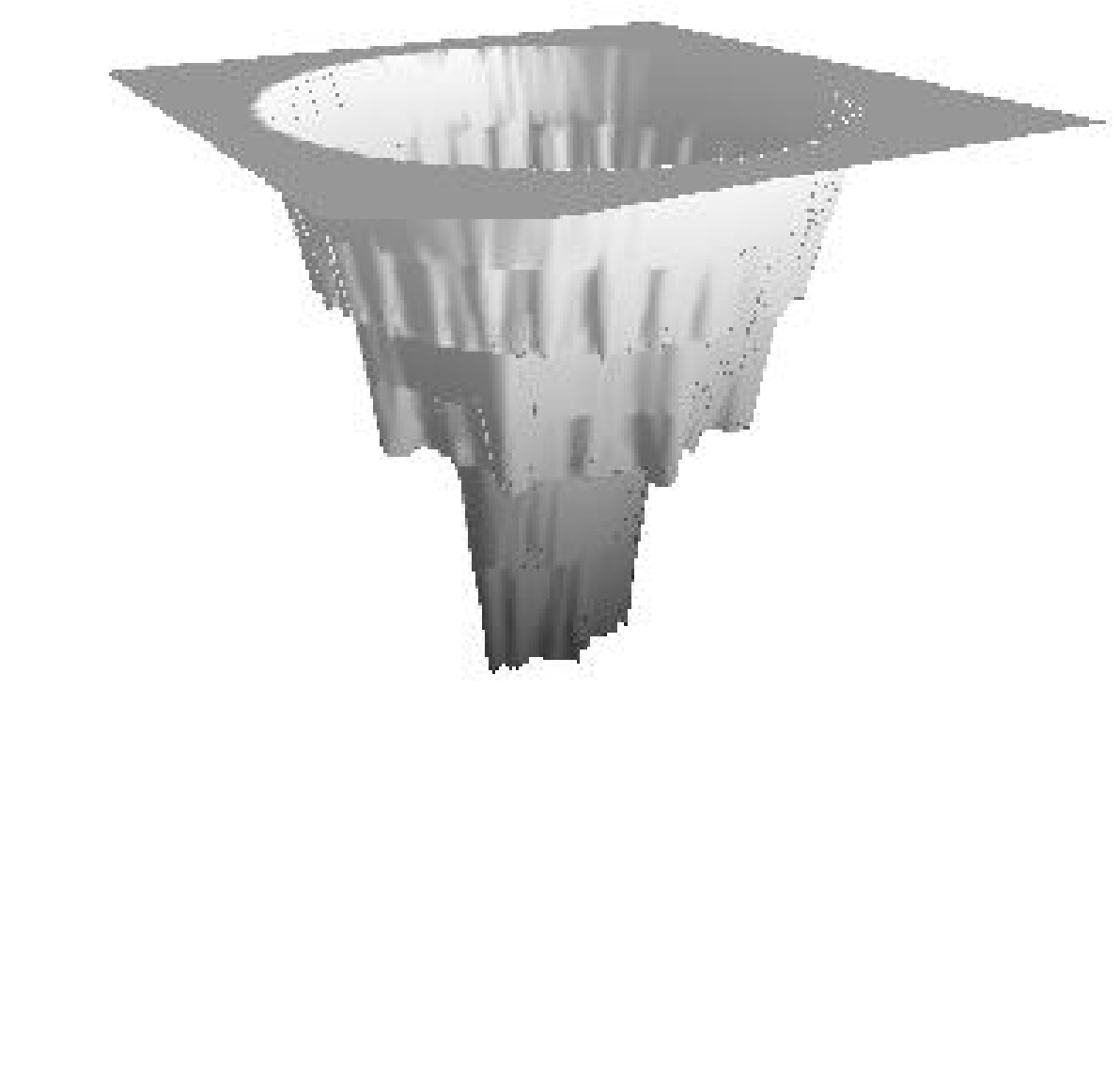}
\label{fig:dcgraphgrey2}
\end{figure}
\begin{figure}[htbp]
\caption{Figure \ref{fig:dcgraphgrey2} upside down. A ``fractal wedding cake''. }
 \ \ \ \ \ \ \ \ \ \ \ \ \ \ \ \ \ \ \ \ \ \ \ \ \ \ \ \ \ \ \ \ \ \ 
\ \ \ \ \ \ \ \ \ \  
\includegraphics[width=4.6cm,width=4.6cm]{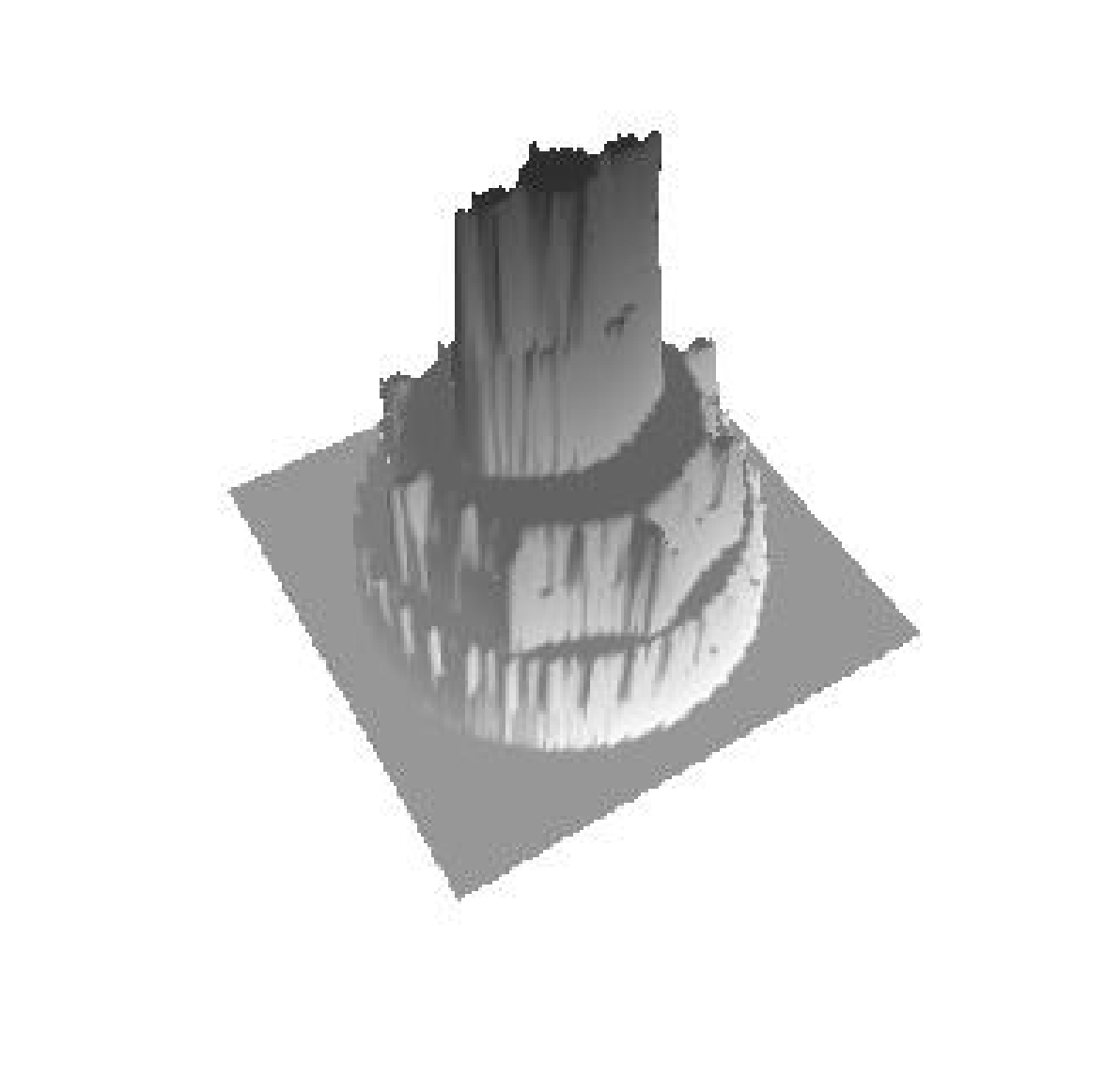}
\label{fig:dcgraphudgrey2}
\end{figure}
We now present a way to construct examples of $(h_{1},h_{2})\in {\cal P}^{2}$ 
such that $G=\langle h_{1},h_{2}\rangle $ is hyperbolic, $\bigcap _{j=1}^{2}h_{j}^{-1}(J(G))=\emptyset $, 
and $\hat{K}(G)\neq \emptyset .$ 
\begin{prop}
\label{p:hyphigh}
Let $g_{1},g_{2}\in {\cal P}$ be hyperbolic. 
Suppose that $(J(g_{1})\cup J(g_{2}))\cap (P(g_{1})\cup P(g_{2}))=\emptyset $, 
$K(g_{1})\subset \mbox{{\em int}}(K(g_{2}))$, and the union $A$ of attracting cycles 
of $g_{2}$ in $\CC $ is included in {\em int}$(K(g_{1})).$ Then, there exists an $m\in \NN $ 
such that for each $n\in \NN $ with $n\geq m$, 
setting $h_{i,n}=g_{i}^{n}$ and $G_{n}=\langle h_{1,n}, h_{2,n}\rangle $, we have that 
$G_{n}$ is hyperbolic, $h_{1,n}^{-1}(J(G_{n}))\cap h_{2,n}^{-1}(J(G_{n}))=\emptyset $, 
and $\emptyset \neq K(g_{1})\subset \hat{K}(G_{n}).$  
\end{prop}
\begin{proof}
Let $\epsilon >0$ be a number such that 
$B(J(g_{1})\cup J(g_{2}),2\epsilon )\cap B(P(g_{1})\cup P(g_{2}),2\epsilon )=\emptyset .$  
Let $m\in \NN $ be a number such that for each $n\in \NN $ with $n\geq m$, we have 
$g_{2}^{n}(K(g_{1}))\subset K(g_{1})$, 
$\bigcap _{i=1}^{2}g_{i}^{-n}(\overline{B(J(g_{1})\cup J(g_{2}),\epsilon )})
=\emptyset $, $\bigcup _{i=1}^{2}g_{i}^{-n}(B(J(g_{1})\cup J(g_{2}),\epsilon ))\subset 
B(J(g_{1})\cup J(g_{2}),\epsilon )$, and 
$\bigcup _{i=1}^{2}g_{i}^{n}(B(P(g_{1})\cup P(g_{2}),\epsilon ))\subset 
B(P(g_{1})\cup P(g_{2}),\epsilon ).$ 
Let $n\geq m.$ 
 Then for each $n\geq m$, $J(G_{n})\subset \overline{B(J(g_{1})\cup J(g_{2}),\epsilon )}.$ Hence 
$h_{i,n}^{-1}(J(G_{n}))\cap h_{2,n}^{-1}(J(G_{n}))=\emptyset .$ 
Since $\bigcup _{i=1}^{2}\mbox{CV}(h_{i,n})\subset P(g_{1})\cup P(g_{2})$, 
we obtain that $P(G_{n})= \overline{G_{n}^{\ast }(\bigcup _{i=1}^{2}\mbox{CV}(h_{i,n}))}
\subset \overline{B(P(g_{1})\cup P(g_{2}),\epsilon )}$, where 
CV$(\cdot )$ denotes the set of all critical values. Therefore 
$J(G_{n})\cap P(G_{n})=\emptyset .$ Thus $G_{n}$ is hyperbolic. 
Furthermore, 
$\emptyset \neq K(g_{1})\subset \hat{K}(G_{n}).$ 
Thus we have proved our proposition. 
\end{proof}
\begin{prop}
\label{p:hypdispur}
Let $m\in \NN $ and let $g=(g_{1},\ldots ,g_{m})\in {\cal P}^{m}.$ 
Let $G=\langle g_{1},\ldots ,g_{m}\rangle .$ 
Suppose that $g_{i}^{-1}(J(G))\cap g_{j}^{-1}(J(G))=\emptyset $ for each $(i,j)$ with $i\neq j$, 
that $G$ is hyperbolic, and that $\hat{K}(G)\neq \emptyset .$ Then, 
there exists a neighborhood $U$ of $g$ in ${\cal P}^{m}$ such that 
for each $h=(h_{1},\ldots ,h_{m})\in U$, 
setting $H=\langle h_{1},\ldots ,h_{m}\rangle $, 
we have that  $h_{i}^{-1}(J(H))\cap h_{j}^{-1}(J(H))=\emptyset $ for each $(i,j)$ with $i\neq j$, 
that $H$ is hyperbolic, and that $\hat{K}(H)\neq \emptyset .$ 
\end{prop}
\begin{proof}
By \cite[Theorem 2.4.1]{S1}, 
there exists a neighborhood $V$ of $g$ such that 
for each $h=(h_{1},\ldots ,h_{m})\in V$, 
setting $H=\langle h_{1},\ldots ,h_{m}\rangle $, 
we have that  $h_{i}^{-1}(J(H))\cap h_{j}^{-1}(J(H))=\emptyset $ for each $(i,j)$ with $i\neq j$, 
and that $H$ is hyperbolic. 
Since $\hat{K}(G)\neq \emptyset $, 
there exists a minimal set $L$ for $(G,\CCI )$ with $L\subset \hat{K}(G).$ 
By Theorem~\ref{t:mtauspec}-\ref{t:mtauspec11}, 
$L\subset A(G)\subset P(G).$ Since $G$ is hyperbolic, 
it follows that $L\subset \mbox{int}(\hat{K}(G)).$ 
Let $\epsilon >0$ be a number such that 
$\overline{B(L, 2\epsilon )}\subset \mbox{int}(\hat{K}(G)).$ 
By Lemma~\ref{l:shatt1},  
there exists an $l\in \NN $ such that 
for each $(i_{1},\ldots ,i_{l})\in \{ 1,\ldots ,m\} ^{l}$, 
$g_{i_{l}}\cdots g_{i_{1}}(B(L,2\epsilon ))\subset B(L,\epsilon ).$ 
Then there exists a neighborhood $W$ of $g$ in ${\cal P}^{m}$ such that 
for each $(i_{1},\ldots ,i_{l})\in \{ 1,\ldots ,m\} ^{l}$ and 
for each $h=(h_{1},\ldots ,h_{m})\in W$, 
$h_{i_{l}}\cdots h_{i_{1}}(B(L,2\epsilon ))\subset B(L,2\epsilon ).$ 
Hence for each $h\in W$, 
$B(L,2\epsilon )\subset \hat{K}(\langle h_{1},\ldots ,h_{m}\rangle ).$ 
Let $U=V\cap W.$ 
Then this $U$ is the desired neighborhood of $g.$   
\end{proof}
We now give an example to which we can apply Lemma~\ref{l:nondiff1}-\ref{l:nondiff1-2}.
\begin{prop}
\label{p:nondiffallj}
Let $(g_{1},g_{2})\in {\cal P}^{2}$ and let $(p_{1},p_{2})\in {\cal W}_{2}.$ 
For each $n\in \NN $, we set 
$h_{1,n}:= g_{1}^{n}, h_{2,n}:=g_{2}^{n}$, $G_{n}:=\langle h_{1,n},h_{2,n}\rangle $, and 
$\tau _{n}:= \sum _{j=1}^{2}p_{j}\delta _{h_{j,n}}.$ 
Suppose that $\bigcap _{j=1}^{2}g_{j}^{-1}(J(G_{1}))=\emptyset $, $G_{1}$ is hyperbolic  
 and 
$\hat{K}(G_{1})\neq \emptyset .$ 
Then, there exists an $m\in \NN $ such that 
for each $n\in \NN $ with $n\geq m$, 
{\em (1)} $G_{n}$ is hyperbolic, {\em (2)} $\bigcap _{j=1}^{2}h_{j,n}^{-1}(J(G_{n}))=\emptyset $, 
{\em (3)} $\hat{K}(G_{n})\neq \emptyset $, 
{\em (4)} $(\emLSfc )_{nc}\neq \emptyset $, 
{\em (5)} for each $j=1,2,$, 
$1<p_{j}\min \{ \| h_{j,n}'(z)\| _{s}\mid 
z\in h_{j,n}^{-1}(J(G_{n}))\} $, 
and 
{\em (6)} for each $z_{0}\in J(G_{n})$ and for each 
$\varphi \in (\emLSfc )_{nc}$, 
$\limsup _{z\rightarrow z_{0}}\frac{|\varphi (z)-\varphi (z_{0})|}{d(z,z_{0})}=\infty $ and 
$\varphi $ is not differentiable at $z_{0}.$ 
\end{prop}
\begin{proof}
Since $G_{1}$ is hyperbolic, by \cite[Theorem 2.17]{S4}, 
there exists an $m\in \NN $ such that for each $n\in \NN $ with $n\geq m$, 
$1<p_{j}\min \{ \| h_{j,n}'(z)\| _{s}\mid z\in h_{j,n}^{-1}(J(G_{n}))\} .$ 
By Lemma~\ref{l:nondiff1}-\ref{l:nondiff1-2}, our proposition holds. 
\end{proof}
\begin{rem}
\label{r:manydhk}
Combining Proposition~\ref{semihyposcexprop}, Proposition~\ref{p:hyphigh}, 
 Proposition~\ref{p:hypdispur}, Proposition~\ref{p:hyppjke}, and Remark~\ref{r:pnckem},  
 we obtain many examples to which 
we can apply Theorem~\ref{t:mtauspec}, Lemma~\ref{l:lsncnonc}, Proposition~\ref{p:erghol}, 
Theorem~\ref{t:hnondiff}, Theorem~\ref{t:hdiffornd}, Corollary~\ref{c:hypdisjc}, 
and Theorem~\ref{t:hholder}. 
Moreover, combining Proposition~\ref{semihyposcexprop}, Proposition~\ref{p:hyphigh}, 
 Proposition~\ref{p:hypdispur} and Proposition~\ref{p:nondiffallj}, 
 we obtain many examples to which we can apply Lemma~\ref{l:nondiff1}-\ref{l:nondiff1-2}. 
\end{rem}
We now give an example of $\tau \in {\frak M}_{1,c}({\cal P})$ such that $J_{\ker }(G_{\tau })=\emptyset $ and such that there exists a minimal set $L\in \Min(G_{\tau },\CCI )$ with 
$L\cap J(G_{\tau })\neq \emptyset .$ 
\begin{ex}
\label{r:parajkere}
Let $f_{1}\in {\cal P}$ and suppose that $f_{1}$ has a parabolic cycle $\alpha .$ 
Let $b$ be a point of the immediate basin of $\alpha $. 
Let $d\in \NN $ with $d\geq 2$ such that $(\deg (f_{1}),d)\neq (2,2).$ 
Then by Proposition~\ref{semihyposcexprop}, 
there exists a $c>0$ such that for each $a\in \CC $ with $0<|a|<c$, 
setting $f_{2}:= a(z-b)^{d}+b$ and $G=\langle f_{1},f_{2}\rangle $, 
we have $f_{1}^{-1}(J(G))\cap f_{2}^{-1}(J(G))=\emptyset $ and 
$G(K(f_{1}))\subset K(f_{1})\subset \mbox{int}(K(f_{2})).$ 
Let $p=(p_{1},p_{2})\in {\cal W}_{2}$ and let 
$\tau =\sum _{i=1}^{2}p_{i}\delta _{f_{i}}.$ 
Then by Lemma~\ref{l:disjker}, 
$J_{\ker }(G_{\tau })=J_{\ker }(G)=\emptyset .$ 
Since $G(K(f_{1}))\subset K(f_{1})\subset \mbox{int}(K(f_{2}))$, 
there exists a minimal set $L$ for $(G_{\tau },\CCI )$ such that 
$L\subset K(f_{1}).$ Since $b$ belongs to the immediate basin of $\alpha $ for $f_{1}$, 
it  follows that $\alpha \subset L.$ In particular, 
$L\cap J(G_{\tau })\neq \emptyset .$   
\end{ex} 
We now give an example of small perturbation of a single map. 
\begin{ex}
\label{ex:FS}
Let $\Bbb{D}:= \{ z\in \CC \mid |z|<1\} .$ 
Let $R:\CCI \times \Bbb{D}\rightarrow \CCI $ be a holomorphic map 
such that for each $z\in \CCI $, 
$c\mapsto R(z,c)$ is non-constant on $\Bbb{D}.$ 
We set $R_{c}(z):= R(z,c)$ for each $(z,c)\in \CCI \times \Bbb{D}.$ 
Let $m\in \NN $ and suppose that $R_{0}$ has exactly $m$ attracting cycles 
$\alpha _{1},\ldots ,\alpha _{m}.$ For each $j$,  
let $A_{j}$ be the immediate basin of $\alpha _{j}$ for $R_{0}.$  Then by \cite[Theorem 0.1]{FS} and Theorem~\ref{t:mtauspec}, 
there exists a $\delta _{0}>0$ such that for each $0<\delta <\delta _{0}$,    
denoting by $\tau _{\delta }$ the normalized $2$-dimensional Lebesgue measure 
on $\overline{D(0,\delta )}$, we have 
(1) $\tau _{\delta }$ is mean stable, (2) $J_{\ker }(G_{\tau _{\delta }})=\emptyset $, 
(3) $\sharp \Min(G_{\tau _{\delta }},\CCI )=m$, 
(4) for each $L\in \Min(G_{\tau _{\delta }},\CCI )$,  
there exists a $j$ such that $L\subset A_{j}$, and 
(5) for each $L\in \Min(G_{\tau _{\delta }},\CCI )$, $r_{L}:= \dim _{\CC }(\LSfl )$ is equal to 
the period of $\alpha _{j}$ for $R_{0}.$    
\end{ex} 
We now give an example of higher dimensional random complex dynamics to which 
we can apply Theorem~\ref{kerJthm1}. 
\begin{ex}
\label{ex:hdjker}
Let $h\in \mbox{NHM}(\CC \Bbb{P}^{n}).$ Suppose that 
int$(J(h))=\emptyset $ and there exist finitely many attracting periodic cycles 
$\alpha _{1},\ldots ,\alpha _{m}$ such that for every $z\in F(h)$, $d(h^{n}(z),\bigcup _{j=1}^{m}\alpha _{j})
\rightarrow 0$ as $n\rightarrow \infty .$ 
Then, there exists a compact neighborhood $\G $ of $h$ 
in  $\mbox{NHM}(\CC \Bbb{P}^{n})$ such that $\G $ is mean stable, 
such that $J_{\ker }(\langle \G \rangle )=\emptyset $, 
and such that for any $\tau \in {\frak M}_{1}( \mbox{NHM}(\CC \Bbb{P}^{n}))$ with 
$\G _{\tau }=\G $, 
Leb$_{2n}(J_{\g })=0$ for $\tilde{\tau }$-a.e. 
$\g \in (\mbox{NHM}(\CC \Bbb{P}^{n}))^{\NN }.$  
For, if $\G $ is small enough, then there exists a neighborhood $U$ of 
$\bigcup _{j=1}^{m}\alpha _{j}$ such that $\overline{\langle \G \rangle (U)} \subset U\subset \overline{U}
\subset F(\langle \G \rangle ).$  
Moreover, for each $z\in \CC \Bbb{P}^{n}$, there exists a $g\in \G $ such that 
$g(z)\in F(h).$ Thus $\G $ is mean stable and $J_{\ker }(\langle \G \rangle )=\emptyset .$ 
By Theorem~\ref{kerJthm1}, it follows that for each $\tau \in {\frak M}_{1}(\mbox{NHM}(\CC \Bbb{P}^{n}))$ 
with $\G _{\tau }=\G $, Leb$_{2n}(J_{\g })=0$ for $\tilde{\tau }$-a.e. 
$\g \in  (\mbox{NHM}(\CC \Bbb{P}^{n}))^{\NN }.$ 
\end{ex} 
We now give an example of $\tau $ with $J_{\ker }(G_{\tau })\neq \emptyset $ 
to which we can apply Theorem~\ref{t:dimjpt0jkn}. 
\begin{ex}
\label{ex:hypcnc}
Let $0<a<1$ and 
let $g_{1}(z)=z^{2}.$ Let $g_{2}\in {\cal P}$ be such that 
$J(g_{2})=\{ z\in \CC \mid |z+a|=|1+a|\} $, $g_{2}(1)=1$ and $g_{2}([1,\infty ))\subset [1,\infty ).$  
Let $l\in \NN $ with $l\geq 2$ and let $\alpha \subset J(g_{2})$ be a 
repelling cycle of $g_{2}$ of period $l.$ Then there exists an $m\in \NN $ such that 
$P(\langle g_{1}^{m},g_{2}^{m}\rangle )\subset F(\langle g_{1}^{m},g_{2}^{m}\rangle )$ and 
$g_{1}^{m}(\alpha )\subset F_{\infty }(\langle g_{1},g_{2}\rangle )\subset 
F_{\infty }(\langle g_{1}^{m},g_{2}^{m}\rangle ).$ 
Let $h_{1}:=g_{1}^{m}$ and $h_{2}:= g_{2}^{m}.$ 
Let $(p_{1},p_{2})\in {\cal W}_{2}$ and let $\tau := \sum _{i=1}^{2}p_{i}\delta _{h_{i}}.$ 
Then we have $1\in J_{\ker }(G_{\tau })\cap \partial F_{\infty }(G_{\tau })$, 
$G_{\tau }$ is hyperbolic, and $\alpha \subset F_{pt}^{0}(\tau )$ (see Lemma~\ref{FJmeaslem2}).  
Thus $T_{\infty ,\tau }$ is discontinuous at $1$, $1\in J_{pt}^{0}(\tau )$, and 
$T_{\infty ,\tau }$ is continuous at each point of $\alpha $ (see Lemma~\ref{lem:fpt0tconti}). 
Moreover, by Theorem~\ref{t:dimjpt0jkn}, we have 
$\dim _{H}(J_{pt}^{0}(\tau ))\leq \mbox{MHD}(\tau )<2$, 
$J_{meas}(\tau )={\frak M}_{1}(\CCI )$, 
and $J_{pt}(\tau )=J(G_{\tau }).$ 
\end{ex}
We now give an example of $\tau $ with $J_{\ker }(G_{\tau })\neq \emptyset $ to which 
we can apply Theorem~\ref{t:jkucuh}.
\begin{ex}
\label{ex:pjne}
Let $g_{1}(z)=z^{2}-1.$ Let 
$a=\frac{1+\sqrt{5}}{2}.$ Then 
$g_{1}(a)=a\in J(g_{1}).$ Moreover, $-1$ is a superattracting fixed point of 
$g_{1}^{2}.$  
Let $b:=\frac{a+(-1)}{2}.$ 
Then it is easy to see that $b$ belongs to 
the immediate basin $A_{1}$ of $0$ for the dynamics of $g_{1}^{2}.$ 
Let $g_{2}\in {\cal P}$ be such that 
$J(g_{2})=\{ z\in \CC \mid |z-b|=a-b\} $, $g_{2}(a)=a$ and 
$g_{2}(-1)=-1.$ 
Let $\epsilon >0$ be a small number so that 
$b-\epsilon $ belongs to $A_{1}.$  
Let $c=b-\epsilon .$ 
Let $g_{3}\in {\cal P}$ be such that 
$J(g_{3})=\{ z\in \CC \mid |z-c|=a-c\} $ and $g_{3}(a)=a.$ 
Then $b$ is an attracting fixed point of $g_{2}$, 
$c$ is an attracting fixed point of $g_{3}$, 
$\{ b,c\} $ is included in $A_{1}$, 
$\{ 0,c\} $ is included in the immediate basin $A_{2}$ of $b$ for $g_{2}$, 
and $\{ 0,b,-1\}$ is included in the immediate basin $A_{3}$ of $c$ for $g_{3}.$ 
 
Let $m\in \NN $ be sufficiently large and let 
$h_{1}=g_{1}^{2m}$, $h_{2}=g_{2}^{m}, $ and $h_{3}=g_{3}^{m}.$  
Let $G=\langle h_{1},h_{2},h_{3}\rangle .$ 
Then  $UH(G)\cap J(G)=P(G)\cap J(G)=\{ -1\} $, $-1\not\in J_{\ker }(G)$ 
and $a\in J_{\ker }(G).$ 
Let $(p_{1},p_{2},p_{3})\in {\cal W}_{3}$ and 
let $\tau =\sum _{i=1}^{3}p_{i}\delta _{h_{i}}.$  
By Theorem~\ref{t:jkucuh}, we obtain that 
(1) for $\tilde{\tau }$-a.e. $\gamma \in {\cal P}^{\NN }$,  
$\mbox{Leb}_{2}(J_{\gamma })
=\mbox{Leb}_{2}(\hat{J}_{\gamma , \Gamma _{\tau }})=0$, 
(2) $\mbox{Leb}_{2}(J_{pt}^{0}(\tau ))=0$, and 
(3) for Leb$_{2}$-a.e. $y\in \CCI $,  
$T_{\infty ,\tau }$ is continuous at $y.$ 
Moreover, since $-1$ is a superattracting fixed point of $h_{1}$ and 
$-1\in J(h_{2})$, 
setting $\rho =(h_{1},h_{1},h_{1},\ldots )\in X_{\tau }$, 
we have $-1 \in \mbox{int}(\hat{J}_{\rho , \Gamma _{\tau }})$ (see \cite[Theorem 1.6(2)]{S7}). 
Therefore for each $\beta \in \bigcup _{n\in \NN }\sigma ^{-n}(\rho )$, 
int$(\hat{J}_{\beta ,\Gamma _{\tau }})\neq \emptyset .$ Note that 
$\bigcup _{n\in \NN }\sigma ^{-n}(\rho )$ is dense in $X_{\tau }.$ 
Thus, (I) for $\tilde{\tau }$-a.e. $\gamma \in X_{\tau }$,  
$\mbox{Leb}_{2}(\hat{J}_{\gamma ,\Gamma _{\tau }})=0$, and 
(II) there exists a dense subset $B$ of $X_{\tau }$ such that for each 
$\beta \in B$, $\mbox{int}(\hat{J}_{\beta ,\Gamma _{\tau }})\neq \emptyset .$ 
\end{ex}

\end{document}